\documentclass[12pt,a4paper,openright,final]{book}
\usepackage{times}        
\usepackage{amssymb}
\usepackage{amsthm}                       
\usepackage{amsmath}                      
\usepackage{amscd}                        
\usepackage{makeidx}                      
\usepackage{fancyhdr}                     
\usepackage{pb-diagram}                   
\usepackage{multirow}
\usepackage[all]{xy}
\usepackage[ansinew]{inputenc}
\usepackage{titlesec}
\usepackage[bookmarks=true,colorlinks=false,naturalnames=true]{hyperref}

\let\oldmarginpar\marginpar
\renewcommand\marginpar[1]{\-\oldmarginpar[\raggedleft\footnotesize #1]%
{\raggedright\footnotesize #1}}

\reversemarginpar

\theoremstyle{plain}
\newtheorem{theorem}                                        {Theorem}      [section]
\newtheorem{proposition}             [theorem]              {Proposition}
\newtheorem{corollary}               [theorem]              {Corollary}
\newtheorem{lemma}                   [theorem]              {Lemma}

\newtheorem{theoremwithoutsection}                          {Theorem}      [chapter]
\newtheorem{corollarywithoutsection} [theoremwithoutsection]{Corollary}

\theoremstyle{definition}
\newtheorem{example}                 [theorem]  {Example}
\newtheorem{remark}                  [theorem]  {Remark}
\newtheorem{definition}              [theorem]  {Definition}

\newtheorem{remarkwithoutsection}    [theoremwithoutsection]{Remark}

\numberwithin{equation}{section}

\def \theoremwithoutnumber#1#2 {\vskip .25cm\noindent{\bf Theorem #1.\ }{\it #2}\vskip .25cm}
\def \corollarywithoutnumber#1#2 {\vskip .25cm\noindent{\bf Corollary #1.\ }{\it #2}\vskip .25cm}
\def \propositionwithoutnumber#1#2 {\vskip .25cm\noindent{\bf Proposition #1.\ }{\it #2}\vskip .25cm}
\def \lemmawithoutnumber#1#2 {\vskip .25cm\noindent{\bf Lemma #1.\ }{\it #2}\vskip .25cm}

\renewcommand{\thechapter}{\Roman{chapter}}

\hypersetup{ pdfauthor = {Bruno Ascenso Simões}, pdftitle =
{phDThesis}, pdfcreator = {pdfLatex}}

\def \g{\mathfrak{g}}
\def \gl{\mathfrak{gl}}

\def \so{\mathfrak{so}}
\def \o{\mathfrak{o}}

\def \m{\mathfrak{m}}
\def \u{\mathfrak{u}}

\def \O{\text{\bf O}}
\def \SO{\text{\bf SO}}
\def \U{\text{\bf U}}
\def \GL{\text{\bf GL}}

\def \G{\text{\bf G}}
\def \K{\text{\bf K}}

\def \cn{\mathbb{C}}
\def \rn{\mathbb{R}}

\def \V{\mathcal{V}}
\def \H{\mathcal{H}}
\def \J{\mathcal{J}}
\def \V{\mathcal{V}}
\def \openU{\mathcal{U}}
\def \L{\mathcal{L}}
\def \C{\mathcal{C}}
\def \N{\mathcal{N}}

\def \d{\mathrm{d}}

\def \equi{\Leftrightarrow}
\def \impl{\Rightarrow}

\DeclareMathOperator{\Hom}{Hom} \DeclareMathOperator{\Id}{Id}
\DeclareMathOperator{\trace}{trace}
\DeclareMathOperator{\proj}{proj} \DeclareMathOperator{\Real}{Re}
\DeclareMathOperator{\Imag}{Im} 
 \DeclareMathOperator{\grad}{grad}
\DeclareMathOperator{\wordspan}{span}

\def \ddtatzero{\left.\frac{\partial}{\partial t}\right|_{t=0}}

\makeindex

\begin{document}

\pagestyle{empty} 

\begin{titlepage}
\centering\Huge
Twistorial constructions of harmonic morphisms and Jacobi fields\\
\Large
\vspace{1.8cm}
Bruno Manuel Ascenso da Silva Simões\\
\large
\vspace{3.7cm}
Submitted in accordance with the requirements for the degree of Doctor of Philosophy\\
\vspace{1cm}
\Large
The University of Leeds \\
Department of Pure Mathematics \\
\vspace{0.8cm}
September 2007\\

\large \vspace{2.0cm}
The candidate confirms that the work submitted is his own and that appropriate
credit has been given where reference has been made to the work of others.
This copy has been supplied on the understanding that it is copyright material
and that no quotation from the thesis may be published without proper
acknowledgement.
\normalsize
\end{titlepage}


\cleardoublepage
\vspace*{\stretch{1}}
\begin{flushright}
\Large{\textit{To my wife Joana}}
\end{flushright}
\vspace*{\stretch{2}} 
\newpage

\vspace*{\stretch{1}}
\parbox{15cm}
{\textit{It is not knowledge, but the act of learning,}\\
\textit{ not possession but the act of getting there,\\ which grants the greatest enjoyment.}\\
Gauss (1808)\\~\\~\\
\textit{O Binómio de Newton é tão belo como a Vénus de Milo.}\\
\textit{O que há é pouca gente para dar por isso.}\\
Álvaro de Campos (Fernando Pessoa) (1928)
}
\vspace*{\stretch{2}}


\chapter*{Acknowledgements}\label{Chapter:Acknowledgements}


\pagestyle{fancy}
\fancyhf{}                                                  
\fancyhead[LE,RO]{\thepage}
\renewcommand{\headrulewidth}{0 pt}
\pagenumbering{roman}
\setcounter{page}{5}
\pagestyle{fancyplain}


First of all, I am very grateful to my supervisor Professor John C. Wood for his infinite patience during these years.
Without his compelling enthusiasm for my ideas throughout my stay in Leeds, this work would have never come to fruition.

To my colleague and friend Martin Svensson, I would like to dedicate as many lines in this work as the hours we spent debating ideas, mathematical and otherwise.

A special thanks to Professor Maria João Ferreira for her interest in my work and for all the support she has always
given me.

I thank all my friends in Leeds without whom I would not have enjoyed my stay half as much.

I thank my family in Lisbon for their boundless care and for making the distance between us seem so short.

Finally, I express my gratitude to my wife Joana, for the joy and the colour she brings to my life.

This work was supported by a grant from Fundação para a Ciência e Tecnologia (SFRH/BD/13633/2003),
which I wish to acknowledge, especially for their efficiency in dealing with any matters I raised with them.


\pdfbookmark[0]{Acknowledgements}{Acknowledgements}

\chapter*{Abstract}\label{Abstract}


\emph{Harmonic maps} are mappings between Riemannian manifolds which extremize a certain natural energy functional generalizing the Dirichlet integral to the setting of Riemannian manifolds. A \emph{harmonic morphism} is a map between Riemannian manifolds which pulls back locally defined \emph{harmonic functions} on the codomain to locally defined harmonic functions on the domain. Harmonic morphisms can be geometrically characterized as harmonic maps which also satisfy the partial conformality condition of \emph{horizontal (weak) conformality} (\cite{Fuglede:81}, \cite{Ishihara:79}). Twistor methods provide a powerful tool in the study of harmonic maps and harmonic morphisms. Indeed, their use has enabled us to produce a variety of examples of harmonic morphisms defined on $4$-dimensional manifolds, and a complete classification in some cases (\cite{BairdWood:03}, \cite{Ville:03}, \cite{Wood:92}). In the first part of this work, we generalize those constructions to obtain harmonic morphisms from higher-dimensional manifolds.

The \emph{infinitesimal deformations} of harmonic maps are called \emph{Jacobi fields}. They satisfy a system of partial differential equations given by the linearization of the equations for harmonic maps. The use of twistor methods in the study of Jacobi fields has proved quite fruitful, leading to a series of results (\cite{LemaireWood:96}, \cite{MontielUrbano:97}, \cite{Wood:02}). In \cite{LemaireWood:02} and \cite{LemaireWood:07} several properties of Jacobi fields along harmonic maps from the $2$-sphere to the complex projective plane and to the $4$-sphere are obtained by carefully studying the twistorial construction of those harmonic maps. In particular, relating the infinitesimal
deformations of the harmonic maps to those of the holomorphic data describing them. In the second part of this work we give a general treatment of these relations between Jacobi fields and variations in the \emph{twistor space}, obtaining \emph{first-order} analogues of twistorial constructions.



\pdfbookmark[0]{Abstract}{Abstract}
\cleardoublepage 


\pagenumbering{arabic} 
\pdfbookmark[0]{Table of Contents}{contents}
\tableofcontents

\chapter*{Introduction}\label{Chapter:Introduction}

\addcontentsline{toc}{chapter}{\numberline{}Introduction}


\emph{Harmonic maps} are mappings between Riemannian manifolds which extremize a certain natural energy functional generalizing the Dirichlet integral to the setting of Riemannian manifolds. A bibliography can be found in \cite{BurstallLemaireRawnsley:Online} and for some useful summaries on this topic, see \cite{EellsLemaire:78}, \cite{EellsLemaire:88}. \textit{Inter alia}, harmonic maps include isometries (distance-preserving maps), harmonic functions (solutions to Laplace's equation on a Riemannian manifold), geodesics, and holomorphic (complex analytic) maps between suitable complex manifolds.

\emph{Harmonic morphisms} are maps between Riemannian manifolds which pull back local harmonic functions on the codomain to local harmonic functions on the domain. Geometrically, harmonic morphisms are characterized as harmonic maps which satisfy the partial conformality condition of \emph{horizontal (weak) conformality} (\cite{Fuglede:81}, \cite{Ishihara:79}). Together, these two conditions amount to an over-determined, non-linear system of partial differential equations. It is thus not surprising that the question of the existence of harmonic morphisms is, in general, very difficult to answer.

Nevertheless, the study of harmonic morphisms from $4$-dimensional manifolds to Riemann surfaces has been greatly advanced with the aid of \emph{twistor} methods. In some cases, even a complete classification of maps has been found, see
\cite{BairdWood:03}, \cite{Ville:03}, \cite{Wood:92}.

To generalize these results notice that, when the domain is a $4$-dimensional Einstein manifold and the codomain is a surface, a harmonic morphism can be characterized as a map which is holomorphic and has \emph{superminimal} fibres with respect to some integrable Hermitian structure on its domain. With this observation in mind, in Chapter \ref{Chapter:HarmonicMorphismsAndTwistorSpaces}, we further the application of twistor methods to the construction of \emph{holomorphic maps with superminimal fibres}, giving large families of superminimal (and so minimal) submanifolds and extend some of the results from the $4$-dimensional case to arbitrary even dimensions.

In Chapter \ref{Chapter:TwistorSpaces}, we define a new notion of \emph{twistor space} $\Sigma^+V$ for a vector bundle $V$ over a manifold $M$. This generalizes the usual notion of twistor space over a Riemannian manifold and will play a crucial role in all subsequent chapters. Using the Koszul-Malgrange Theorem (\cite{KoszulMalgrange:58}), we prove that a holomorphic structure can be introduced on $\Sigma^+V$, provided a certain condition on the curvature of $V$ is satisfied (Theorem \ref{Theorem:KoszulMalgrangeComplexStructureOnSigmaPlusV}). Moreover, we shall prove a parametric version of such a result and use it later on when studying variations of harmonic maps.

When $N^{2n}$ is a Riemannian manifold, its twistor space $\Sigma^+N$ is equipped with two almost complex structures $\J^1$ and $\J^2$, the second of which is never integrable (see, \textit{e.g.}, \cite{BurstallRawnsley:90}). It is well known (\cite{DavidovSergeev:93}, \cite{Salamon:85}) that twistor methods relate \emph{conformal} harmonic maps $\varphi:M^2\to{N}^{2n}$ from Riemann surfaces with holomorphic maps to the twistor space $\Sigma^+N$. In higher dimensions, \textit{conformal} and \textit{harmonic} must be replaced by \textit{pluriconformal} and either \textit{pluriharmonic} or \textit{$(1,1)$-geodesic}. We investigate the properties of $(1,1)$\emph{-geodesic maps} $\varphi:M^{2m}\to N^{2n}$ and their relations with twistor lifts.

In Chapter \ref{Chapter:HarmonicMapsAndTwistorSpaces}, we prove the following result (\cite{SimoesSvensson:06}, see also Remark 5.7 of \cite{Rawnsley:85}):
\theoremwithoutnumber{\ref{Theorem:NewTheorem35}}{Let $(M,g,J^M)$ be a $(1,2)$-symplectic manifold and $N^{2n}$ an oriented even-dimensional Riemannian manifold. Consider a $(J^M,\J^2)$-holomorphic map $\psi:M^{2m}\to\Sigma^+N^{2n}$. Then, the projected map $\varphi:M^{2m}\to N^{2n}$, $\varphi=\pi\circ\psi$, is pluriconformal and $(1,1)$-geodesic.}

Under some further conditions on the curvature of the normal bundle $V=\big(\d\varphi(TM)\big)^\perp$, we also have a converse to this theorem (Theorem \ref{Theorem:NewTheorem4.1}). We show that the condition on the curvature is always satisfied when $N$ is a Riemannian symmetric space of Euclidean, compact or non-compact type.

In the same way that harmonicity of a map $\varphi:M\to N$ can be interpreted as $\J^2$-holomorphicity of its twistor lift $\psi:M\to\Sigma^+N$, we relate pluriconformality and real-isotropy of $\varphi$ with, respectively, $\H$ and $\J^1$-holomorphicity of its twistor lift. This is done in Section \ref{Section:SectionInChapter:TwistorSpacesAndHarmonicMaps:TheHAndJ1CasesRealIsotropyAndTotallyUmbilicMaps} and will also play a crucial role in Chapter \ref{Chapter:JacobiVectorFieldsAndTwistorSpaces}.

In Chapter \ref{Chapter:HarmonicMorphismsAndTwistorSpaces}, we apply these results to the study of holomorphic maps with
superminimal fibres. We provide a twistor method for the construction of such maps, generalizing the methods in the $4$-dimensional case, see \cite{BairdWood:03}, \cite{Wood:92}. More precisely, we prove the following result (\cite{SimoesSvensson:06}):
\theoremwithoutnumber{\ref{Theorem:ConstructionOfHarmonicMorphismsWithSuperminimalFibresBruSve}}{
Let $N^{2n}$ and $P^{2p}$ be complex manifolds and $M^{2(n+p)}$ an oriented Riemannian manifold. Denote by $\pi_1:N\times P\to N$ the projection onto the first factor. Assume that we have a $\J^1$-holomorphic map
\[H:W\subseteq N\times P\to\Sigma^+M,\quad(z,\xi)\mapsto H(z,\xi)\text{,}\]
defined on some open subset $W$, such that \newline\indent\textup{(i)} for each $z$, the map $P\ni\xi\to{H}(z,\xi)\in\Sigma^+M$ is horizontal on its domain;
\newline\indent\textup{(ii)} the map $h=\pi\circ H$ is a diffeomorphism onto its image.
\newline\noindent Then, the map $\pi_1\circ h^{-1}:h(W)\subseteq M\to N$ is holomorphic and has superminimal fibres with respect to the Hermitian structure on $h(W)$ defined by the section $H\circ{h}^{-1}$. In particular, when $n=1$, $\pi_1\circ h^{-1}$ is a harmonic morphism.}

We also prove a converse to this result (Theorem \ref{Theorem:ConverseToThe:Theorem:ConstructionOfHarmonicMorphismsWithSuperminimalFibresBruSve}).

In the last chapter, we apply and extend the results in Chapters \ref{Chapter:TwistorSpaces}--\ref{Chapter:HarmonicMapsAndTwistorSpaces} to the study of \emph{infinitesimal variations} of harmonic maps. A \emph{Jacobi field} $v$ along a harmonic map $\varphi_0:M\to N$ is characterized as being a solution to the linear equation $J_{\varphi_0}(v)=0$. A Jacobi field $v$ along $\varphi_0$ is said to be integrable if there is a smooth variation $\varphi$ of $\varphi_0$ tangent to $v$ (\textit{i.e.}, $\varphi:I\times M\to N$, $(t,x)\to\varphi(t,x)=\varphi_t(x)$, with $\varphi(0)=\varphi_0$ and $\ddtatzero\varphi=v$) by harmonic maps (\textit{i.e.}, $\varphi_t:M\to N$ is harmonic for all $t\in{I}\subseteq\rn$). For two real-analytic manifolds, all Jacobi fields are integrable if and only if the space of harmonic maps is a manifold whose tangent bundle is given by the Jacobi fields (\cite{AdamsSimon:88}, see \cite{LemaireWood:02}). L.~Lemaire and J.~C. Wood studied the integrability of harmonic maps from the $2$-sphere to the complex projective plane and to the $4$-sphere, obtaining a positive answer for the first case \cite{LemaireWood:02} and negative for second \cite{LemaireWood:07}. This was achieved by carefully studying the twistorial construction of those harmonic maps; in particular, relating the infinitesimal deformations of the harmonic maps to those of the holomorphic data describing them. For maps $\psi:I\times (M,J^M)\to (N,h,J^N)$, we introduce the concept of maps \emph{holomorphic to first order} as those for which simultaneously $\d\psi_0(J^MX)=J^N\d\psi_0(X)$ and $\nabla_{\ddtatzero}\big(\d\psi_t(J^MX)-J^N\d\psi_t(X)\big)=0$. In order to advance this programme, we then prove a series of relations between infinitesimal properties of the map $\varphi$ and those of its twistor lift $\psi$, obtaining \emph{first-order} analogues of twistorial constructions; thus, providing a unified twistorial framework for the results in \cite{LemaireWood:02} and \cite{LemaireWood:07}.


\renewcommand{\chaptermark}[1]{\markboth{#1}{}}             
\renewcommand{\sectionmark}[1]{\markright{\thesection\ #1}} 

\chapter{Twistor spaces}\label{Chapter:TwistorSpaces}


\fancyhf{}                                                  
\fancyhead[LE,RO]{\thepage}
\fancyhead[LO]{\nouppercase{\slshape\rightmark}} 
\fancyhead[RE]{\nouppercase{\slshape\leftmark}}
\renewcommand{\headrulewidth}{0 pt}
\pagestyle{fancyplain} 



In this chapter, we define a new notion of \emph{twistor space} $\Sigma^+V$ for an orientable vector bundle $V$ of even rank over a manifold $M$. This generalizes the usual notion of twistor space over a Riemannian manifold where $V=TM$ and
will become quite useful in Chapters \ref{Chapter:HarmonicMapsAndTwistorSpaces} and \ref{Chapter:JacobiVectorFieldsAndTwistorSpaces}. In Section \ref{Section:SectionInChapter:TwistorSpaces:HolomorphicStructureOnSigmaPlusVAndKoszulMalgrangeTheorem} we shall see under what circumstances we can give $\Sigma^+V$ the structure of a holomorphic bundle over a complex manifold $M$ and the importance of the Koszul-Malgrange Theorem. We also take the opportunity to generalize this theorem into a ``time-dependent" version, that will come into play when we study variations of harmonic maps in Chapter \ref{Chapter:JacobiVectorFieldsAndTwistorSpaces}.

Throughout this and subsequent chapters, all vector spaces are assumed to be finite-dimensional. We shall also use the following standard notation: if $E$ is an oriented Euclidean vector space and $J$ a \emph{Hermitian structure} on $E$ (Definition \ref{Definition:ComplexStructureOnAVectorSpace}):\index{$\GL(E)$}\index{$\O(E)$}\index{$\SO(E)$}\index{$k$-linear~maps}\index{$\L^k(E,E')$}\index{$\U_J(E)$}\index{$\gl(E)$}\index{$\L(E,E)$}\index{$\so(E)$}\index{$\u_J(E)$}\index{$\m_J(E)$!}\index{Lie~groups~and~Lie~algebras~notation}

\addtolength{\arraycolsep}{-1.2mm}$\begin{array}{cll}
\GL(E)&=& \{\lambda:E\to E:\,\lambda\text{ is an isomorphism}\}\\
\O(E)&=& \{\lambda\in\GL(E):\lambda\circ\lambda^\top=Id\}\\
\SO(E)&=&\{\lambda\in\O(E):\,\lambda\text{~is~orientation~preserving}\}\\
\U_J(E)&=&\{\lambda\in\SO(E):\,\lambda\circ J=J\circ\lambda\}\\
\L(E,E)\footnotemark&=&T_{Id}\GL(E)=\gl(E)\\
\o(E)&=&\so(E)=T_{Id}\O(E)=T_{Id}\SO(E)=\{\lambda\in\L(E,E):\lambda+\lambda^\top=0\}\\
u_J(E)&=&T_{Id}\U_J(E)=\{\lambda\in\so(E):\lambda\circ{J}=J\circ\lambda\}\\
m_J(E)&=&\{\lambda\in\so(E):\,\lambda\circ J=-J\circ\lambda\}
\end{array}$\addtolength{\arraycolsep}{1.2mm}

\footnotetext{More generally, when $E$ and $E'$ are two vector spaces, $\L^k(E,E')$ denotes the space of $k$-linear maps from $E\times...\times{E}$ ($k$ copies) to $E'$.}\index{Complexified!vector~space}\index{Complexified!metric}\index{Metric!complexified}\index{$E^\cn$}\noindent
When $W$ is a complex vector space, we shall add the prefix $\cn$ to indicate that we are dealing with complex-linear maps, thus obtaining the usual groups $\GL(\cn,W)$, $\O(\cn,W)$, $\SO(\cn,W)$, \textit{etc.}. To be more precise, any complex vector space $W$ that we shall be dealing with will be the \emph{complexification} of some Euclidean vector space $E$, \textit{i.e.}, $W=E^\cn=E\otimes\cn$. Hence, on $W$, we have the complex bilinear extension of the Euclidean metric $<,>$ on $E$ (the \emph{complexified metric}, that we still denote by $<,>$). Then we can define $<\lambda^\top(u),v>=<u,\lambda(v)>$ for $\lambda\in\L(\cn,W)$ and obtain the above groups. When $W=\cn^k\simeq\rn^k\otimes\cn$ we shall use the more usual notation $\GL(\cn,k)$ instead of $\GL(\cn,\cn^k)$, and similarly for the other standard groups.


\section{Hermitian~structures~on~a~vector~space}\label{Section:SectionIn:Chapter:TwistorSpaces:HermitianStructuresonAVectorSpace}


The following definitions and results are standard in the literature (\textit{e.g.} \cite{BairdWood:03}, \cite{DavidovSergeev:93}, \cite{KobayashiNomizu:69}, \cite{Salamon:85}).

\begin{definition}[\emph{Hermitian~structure~on~a~vector~space}]\label{Definition:ComplexStructureOnAVectorSpace}\index{$\Sigma~E$}\index{Complex~structure!on~a~vector~space}\index{Hermitian~structure!on~a~Euclidean~vector~space}\index{Compatible!complex~structure~and~metric}\index{$\Sigma~E$}
Let $E$ be a real vector space. A \emph{complex structure} on $E$ is a linear map $J:E\to E$ such that $J^2=-\Id$. If $E$ is a Euclidean vector space, we call a complex structure $J$ \emph{Hermitian} if it is \emph{compatible} with the Euclidean metric:
\begin{equation}\label{Equation:EquationIn:Definition:ComplexStructureOnAVectorSpace}
<J u,J v>=<u,v>,\quad\forall\,u,v\in{E}.
\end{equation}
We shall represent the set of Hermitian structures on $E$ by
$\Sigma~E$.
\end{definition}\index{Complex~structure!orientation~preserving}\index{Complex~structure!even~dimension}\index{Orientation!induced~by~a~complex~structure}

It is easy to see that if $E$ is a real vector space which admits an almost complex structure $J$, then $E$ admits the structure of a complex vector space, by defining $i.u=J u$ for $u\in{E}$. In particular, it is easy to see that if we consider a basis $\{e_{1},...,e_k\}\in{E}$ for $E$ as a complex vector space, then $\{e_1,J e_1,...,e_k,J e_k\}$ is a basis for $E$ as a real vector space. Hence $E$ must have even real dimension. Two bases of $E$ of the form $\{e_1,Je_1,...,e_k,Je_k\}$ always differ by a matrix of positive determinant. Consequently, we can define an orientation on $E$ by declaring $\{e_1, J e_1,...,e_k,J e_k\}$ to be positive; we call this orientation the \emph{natural orientation induced by the almost complex structure} $J$ on the vector space $E$. If $J$ is some complex structure on $E$, it has two and only two eigenvalues (as a linear isomorphism on $E^\cn$), namely $i$ and $-i$; their
corresponding eigenspaces will be denoted by $E^{10}$ and $E^{01}$,
respectively, so that we have the decomposition
\[E^\cn=E^{10}\oplus E^{01},\,E^{10}=\{u-iJu\in{E}^{\cn}:\,u\in{E}\},\,E^{01}=\{u+iJu\in{E}^\cn:\,u\in{E}\}\text{.}\]
Finally, if $J$ is a Hermitian structure on $E$, it is easy to prove the existence of an orthonormal basis $\{e_1,Je_1,...,e_k,Je_k\}$ for $E$ (\textit{e.g.}, \cite{KobayashiNomizu:69}, Vol 2, Proposition 1.8).\index{$E^\cn$,~$E^{10}$,~$E^{01}$}\index{Decomposition!into~$(1,0)$~and~$(0,1)$~subspaces}\index{$(1,0)$,~$(0,1)$~subspaces}

\begin{definition}[\emph{Positive~Hermitian~structure~on~a~vector~space}]\label{Definition:PositiveHermitianStructureOnAEuclideanVectorSpace}\index{Positive~Hermitian~structure!on~a~vector~space}\index{Hermitian~structure!positive~(on~a~vector~space)}\index{Hermitian~structure!negative}
Let $J$ be a Hermitian structure on an oriented Euclidean space $E$. We say that $J$ is \emph{positive} if the natural orientation it induces on $E$ is the orientation of $E$. Equivalently, $J$ is positive if and only if there is a positive basis of the form $\{e_1,J e_1,...,e_k,J e_k\}$. The set of all positive Hermitian structures on an oriented Euclidean space will be represented by $\Sigma^+ E$. Similarly we can define $\Sigma^-E$; clearly, if $\tilde{E}$ represents $E$ with the opposite orientation, $\Sigma^+\tilde{E}=\Sigma^-E$.
\end{definition}

Now, $\SO(E)$ acts on $\Sigma^+E$ by the formula\index{Action!of~$\SO(E)$~on~$\Sigma^+E$}\addtolength{\arraycolsep}{-1.2mm}
\begin{equation}\label{Equation:ActionOfSOVOnSigmaPlusV}
\begin{array}{cccll}
\SO(E)&\times&\Sigma^+ E&\to&\Sigma^+ E\\
 (S&,& J)&\to& S\cdot J=S\circ J\circ S^{-1}\text{.}
\end{array}
\end{equation}
It is easy to check that this is indeed a well-defined transitive action with isotropy subgroup at an element $J\in\Sigma^+E$ given by
\begin{equation}\label{Equation:GroupUJE}\index{$\U_J(E)$}
\U_J(E)=\{S\in\SO(E):S\cdot J=J\cdot S\}\text{.}
\end{equation}\addtolength{\arraycolsep}{1.2mm}\noindent
It is clear that $\U_J(E)$ is a closed subgroup of $\SO(E)$ so that $\Sigma^+E$ has the unique differentiable structure such that this action is smooth (\textit{e.g.}, \cite{BairdWood:03}). On the other hand, since the action \eqref{Equation:ActionOfSOVOnSigmaPlusV} admits a smooth extension to an open set of $\L(E,E)$ and $\SO(E)$ is compact, $\Sigma^+E$ is a submanifold of $\L(E,E)$.
\label{Page:DefinitionOfTheDifferentiableStructureOnSigmaPlusV}\index{$\Sigma^+E$!differentiable~structure}\index{$\Sigma^+E$!as~a~submanifold~of~$\L(E,E)$}

Fix $J\in\Sigma^+E$ and consider $\U_J(E)$ as well as the \emph{involution}
$\sigma$ defined by\index{Involution}\addtolength{\arraycolsep}{-1.2mm}
\begin{equation}\label{Equation:InvolutionInSOV}
\begin{array}{lcll}
\sigma:&\SO(E)&\to&\SO(E)\\
~&S&\to& J\circ S\circ J^{-1}\quad(=-J\circ S\circ J)\text{.}
\end{array}
\end{equation}\addtolength{\arraycolsep}{1.2mm}\noindent
Then, $\big(\SO(E),\U_J(E),\sigma\big)$ is a symmetric space (see, for example, \cite{KobayashiNomizu:69}, Vol. 2 Ch. XI for standard notation and terminology) and it gives rise to a symmetric Lie algebra, $\big(\so(E),u_J(E),\sigma^\star\big)$ where $\sigma^{\star}=\d\sigma_{\Id}$. Let $\m_J(E)$ denote the eigenspace of $\so(E)$ associated to the eigenvalue $-1$ of
$\sigma^\star$:\index{$\m_J(E)$}\index{$\u_J(E)$}\index{$\so(E)$}\index{Symmetric!space}\index{Symmetric!Lie~algebra}
\[\m_J(E)=\{\lambda\in\so(E):\,\sigma^\star(\lambda)=-\lambda\}=\{\lambda\in\so(E):\,\lambda\circ{J}=-J\circ\lambda\}\text{;}\]
then, $\so(E)=\u_J(E)\oplus \m_J(E)$. We can now define a complex structure on $\m_J(E)$ (see ,\textit{e.g.}, \cite{BairdWood:03} Ch. 7 or \cite{DavidovSergeev:93}):

\begin{definition}[\emph{Complex~structure~on~$\m_J(E)$}]\label{Definition:AlmostComplexStructureOnMJV}\index{Complex~structure!on~$\m_J(E)$}\index{$\m_J(E)$!complex~structure}
We can define a complex structure $\J^\V$ on $\m_J(E)$ by\addtolength{\arraycolsep}{-1.2mm}
\begin{equation}\label{Equation:EquationIn:Definition:AlmostComplexStructureOnMJV}
\begin{array}{lcll}
\J^\V: &\m_J&\to& \m_J \\
~& \lambda&\to& \J^\V(\lambda)=J\circ\lambda\,\,\big(=-\lambda\circ{J}\big)\text{.}
\end{array}
\end{equation}\addtolength{\arraycolsep}{1.2mm}
\end{definition}

With the (unique) differentiable structure that makes the action \eqref{Equation:ActionOfSOVOnSigmaPlusV} smooth, $\Sigma^+E$ is a submanifold of $\L(E,E)$. Choose $J\in\Sigma^+E$ and consider the smooth submersion \index{$R_J$,~right~action}\index{Action!$R_J$}$R_J:\SO(E)\to\Sigma^+E$ obtained by the
right action of $J$, $R_J(S)=S\cdot J=S\circ J\circ S^{-1}$. Clearly, $R_J(\Id)=J$ and $\d{R}_{J_{\Id}}:T_{\Id}\SO(E)=\so(E)\to T_J\Sigma^+E$ is a surjective map. Hence, $T_J\Sigma^+E=\{\lambda\in\L(E,E):\,\exists\,\beta\in\so(V)\text{~with~}\d{R}_{J_{\Id}}(\beta)=\lambda\}$. Now, as
$\d{R}_{J_{\Id}}(\beta)=\beta J-J\beta$, we can rewrite
$T_J\Sigma^+E$ as the set $\{\lambda\in\L(E,E):\,\exists\,\beta\in\so(E)\text{~with~}\beta{J}-J\beta=\lambda\}$. Finally, it is easy to check that this set is nothing but $\m_J (E)$ and therefore
\begin{equation}\label{Equation:TheTangentSpaceToSigmaPlusVAtAPointJ}
T_J\Sigma^+E=\m_J(E)\text{.}
\end{equation}
\noindent As a consequence, from Definition \ref{Definition:AlmostComplexStructureOnMJV}, we conclude that $\Sigma^+E$ is an almost complex manifold. We shall see in the next section that $\Sigma^+E$ is, indeed, a \emph{complex} manifold and that this complex structure can be given in a different way\footnote{We can also give $\Sigma^-E$ an almost complex
structure: just consider $\Sigma^-E$ as $\Sigma^+\tilde{E}$ where $\tilde{E}$ is the Euclidean space $E$ with the opposite orientation.}.\label{Page:DefinitionOfAlmostComplexStructureOnSigmaPlusE}\index{Almost~complex~structure!on~$\Sigma^+E$}\index{Almost~complex~structure!on~$\Sigma^-E$}\index{$\Sigma^+E$!almost~complex~structure}


\subsection{\texorpdfstring{$\Sigma^+E$ as the complex manifold $G^+_{iso}(E^\cn)$}{Sigma+E as the complex manifold G+iso(Ec)}}\label{Subsection:SigmaPlusEAsTheComplexManifoldGkOplusIsoEcn}\index{$G^+_{iso}(E^\cn)$}


\label{Page:kDimensionalIsotropicSubspacesAndHermitianStructures}\index{Isotropic~subspace}\index{Positive~isotropic~subspace|see{Isotropic~subspace}}\index{Isotropic~subspace!positive}\index{Isotropic~subspace!}\index{Hermitian~structure!induced~by~a~isotropic~subspace}\index{Isotropic~subspace!and~Hermitian~structures}\index{Positive~isotropic~subspace|see{Isotropic~subspace}}\index{Isotropic~subspace!positive}\index{$J_F$}\noindent
Let $(E,<,>)$ be an oriented even-dimensional Euclidean vector space with (real) dimension $2k$. Consider its complexification $E^\cn$, equipped with the complexified metric. Given a complex linear subspace $F\subseteq E^\cn$, we shall say that $F$ is \emph{isotropic} if $<F,F>=0$; \textit{i.e.}, $<u,v>=0$ for every pair of vectors $u,v\in{F}$. We now wish to introduce the notion of \emph{positive isotropic subspace}. We start by noticing that to each $k$-dimensional isotropic linear subspace on $E^\cn$ corresponds one and only one Hermitian structure $J$ (the \emph{induced Hermitian structure}) on $E$ such that $F$ is the corresponding $(1,0)$-subspace. Explicitly, to each $J$ in $\Sigma~E$ we make correspond the isotropic $k$-dimensional subspace in $E^\cn$ given by $E^{10}$ and, conversely, given an isotropic $k$-dimensional subspace $F$ in $E^\cn$, we define $J^\cn$ as acting as $i$ on $F$ and $-i$ on its conjugate $\overline{F}:=\{\overline{u}\in{E}^\cn:\,u\in{F}\}$. It is easy to check that $J^\cn$ preserves $E$ and has $F$ as its associated $(1,0)$-subspace. Hence, we can now define an isotropic $k$-dimensional complex linear subspace $F$ in $E^\cn$ as being \emph{positive} if the induced Hermitian structure $J_F$ is positive. Therefore, we can construct the Grassmannian of the $k$-dimensional positive isotropic subspaces on $E^\cn$,
\begin{equation*}\index{Grassmannian!of~the~positive~isotropic~$k$-dimensional~subspaces~in~$E^\cn$}\index{Isotropic~subspace!Grassmannian}\index{$\Sigma^+E$!as~$G^+_{iso}(E^\cn)$}\index{$G^+_{iso}(E^\cn)$}\index{Twistor~space!of~a~Euclidean~space}
G^+_{iso}(E^\cn):=\{F\,\subseteq{E}^\cn,\,F\text{~isotropic~}k\text{-dimensional~positive~subspace}\}\text{,}
\end{equation*}
which, by its very definition and the preceding discussion, can be identified with the set $\Sigma^+E$ of all positive Hermitian structures on $E$. We now wish to introduce on this set a complex structure (and not just an almost complex structure as previously). We can do so by showing that there is a left transitive action of a complex Lie group which has a complex Lie subgroup as isotropy subgroup, as follows:

\begin{theorem}[\textup{$G^+_{iso}(E^\cn)$~as~a~complex~manifold}]\label{Theorem:ComplexStructureOnTheGrassmannianOfThekDimensionalIsotropicPositiveSubspacesOfcnTwok}\index{Complex~structure!on~the~manifold~$G^+_{iso}(E^\cn)$}\index{$G^+_{iso}(E^\cn)$!complex~structure}
Let $G^+_{iso}(E^\cn)$ be defined as above. Consider the complex Lie group $\SO(\cn,E^\cn)$ and the map\addtolength{\arraycolsep}{-1.2mm}
\begin{equation}\label{Equation:FirstEquationIn:Theorem:ComplexStructureOnTheGrassmannianOfThekDimensionalIsotropicPositiveSubspacesOfcnTwok}
\begin{array}{ccclc}
\SO(\cn,E^\cn)&\times& G^+_{iso}(E^\cn)& \to & G^+_{iso}(E^\cn) \\
(\lambda&,&F)&\to&\lambda(F)\text{.}
\end{array}
\end{equation}\addtolength{\arraycolsep}{1.2mm}\noindent
This is a well-defined transitive left action of the group $\SO(\cn,E^\cn)$ on the set $G^+_{iso}(E^\cn)$. Moreover, its isotropy group is a complex Lie subgroup of $\SO(\cn,E^\cn)$. In particular, $G^+_{iso}(E^\cn)$ is a complex manifold, whose differentiable (holomorphic) structure is the only one for which the action
\eqref{Equation:FirstEquationIn:Theorem:ComplexStructureOnTheGrassmannianOfThekDimensionalIsotropicPositiveSubspacesOfcnTwok}
is holomorphic\footnote{Notice that we could do an analogous construction for $\Sigma~E$ instead of $\Sigma^+E$: in this case, we should drop the condition ``$F$ is positive" and consider the Grassmannian $G_{iso}(E^\cn)$ of isotropic $k$-dimensional subspaces instead of $G^+_{iso}(E^\cn)$.}.
\end{theorem}\index{Compatible!$\Sigma^+E$~and~$G^+_{iso}(E^\cn)$!the~differentiable~structure}\index{$G^+_{iso}(E^\cn)$!and~$\Sigma^+E$~differentiable~structures}

See Section \ref{Section:SectionIn:Chapter:Appendix:ComplexLieGroupsAndProofOf:Theorem:ComplexStructureOnTheGrassmannianOfThekDimensionalIsotropicPositiveSubspacesOfcnTwok}
for the proof of this theorem and some more details. The transitivity of the above action
\eqref{Equation:FirstEquationIn:Theorem:ComplexStructureOnTheGrassmannianOfThekDimensionalIsotropicPositiveSubspacesOfcnTwok}
follows from that of $\SO(E)$ on $G^+_{iso}(E^\cn)$; therefore, we could have defined a left transitive action $\SO(E)$ on $G^+_{iso}(E^\cn)$ by the same expression. For this action, the isotropy group at $F\in{G}^+_{iso}(E^\cn)$ would be $\U_J(E)$, where $J$ is the complex structure induced by $F$. In particular, as $\SO(E)$ is a (real) Lie subgroup of $\SO(\cn,E^\cn)$, we deduce that the underlying differentiable structure on $G^+_{iso}(E^\cn)$
makes the identification $\Sigma^+E\longleftrightarrow{G}^+_{iso}(E^\cn)$ smooth. The following proposition shows that this identification is anti-holomorphic for the \emph{almost complex} structure introduced on $\Sigma^+E$ \textit{via} \eqref{Equation:EquationIn:Definition:AlmostComplexStructureOnMJV} and \eqref{Equation:TheTangentSpaceToSigmaPlusVAtAPointJ} and the \emph{complex} structure on $G^+_{iso}(E^\cn)$ given by the above Theorem \ref{Theorem:ComplexStructureOnTheGrassmannianOfThekDimensionalIsotropicPositiveSubspacesOfcnTwok}. For the following, see also \cite{DavidovSergeev:93}, p. 22:

\begin{proposition}[\textup{$G^+_{iso}(E^\cn)$~and~$\Sigma^+E$}]\label{Proposition:ComptibilityBetweenSigmaPlusEAndGrassmannianOfTheIsotropicSubspacesInEcnTheComplexStructure}\index{Compatible!$\Sigma^+E$~and~$G^+_{iso}(E^\cn)$!the~complex~structure}
Take $\Sigma^+E$ with the almost complex structure introduced \textit{via} \eqref{Equation:EquationIn:Definition:AlmostComplexStructureOnMJV} and \eqref{Equation:TheTangentSpaceToSigmaPlusVAtAPointJ}, $G^+_{iso}(E^\cn)$ with the (almost) complex structure in Theorem
\ref{Theorem:ComplexStructureOnTheGrassmannianOfThekDimensionalIsotropicPositiveSubspacesOfcnTwok}.
Let $\eta:\Sigma^+E\to G^+_{iso}(E^\cn)$ be the identification map that to each complex structure $J$ makes correspond $E^{10}$. Then, not only is $\eta$ smooth (as in the previous discussion) but it is anti-holomorphic. In other words, both these almost complex structures are integrable and conjugates of each other.
\end{proposition}
\begin{proof}
Since the constructions are independent of choice of basis, by fixing an orthonormal positive basis for $E$ and identifying $E$ with $\rn^{2k}$, we can reduce to the case $\Sigma^+\rn^{2k}$. Now, all we have to prove is that for any holomorphic map $f:\openU\subseteq\cn\to G^+_{iso}(\cn^{2k})$ the corresponding map $J_f$ to $\Sigma^+ \rn^{2k}$ is anti-holomorphic for the structure defined \textit{via} \eqref{Equation:EquationIn:Definition:AlmostComplexStructureOnMJV} and \eqref{Equation:TheTangentSpaceToSigmaPlusVAtAPointJ}. Now, $f$ is holomorphic if and only if it has a holomorphic lift (also denoted by $f$) to $\SO(\cn,2k)\subseteq \L(\cn,2k)$. As $\L(\cn,2k)$ is a complex vector space, holomorphicity of $f$ is equivalent to the Cauchy-Riemann equations:
\begin{equation}\label{Equation:FirstEquationInTheProofOf:Proposition:ComptibilityBetweenSigmaPlusEAndGrassmannianOfTheIsotropicSubspacesInEcnTheComplexStructure}\index{Cauchy-Riemann~equations}\index{$F_0$,~$(1,0)$-subspace~of~$J_0$}\index{$J_0$,~canonical~complex~structure~on~$\rn^{2k}$}
\left\{\begin{array}{l}
\partial_x \Real f=\partial_y \Imag f\\
\partial_y \Real f=-\partial_x \Imag f\text{,}
\end{array}\right.
\end{equation}
where $f=\Real f+i\Imag f\in\L(\cn,2k)$ with $\Real f,\Imag{f}\in\L(\rn,2k)$ and $z=x+iy\in\cn$. On the other hand, given the map $f:z\to f(z)\in\SO(\cn,2k)$ and assuming for simplicity that $f(0)=\Id$, the induced map to $G^+_{iso}(\cn^{2k})$ is precisely $f(z)F_0$, where $F_0=\wordspan \{e_{2i-1}-ie_{2i}\}_{i=1,...,k}$ so\footnote{$F_0$ is the $(1,0)$-subspace associated with the canonical complex structure $J_0$ of $\rn^{2k}$ defined by $J_0 e_{2i-1}=e_{2i}$, $i=1,...,k$.} that $f(z)F_0=\wordspan\{f(z)e_{2i-1}-if(z)e_{2i}\}_{i=1,...,k}=\wordspan\{\Real f(z)e_{2i-1}+\Imag f(z)e_{2i}-i(\Real f(z) e_{2i}-\Imag f(z)e_{2i-1})\}_{i-1,...,k}$ and therefore $J_f:\openU\subseteq\cn\to\Sigma^+\rn^{2k}$ is given by $J_f(\Real f e_{2i-1}+\Imag fe_{2i})=\Real f e_{2i}-\Imag fe_{2i-1}$, as $T^{10}_{J_f(z)}\rn^{2k}=f(z)F_0$. Equivalently, letting $\cdot$ denote the usual matrix multiplication in $\L(\rn,2k)$, $J_f$ is determined by the conditions\addtolength{\arraycolsep}{-1.0mm}
\[\left\{\begin{array}{ccc}
J_f\cdot \Real fe_{2i-1}&=&-\Imag f{e}_{2i-1}\\
J_f\cdot \Imag fe_{2i}&=&\Real f e_{2i}
\end{array}\right.\]\addtolength{\arraycolsep}{1.0mm}\noindent
which are equivalent to $J_f\cdot \Real f=-\Imag f$. Applying the
derivatives $\partial_x$ and $\partial_y$ yields\addtolength{\arraycolsep}{-1.0mm}
\[\left\{\begin{array}{l}
\partial_x J_f\cdot \Real f+J_f \cdot\partial_x \Real f=-\partial_x \Imag f\\
\partial_y J_f\cdot \Real f+J_f \cdot\partial_y \Real f=-\partial_y \Imag f\text{.}\\
\end{array}\right.\]
Applying $J_f$ to both sides of the first equation and using \eqref{Equation:FirstEquationInTheProofOf:Proposition:ComptibilityBetweenSigmaPlusEAndGrassmannianOfTheIsotropicSubspacesInEcnTheComplexStructure}
on the second identity gives
\[\left\{\begin{array}{l}
J_f \cdot\partial_x J_f\cdot \Real f=\partial_x \Real f-J_f\partial_x \Imag f\\
\partial_y J_f\cdot \Real f=J_f \cdot\partial_x \Imag f-\partial_x \Real f\\
\end{array}\right.\]
and using the second equation to simplify the first as well as the fact that $\Real f(0)=\Id$ (so that $\Real f$ is invertible near the origin) yields
\[J_f\cdot\partial_xJ_f\cdot\Real{f}=-\partial_yJ_f\cdot\Real{f}\,\impl\,J_f\cdot\partial_xJ_f=-\partial_yJ_f\text{.}\]\addtolength{\arraycolsep}{1.0mm}\noindent
Hence, $\d J_f (\partial_y)=\d{J}_f(J^{\rn^2}_0\partial_x)=-J_f \d J_f(\partial_x)=-\J^\V \d J_f(\partial_x)$ by
\eqref{Equation:EquationIn:Definition:AlmostComplexStructureOnMJV},
proving that $J_f$ is anti-holomorphic for the complex structure introduced \textit{via} \eqref{Equation:EquationIn:Definition:AlmostComplexStructureOnMJV} and \eqref{Equation:TheTangentSpaceToSigmaPlusVAtAPointJ} and thus concluding our proof.
\end{proof}


\section{Fundamental $2$-form on vector bundles}


Let $(M,g,J)$ be an almost Hermitian manifold, thus $J_x\in\Sigma~T_xM$ and the assignment $x\to J_x$ is smooth as a
section of the bundle $\L(TM,TM)$. Then we have the usual fundamental $2$-form (or \emph{Kähler form}) $\omega$ associated with $J$:
\begin{equation}\label{Equation:DefinitionOfKahlerFundamentalForm}\index{Fundamental~Kähler~form}\index{Fundamental~$2$-form!}\index{Kähler~form!}\index{$\omega$,~fundamental~$2$-form}\index{Compatible!linear~connection}\index{Metric!compatible~with~linear~connection}\index{Riemannian!vector~bundle}\index{Vector~bundle!Riemannian}\index{Compatible!metric~and~connection}\index{Compatible!complex~structure~the~metric}
\omega(X,Y):=g(JX,Y),\, X,Y\in\Gamma(TM)\text{.}
\end{equation}
In a more general context, consider an (orientable even-dimensional) vector bundle $V$ over a manifold $M$ (not necessarily equipped with an almost complex structure for now); let $V$ be equipped with a metric
$g_V$ and a \emph{compatible linear connection}, $\nabla^V$, in the sense that $\nabla^V g_V=0$; \textit{i.e.},
\begin{equation}\label{Equation:CompatibilityOfMetricAndConnectionOnAVectorBundle}\index{Almost~Hermitian~structure!on~a~vector~bundle}\index{Vector~bundle!almost~Hermitian~structure}
X(g_V(U,W))=g(\nabla^V_X{U},W)+g_V(U,\nabla^V_XW),\quad\forall\,X\in\Gamma(TM),\,U,W\in\Gamma(V)
\end{equation}
(such a vector bundle $(V,g_V,\nabla^V)$ will be called a \emph{Riemannian vector bundle} as in, for example, \cite{doCarmo:88}). Moreover, assume that $V$ is equipped with an almost Hermitian structure $J_V$. In other words, $J_{V_x}:V_x\to{V}_x$ with $J^2_{V_x}=-\Id_{V_x}$, $J_{V_x}$ compatible with the metric $g_{V_x}$ and, of course, $J$ is a smooth section of $\L(V,V)$. We shall call such a vector bundle $(V,g_V,\nabla^V,J_V)$ an \emph{almost Hermitian vector bundle}. On such a vector bundle, we also have a fundamental $2$-form defined by
\begin{equation}\label{Equation:DefinitionOfKahlerFormForVectorBundles}\index{Kähler~form!on~vector~bundles}\index{Fundamental~$2$-form!on~vector~bundles}\index{Almost~Hermitian~vector~bundle}\index{Vector~bundle!almost~Hermitian}
\omega_{V}(U,W)=g_V(J_V U,W),\quad\forall\,U,W\in\Gamma(V)\text{.}
\end{equation}
As in the usual case when $V=TM$ we also have a natural splitting of the bundle $V^{\cn}$ into \emph{$(1,0)$} and \emph{$(0,1)$-subspaces},
\[V^\cn=V^{10}\oplus V^{01}\text{,}\]
where $V^{10}$ (respectively $V^{01}$) is the eigenspace associated to the eigenvalue $i$ (respectively $-i$) of $J_V$.

We establish the following result, which is known for the case $V=TM$ (\cite{Salamon:85}, Lemma 1.1):
\begin{lemma}\label{Lemma:SalamonsLemma.1.1Generalized}
Let $(V,g_V,\nabla^V,J_V)$ be an almost Hermitian vector bundle over the Riemannian manifold $(M,g)$. Then, for all $X\in\Gamma(TM)$, $U^{10},W^{10}\in\Gamma(V^{10})$ and $W^{01}\in\Gamma(V^{01})$,
\begin{equation}\label{Equation:FirstEquationIn:Lemma:SalamonsLemma1.1Generalized}
(\nabla_X\omega_V)(U^{10},W^{10})=2ig_V(\nabla^V_XU^{10},W^{10})\text{~and}
\end{equation}
\begin{equation}\label{Equation:SecondEquationIn:Lemma:SalamonsLemma1.1Generalized}
(\nabla_X\omega_V)(U^{10},W^{01})=0\text{,}
\end{equation}
where the connection $\nabla^V$, the metric $g_V$ and the fundamental $2$-form $\omega_V$ are just the complex bilinear extensions of the given ones on $V$.
\end{lemma}
The same result holds if we allow $X\in\Gamma(T^{\cn}M)$ and then the first equation will easily lead to
\begin{equation}\label{Equation:ThirdEquationIn:Lemma:SalamonsLemma1.1Generalized}
(\nabla_X\omega_V)(U^{01},W^{01})=-2ig_V(\nabla^V_XU^{01},W^{01})\text{,}
\end{equation}
for $X\in\Gamma(T^\cn{M})$), $U^{01},W^{01}\in\Gamma(V^{01})$. In particular, we deduce Salamon's Lemma 1.1 in \cite{Salamon:85}
when $V$ is just the usual tangent bundle over an almost Hermitian manifold, expressed in terms of forms (instead of using what he calls ``fundamental $2$-vectors"). Notice, moreover, that the preceding identities are tensorial. For instance, identity \eqref{Equation:FirstEquationIn:Lemma:SalamonsLemma1.1Generalized} can be written as
\[(\nabla_X\omega_V)(U^{10},W^{10})=2ig_V(\nabla^V_XU^{10},W^{10}),\quad\forall\,x\in{M},\,X\in{T}_{x}M,\,U^{10},W^{10}\in{V}^{10}_{x}\text{.}\]
The left-hand side of each equation is tensorial. For the right-hand side, we can directly check its tensoriality using again the fact that $g_V(V^{10},V^{10})=0$ (or the $(0,1)$-analogue):
\[g_V(\nabla_X(f.U^{10}),W^{10})=g_V(X(f).U^{10}\hspace{-1mm},W^{10})+fg_V(\nabla_XU^{10}\hspace{-1mm},W^{10})=fg_V(\nabla_XU^{10}\hspace{-1mm},W^{10})\text{,}\]
for all smooth functions $f$ (locally) defined on $M$.

\noindent\textit{Proof of Lemma
\ref{Lemma:SalamonsLemma.1.1Generalized}.} We have\addtolength{\arraycolsep}{-1.0mm}
\[\begin{array}{lll}
(\nabla_X\omega_V)(U^{10},W^{10})&=&X\big(\omega_V(U^{10},W^{10})\big)-\omega_V(\nabla^V_XU^{10},W^{10})-\omega_V(U^{10},\nabla^V_XW^{10})\\
~&=&X(g_V(J_VU^{10}\hspace{-1mm},W^{10}))-g_V(J\nabla_XU^{10}\hspace{-1mm},W^{10})-g_V(JU^{10}\hspace{-1mm},\nabla_XW^{10})\\
~&=&X(ig_V(U^{10},W^{10}))+g_V(\nabla_XU^{10},JW^{10})-ig_V(U^{10},\nabla_XW^{10})\\
~&=&iX(g_V(U^{10},W^{10}))+ig_V(\nabla_XU^{10},W^{10})-ig_V(U^{10},\nabla_XW^{10})\\
~&=&2ig_V(\nabla_XU^{10},W^{10})\text{,}\end{array}\]\addtolength{\arraycolsep}{1.0mm}\noindent
as desired. For the other two equations \eqref{Equation:SecondEquationIn:Lemma:SalamonsLemma1.1Generalized}
and \eqref{Equation:ThirdEquationIn:Lemma:SalamonsLemma1.1Generalized} the arguments are similar.\qed

A useful way of stating the last lemma is the following: consider $V^{10}$ and its dual, $V^{10^{\star}}$; $V^{10^\star}_x$ is the set of $\cn$-valued linear maps from $V^{10}_x$ or, if we prefer, the set of $\cn$-valued linear maps from $V^\cn_x$\index{Dual~of~$V^{10}$~and~$V^{01}$} that vanish on every vector of $V^{01}_x$. Define $V^{01^{\star}}$ in an analogous fashion. Now, take $V^{20^{\star}}$ as given by\index{$V^{20^{\star}}$,~$V^{02^{\star}}$,~$V^{11^{\star}}$}\addtolength{\arraycolsep}{-1.0mm}

\noindent$\begin{array}{llll}~&V^{20^{\star}}&:=&\{\lambda:V^\cn\times V^\cn\to\cn\text{, bilinear and skew-symmetric,~}\lambda(V^{01},V^\cn)=0\}\\
~&~&\simeq&\{\lambda:V^{10}\times{V}^{10}\to\cn\text{,~bilinear~and~skew-symmetric}\}\text{.}\\
\multicolumn{4}{l}{\text{\put(-5,0){Analogously,}}}\\
~&V^{02^{\star}}&:=&\{\lambda:V^\cn\times V^\cn\to\cn\text{, bilinear and skew-symmetric,~}\lambda(V^{10},V^\cn)=0\}\\
~&~&\simeq&\{\lambda:V^{01}\times{V}^{01}\to\cn\text{,~bilinear~and~skew-symmetric}\}\quad\text{and}\\
~&V^{11^{\star}}&:=&\left\{\begin{array}{ll}\lambda:V^\cn\times V^\cn\to\cn&\text{,~bilinear~and~skew-symmetric,~}\\~&\lambda(V^{01},V^{01})=\lambda(V^{10},V^{10})=0\end{array}\right\}\text{.}\\
\end{array}$

\addtolength{\arraycolsep}{1.0mm}\noindent Then $\Lambda^2(V^\cn)$ can be decomposed into the direct sum
\[\Lambda^2(V^\cn)=V^{20^{\star}}\oplus V^{11^\star}\oplus V^{02^\star}\]
and, since $\omega_V$ is skew-symmetric, so is $\nabla_X\omega_V$. Lemma \ref{Lemma:SalamonsLemma.1.1Generalized} states that $\nabla_X\omega_V\in{V}^{20^\star}\oplus V^{02^\star}$ and that its $(2,0)$ and $(0,2)$-parts are given respectively by \eqref{Equation:FirstEquationIn:Lemma:SalamonsLemma1.1Generalized} and \eqref{Equation:ThirdEquationIn:Lemma:SalamonsLemma1.1Generalized}.

To proceed further, assume now that the base manifold also has an almost Hermitian structure. Then, not only $V^\cn$ can be split into its $(1,0)$ and $(0,1)$-parts, but the same happens to $T^\cn M$. In particular we can write, for any vector fields $X\in\Gamma(T^{\cn}M),\,U,\,W\in\Gamma(V^\cn)$, denoting by $X^{10}$ the $(1,0)$-part of $X$ in $T^\cn{M}$, $U^{10}$ the $(1,0)$-part of $U$ in $V^\cn$, \textit{etc.} (again, for any vectors at a point, with the same argument as after Lemma \ref{Lemma:SalamonsLemma.1.1Generalized}):\addtolength{\arraycolsep}{-1.5mm}
\begin{equation}\label{Equation:DecompositionOfDerivativeOfTheFundamentalKahlerFormOnIts1.2.11.12.etcPartsMorePrecise}
\begin{array}{lll}
(\nabla_{\hspace{-1mm}X}\omega)(U,W)&=&\d_1^1\omega(X,\hspace{-0.4mm}U,\hspace{-0.4mm}W)\hspace{-0.5mm}+\hspace{-0.5mm}\d_1^2\omega(X,\hspace{-0.4mm}U,\hspace{-0.4mm}W)\hspace{-0.5mm}+\hspace{-0.5mm}\d_2^1\omega(X,\hspace{-0.4mm}U,\hspace{-0.4mm}W)\hspace{-0.5mm}+\hspace{-0.5mm}\d_2^2\omega(X,\hspace{-0.4mm}U,\hspace{-0.4mm}W)\\
~&=&d_1\omega(X,\hspace{-0.4mm}U,\hspace{-0.4mm}W)\hspace{-0.5mm}+\hspace{-0.5mm}\d_2\omega(X,\hspace{-0.4mm}U,\hspace{-0.4mm}W)\text{, where}
\end{array}
\end{equation}\addtolength{\arraycolsep}{1.5mm}\noindent
\[\begin{array}{ll}
\d_1^1\omega(X,U,W)=(\nabla_{X^{10}}\omega)(U^{10},W^{10})\text{,}&\d_1^2\omega(X,U,W)=(\nabla_{X^{01}}\omega)(U^{01},W^{01})\text{,}\\ \d_2^1\omega(X,U,W)=(\nabla_{X^{10}}\omega)(U^{01},W^{01})\text{,}&\d_2^2\omega(X,U,W)=(\nabla_{X^{01}}\omega)(U^{10},W^{10}) \end{array}\]
and $\d_1\omega=\d_1^1\omega+\d_1^2\omega$, $\d_2\omega=\d_2^1\omega+\d_2^2\omega$.
Equation \eqref{Equation:DecompositionOfDerivativeOfTheFundamentalKahlerFormOnIts1.2.11.12.etcPartsMorePrecise} is valid since all the other parts vanish, in virtue of \eqref{Equation:SecondEquationIn:Lemma:SalamonsLemma1.1Generalized}.
We then have the following lemma, which is known for the case $V=TM$
(\cite{Salamon:85}, Lemma 1.2):
\begin{lemma}\label{Lemma:Salamon1.2Generalized}
Let $(M,J^M,V,g_V,\nabla^V,J_V)$ be as before. Then,
\begin{equation}\label{Equation:FirstEquationIn:Lemma:Salamon1.2Generalized}
\d_1\omega=0\,\,\equi\,\,\nabla^V_{T^{10}M}V^{10}\subseteq V^{10}\text{~and}
\end{equation}
\begin{equation}\label{Equation:SecondEquationIn:Lemma:Salamon1.2Generalized}
\d_2\omega=0\,\,\equi\,\,\nabla^V_{T^{01}M}V^{10}\subseteq{V}^{10}\text{.}\hspace{6mm}
\end{equation}
Moreover\footnote{We could also write \eqref{Equation:SecondEquationIn:Lemma:Salamon1.2Generalized} as $\d_2\omega=0\,\,\equi\,\,\nabla^V_{T^{10}M}V^{01}\subseteq{V}^{01}$.}, these identities are valid pointwise on $M$:
\begin{equation}\label{Equation:ThirdEquationIn:Lemma:Salamon1.2Generalized}
\d_1\omega(x)=0\,\,\equi\,\,\nabla^V_{X^{10}_{x}}U^{10}\subseteq{V}^{10}_{x},\quad\forall\,U^{10}\,\in\Gamma(V^{10})\text{~and}
\end{equation}
\begin{equation}\label{Equation:FourthEquationIn:Lemma:Salamon1.2Generalized}
\d_2\omega(x)=0\,\,\equi\,\,\nabla^V_{X^{01}_{x}}U^{10}\subseteq{V}^{10}_{x},\quad\forall\,U^{10}\,\in\Gamma(V^{10})\text{.}\,\,\,\quad
\end{equation}
\end{lemma}
\begin{proof}Let us start by proving \eqref{Equation:FirstEquationIn:Lemma:Salamon1.2Generalized}:

Let us prove $\d_1^1\omega=0\,\equi\,\d_1^2\omega=0\,\equi\,\d_1\omega=0$. We have $\d_1\omega=0$ if and only if both $\d_1^1\omega$ and $\d_1^2\omega$ vanish, since $\d_1^1\omega$ and $\d_1^2\omega$ are just the different components of $\d_1\omega$. On the other hand,\addtolength{\arraycolsep}{-1.0mm}
\[\begin{array}{ll}~&\d_1^1\omega=0\,\equi\,(\nabla_{X^{10}}\omega)(U^{10},W^{10})=0\,\equi\,\overline{(\nabla_{X^{10}}\omega)(U^{10},W^{10})}=0\\
\equi&(\nabla_{\overline{X^{10}}}\omega)(\overline{U^{10}},\overline{W^{10}})=0\,\equi\,(\nabla_{X^{01}}\omega)(U^{01},W^{01})=0\,\equi\,\d_1^2\omega=0\text{.}\end{array}\]\addtolength{\arraycolsep}{1.0mm}\noindent

Now, we have $\d_1\omega=0$ if and only if $\d_1^1\omega=0$; equivalently, using \eqref{Equation:FirstEquationIn:Lemma:SalamonsLemma1.1Generalized}, if and only if $g_V(\nabla_{X^{10}}U^{10},W^{10})=0$, the latter condition being equivalent to the requirement $\nabla_{X^{10}}U^{10}\in{V}^{10}$, concluding the proof of \eqref{Equation:FirstEquationIn:Lemma:Salamon1.2Generalized}.

To prove the second identity \eqref{Equation:SecondEquationIn:Lemma:Salamon1.2Generalized}, we apply similar arguments, now using \eqref{Equation:ThirdEquationIn:Lemma:SalamonsLemma1.1Generalized}. Finally, the fact that these identities are valid pointwise on $M$ is easy to check.
\end{proof}


\section{Types~of~almost~Hermitian~manifolds}\label{Section:SectionIn:Chapter:TwistorSpaces:TypesOfAlmostHermitianManifolds}


We shall need in the sequel the definition and characterization of the following types of almost Hermitian manifolds (see, \textit{e.g.}, \cite{BairdWood:03}, \cite{Salamon:85}):

\begin{definition}[\emph{Types~of~almost~Hermitian~manifolds}]\label{Definition:TheMainTypesOfAlmostComplexManifolds}\index{$(1,2)$-symplectic~manifold!}\index{Cosymplectic~manifold}\index{Almost~Hermitian~manifold!types}\index{Hermitian~manifold}\index{Complex~structure!induced}\index{Complex~manifold}
Let $(M,g,J)$ be an almost Hermitian manifold. Then, $M$ is said to be

(i) \emph{Hermitian} if $J$ is integrable; \textit{i.e.}, if and only if there is a complex structure on $M$ with $J$ as the induced almost complex structure \footnote{If we consider only $(M,J)$ with $J$ integrable then, as usual, $(M,J)$ is said to be a \emph{complex manifold}. When $M$ is a complex manifold, the \emph{induced complex structure} $J$ on $M$ is defined by $J\partial_{x_i}=\partial_{y_i}$ where $z=(x_1+iy_1,...,x_{m}+iy_{m})$ is a complex chart for $M$.}.

(ii) \emph{$\mathit{(1,2)}$-symplectic} if $\nabla_{Z^{01}}Y^{10}\in\Gamma(T^{10}M)$, for every $Z^{01}\,\in\Gamma(T^{01}M)$, $Y^{10}\in\Gamma(T^{10}M)$ (see \cite{BairdWood:03}, p. 251).

(iii) \emph{Cosymplectic} if $\trace\nabla J=0$. In other words, if $\{X_i\}$ is an orthonormal basis for $T_xM$, then $\sum_i(\nabla_{X_i}J)X_i=0$.
\end{definition}

Let us recall the following properties of almost Hermitian manifolds (see \cite{KobayashiNomizu:69}, \cite{Salamon:85}):

\begin{proposition}[\textup{Properties~of~almost~Hermitian~manifolds}]\label{Proposition:AnAlternativeCharacterizationOfAlmostComplexManifolds}\index{Nijenhuis~tensor}\index{$\N$,~Nijenhuis~tensor}
Let $(M,g,J)$ be as before. Then:

\textup{(i)} the following conditions are equivalent.

\textup{$~$(i$_\text{a}$)} $M$ is Hermitian\footnote{We could replace the word ``Hermitian" with ``complex".}.

\textup{$~$(i$_\text{b}$)} The Nijenhuis tensor $\N$ associated with $J$ vanishes, where
\begin{equation}\label{Equation:DefinitionofTheNijenhuisTensor}
\N(X,Y):=[X,Y]+J[JX,Y]+J[X,JY]-[JX,JY].
\end{equation}

\textup{$~$(i$_\text{c}$)} The bundle $T^{10}M$ is closed under the Lie bracket: $[T^{10}M,T^{10}M]\subseteq{T}^{10}M$.

\textup{(ii)} $M$ is $(1,2)$-symplectic if and only if $(\d\omega)^{12}=0$, where $(\d\omega)^{12}$ denotes the $(1,2)$-part of $\d\omega$ given by\addtolength{\arraycolsep}{-1.0mm}
\[(\d\omega)^{12}(X^{10},Y^{01},Z^{01})\hspace{-0.1cm}=\hspace{-0.1cm}\big(\nabla_{X^{10}}\omega\big)(Y^{01},Z^{01})-\big(\nabla_{Y^{01}}\omega\big)(X^{10},Z^{01})+\big(\nabla_{Z^{01}}\omega\big)(X^{10},Y^{01})\text{.}\]\addtolength{\arraycolsep}{1.0mm}\noindent

\textup{(iii)} Every $(1,2)$-symplectic manifold is cosymplectic.
\end{proposition}
\begin{proof}
The equivalence between (i$_\text{a}$) and (i$_\text{b}$) is the well-known Newlander-Nirenberg theorem (\cite{NewlanderNirenberg:57}; see also \cite{KobayashiNomizu:69}). As for the equivalence between (i$_\text{b}$) and (i$_\text{c}$): $\N=0$ if and only if $<\N(X,Y),Z>=0$ for all $X,Y,Z$ in $\Gamma(TM)$. But \[<[X-iJX,Y-iJY],Z-iJZ>=<\N(X,Y),Z>-i<\N(X,Y),JZ>\text{.}\]
Hence, $\N=0\,\impl\,<[T^{10}M,T^{10}M],T^{10}M>=0\,\impl\,[T^{10}M,T^{10}M]\subseteq{T}^{10}M$ and, conversely, $[T^{10}M,T^{10}M]$ $\subseteq{T}^{10}M\,\impl\,0=\Real <[X-iJX,Y-iJY],Z-iJZ>=$ $<\N(X,Y),Z>\,\impl\,\N=0$, as wanted. The proof of (ii) is direct consequence of \eqref{Equation:SecondEquationIn:Lemma:Salamon1.2Generalized} and
\eqref{Equation:SecondEquationIn:Lemma:SalamonsLemma1.3} (see next lemma):
\[(\d\omega)^{12}=0\text{~if~and~only~if~}\d_2\omega=0\text{~if~and~only~if~}\nabla_{T^{01}M}T^{10}M\subseteq{T}^{10}M\text{.}\]
Finally, (iii) follows from the first characterization we gave of $(1,2)$-symplectic manifolds and from the fact that, on an almost Hermitian manifold $(M,g,J)$,
\begin{equation}\label{Equation:FirstEquationIn:Lemma:ProofThat1.2SymplecticAreCosymplecticThe01PartOfJTraceJ}
\big(J\,\trace\nabla{J}\big)^{01}=4\big(\sum_j\nabla_{\overline{Z}_j}Z_j\big)^{01}
\end{equation}
where $Z_j=\frac{1}{2}(X_j-iJX_j)$, with $\{X_j,JX_j\}$ an orthonormal frame for $TM$ (\cite{BairdWood:03}, Lemma 8.1.2). In fact, assuming that $(M,g,J)$ is $(1,2)$-symplectic, we have $(J\trace\nabla{J})^{01}=4\sum_j(\nabla_{\overline{Z}_j}Z_j)^{01}=0$ since $\nabla_{\overline{Z}_j}Z_j\in{T}^{10}M$ as $M$ is $(1,2)$-symplectic. Therefore, $J\trace\nabla J$ lies in $TM$ and has vanishing $(0,1)$-part, which by reality implies it must vanish and so must $\trace\,\nabla J$, as desired.
\end{proof}

The following result is proved using forms in \cite{Salamon:85} (Proposition 1.3); we give a different, more explicit, proof using vector fields:

\begin{proposition}\label{Proposition:SalamonsLemma1.3}
Let $(M,g,J)$ be an almost Hermitian manifold. Then
\begin{equation}\label{Equation:FirstEquationIn:Lemma:SalamonsLemma1.3}
\d_1\omega=0\text{ if and only if } J\text{ is integrable (\textit{i.e.}, $M$ is a Hermitian manifold).}
\end{equation}
\begin{equation}\label{Equation:SecondEquationIn:Lemma:SalamonsLemma1.3}
\begin{array}{l}
\d_2\omega=0\text{ if and only if }(\d\omega)^{12}=0\text{ if and only if }\nabla_{T^{01}M}T^{10}M\,\subseteq\,T^{10}M \\
\text{(\textit{i.e.,~}$(M,g,J)$ is a $(1,2)$-symplectic manifold).} \
\end{array}
\end{equation}
\end{proposition}
Notice that there is not a version of this lemma for a general vector bundle $V$, in contrast with Lemmas
\ref{Lemma:SalamonsLemma.1.1Generalized} and \ref{Lemma:Salamon1.2Generalized}.
\begin{proof}
Let us start by proving \eqref{Equation:FirstEquationIn:Lemma:SalamonsLemma1.3}: by Proposition
\ref{Proposition:AnAlternativeCharacterizationOfAlmostComplexManifolds} (i), we know that $J$ is integrable if and only if $T^{10}M$ is stable under the Lie bracket. Now, if $\d_1 \omega=0$, using Lemma \ref{Lemma:Salamon1.2Generalized}, we know that $\nabla_{T^{10}M}T^{10}M\subseteq T^{10}M$ and so $[X^{10},Y^{10}]=\nabla_{X^{10}}Y^{10}-\nabla_{Y^{10}}X^{10}\in{T}^{10}M$ which proves the integrability of $J$. Conversely, if $J$ is integrable, $\N=0$ and consequently $[T^{10}M,T^{10}M]\subseteq{T}^{10}M$ so that we obtain\addtolength{\arraycolsep}{-1.2mm}
\[\begin{array}{rll}0&=&<[X^{10},Y^{10}],Z^{10}>\\
~&=&<\nabla_{X^{10}}Y^{10},Z^{10}>-Y^{10}<X^{10},Z^{10}>+<X^{10},\nabla_{Y^{10}}Z^{10}>.\\
\multicolumn{3}{l}{\text{\put(-47,0){Since $<X^{10},Z^{10}>=0$ and $[Y^{10},Z^{10}]\in{T}^{10}M$, we obtain}}}\\
0&=&<\nabla_{X^{10}}Y^{10},Z^{10}>+<X^{10},[Y^{10},Z^{10}]+\nabla_{Z^{10}}Y^{10}>\\
~&=&<\nabla_{X^{10}}Y^{10},Z^{10}>+Z^{10}<X^{10},Y^{10}>-<\nabla_{Z^{10}}X^{10},Y^{10}>\\
~&=&<\nabla_{X^{10}}Y^{10},Z^{10}>-<[Z^{10},X^{10}]+\nabla_{X^{10}}Z^{10},Y^{10}>\\
~&=&<\nabla_{X^{10}}Y^{10},Z^{10}>-X^{10}<Z^{10},Y^{10}>+<Z^{10},\nabla_{X^{10}}Y^{10}>\\
~&=&2<\nabla_{X^{10}}Y^{10},Z^{10}>\end{array}\]\addtolength{\arraycolsep}{1.2mm}\noindent
which implies that $\nabla_{X^{10}}Y^{10}\in{T}^{10}M$, as desired. To prove the second statement: from Lemma \ref{Lemma:Salamon1.2Generalized}, we know that $\d_2\omega=0$ if
and only if $\nabla_{T^{01}M}T^{10}M\subseteq T^{10}M$. On the other hand, using \eqref{Equation:SecondEquationIn:Lemma:SalamonsLemma1.1Generalized}, we know that $\big(\nabla_{Y^{01}}\omega\big)(X^{10},Z^{01})=0$ for all $Y^{01},Z^{01}$ in $T^{01}M$, $X^{10}$ in $T^{10}M$. Thus,\addtolength{\arraycolsep}{-1.0mm}
\[\begin{array}{ll}~&(\d\omega)^{12}=0\,\equi\,\d\omega(X^{10},Y^{01},Z^{01})=0\\
\equi&\big(\nabla_{X^{10}}\omega\big)(Y^{01},Z^{01})-\big(\nabla_{Y^{01}}\omega\big)(X^{10},Z^{01})+\big(\nabla_{Z^{01}}\omega\big)(X^{10},Y^{01})=0\\
\equi&\d_2^1\omega=0\,\equi\text{~(as~in~the~first~part~of~the~proof~of~Lemma~\ref{Lemma:Salamon1.2Generalized})~}\d_2\omega=0,
\end{array}\]\addtolength{\arraycolsep}{1.0mm}\noindent
concluding our proof.
\end{proof}

As a consequence of the two previous results, we get the diagram (\cite{Salamon:85})

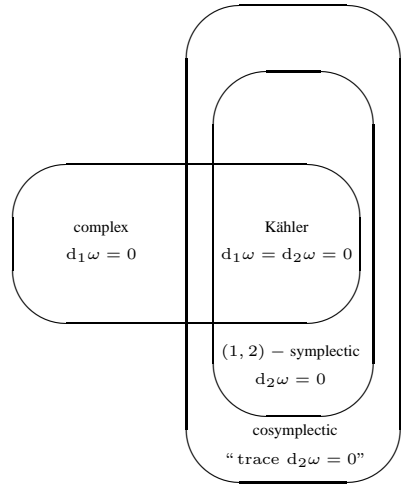
\begin{picture}(100,100)(0,-10)
\put(200,00){\oval(60,130)} \put(160,00){\oval(130,60)}
\put(200,00){\oval(80,180)}
\put(169,-77){\text{\tiny{$\begin{array}{c}
  \text{cosymplectic} \\
  ``\trace\,\d_2\omega=0"
\end{array}$}}} \put(110,00){\text{\tiny{$\begin{array}{c}
  \text{complex} \\
  \d_1\omega=0
\end{array}$}}}
\put(168,00){\text{\tiny{$\begin{array}{c}
  \text{Kähler} \\
  \d_1\omega=\d_2\omega=0
\end{array}$}}}
\put(168,-47){\text{\tiny{$\begin{array}{c}
  (1,2)-\text{symplectic} \\
  \d_2\omega=0
\end{array}$}}}
\end{picture}
\vspace{2.5cm}


\section{The bundle \texorpdfstring{$\Sigma^+M$}{Sigma+M}}\label{Section:TheBundleSigmaPlusM}


If $(M,g)$ is an oriented even-dimensional Riemannian manifold, for each $x$ on $M$ we can set $E=T_xM$ and proceed as in Section \ref{Section:SectionIn:Chapter:TwistorSpaces:HermitianStructuresonAVectorSpace} to define $\Sigma^+ T_x M$. We can then consider the total bundle\addtolength{\arraycolsep}{-1.0mm}
\begin{equation}\index{$\Sigma^+M$}\index{Twistor~space!of~a~Riemannian~manifold}
\begin{array}{lll}\Sigma^+M&=&\SO(M)\times_{\SO(2m)}\Sigma^+\rn^{2m}=\SO(M)\times_{\SO(2m)}\SO(2m)/\U(2m)\\
~&=&\{(x,J_x),\,x\in{M},\,J_x\in\Sigma^+T_xM\}\end{array}
\end{equation}\addtolength{\arraycolsep}{1.0mm}\noindent
whose fibre at $x$ is precisely $\Sigma^+ T_xM$ and whose projection map $\pi:\Sigma^+M\to M$ is defined by $\pi(x,J_x)\to x$. This is a subbundle of $\L(TM,TM)$ and the connection in the latter defines a splitting of $T\Sigma^+M$ into its \emph{horizontal} and \emph{vertical} parts: taking $\H_{(x,J_x)}$ as the set
\begin{equation}\label{Equation:DefinitionOfTheHorizontalSpaceOfTheTwistorSpace}\index{Decomposition!of~$T\Sigma^+M$}\index{$\H_{(x,J_x)}$}\index{$\V_{(x,J_x)}$}
\H_{(x,J_x)}=\left\{\d\sigma_x(X),\begin{array}{l}\sigma\text{~smooth~section\footnotemark~of~}\Sigma^+M\text{~with}\\
\nabla^{\L(TM,TM)}_{X_x}\sigma=0\text{~and~}\sigma(x)=(x,J_x)\end{array}\right\}
\end{equation}
\footnotetext{Hence, also a smooth section of $\L(TM,TM)$.}and $\V_{(x,J_x)}$ as the kernel of $\d\pi_{(x,J_x)}$, we have
\begin{equation}
T_{(x,J_x)}\Sigma^+=\H_{(x,J_x)}\oplus\V_{(x,J_x)}\text{.}
\end{equation}
For this decomposition, $\d\pi_{(x,J_x)}:T_{(x,J_x)}\Sigma^+M\to{T}_xM$ is an isomorphism when restricted to $\H_{(x,J_x)}$. In particular, we can define a complex structure on $\H_{(x,J_x)}$ by transporting the almost complex structure $J_x$ on $T_xM$:
\begin{equation}\label{Equation:DefinitionOfTheAlmostComplexStructureOnH}
\J^\H_{(x,J_x)}=\d\pi_{(x,J_x)}|_{\H}^{-1}\circ J_x\circ\d\pi_{(x,J_x)}|_\H\text{.}
\end{equation}
On the other hand, $\V_{(x,J_x)}$ is the tangent space to the fibre through $(x,J_x)$ at $J_x$. Hence, $\V_{(x,J_x)}\simeq T_{J_x}\Sigma^+T_x M=\m_{J_x}(T_xM)$ has the complex structure $\J^\V$ of Definition \ref{Definition:AlmostComplexStructureOnMJV}. We can define two almost complex structures $\J^1$ and $\J^2$ on $T_{(x,J_x)}\Sigma^+M$:
\begin{equation}\label{Equation:DefinitionOfTheAlmostComplexStructuresJ1AndJ2OnSigmaPlusM}
\begin{array}{ll}
\J^1=\left\{\begin{array}{ll}\J^\H\text{ on }\H \\
\J^\V\text{ on }\V\end{array}\right. \text{ and }
\J^2=\left\{\begin{array}{l}\,\,\J^\H\,\text{ on }\H \\
-\J^\V\text{ on }\V\text{.}\end{array}\right.
\end{array}
\end{equation}
It follows from the above definition that $\pi$ is a ``holomorphic map", in the sense that
\begin{equation}\label{Equation:HolomorphicityOfTheProjectionMap}\index{Holomorphic!projection~map~from~$\Sigma^+M$}
\d\pi_{(x,J_x)}(\J^a X)=J_x\d\pi_{(x,J_x)}X,\quad\forall\,X\in{T}\Sigma^+ M\text{,}
\end{equation}
as on the horizontal space $\J^a_{(x,J_x)}$ is transported from $J_x$ \textit{via} $\pi$ and $\V_{(x,J_x)}$ is precisely the kernel of $\d\pi_{(x,J_x)}$.

\index{$\sigma_J$,~associated~section}\index{$J_\sigma$,~associated~almost~Hermitian~structure}\index{Associated!almost~Hermitian~structure}\index{Associated!section}
An almost Hermitian manifold $(M,g,J)$ is thus a Riemannian manifold $(M,g)$ equipped with a smooth section $J$ of the bundle $\L(TM,TM)$ with $J^2=-Id$ and $J_x$ positive at each $x$ on $M$. To each such $J$ corresponds a section $\sigma_J$ of the bundle $\Sigma^+M$ and we shall refer to the latter as the \emph{associated section}. Similarly, if $\sigma$ is a section of $\Sigma^+M$, there is a corresponding section $J_\sigma$ of $\L(TM,TM)$ and we shall call $J_\sigma$
the \emph{almost Hermitian structure defined by} $\sigma$.

If $\sigma_J$ is a section of $\Sigma^+M$ with associated almost Hermitian structure $J$, identifying $\V_{(x,J_x)}$ with
$T_{J_x}\Sigma^+T_xM$, we have (see \textit{e.g.} \cite{Salamon:85}, p. 182))
\begin{equation}\label{EquationIn:Remark:DecompositionOfASectionIntoVerticalAndHorizontalParts}
(\d\sigma_J X)^\V=\nabla^{\L(TM,TM)}_X J\text{.}
\end{equation}
More precisely, given such a section $\sigma_J$,
\begin{equation}\label{Equation:DecompositionIntoHorizontalAndVerticalPartsOfTheDerivativeOfASection}\index{Decomposition!of~a~section~of~$\Sigma^+M$}
\d\sigma_{J_x}(X_x)=\{\d\sigma_{J_x}(X_x)-\nabla_{X_x}J\}\oplus\nabla_{X_x}J\in\H\oplus\V
\end{equation}
gives the decomposition of $\d\sigma_{J_x}(X_x)$ into its horizontal and vertical parts.

\begin{remark}
Another way of defining these structures is the following: consider the usual canonical isomorphism between $\L(TM,TM)$ and $\L^2(TM,\rn)$ given by the metric,\addtolength{\arraycolsep}{-1.2mm}
\begin{equation}\label{Equation:IsomorphismTildegGivenByTheMetric}\index{Metric!isomorphism~induced~by}\index{$\tilde{g}$}
\begin{array}{lcll}
\tilde{g}_x: & \L(T_xM,T_xM)&\to&\L(T_xM\times T_xM,\rn)\\
~ & \lambda_x&\to&\tilde{g}_x(\lambda_x):T_xM\times T_xM \to \rn\\
~ & ~ & ~& \hspace{15mm} (X_x\hspace{4mm},\hspace{3mm} Y_x)\,\to\tilde{g}(\lambda_x)(X_x,Y_x)=g_x(\lambda_x X_x,Y_x)\text{.}
\end{array}
\end{equation}\addtolength{\arraycolsep}{1.2mm}\noindent
Then, $\lambda_x\in\so(T_xM)$ if and only if $\tilde{g}_x(\lambda_x)$ is skew-symmetric and such a $\lambda_x$ lies in $\m_{J_x}(T_xM)$ if and only if, considering the complex bilinear extended map $(\tilde{g}_x(\lambda_x))^\cn$, we have $(\tilde{g}_x(\lambda_x))^\cn(X_x^{10},Y_x^{01})=0$. Equivalently, under the isomorphism $\tilde{g}$,
\begin{equation}\label{Equation:mJTxMAsT20StarPlusT02Star}
\m_{J_x}(T_xM)\simeq T^{20^\star}_xM\oplus T^{02^\star}_xM\text{.}
\end{equation}
Hence, we can define $\J^\V_{(x,J_x)}$ as acting as $i$ on $T^{02^\star}_xM$ and as $-i$ on $T^{20^\star}_xM$ and then define $\J^1$ and $\J^2$ just as before. It is easy to check that both definitions agree.
\end{remark}

\begin{remark}\label{Remark:TheCaseOfGeneralVectorBundles}\index{Almost~complex~structure!on~$\Sigma^+V$}\index{$\Sigma^+V$!almost~complex~structures~$\J^a$}
Imagine that we have a vector bundle $V$ on an almost Hermitian manifold $M$. We could then introduce two almost complex structures $\J^a$ on $\Sigma^+ V$, defined as before on the vertical spaces but very differently in the horizontal spaces. In fact, we could identify each horizontal space $\H_{(x,J_x)}$ with $T_xM$ but now the only natural almost complex structure is precisely the one from $M$ at $x$, since $J_x\in\Sigma^+V_x$ is not even defined over $T_xM$. So, we would get an almost complex structure whose horizontal part ``does not change along the fibre" and would therefore be in a much more rigid situation. We shall come back to this topic in Section \ref{Section:SectionIn:TwistorSpacesAndHarmonicMaps:FundamentalLemma}.
\end{remark}


\subsection{Holomorphic sections of \texorpdfstring{$\Sigma^+M$}{Sigma+M}}


The goal of this section is to give a more transparent proof of the following theorem (see \cite{Salamon:85}, Proposition 3.2 for the equivalence between (i) and (ii)):

\begin{theorem}[\textup{Holomorphic~sections~of~$\Sigma^+M$}]\label{Theorem:SalamonsTheorem3.2}\index{Holomorphic!sections~of~$\Sigma^+M$}
Let $(M,g)$ be an oriented even-dimensional Riemannian manifold as before and consider its positive twistor bundle, $\Sigma^+M$. For each local section $\sigma$ of this bundle, the following conditions are equivalent:

\textup{(i)} $\d_a\omega_\sigma=0$, where $\omega_\sigma$ is the fundamental $2$-form associated with the almost Hermitian structure defined by $\sigma$.

\textup{(ii)} $\sigma$ is a holomorphic map from $(M,J_\sigma)$ to $\big(\Sigma^+M,\J^a\big)$, where $J_\sigma$ is the almost Hermitian structure defined by $\sigma$ on $M$.

\textup{(iii)} $\sigma$ is a $\J^a$\emph{-stable map}; \textit{i.e.}, $\d\sigma(TM)$ is a $\J^a$-stable subspace of
$T\Sigma^+M$ (see Definition \ref{Definition:HolomorphicityInCertainSubbundles} below).\index{Stable!map}\index{$\J^a$!-stable map}
\end{theorem}\label{Page:CommentsTo:Theorem:SalamonsTheorem3.2:TheImportanceOfTheConditionDaEqualsZeroInsteadOfJaHolomorphicity}

The proof of this theorem will be done in several steps. Since some of them are important on their own, we shall present them as lemmas. As we already noticed in Remark \ref{Remark:TheCaseOfGeneralVectorBundles}, when we deal with the general case of a vector bundle $V$ over the manifold $M$, there is no ``canonical" almost complex structure on $\Sigma^+V$; however, conditions $\d_a\sigma=0$ still arise in a very natural way, as we saw in Section
\ref{Section:SectionIn:Chapter:TwistorSpaces:TypesOfAlmostHermitianManifolds}. Indeed, we shall see in Chapter \ref{Chapter:HarmonicMapsAndTwistorSpaces} that these last conditions are the correct generalization to bundles $\Sigma^+V$ of the conditions $\J^a$-holomorphicity on the special case $V=TM$. We begin with a definition that will be useful in the sequel:
\begin{definition}\label{Definition:HolomorphicityInCertainSubbundles}\index{$\H$-holomorphic!map}\index{$\V$-holomorphic!map}\index{Holomorphic!in~subbundles}\index{Stable!subbundle}\index{Stable!decomposition} Let $(M,J^M)$ and $(Z,J^Z)$ be two almost complex manifolds. Let $TZ=\H\oplus\V$
be a decomposition of $TZ$ into $J^Z$\emph{-stable} subbundles;
\textit{i.e.}, $J^Z\H\subseteq \H$ and
$J^Z\V\subseteq\V$\footnote{Equivalently, $J^Z\H=\H$ and
$J^Z\V=\V$.}. We shall call such a decomposition into stable subbundles a ($J^Z$-)\emph{stable decomposition}. Let $\psi:M\to Z$ be a smooth map. We shall say that $\psi$ is $\H$\emph{-holomorphic} (or $(J^M,J^Z)$\emph{-horizontally holomorphic}) if
\begin{equation}\label{Equation:EquationIn:Definition:HolomorphicityInCertainSubbundles}
(\d\psi (J^MX))^\H=J^Z (\d\psi X)^\H,\quad\forall\,X\in{T}M.
\end{equation}
Changing $\H$ to $\V$ in the above equation gives the definition for ``$\psi$ is $\V$\emph{-holomorphic}" (or $(J^M,J^Z)$\emph{-vertically holomorphic}).
\end{definition}
A smooth map $\psi:M\to Z$ is holomorphic if and only if it is both $\H$ and $\V$-holomorphic for some, and so any, stable decomposition $TZ=\H\oplus\V$. Taking $Z=\Sigma^+M$, the decomposition $T\Sigma^+M=\H\oplus\V$ is stable for both the almost complex structures $\J^1$ and $\J^2$ on $\Sigma^+M$, from their very definition.

\begin{lemma}[\textup{Horizontal~holomorphicity~of~any~section~of~$\Sigma^+M$}]\label{Lemma:HorizontalHolomorphicityOfAnySection}\index{Horizontal~holomorphicity~of~any~section}\index{$\H$-holomorphic!sections}
Let $\sigma$ be any local section of $\Sigma^+M$. Then, $\sigma$ is always $(J_\sigma,\J^a)$-horizontally holomorphic, $a=1,2$, where $J_\sigma$ is the almost complex structure defined by $\sigma$ on $M$ (more precisely, on the open set $\openU$ of $M$ where $\sigma$ is defined). In other words,
\begin{equation}\label{Equation:EquationIn:Lemma:HorizontalHolomorphicityOfAnySection}
(\d\sigma(J_{\sigma}X))^\H=\J^{\H,a}(\d\sigma{X})^\H,\quad\forall\,X\in{T}M,\,a=1,2\text{.}
\end{equation}
\end{lemma}
\begin{proof}
Let $\sigma$ be a smooth section of $\Sigma^+M$ and $x\in{M}$. Since the horizontal part of $\J^a$ does not depend on $a=1,2$, we can just write $\J^\H$. Since $\d\pi_{\sigma(x)}$ vanishes on the vertical space, we have $\d\pi_{\sigma(x)}|_\H\big((\d\sigma_xX_x)^\H\big)=\d\pi_{\sigma(x)}(\d\sigma_x X_x)=\d(\pi\circ\sigma)_xX_x=X_x$ for all $X_x\in{T}_xM$ so that
\[\big(\d\sigma_x(J_{\sigma_x}X_x)\big)^\H=\d\pi_{\sigma(x)}|_\H^{-1}\big(\d\pi_{\sigma(x)}|_\H\big(\d\sigma_x(J_{\sigma_x}X_x)\big)^\H\big)=\d\pi_{\sigma(x)}|_\H^{-1}(J_{\sigma_x}X_x)\text{.}\]
On the other hand,
\[\J^{\H}(\d\sigma_xX_x)^\H=\d\pi_{\sigma(x)}|_\H^{-1}\big(J_{\sigma_x}\d\pi_{\sigma(x)}|_\H(\d\sigma_xX_x)^\H\big)=\d\pi_{\sigma(x)}|_\H^{-1}(J_{\sigma_x}X_x\big)\text{.}\]
Hence, equation \eqref{Equation:EquationIn:Lemma:HorizontalHolomorphicityOfAnySection} holds, as desired.
\end{proof}

\begin{lemma}[\textup{Vertical~$\J^a$-holomorphicity~of~a~section~and~vanishing~of~$\d_a$}]\label{Lemma:VerticalJaHolomorphicityNadNullityOfda} Let $\sigma$ be a smooth local section of $\Sigma^+M$. Then, $\sigma$ is $(J_\sigma,\J^a)$-vertically holomorphic if and only if $\d_a\omega_\sigma=0$, where $\omega_\sigma$ denotes the fundamental $2$-form associated with the almost complex structure $J_\sigma$;
\textit{i.e.},
\begin{equation}\label{Equation:EquationIn:Lemma:VerticalJaHolomorphicityNadNullityOfda}
(\d\sigma(J_{\sigma}X))^\V=\J^{\V,a}(\d\sigma{X})^\V,\quad\forall\,X\in{T}M,\,a=1,2\text{.}
\end{equation}
\end{lemma}\index{$\V$-holomorphic!sections~of~$\Sigma^+M$}\index{$\Sigma^+M$!$\V$-holomorphic~sections}
Notice that, using \eqref{Equation:DefinitionOfTheAlmostComplexStructuresJ1AndJ2OnSigmaPlusM},
\eqref{EquationIn:Remark:DecompositionOfASectionIntoVerticalAndHorizontalParts} and \eqref{Equation:EquationIn:Definition:AlmostComplexStructureOnMJV} we can rewrite equation
\eqref{Equation:EquationIn:Lemma:VerticalJaHolomorphicityNadNullityOfda} as
\begin{equation}\label{Equation:FirstEquationAfter:Lemma:VerticalJaHolomorphicityNadNullityOfda}
\nabla_{JX}J=(-1)^{a+1}J\nabla_XJ,\quad\forall\,X\in{T}M,\,a=1,2\text{.}
\end{equation}
\begin{proof}
Equation \eqref{Equation:EquationIn:Lemma:VerticalJaHolomorphicityNadNullityOfda} holds if and only if for each $x$ on $M$, $(\d\sigma_xX_x^{10,J_\sigma})^\V\in\V^{10,a}$. On the other hand, we know that $(\d\sigma_xX^{10}_x)^\V=\nabla_{X^{10}_x} J_\sigma$ and that under the isomorphism $\tilde{g}$ in
\eqref{Equation:IsomorphismTildegGivenByTheMetric}, it will belong to $\V^{10,a}$ if and only if it belongs to $T^{02^\star}_xM$ (if $a=1$) or to $T^{20^\star}_xM$ (if $a=2$), where the decompositions are relative to the almost Hermitian structure $J_\sigma$. Hence, in the case $a=1$ we have\addtolength{\arraycolsep}{-1.2mm}
\[\begin{array}{ll}
~&\sigma\text{ is vertically holomorphic if and only if }\tilde{g}\big(\nabla_{X^{10}}J_\sigma\big)\in{T}^{02^\star}_xM\\
\equi&
 \left\{\begin{array}{llll}
       <(\nabla_{X^{10}}J)Y,Z>&=&-<(\nabla_{X^{10}}J)Z,Y>&\multirow{2}{*}{$\left\}\begin{array}{l}~\\~\end{array}\right.\text{(A,~always~true)}$}\\
       <(\nabla_{X^{10}}J)Y^{10},Z^{01}>&=&0&~ \\
       <(\nabla_{X^{10}}J)Y^{10},Z^{10}>&=&0&~
       \end{array}\right. \\
\equi&<(\nabla_{X^{10}}J)Y^{10},Z^{10}>=0\equi\d_1^1\omega_\sigma=0\equi\d_1\omega_\sigma=0\text{,}\end{array}\]\addtolength{\arraycolsep}{1.2mm}\noindent
where we have used\addtolength{\arraycolsep}{-1.0mm}
\begin{equation}\label{Equation:DerivativeOfJAndOmega}
\begin{array}{lll}<(\nabla_XJ)Y,Z>&=&X<JY,Z>-<JY,\nabla_XZ>-<J\nabla_XY,Z>\\
~&=&(\nabla_X\omega_J)(Y,Z)\text{.}\end{array}
\end{equation}\addtolength{\arraycolsep}{1.0mm}\noindent
Thus, we are left with proving the equations in (A). We could argue that these equations are always true since $\nabla_{X}J$ belongs to $T_{J_x}(\Sigma^+T_xM)$ (see \eqref{Equation:mJTxMAsT20StarPlusT02Star}) or we can just check it
directly: as a matter of fact, for all $X,Y,Z$ in $\Gamma(T^{\cn}M)$,\addtolength{\arraycolsep}{-1.0mm}
\[\begin{array}{rll}<(\nabla_X J)Y,Z>&=&<\nabla_XJY-J\nabla_XY,Z>\\
~&=&-X<Y,JZ>+<Y,J\nabla_XZ>+\\
~&~&+X<Y,JZ>-<Y,\nabla_XJZ>\\
&=&-<Y,(\nabla_XJ)Z>,\\
<(\nabla_{X^{10}}J)Y^{10},Z^{01}>&=&i<\nabla_{X^{10}}Y^{10},Z^{01}>+<\nabla_{X^{10}}Y^{10},JZ^{01}>=0\text{,}\end{array}\]\addtolength{\arraycolsep}{1.0mm}\noindent
since $(JZ^{01}=-iZ^{01})$. For the case $a=2$, we have\addtolength{\arraycolsep}{-1.2mm}
\[\begin{array}{ll}
~&\sigma\text{~is~vertically~holomorphic~if~and~only~if~}\tilde{g}\big(\nabla_{X^{10}}J_\sigma\big)\in{T}^{20^\star}_xM\\
\equi&\left\{\begin{array}{llll}
<(\nabla_{X^{10}}J)Y,Z>&=&-<(\nabla_{X^{10}}J)Z,Y>&~\multirow{2}{*}{$\left\}\begin{array}{l}~\\~\end{array}\right.\text{(B,~always~true)}$}\\
<(\nabla_{X^{10}}J)Y^{01},Z^{10}>&=&0&~ \\
<(\nabla_{X^{10}}J)Y^{01},Z^{01}>&=&0&~
\end{array}\right.\\
\equi&<(\nabla_{X^{10}}J)Y^{01},Z^{01}>=0\equi\d_2^1\omega_\sigma=0\equi\d_2\omega_\sigma=0\text{.}\end{array}\]\addtolength{\arraycolsep}{1.2mm}\noindent
We can then prove (B) using similar arguments to those used to prove (A), concluding the proof.
\end{proof}

We are now ready to prove Theorem \ref{Theorem:SalamonsTheorem3.2}:

\noindent\textit{Proof of Theorem \ref{Theorem:SalamonsTheorem3.2}.}
Since we are dealing with sections of the bundle $\Sigma^+M$, they are always $\H$-holomorphic, by Lemma \ref{Lemma:HorizontalHolomorphicityOfAnySection}. Hence, given a section $\sigma$, it will be $(J_\sigma,\J^a)$-holomorphic as a map $M\to\Sigma^+M$ if and only if it is vertically holomorphic which, using Lemma \ref{Lemma:VerticalJaHolomorphicityNadNullityOfda}, is equivalent to the condition $\d_a\omega_\sigma=0$. Therefore, (i)
and (ii) are equivalent. If (ii) holds, $\sigma$ is a $\J^a$-holomorphic map and therefore $\J^a$-stable; consequently,
(iii) is satisfied. Conversely, let us prove that (iii) implies (ii).

Suppose then that we have a $\J^a$-stable smooth section $\sigma$. As $\sigma$ is a section, Lemma \ref{Lemma:HorizontalHolomorphicityOfAnySection} guarantees that it is $\H$-holomorphic. Hence, we are left with checking that $\J^a$-stability implies $\J^a$-vertical holomorphicity. We know, since $\sigma$ is $\J^a$-stable, that for each vector $Y_x=J_\sigma{X}_x\in{T}_xM$ there is a new vector $Z_x\in{T}_xM$ such that $\J^a(\d\sigma_xY_x)=\d\sigma_xZ_x$. On the other hand, since $\sigma$ is $\H$-holomorphic, $(\d\sigma_x(J_\sigma{X}_x)\big)^\H=\J^\H(\d\sigma_x X_x)^\H$. Therefore, we have\addtolength{\arraycolsep}{-1.0mm}
\[\begin{array}{llll}
~& \J^a(\d\sigma_x(J_\sigma X_x))=\d\sigma_x(Z_x)&\impl&\J^{\H}(\d\sigma_x(J_\sigma X_x))^\H=(\d\sigma_xZ_x)^\H \\
\impl&(-\d\sigma_xX_x)^\H=(\d\sigma_xZ_x)^\H&~&~\\
\impl&\multicolumn{3}{l}{\d\pi_{\sigma(x)}|_\H^{-1}(\d\pi_{\sigma(x)}(\d\sigma_x(-X_x))^\H)=\d\pi_{\sigma(x)}|_\H^{-1}(\d\pi_{\sigma(x)}(\d\sigma{Z}_x)^\H)}\\
\impl&\multicolumn{3}{l}{\text{(as~$\d\pi_{\sigma(x)}|_\H$~is~an~isomorphism)~}\,\d\Id_{M_x}(-X_x)=\d\Id_{M_x}(Z_x)\,\impl\,X_x=-Z_x}
\end{array}\]\addtolength{\arraycolsep}{1.0mm}\noindent
and consequently $\J^a(\d\sigma_x(J_\sigma{X}_x))=\d\sigma_x(-X_x)=\J^a.\J^a\d\sigma_x X$ so that $\d\sigma_x(J_\sigma{X}_x)=\J^a\d\sigma_x X_x$, as desired.\qed

\begin{remark}\label{Remark:TheImportanceOfBeingDealingWithSectionsInSalamonTheorem3.2}
Notice that it is very important that we are dealing with sections of this bundle: this is what allows us to deduce the horizontal holomorphicity and consequently establish vertical holomorphicity only knowing $\J^a$-stability.
\end{remark}

From Theorem \ref{Theorem:SalamonsTheorem3.2} and Proposition \ref{Proposition:SalamonsLemma1.3}, we deduce the following characterization of Hermitian and $(1,2)$-symplectic manifolds (see \cite{Salamon:85}):

\begin{proposition}[\textup{Types~of~manifolds~and~$\J^1$,~$\J^2$-holomorphic~sections}]\label{Proposition:CorollaryToSalamonTheorem3.2ComplexAnd12SymplecticVersusJ1AndJ2Holomorphicity}
Let $(M,g,J^M)$ be an almost Hermitian manifold and consider its
twistor space $\Sigma^+M$. Consider the section $\sigma_J$
associated with the almost complex structure $J^M$. Then:
\begin{equation}\label{Equation:FirstEquationIn:Proposition:CorollaryToSalamonTheorem3.2ComplexAnd12SymplecticVersusJ1AndJ2Holomorphicity}
M\text{~is~Hermitian~if~and~only~if~}\sigma_{J}:M\to\Sigma^+M\text{~is~}(J_\sigma,\J^1)\text{-holomorphic.}
\end{equation}
\begin{equation}\label{Equation:SecondEquationIn:Proposition:CorollaryToSalamonTheorem3.2ComplexAnd12SymplecticVersusJ1AndJ2Holomorphicity}
M\text{~is~$(1,2)$-symplectic~if~and~only~if~}\sigma_{J}:M\to\Sigma^+M\text{~is~}(J_\sigma,\J^2)\,\text{-holomorphic.}
\end{equation}
\end{proposition}\label{Page:CommentsToProposition:CorollaryToSalamonTheorem3.2ComplexAnd12SymplecticVersusJ1AndJ2Holomorphicity}\index{$(1,2)$-symplectic~manifold!\textit{via}~covariant~derivative}
Notice that, as any section is $\H$-holomorphic (Lemma \ref{Lemma:HorizontalHolomorphicityOfAnySection}), we only need to
be concerned with $\V$-holomorphicity. The latter condition can be expressed as \eqref{Equation:FirstEquationAfter:Lemma:VerticalJaHolomorphicityNadNullityOfda}:
\[\nabla_{JX}J=(-1)^{a+1}J\nabla_XJ,\quad\forall\,X\in{T}M,\,a=1,2\text{,}\]
so that $J$ is integrable if and only if $\nabla_{JX}J=J\nabla_XJ,\quad\forall\,X\in{T}M$ and $(1,2)$-symplectic
if and only if $\nabla_{JX}J=-J\nabla_XJ,\quad\forall\,X\in{T}M$.

Theorem \ref{Theorem:SalamonsTheorem3.2} also allows an immediate proof of the following well-known fact (see \cite{BairdWood:03}, Theorem 7.1.3 (iii), p. 209):
\begin{corollary}\label{Corollary:Theorem7.1.3iiiBairdWoodComplexSubmanifoldsOfSigmaPlusMAndTheStructureJ1}
There is a one-to-one correspondence between locally defined (integrable) Hermitian structures on $M$ and complex submanifolds of $(\Sigma^+M,\J^1)$ which are diffeomorphically mapped onto $M$ via the projection map $\pi$.
\end{corollary}
\begin{proof}Proving this result based on our previous discussion is easy: it is obvious that each integrable structure $J$ gives rise to a submanifold of $\Sigma^+M$ that is mapped diffeomorphically by $\pi$ into $M$: take $\sigma_J$ to be the corresponding section and $\sigma_J(\openU)$ (and conversely). Moreover, since $J$ is integrable if and only if $\sigma_J$ is $(J,\J^1)$-holomorphic, it is obvious that these submanifolds must be $\J^1$-stable. Finally, since $J$ is integrable and $\sigma_J$ holomorphic, its image must be a (complex) submanifold on which the restriction of $\J^1$ is integrable.
\end{proof}

Notice that, in general, $\J^1$ is not an integrable complex structure on $\Sigma^+M$; however, when we refer to ``complex submanifolds" we are indeed claiming that we have a submanifold on which (the restriction of) $\J^1$ is integrable\footnote{In fact, $\J^1$ is integrable on $\Sigma^+M^{2m}$ if and only if $M^{2m}$ is conformally flat ($m\geq3$) or anti-self-dual ($m=2$). As for $\J^2$, it is never integrable. For more details, see \cite{AtiyahHitchinSinger:78}, \cite{OBrianRawnsley:85}, \cite{Salamon:85}; a discussion on this topic can also be found in \textit{e.g.} \cite{DavidovSergeev:93} and references therein.}.

\label{Page:RemarkTo:Corollary:Theorem7.1.3iiiBairdWoodComplexSubmanifoldsOfSigmaPlusMAndTheStructureJ1}\index{Integrability!of~$\J^1$~and~$\J^2$}\index{$\J^a$!integrability}\index{Complex~submanifold!of~$\Sigma^+M$~and~$\J^1$,~$\J^2$}\index{$\J^a$!complex~submanifolds}
Moreover, this same result does not hold for the $\J^2$ case. As a matter of fact, we can still say that there is a one-to-one correspondence between almost complex submanifolds of $(\Sigma^+M,\J^2)$ which are diffeomorphically mapped into $M$ and $(1,2)$-symplectic structures on $M$ but now we cannot deduce that these submanifolds are complex submanifolds of $(\Sigma^+M,\J^2)$: they are just the image under a holomorphic map of a $(1,2)$-symplectic manifold.


\section{The bundle \texorpdfstring{$\Sigma^+V^\cn$}{Sigma+Vc} and the Koszul-Malgrange Theorem}\label{Section:SectionOf:Chapter:TwistorSpaces:TheBundleSigmaPlusVCnAndKoszulMalgrangeTheorem}


In the sequel, we shall construct the bundle $\Sigma^+V$\index{$\Sigma^+V$}
\[\Sigma^+ V=\SO(V^\cn)\times_{\SO(\cn,2k)} G^+_{iso}(\cn^{2k})\]
for a given vector bundle $V$ over an almost complex manifold $M$. In addition, we shall introduce, under special circumstances, a complex structure on such a bundle. In this chapter we discuss those circumstances and the importance of the Koszul-Malgrange Theorem (\cite{KoszulMalgrange:58}, Theorem 1). In Chapter \ref{Chapter:JacobiVectorFieldsAndTwistorSpaces}, a parameter-dependent version of this last result will be necessary and we therefore start by establishing such a generalization.


\subsection{Parametric~Koszul-Malgrange~Theorem}\label{Subsection:SubsectionOf:Section:SectionOf:Chapter:TwistorSpaces:TheBundleSigmaPlusVCnAndKoszulMalgrangeTheorem:ParametricKoszulMalgrangeTheorem}


In what follows, $\G$ will denote a complex Lie group with $\g$ its Lie algebra and $M$ a complex manifold with complex dimension $m$. Given a $\g$-valued $1$-form $\alpha$ on $M$ we shall say that $\alpha$ is of type $(0,1)$ if $\alpha(\partial_{z_i})=0$ for all $1\leq i\leq m$. We shall denote by $\d\alpha^{02}$ the $\g$-valued $2$-form obtained by restriction of $\d\alpha$ to $T^{01}M\times{T}^{01}M$; in other words, $\d\alpha^{02}:T^{01}M\times{T}^{01}M\to\g$, $\d\alpha^{02}(\partial_{\bar{z}_i},\partial_{\bar{z}_j})=\d\alpha(\partial_{\bar{z}_i},\partial_{\bar{z}_j})$. We wish to find a (locally defined) function $f:M\to G$ with
\begin{equation}\label{Equation:TheInitialKoszulMalgrangeEquation}
f(z_0)=e\text{ and }f^{-1}\d f^{01}=\alpha
\end{equation}
where $\d f^{01}$ is the restriction of $\d f$ to $T^{01}M$. The usual Koszul-Malgrange Theorem (\cite{KoszulMalgrange:58}, Theorem 1) tells us that a solution to equation \eqref{Equation:TheInitialKoszulMalgrangeEquation} exists if and only if
\begin{equation}\label{Equation:TheInitialKoszulMalgrangeConditionToExistenceOfSolutions}
\d\alpha^{02}+[\alpha,\alpha]=0\text{.}
\end{equation}
We establish the following parametric version of this result:

\begin{theorem}[\textup{Parametric~Koszul-Malgrange~Theorem}]\label{Theorem:KoszulMalgrangeTheorem:ParametricVersion}\index{Koszul-Malgrange!Theorem!parametric~version}
Let $F$ be a (real) vector space and let $F_0$ be any subset of $F$ containing its origin. Let $\alpha$ be a smooth $\g$-valued $(0,1)$-form on $M$ defined on an open set $A\subseteq{F}_0\times M$ containing $(0,x_0)$ (\textit{i.e.},
$\alpha(t,x)\in\L(T_x^{01}M,\g)$ and $\alpha$ is smooth in $(t,x)$). Then, there is a smooth solution $f(t,x):\V\subseteq{F}_0\times M\to\G$ on a neighbourhood $\V\subseteq A$ containing $(0,x_0)$ to the equation
\begin{equation}\label{Equation:FirstEquationIn:Theorem:KoszulMalgrangeTheorem:ParametricVersion}
f^{-1}\d f_t^{01}=\alpha_t,\quad\forall\,t
\end{equation}
for the initial conditions $f_t(x_0)=e$ $\forall t$ if and only if
\begin{equation}\label{Equation:SecondEquationIn:Theorem:KoszulMalgrangeTheorem:ParametricVersion}
\d\alpha_t^{02}+[\alpha_t,\alpha_t]=0\quad\forall\,t\text{.}
\end{equation}
\end{theorem}
For a proof, see Section \ref{Section:SectionIn:Chapter:Addendums:KoszulMalgrangeVariationalTheorem}. Our proof of the existence of a holomorphic structure on the principal bundle $\SO(V^\cn)$ will depend on the existence of orthonormal holomorphic frames. This in turn depends on the existence of solutions to a certain system of differential equations, which can be guaranteed by the Koszul-Malgrange Theorem, as we shall see.
\index{Orthonormal~and~holomorphic~sections}\index{Complexified!bundle}\index{Orientation!preserving}\index{$\SO(V^\cn)$!}
Suppose, then, that $V$ is an oriented even-dimensional Riemannian real vector bundle with rank $2k$ over a manifold $M$ and form the \textit{complexified bundle}, $V^\cn=V\otimes\cn$. We can then consider the bundle
\[\SO(V^\cn):=\left\{\begin{array}{ll}(x,\lambda_x),&x\in{M},\,\lambda_x:\cn^{2k}\to{V}^\cn_x,\,\lambda\text{~$\cn$-linear}\\
~&\text{and preserves metric and orientation}\end{array}\right\}\]
where $\cn^{2k}$ has metric induced from $\rn^{2k}$ by complex bilinear extension and \emph{preserving orientation} means that, choosing the canonical basis $\{e_i\}_{i=1,...,2k}$ for $\rn^{2k}$ and another orthonormal oriented basis $\{U_j\}_{j=1,...,{2k}}$ for $V_x$, the determinant of the matrix of $\lambda_x$ in these basis equals $1$.

$\SO(V^\cn)$ is a principal bundle over $M$ with group $\SO(\cn,2k)$ in the natural way. Trivializations are given as usual: for $x\in{M}$, take $\{e_i\}_{i=1,...,2k}$ and $\{U_j\}_{j=1,...,2k}$ as before and\addtolength{\arraycolsep}{-1.2mm}
\[\begin{array}{lcll}
\eta: & \pi^{-1}(\openU)&\to&\openU\times\SO(\cn,2k) \\
~&(x,\lambda_x)&\to&(x,\mathcal{M}(\lambda_x,e_i,X_j):=\text{~matrix~of~}\lambda_x\,\text{~with~respect~to~the~basis~}\{e_i\},\{U_j\})\text{,}
\end{array}\]\addtolength{\arraycolsep}{1.2mm}\noindent
where $\pi:\SO(V^\cn)\to M$ is the canonical projection, $\pi(x,\lambda_x)=x$. Of course, we could have done the exact same construction without imposing that $V$ be even-dimensional but that will be the case that will interest us in the sequel.


\subsection{The Koszul-Malgrange Theorem and \texorpdfstring{the existence of holomorphic orthonormal frames}{holomorphic o.n. frames}}\label{Subsection:SubsectionOf:Section:SectionOf:Chapter:TwistorSpaces:TheBundleSigmaPlusVCnAndKoszulMalgrangeTheorem:TheKoszulMalgrangeTheoremAndTheExistenceOfHolomorphicOrthonormalFrames}


Let us start by stating our problem to show why the Koszul-Malgrange Theorem is important. We have seen that a differentiable structure can be introduced in the bundle $\SO(V^\cn)$ defining coordinates on this bundle by
\[(x,\lambda_x)\to (\varphi(x),\lambda_{ij}(x))\]
where $\varphi$ is a chart for the manifold and $\lambda_{ij}(x)$ are the entries of the matrix of $\lambda_x$ with respect to suitable bases $\{e_i\}$, $\{U_j\}$. However, even when $M$ is a complex manifold, we cannot be sure that we get a system of holomorphic charts for our bundle just by requiring $\varphi$ to be a holomorphic chart for the manifold. Suppose, however, that $M$ is a complex manifold and that around each point $z\in{M}$ we can choose a positive orthonormal frame for $V$ such that
\[\nabla^V_{\partial_{\bar{z}_i}} U_j=0,\quad\forall\,i,j\text{.}\]
In that case, choosing a system of charts for $\SO(V^\cn)$ as before, we shall get a holomorphic structure for our bundle. Indeed, if $\eta_{\varphi,U}\simeq(\varphi,U_j)$ and $\eta_{\psi,W}\simeq(\psi,W_j)$ are two such charts (\textit{i.e.},
$\varphi$ and $\psi$ are holomorphic charts for $M$ and $\{U_j\},\{W_j\}$ are two orthonormal positive frames with vanishing $\partial_{\bar{z}_i}$ covariant derivatives), then the transition map
$\eta_{\varphi,U_j}\circ\eta_{\psi,W_j}^{-1}:\openU\subseteq\cn^m\times\SO(\cn,2k)\to\V\times\SO(\cn,2k)$ is given by
\[\eta_{\varphi,U_j}\circ\eta_{\psi,W_j}^{-1}(z,\lambda_{ij})=\big(\varphi\circ\psi^{-1}(z),\tilde{\lambda}_{ij}(z)\big)\text{.}\]
Now, $\varphi\circ\psi^{-1}$ is holomorphic and $\tilde{\lambda}_{ij}(z)=\sum_k \lambda_{ik}<U_k(z),W_j(z)>$ is also holomorphic as a function of $\lambda_{ij}$ as well as as a function of $z$ since
\[\partial_{\bar{z}_i}<U_j(z),W_k(z)>=<\nabla^V_{\partial_{\bar{z}_i}}U_j,W_k>+<U_j,\nabla^V_{\partial_{\bar{z}_i}}W_k>=0\text{.}\]
Hence, if we can guarantee the existence of such local frames, $\SO(V^\cn)$ will become a holomorphic bundle over the complex manifold $M$. Now, let us see how to use the Koszul-Malgrange Theorem to get these frames. Let $M$ be a complex manifold and fix a local orthonormal and positive frame $\{W_i\}$ of $V$ around a point $z_0\in{M}$. We want a new frame $\{U_j\}$ for $V^\cn$ such that $<U_j,U_k>=\delta_{jk}$ and $\nabla_{\partial_{\bar{z}_i}}^VU_j=0$. Let $U_j=\sum_lU_{jl}W_l$ and define $\Gamma^n_{im}$ by
\[\nabla_{\partial_{\bar{z}_i}}W_m=\sum_n\Gamma^n_{im}W_n\text{.}\]
Since our frame $\{W_m\}$ is orthonormal we have
\begin{equation}\label{Equation:GammaijkAreSkewSymmetricInTheCaseForOnVectors}
\Gamma^n_{im}=<\nabla_{\partial_{\bar{z}_i}}W_m,W_n>=-<W_m,\nabla_{\partial_{\bar{z}_i}}W_n>=-\Gamma^m_{in}\text{.}
\end{equation}
With this notation fixed, the problem of finding our suitable frame $\{U_j\}$ transforms into that of finding complex functions $U_{jm}$ defined around $z_0$ for which
\[\begin{array}{l}<\sum{U}_{jm}W_m,\sum{U}_{kn}W_n>=\delta_{jk}\text{,~equivalently,}\,\sum_{mn}U_{jm}U_{km}=\delta_{jk}\text{~and}\\
\nabla_{\partial_{\bar{z}_i}}\big(\sum_mU_{jm}W_m\big)=0\text{,~equivalently}\,\partial_{\bar{z}_i}U_{jn}+\sum_mU_{jm}\Gamma^n_{im}=0.\end{array}\]
In matrix notation, we can rewrite the above equations as
\begin{equation}\label{Equation:TheEquationForONHolomorphicFrameOfVcnOnTheNonParametricCaseInMatricialNotation}
\begin{array}{l}U\cdot U^\top=\Id\text{~and}\\
\partial_{\bar{z}_i}U+U\cdot\Gamma_i=0\,\equi\,U^{-1}\partial_{\bar{z}_i}U=-\Gamma_i\,\equi\,U^{-1}\big\{\sum_i\partial_{\bar{z}_i}U\d\bar{z}_i\big\}=-\sum_i\Gamma_i\d\bar{z}_i.\end{array}
\end{equation}
Hence, writing $\alpha=-\sum_i\Gamma_i\d\bar{z}_i$, we are precisely under the conditions of Theorem \ref{Theorem:KoszulMalgrangeTheorem:ParametricVersion} (non-parametric version) where the group under consideration is $\SO(\cn,2k)$ (see \eqref{Equation:GammaijkAreSkewSymmetricInTheCaseForOnVectors}). Thus, a solution to our system of equations exists if and only if\addtolength{\arraycolsep}{-1.0mm}
\[\begin{array}{ll}~&\d^{02}\alpha(\partial_{\bar{z}_i},\partial_{\bar{z}_j})+[\alpha(\partial_{\bar{z}_i}),\alpha(\partial_{\bar{z}_j})]=0\\
\equi&\partial_{\bar{z}_i}\big(\alpha(\partial_{\bar{z}_j})\big)-\partial_{\bar{z}_j}\big(\alpha(\partial_{\bar{z}_i})\big)+\alpha(\partial_{\bar{z}_i}).\alpha(\partial_{\bar{z}_j})-\alpha(\partial_{\bar{z}_j}).\alpha(\partial_{\bar{z}_i})=0\\
\equi&-\partial_{\bar{z}_i}\Gamma_j+\partial_{\bar{z}_j}\Gamma_i+\Gamma_i\cdot\Gamma_j-\Gamma_j\cdot\Gamma_i=0.\end{array}\]\addtolength{\arraycolsep}{1.0mm}\noindent
If $R_V^{02}=0$ (\textit{i.e.}, $R_V(T^{01}M,T^{01}M)=0$), we can deduce\addtolength{\arraycolsep}{-1.0mm}
\[\begin{array}{ll}~&\nabla_{\partial_{\bar{z}_i}}\nabla_{\partial_{\bar{z}_j}}W_k-\nabla_{\partial_{\bar{z}_j}}\nabla_{\partial_{\bar{z}_i}}W_k=0\\
\equi&\partial_{\bar{z}_i}\Gamma^n_{jk}W_n+\sum_{n,m}\Gamma^m_{jk}\Gamma^n_{im}W_n-\partial_{\bar{z}_j}\Gamma^n_{ik}W_n-\sum_{n,m}\Gamma^m_{ik}\Gamma^n_{jm}W_n=0\\
\equi&\partial_{\bar{z}_i}(\Gamma_j)^n_k+(\Gamma_j\Gamma_i)^n_k-\partial_{\bar{z}_j}(\Gamma_i)^n_k-(\Gamma_i\Gamma_j)^n_k=0\\
\equi&\partial_{\bar{z}_i}\Gamma_j-\partial_{\bar{z}_j}\Gamma_i+\Gamma_j\Gamma_i-\Gamma_i\Gamma_j=0,\end{array}\]\addtolength{\arraycolsep}{1.0mm}\noindent
so that we can, indeed, obtain such an orthonormal holomorphic frame for $V^\cn$.

\begin{remark}\label{Remark:ConstructionOfFramesInTheParametricCase}
From Theorem \ref{Theorem:KoszulMalgrangeTheorem:ParametricVersion}, if $V$ is, as above, a smooth oriented Riemannian vector bundle over $I\times M$, $I$ open interval around the origin of $\rn$, we can do a similar construction to the above for each $t$, provided $R^{20}_{V_t}=0$ for each $t$. Moreover, if we are given a smooth orthonormal frame $W(t,z)$ we can again obtain a new \textit{smooth} frame $U(t,z)$, with $U_t(z)=U(t,z)$ orthonormal for all $t$ and with $\nabla^{V_t}_{\partial\bar{z}_i}U_t=0$.
\end{remark}

\begin{remark}\label{Remark:AboutTheUsualKoszulMalgrangeTheorem}
If we do not know anything about the metric or orientation but only the condition on the curvature, a similar argument shows that there exist (local) frames for which $\nabla^V_{\partial_{\bar{z}_i}}U_j=0$, now not necessarily orthonormal (the group will simply be $\GL(\cn,2k)$ as $\Gamma$ will not necessarily verify \eqref{Equation:GammaijkAreSkewSymmetricInTheCaseForOnVectors}). Moreover, notice that we did not use the fact that $V$ is even-dimensional: we are only stating everything in this particular case as it will be our main interest. Hence, we get the usual version of Koszul-Malgrange Theorem:
\end{remark}

\begin{theorem}\textup{(\cite{KoszulMalgrange:58})}\label{Theorem:TheUsualKoszulMalgangeTheorem}\index{Koszul-Malgrange!Theorem!}
If $V$ is a complex vector bundle over a complex manifold $M$ with connection $\nabla^V$ such that $R^{02}_V=0$ then $V$ is a holomorphic vector bundle for the complex structure induced by the condition
\begin{equation}\label{Equation:EquationIn:Theorem:TheUsualKoszulMalgangeTheorem}
\text{a~section~}U\text{~of~}V\text{~is~holomorphic~if~and~only~if~}\nabla^V_{\partial_{\bar{z}_i}}U=0\text{.}
\end{equation}
\end{theorem}


\subsection{The holomorphic bundle \texorpdfstring{$\SO(V^\cn)$}{SO(Vc)}}\label{Subsection:TheHolomorphicBundleSOVCn}


From the previous discussion, we can state the following result:

\begin{proposition}[\textup{Existence~of~holomorphic~orthonormal~frames}]\label{Proposition:ExistenceOfHolomorphicOnFrames}\index{Holomorphic!and~orthonormal~sections}\index{Orthonormal~and~holomorphic~sections}
Let $M$ be a complex manifold and $V$ be an oriented even-dimensional Riemannian bundle. Consider the complexified bundle
$V^\cn$ with connection and metric obtained from those on $V$ by $\cn$-bilinear extension. Then, if $R_V^{02}=0$, there is (locally) an orthonormal frame $U_1,...,U_{2k}$ for $V^\cn$ such that
\begin{equation}\label{Equation:EquationIn:Proposition:ExistenceOfHolomorphicOnFrames}
\nabla^V_{\partial_{\bar{z}_i}}U_j=0,\quad\forall\,i,j\text{.}
\end{equation}
Moreover, the same holds in a parameter-dependent bundle: if $R^{20}_{V_t}=0$ for all $t$, $V$ bundle over $I\times M$, there is a smooth orthonormal frame $U_{1}(t,z),...,U_{2k}(t,z)$ for $V^{\cn}$ with equation
\eqref{Equation:EquationIn:Proposition:ExistenceOfHolomorphicOnFrames} verified for each $t$.
\end{proposition}
\begin{proof}
As before, at each $(t,z)$, we take any local orthonormal and positive frame $\{W_{m,t}\}$ of $V$ and define $\Gamma^{n,t}_{im}$ by
\[\nabla_{\partial_{\bar{z}_i}}W_{m,t}=\sum_n\Gamma^{n,t}_{im}W_{n,t}\text{.}\]
Since our frame $\{W_{m,t}\}$ is orthonormal, equation \eqref{Equation:GammaijkAreSkewSymmetricInTheCaseForOnVectors}
\[\Gamma^{n,t}_{im}=<\nabla_{\partial_{\bar{z}_i}}W_{m,t},W^{n,t}>=-<W_{m,t},\nabla_{\partial_{\bar{z}_i}}W_{n,t}>=-\Gamma^{m,t}_{in}\]
is satisfied for each $t$. On the other hand, $R^{20}_{V_t}=0$ for all $t$ implies that
\[\partial_{\bar{z}_i}\Gamma^{t}_j-\partial_{\bar{z}_j}\Gamma^{t}_i+\Gamma^{t}_j\Gamma^{t}_i-\Gamma^{t}_i\Gamma^{t}_j=0\text{.}\]
Thus, writing $\alpha_t=-\sum_i\Gamma^{t}_i\d\bar{z}_i$, $\alpha$ is $\so(\cn,2k)$-valued and satisfies \eqref{Equation:SecondEquationIn:Theorem:KoszulMalgrangeTheorem:ParametricVersion}
\[\d\alpha_t^{02}+[\alpha_t,\alpha_t]=0\]
for all $t$. Hence, using Theorem \ref{Theorem:KoszulMalgrangeTheorem:ParametricVersion}, a smooth solution to (compare with \eqref{Equation:TheEquationForONHolomorphicFrameOfVcnOnTheNonParametricCaseInMatricialNotation})
\[U^{-1}_t\big\{\sum_i\partial_{\bar{z}_i}U_t\d\bar{z}_i\big\}=-\sum_i\Gamma^{t}_i\d\bar{z}_i\]
exists in $\SO(\cn,2k)$. Letting $U_{jl,t}$ denote the entries of $U_t$, we can define a new \emph{smooth} frame of $V^\cn$ by $U_{j,t}=\sum_lU_{jl,t}W_{l,t}$. This new frame satisfies \eqref{Equation:EquationIn:Proposition:ExistenceOfHolomorphicOnFrames} and therefore concludes our proof.
\end{proof}

\begin{proposition}[\textup{Holomorphic~structure~of~the~bundle~$\SO(V^\cn)$}]\label{Proposition:HolomorphicStructureOfTheBundleSOVn}\index{Koszul-Malgrange!complex~structure~on~$\SO(V^\cn)$}\index{$\SO(V^\cn)$!Koszul-Malgrange\\holomorphic~structure}
As before, let $M$ be a complex manifold and let $V$ be an oriented even-dimensional Riemannian vector bundle over $M$ with $R_V^{02}=0$. Then, $\SO(V^\cn)$ has a canonical structure as a holomorphic bundle over $M$.

More generally, if $V$ is an oriented even-dimensional Riemannian vector bundle over $I\times M$ with $R_{V_t}^{02}=0$ for all $t$, then $\SO(V^\cn)$ has a canonical structure as a \emph{smooth} bundle over $I\times M$ for which each $\SO(V^\cn_t)$ is a \emph{holomorphic} bundle over $M$.
\end{proposition}
\begin{proof}
We define charts for our bundle which are smooth in $(t,z)$ and holomorphic in $z$ in the following way: take a point $(t,z)$ on $M$. Around that point, take a complex chart $(\openU,\varphi)$ for $M$ and, if necessary reducing $I$ and $\openU$, take $\{U_{j,t}\}$ the smooth frame constructed in the previous proposition. Define, as before,
\begin{equation}\label{Equation:EquationInTheProofOf:Proposition:HolomorphicStructureOfTheBundleSOVn}\addtolength{\arraycolsep}{-1.2mm}
\begin{array}{llll}
\eta: & \pi^{-1}(\openU)&\to&I\times\varphi(\openU)\times\SO(\cn^{2k}) \\
~ & (t,z,\lambda^t_z)&\to&(t,\varphi(z),\mathcal{M}(\lambda^t_z,e_i,U_{t,z}))
\end{array}
\end{equation}\addtolength{\arraycolsep}{1.2mm}\noindent
Then, these charts give a smooth structure to $\SO(\V^\cn)$ and, for each fixed $t$, a holomorphic structure to $\SO(\V^\cn)$, as we have seen in the previous section.
\end{proof}

\subsection{The bundle \texorpdfstring{$\Sigma^+
V$}{Sigma+V}}\label{Section:TheBundleSigmaPlusV}\label{Section:SectionInChapter:TwistorSpaces:HolomorphicStructureOnSigmaPlusVAndKoszulMalgrangeTheorem}


Let $V$ be an oriented Riemannian vector bundle over a complex manifold $M$ with rank $2k$. On each $V_x$, define $\Sigma^+ V_x$ as in Section \ref{Section:SectionIn:Chapter:TwistorSpaces:HermitianStructuresonAVectorSpace}
and then define the fibre bundle $\Sigma^+V$ over $M$ as
\begin{equation}
\Sigma^+ V:=\SO(V)\times_{\SO(2k)}\SO(2k)/\U(2k)\text{.}
\end{equation}
The fibre of $\Sigma^+V$ at $x$ is $\Sigma^+V_x$ and the projection map $\pi:\Sigma^+V\to M$ is defined by $\pi(x,J_x)\to x$. This is a subbundle of $\L(V,V)$ and the connection in the latter allows a splitting of $T\Sigma^+V$ into its \textit{horizontal} and \textit{vertical} parts as in the case $V=TM$.

In Section \ref{Subsection:SigmaPlusEAsTheComplexManifoldGkOplusIsoEcn}, we have seen that at each point $x\in{M}$, $\Sigma^+V_x\simeq G^+_{iso}(V_x^\cn)$. Hence, we have a natural identification between
$\Sigma^+V$ and $G^+_{iso}(V^\cn)$,
\begin{equation}\label{Equation:TheBundleGkOplusIsoVcn}\index{$G^+_{iso}(\cn^{2k})$}
G^+_{iso}(V^\cn)=\SO(V^\cn)\times_{\SO(\cn,2k)}G^+_{iso}(\cn^{2k}).
\end{equation}
On the other hand, in Section \ref{Subsection:TheHolomorphicBundleSOVCn}, we have seen that a complex structure can be introduced on $\SO(V^\cn)$ as long as $R_V^{02}=0$ so that a complex structure can be given, in that case, to $G^+_{iso}(V^\cn)$. Moreover, as we shall see, this structure is such that, identifying each almost Hermitian structure $J_x$ on $V_x$ with the corresponding $(1,0)$-subspace, $s^{10}_x$, of $V^\cn_x$ a section $s^{10}$ is holomorphic if and only if\index{$s^{10}$}
\begin{equation}\label{Equation:ConditionForHolomorphicityOfSectionsOfTheBundleSoVcnWithItsKoszulMalgrangeStructure}\index{Koszul-Malgrange!holomorphic~sections~of~$\Sigma^+V$}\index{Holomorphic!sections~of~$\Sigma^+V$}\index{$\Sigma^+V$!holomorphic~sections}
\nabla^V_{X^{01}}s^{10}\subseteq{s}^{10},\quad\forall\,X^{01}\in{T}^{01}M.
\end{equation}
More precisely, we may state the following

\begin{theorem}[\textup{Koszul-Malgrange~complex~structure~on~$\Sigma^+V$}]\label{Theorem:KoszulMalgrangeComplexStructureOnSigmaPlusV}
\index{Koszul-Malgrange!complex~structure~on~$\Sigma^+V$}\index{$\J^{KM}$}\index{$\Sigma^+V$!$\J^{KM}$} Let
$M^{2m}$ be a complex manifold and $V$ an oriented even-dimensional Riemannian vector bundle whose curvature tensor has vanishing $(0,2)$-part. Then, $\Sigma^+V$ can be given a unique structure as a holomorphic bundle such that a section $s^{10}$ is holomorphic if and only if \eqref{Equation:ConditionForHolomorphicityOfSectionsOfTheBundleSoVcnWithItsKoszulMalgrangeStructure} is satisfied.

More generally, if $V$ is an oriented even-dimensional Riemannian vector bundle over $I\times M$ with $R_{V_t}^{02}=0$ for all $t$, then $\Sigma^+V$ has a canonical structure as a \emph{smooth} bundle over $I\times M$ for which each $\Sigma^+V$ is a \emph{holomorphic} bundle over $M$.
\end{theorem}

In particular, under the conditions in Theorem \ref{Theorem:KoszulMalgrangeComplexStructureOnSigmaPlusV}, $\Sigma^+V$ is an (almost) complex manifold. We shall call this (almost) complex structure the Koszul-Malgrange structure and denote by $(\Sigma^+V,\J^{KM})$ this (almost) complex manifold.

We shall prove the non-parametric version of the above result, since the general case follows with the appropriate modifications. Before that, we state a lemma that clarifies equation \eqref{Equation:ConditionForHolomorphicityOfSectionsOfTheBundleSoVcnWithItsKoszulMalgrangeStructure}:

\begin{lemma}\label{Lemma:KoszulMalgrangeTheoremAndTheConditionNablaS10ContainedInS10}
As before, let $V$ be an oriented Riemannian vector bundle over a complex manifold $M$ with $R^{20}_V=0$. Then, given a smooth subbundle $F$ of $V^{\cn}$
\begin{equation}\label{Equation:EquationIn:Lemma:KoszulMalgrangeTheoremAndTheConditionNablaS10ContainedInS10}
\nabla_{\partial_{\bar{z}_i}}F\subseteq F
\end{equation}
if and only if $F$ can be locally given as $\wordspan\{U_1,...,U_n\}$ with $\nabla_{\partial_{\bar{z}_i}}U_j=0$.
\end{lemma}
\begin{proof}
The ``if" part is obvious. As for the ``only if" part, we shall once again make use of Theorem
\ref{Theorem:KoszulMalgrangeTheorem:ParametricVersion} (non-parametric version): suppose that $F$ satisfies
\eqref{Equation:EquationIn:Lemma:KoszulMalgrangeTheoremAndTheConditionNablaS10ContainedInS10}; we then know that there are $\Gamma^k_{ij}$ with $\nabla_{\partial_{\bar{z}_i}}W_j=\sum_k \Gamma^k_{ij}W_k$ for any local frame $\{W_1,...,W_n\}$ for $F$. As in Section \ref{Subsection:SubsectionOf:Section:SectionOf:Chapter:TwistorSpaces:TheBundleSigmaPlusVCnAndKoszulMalgrangeTheorem:TheKoszulMalgrangeTheoremAndTheExistenceOfHolomorphicOrthonormalFrames},
since $R^{20}_V=0$, we can guarantee that there is a solution to the system of equations $U^{-1}\partial_{\bar{z}_i}U=-\Gamma_i$, $U\in{G}L(\cn,n)$. Hence, taking $U_j=U W_j$ is easy to check that
$\nabla_{\partial_{\bar{z}_i}}U_j=0$ as required.
\end{proof}

\noindent \textit{Proof of Theorem
\ref{Theorem:KoszulMalgrangeComplexStructureOnSigmaPlusV}.} We have seen that $\SO(V^\cn)$ can be given the holomorphic structure characterized in the following way: given $z\in{M}$, fixing a orthonormal frame $\{U_j\}_{j=1,...,2k}$ of $V^{\cn}$ with $\nabla_{\partial_{\bar{z}_i}}U_j=0$ then a section $z\to (z,\lambda_z)$ of $\SO(V^{\cn})$ is holomorphic if and only if $\mathcal{M}(\lambda_z,e_i,U_j)$ is holomorphic. On $\Sigma^+V=G^+_{iso}(V^{\cn})$ we have the holomorphic structure given by the construction
\[\SO(V^{\cn})\times_{\SO(\cn,2k)} G^{+}_{iso}(\cn^{2k})\]
so that a section $z\to(z, s^{10}_z)$ of $G^{+}_{iso}(V^{\cn})$ is holomorphic if and only if it admits a holomorphic lift to $\SO(V^\cn)\times \SO(\cn,2k)$: notice that we have the following holomorphic projections:\addtolength{\arraycolsep}{-1.2mm}
\[\begin{array}{cccll}\SO(V^\cn)&\times& G^{+}_{iso}(\cn^{2k})&\to& G^+_{iso}(V^\cn)\\
\big((z,\lambda_z)&,&F\big)&\to &(z,\lambda_z F)\\
\SO(V^\cn)&\times& \SO(\cn,2k)&\to& G^+_{iso}(V^\cn)\\
\big((z,\lambda_z)&,&\beta\big)&\to &(z,(\lambda_z \circ \beta) F_0)
\end{array}\]\addtolength{\arraycolsep}{1.2mm}\noindent
where $F_0$ is as in the proof of Proposition
\ref{Proposition:ComptibilityBetweenSigmaPlusEAndGrassmannianOfTheIsotropicSubspacesInEcnTheComplexStructure}.
Hence, $s^{10}$ is holomorphic if and only if it is spanned by some $W_j(z)=(\lambda_z\circ\beta_z)f_j$ with $z\to(z,\lambda_z)$, $z\to\beta_z$ holomorphic and $F_0=\wordspan\{f_j\}$, $f_j=e_{2j-1}-ie_{2j}$. Now, we can easily check that $\nabla_{\partial_{\bar{z}_i}}W_j=0$ as we have\addtolength{\arraycolsep}{-1.2mm}
\[\begin{array}{ll}
~&\nabla_{\partial_{\bar{z}_i}}W_j=0\,\equi\,\partial_{\bar{z}_i}<W_j,U_k>=0\,\equi\,\partial_{\bar{z}_i}<\lambda_z\beta_zf_j,U_k>=0\\
\equi&\text{(since~$\beta_zf_j=\sum_r<\beta_zf_j,e_r>e_r$)~}\partial_{\bar{z}_i}(\sum_r<\beta_zf_j,e_r><\lambda_ze_r,U_j>)=0,
\end{array}\]\addtolength{\arraycolsep}{1.2mm}\noindent
which is trivially true as $\beta_z$ is holomorphic as well as $\mathcal{M}(\lambda_z,e_i,U_j)$. From the preceding lemma, the result follows.\qed

\chapter{Harmonic and pluriharmonic maps. Harmonic morphisms}\label{Chapter:SeveralMaps}



In this chapter we quickly review some of the main concepts that we shall need in subsequent chapters. Namely, the concepts of \emph{harmonic} and \emph{conformal} maps as well as some of their natural generalizations and particular cases are introduced and analyzed (Sections \ref{Section:SectionIn:Chapter:SeveralMaps:ConformalPluriconformalAndRealIsotropicMaps} and \ref{Section:SectionIn:Chapter:SeveralMaps:HarmonicAndPluriharmonicMaps}). We pay particular attention to the case of maps defined on a \emph{Riemann surface}\footnote{By a Riemann surface $M^2$ we mean an oriented two-dimensional manifold equipped with a \emph{conformal structure} (\textit{i.e.}, a class of conformally equivalent metrics; see Section \ref{Section:SectionIn:Chapter:Addendums:RiemannSurfacesAndPluriharmonicMaps} for further details) so that a complex structure is well defined on $M^2$.}. Finally, in Section \ref{Section:SectionIn:Chapter:HarmonicAndPluriharmonicMapsHarmonicMorphisms:HarmonicMorphisms} we recall the notion of \textit{harmonic morphism} as well as some fundamental properties of this special class of harmonic maps.

\index{Riemann~surface!}\index{Conformal~structure}\index{Metric!conformally~equivalent}

\section{Conformal, pluriconformal and real isotropic maps}\label{Section:SectionIn:Chapter:SeveralMaps:ConformalPluriconformalAndRealIsotropicMaps}


\begin{definition}[\emph{(Weakly)~conformal~map}]\label{Definition:ConformalMap}\index{Conformal~map}\index{Regular~point}\index{Branch~point}\index{Conformality~factor|see{Conformal~map}}\index{Conformal~map!weakly}\index{Conformal~map!conformality~factor}
(see \cite{BairdWood:03}) Let $\varphi:M\to N$ be a smooth map between two Riemannian manifolds $M$ and $N$. Then, $\varphi$ is said to be \emph{weakly conformal} at $x\in{M}$ if there is $\Lambda_{x}\in\rn$ with \begin{equation}\label{Equation:EquationIn:Definition:ConformalMap}
<\d\varphi_{x}X,\d\varphi_{x}Y>=\Lambda_{x}<X,Y>,\quad\forall\,X,Y\in{T}_{x}M\text{.}
\end{equation}
Taking $X=Y$ shows that $\Lambda_{x}\geq0$ so that there is $\lambda_{x}\geq0$ with $\lambda^2=\Lambda$; $\lambda_{x}$ is called the \emph{conformality factor (of $\varphi$ at $x$)}. If $\Lambda_{x}\neq0$ then $x$ is said to be a \emph{regular point}\footnote{More generally, given any smooth map $\varphi:M\to N$, we shall always refer to points $x\in{M}$ for which $\d\varphi_x\neq0$ as \emph{regular}.} (of $\varphi$) and the map $\varphi$ is called \emph{conformal} at $x$. If $\Lambda_{x}=0$ then $\d\varphi_{x}=0$ and $x$ is called a \emph{branch point}. Moreover, a map which is conformal (respectively, weakly conformal) at all points $x\in{M}$ is said to be a \emph{conformal map} (respectively, a \emph{weakly conformal map}). A conformal map is always an immersion and, in the case $\lambda_x$ is constant, it is called a \emph{homothetic immersion} (or a \emph{homothety} when $\varphi$ is also a diffeomorphism). Finally, when $\lambda_x\equiv1$ we get a \emph{Riemannian} or \emph{isometric} immersion (or an
\emph{isometry} when $\varphi$ is also a diffeomorphism).
\end{definition}\index{Riemannian!immersion}\index{Homothety}\index{Isometry}\index{Isometric~immersion}\index{Homothetic~immersion}

A weaker condition than that of conformal map is the following (see \cite{EschenburgTribuzy:95}, \cite{OhnitaValli:90}):

\begin{definition}[(\emph{Weakly})~\emph{pluriconformal~map}]\label{Definition:PluriconformalMap}\index{Complexified~derivative}\index{Pluriconformal~map}\index{Pluriconformal~map!weakly}
Let $\varphi:M^{2m}\to N^n$ a smooth map from an almost Hermitian manifold $M^{2m}$ to a Riemannian manifold $N^n$ and consider at each point $x\in{M}$ its complexified derivative, $\d\varphi_x:T^\cn_x M\to T^\cn_{\varphi(x)} N$. We say that $\varphi$ is \emph{weakly pluriconformal} at $x$ if $\d\varphi_{x} (T^{10}_{x}M)$ is an isotropic subspace of
$T_{\varphi({x})}^\cn N$:
\begin{equation}\label{Equation:SecondEquationIn:Definition:PluriconformalMap}\index{Pluriconformal~map!and~isotropic~subspace}
<\d\varphi_{x}{X}^{10},\d\varphi_{x}{Y}^{10}>=0,\quad\forall\,X^{10},Y^{10}\in{T}^{10}_{x}M\text{.}
\end{equation}
If $\d\varphi_{x}$ is also an injective linear map we shall say that $\varphi$ is \emph{pluriconformal} at ${x}$. In addition, a map that is weakly pluriconformal (respectively, pluriconformal) at all points $x\in{M}$ is called a \emph{weakly pluriconformal map} (respectively, a \emph{pluriconformal map}).
\end{definition}

\begin{remark}\label{Remark:ConformalAndPluriconformalMapsFromARiemannSurface}\index{Conformal~map!and~pluriconformal~map}\index{Riemann~surface!conformal~and~pluriconformal~maps}\index{Conformal~map!from~a~Riemann~surface}\index{Pluriconformal~map!from~a~Riemann~surface}

(i) If $M$ and $N$ are almost Hermitian manifolds, any holomorphic (respectively, anti-holomorphic) map $\varphi:M\to N$ is (weakly) pluriconformal as it maps $T^{10} M$ to $T^{10}N$ (respectively, to $T^{01}N$).

(ii) If $\varphi:M^2\to N^n$ is a map from a \emph{Riemann surface} $M^2$, the concepts of conformal (weakly conformal) and pluriconformal (weakly pluriconformal) map coincide. However, in higher dimensions, conformality implies pluriconformality but not conversely. In fact, if $\varphi$ is conformal then\addtolength{\arraycolsep}{-1.0mm}
\[\begin{array}{ll}
~&<\d\varphi(X-iJX),\d\varphi(Y-iJY)>\\
=&\lambda<X,Y>-\lambda<JX,JY>-i(\lambda<JX,Y>+<X,JY>)=0,
\end{array}\]\addtolength{\arraycolsep}{1.0mm}\noindent
which shows that $\varphi$ is (weakly) pluriconformal (see also comment (iii) below). To show that the converse is false, consider the holomorphic (and therefore pluriconformal) map $\varphi(z_1,z_2)=(z_1,z_1.z_2)$, that in real coordinates is given by $\varphi(x_1,x_2,x_3,x_4)=(x_1,x_2,x_1x_3-x_2x_4,x_1x_4+x_2x_3)$. Now, $<\d\varphi_xe_1,\d\varphi_xe_3>=x_1x_3+x_2x_4$ which is, in general, nonzero and therefore cannot be written as $\lambda_x<e_1,e_3>\equiv 0$, showing that $\varphi$ is not conformal.

(iii) If $M$ is a complex manifold, a map $\varphi:M\to N$ is pluriconformal if and only if it is conformal along each \emph{complex curve} in $M$ (\textit{i.e.}, a one dimensional complex submanifold in $M$); this is sometimes taken to be the definition of pluriconformal map, \textit{e.g.} \cite{OhnitaValli:90}.
\end{remark}

\emph{Real isotropic maps} will play a crucial role in what follows. We start by recalling the definition of such maps defined on a Riemann surface (see~\cite{EellsWood:83},~\cite{Rawnsley:85},~\cite{Salamon:85}):
\begin{definition}[\emph{Real~isotropic~map~from~a~Riemann~surface}]\label{Definition:RealIsotropicMapFromARiemannSurface}\index{Real~isotropic~map!}\index{Real~isotropic~map!from~a~Riemann~surface}
Let $\varphi:M^2\to N^n$ be a smooth map from a Riemann surface. Then $\varphi$ is said to be \emph{real isotropic} if
\begin{equation}\label{Equation:EquationIn:Definition:RealIsotropicMapFromARiemannSurface}
<\partial_z^r\varphi,\partial_z^s\varphi>=0,\quad\forall\,r,s\geq 1
\end{equation}
where $\partial_z^r\varphi=\nabla_{\partial_z}\big(\partial_z^{r-1}\varphi\big)$ so that, for instance, $\partial^2_z\varphi=\nabla_{\partial_z}\partial_z\varphi=\nabla\d\varphi(\partial_z,\partial_z)$ (as $\nabla^M_{\partial_z}\partial_z=0$).
\end{definition}
In particular, a real isotropic map from a Riemann surface is (weakly) conformal since putting $r=s=1$ in
\eqref{Equation:EquationIn:Definition:RealIsotropicMapFromARiemannSurface} gives
\begin{equation}\label{Equation:EquationAfter:Definition:RealIsotropicMapFromARiemannSurface:ConformalityConditionOnARiemannSurface}\index{Conformal~map!from~a~Riemann~surface}\index{Riemann~surface!conformal~map~from}\index{Riemann~surface!real~isotropic~map~from}
<\partial_z\varphi,\partial_z\varphi>=0,
\end{equation}
which implies \eqref{Equation:EquationIn:Definition:ConformalMap}.

A stronger property (Proposition \ref{Proposition:ComplexIsotropicMapsAreRealIsotropic}) than that of real isotropy is complex isotropy: given a Kähler manifold $N$ and a smooth map $\varphi:M^2\to N$, $\varphi$ is called \emph{complex
isotropic} (see \cite{EellsWood:83}, \cite{Wood:02}) if
\begin{equation}\label{Equation:ComplexIsotropicMapDefinition}\index{Complex~isotropic~map}\index{Isotropic~map!complex|see{Complex~isotropic~map}}\index{Isotropic~map!real|see{Real~isotropic~map}}
<\nabla^{r-1}_{\partial^{r-1}_z}\partial^{10}_z\varphi,\nabla^{s-1}_{\partial^{s-1}_{\bar{z}}}\partial^{10}_{\bar{z}}\varphi>_{Herm}=0,\quad\forall\,r,s\geq1\text{.}
\end{equation}
Here, $<u,v>_{Herm}=<u,\bar{v}>$ (and the latter is, as before, obtained by complex bilinear extension of the usual metric) and $\partial^{10}_z\varphi$ (respectively, $\partial^{10}_{\bar{z}}\varphi$) denotes the $(1,0)$-part of $\d\varphi(\partial_z)$ (respectively, of $\d\varphi(\partial_{\bar{z}})$) (for more on real and complex isotropic maps from a Riemann surface see Section \ref{Section:SectionIn:Chapter:Addendums:RealAndComplexIsotropicMaps}).

We generalize the concept of real isotropic map for maps defined in arbitrary Hermitian manifolds:
\begin{definition}[\emph{Real~isotropic~map}]\label{Definition:RealIsotropicMapFromAHermitianManifold}\index{Real~isotropic~map!general~case}
Let $\varphi:M^{2m}\to N$ be a smooth map from a Hermitian manifold $M^{2m}$. We shall say that $\varphi$ is \emph{real isotropic} if
\begin{equation}\label{Equation:FirstEquationIn:Definition:RealIsotropicMapFromAHermitianManifold}
<\nabla^{r-1}_{T^{10}M}\d\varphi(T^{10}M),\nabla^{s-1}_{T^{10}M}\d\varphi(T^{10}M)>=0,\quad\forall\,r,s\geq1\text{,}
\end{equation}\addtolength{\arraycolsep}{-1.0mm}\noindent
\[\hspace{-2mm}\begin{array}{rll}\text{where }\qquad(\nabla^0\d\varphi)(X^{10})&=&\d\varphi(X^{10})\text{,}\\ (\nabla^1_{Y^{10}_1}\d\varphi)(X^{10})&=&(\nabla\d\varphi)(Y^{10}_1,X^{10})=\nabla_{Y^{10}_1}(\d\varphi{X}^{10})-\d\varphi(\nabla_{Y^{10}}X^{10})\,\footnotemark\text{and,}
\end{array}\]\addtolength{\arraycolsep}{1.0mm}\footnotetext{Notice that, since $M$ is Hermitian, $\nabla_{Y^{10}}X^{10}\in{T}^{10}M$ for all $X^{10},Y^{10}\in{T}^{10}M$.}recursively \footnote{This is just the usual connection induced from the Levi-Civita connection.},\addtolength{\arraycolsep}{-1.0mm}
\[\begin{array}{lll}(\nabla^r_{Y_1^{10},...,Y_r^{10}}\d\varphi)(X^{10})&=&(\nabla^r\d\varphi)(Y_1^{10},...,Y_r^{10},X^{10})\\
~&=&\nabla_{Y^{10}}\big(\nabla^{r-1}\d\varphi(Y_2,...,Y_r,X)\big)-\nabla^{r-1}\d\varphi(\nabla_{Y_1^{10}}Y_2^{10},...,X^{10})\\
~&~&-...-\nabla^{r-1}\d\varphi(Y_2^{10},...,\nabla_{Y_1^{10}}X^{10})\text{.}\end{array}\]\addtolength{\arraycolsep}{1.0mm}\noindent
Introducing complex coordinates $(z_1,...,z_m)$ on $M$, equation \eqref{Equation:FirstEquationIn:Definition:RealIsotropicMapFromAHermitianManifold} is equivalent to
\begin{equation}\label{Equation:SecondEquationIn:Definition:RealIsotropicMapFromAHermitianManifold}
<\partial^{i_1+...+i_m}_{z_1^{i_1}...z_m^{i_m}}\varphi,\partial^{j_1+...+j_m}_{z_1^{j_1}...z_m^{j_m}}\varphi>=0,\quad\forall\,(i_1,...,i_m),(j_1,...,j_m)\text{~with~\footnotesize{$\sum$}}i_k,\text{\footnotesize{$\sum$}}j_k\geq1.
\end{equation}
\end{definition}
\noindent Putting $m=1$ in \eqref{Equation:SecondEquationIn:Definition:RealIsotropicMapFromAHermitianManifold}, it gives \eqref{Equation:EquationIn:Definition:RealIsotropicMapFromARiemannSurface}. In particular, Definitions
\ref{Definition:RealIsotropicMapFromAHermitianManifold} and \ref{Definition:RealIsotropicMapFromARiemannSurface} agree when $m=1$.

We now turn our attention to \emph{totally umbilic maps}. The following definition can be found, for example, in \cite{Salamon:85}:
\begin{definition}[\emph{Totally~umbilic~map~from~a~Riemann~surface}]\label{Definition:TotalUmbilicMapFromARiemannSurface}\index{Totally~umbilic~map!from~a~Riemann~surface}\index{Totally~umbilic~map!}\index{Umbilic!map|see{Totally~umbilic~map}}\index{Umbilic!point|see{Totally~umbilic~map}}
Let $\varphi:M^2\to N$ be a smooth map from a Riemann surface $M^2$. We shall say that $z\in{M}^2$ is an \emph{umbilic point} (of $\varphi$) if $\{\partial_z\varphi(z),\partial^2_z\varphi(z)\}$ is a $\cn$-linearly dependent set. If $\varphi:M^2\to N$ is such that all points $z\in{M}^2$ are umbilic, we shall say that $\varphi$ is \textit{totally umbilic}.
\end{definition}

We can also generalize this notion to an arbitrary almost Hermitian manifold:

\begin{definition}[\emph{Totally~umbilic~map}]\label{Definition:TotalUmbilicMapFromAnAlmostHermitianManifold}\index{Totally~umbilic~map!from~an~almost~Hermitian~manifold}
Let $\varphi:M^{2m}\to N$ be a smooth map from an almost Hermitian manifold $M^{2m}$. We shall say that $x\in{M}^{2m}$ is an \emph{umbilic point} (of $\varphi$) if
\begin{equation}\label{Equation:FirstEquationIn:Definition:TotalUmbilicMapFromAnAlmostHermitianManifold}
\nabla\d\varphi_{x}(T^{10}M,T^{10}M)\subseteq\d\varphi_{x}(T^{10}M).
\end{equation}
If $\varphi$ is such that all points $x\in{M}^{2m}$ are umbilic, $\varphi$ is said to be \emph{totally umbilic}.
\end{definition}

In the case $m=1$, \eqref{Equation:FirstEquationIn:Definition:TotalUmbilicMapFromAnAlmostHermitianManifold}
is equivalent to requiring $\{\partial^2_z\varphi,\partial_z\varphi\}$ to be linearly dependent
so that both Definitions \ref{Definition:TotalUmbilicMapFromARiemannSurface} and \ref{Definition:TotalUmbilicMapFromAnAlmostHermitianManifold} agree in this case.


\section{Harmonic maps}\label{Section:SectionIn:Chapter:SeveralMaps:HarmonicAndPluriharmonicMaps}


Given two Riemannian manifolds $(M,g)$ and $(N,h)$, a smooth map $\varphi:M\to N$ is said to be \emph{harmonic} if it satisfies $\tau(\varphi)=0$, where $\tau(\varphi)=\trace\nabla\d\varphi$ is the \emph{tension field} of $\varphi$ and
$\nabla\d\varphi$ is the second fundamental form of $\varphi$ (\cite{EellsLemaire:78}). Variationally, harmonic maps are critical points for the \emph{energy functional}\index{Tension~field!}
\begin{equation}\label{Equation:EnergyFunctional}\index{Energy~functional}\index{Harmonic~map!as~critical~point~for~the~energy}\index{Harmonic~map!}\index{Tensor~field}\index{Second~fundamental~form!of~a~smooth~map}
E(\varphi)=\frac{1}{2}\int_D \|\d\varphi\|^2(x)\d\mu_M(x),
\end{equation}
\noindent where $D\subseteq M$ is compact, $\|\d\varphi\|$ is the Hilbert-Schmidt norm of $\d\varphi$ and $\mu_M$ is the measure on $M$ induced by the metric $g$ (see, \textit{e.g.}, \cite{EellsLemaire:88}). Indeed, for a $1$-parameter variation $\{\varphi_t\}$ of a smooth map $\varphi_0$ supported in $D$, we have \begin{equation}\label{Equation:OneParameterVariationOfEnergy}\index{One~parameter~variation!for~the~energy~functional}\index{Energy~functional!one-parameter~variations}
\left.\frac{\d{E}(\varphi_t)}{\d{t}}\right|_{t=0}=-\int_D<\tau(\varphi),\left.\frac{\partial\varphi_t}{\partial{t}}\right|_{t=0}>(x)\d\mu_M(x)\text{.}
\end{equation}
For example, when $N=\rn$ the energy functional reduces to the Dirichlet integral and a smooth function $f:M\to\rn$ is harmonic if it satisfies \emph{Laplace's equation}: $\triangle f=0$; \textit{i.e.}, it is a solution of the equation
\begin{equation}\label{Equation:LaplaceEquationForFunctionsOnARiemannianManifold}\index{$\triangle{f}$}\index{Laplace's~equation}\index{Harmonic~function}
g^{ij}\big(\frac{\partial^2 f}{\partial x_i\partial{x}_j}-\Gamma^k_{ij}\frac{\partial f}{\partial x_k}\big)=0\text{.}
\end{equation}

The following is a well-known result (see, for example, \cite{BairdWood:03}, p. 85) characterizing harmonic maps from a Riemann surface:
\begin{proposition}[\textup{Harmonic~map~from~a~Riemann~surface}]\label{Proposition:HarmonicMapsFromRiemannSurfaces}\index{Harmonic~map!from~Riemann~surfaces}
Let $\varphi:M^2\to N$ be a harmonic map from a Riemann surface $M^2$ to a Riemannian manifold $N$. Then, $\varphi$ is harmonic if and only if
\begin{equation}\label{Equation:EquationIn:Proposition:HarmonicMapsFromRiemannSurfaces}
\nabla_{\partial_{\bar{z}}}\partial_z\varphi=0.
\end{equation}
In particular, the concept of harmonic map on a Riemann surface is well-defined.
\end{proposition}\index{Conformal~structure!}
Notice that $\nabla$ is the Levi-Civita connection associated to any metric within the \emph{conformal structure} of $(M,g)$ (see Section \ref{Section:SectionIn:Chapter:Addendums:RiemannSurfacesAndPluriharmonicMaps} for more on conformal structures). Note that, given a one-dimensional complex manifold $M$, any two Hermitian metrics on $M$ are conformally equivalent and, from the above, as harmonicity does not depend on the choice of metric within a conformal class, it makes sense to speak of harmonic map from a one-dimensional complex manifold.

Recall that an embedded or immersed submanifold $M^m$ of $N^n$ is said to be \textit{minimal} if $\mu^M=0$, where $\mu^M=\frac{1}{m}\trace B$ denotes the mean curvature of $M$ and $B$ the second fundamental form of $M$ given by $B(X,Y)=(\nabla_XY)^{\bot}$. The following is well-known (\cite{EellsSampson:64}):

\begin{proposition}[\textup{Harmonic~maps~and~minimal~submanifolds}]\label{Proposition:HarmonicMapsAndMinimalSubmanifolds}\index{Harmonic~map!and~minimal~submanifolds}\index{Minimal~submanifold!and~harmonic~maps}\index{Minimal~submanifold}\index{Submanifold!minimal}\index{Second~fundamental~form!of~a~submanifold}
If $M\subset N$ is a smooth manifold then $M$ is minimal if and only if the inclusion map is harmonic. More generally, if $\varphi:M\to N$ is a Riemannian immersion, $\varphi$ is harmonic if and only if the immersed submanifold $\varphi(M)$ is minimal.
\end{proposition}
In the special case of a conformal map $\varphi:M^2\to N$ from a Riemann surface $M^2$ into $N$, we can take the pull-back metric $\varphi^\ast g_N$ on $M^2$, that belongs to the conformal structure on $M^2$ and makes $\varphi$ a Riemannian
immersion. Hence, $\varphi(M)$ is a minimal submanifold if and only if $\varphi$ is harmonic.

When $M^2=S^2$, any harmonic map $\varphi:S^2\to N^n$ is automatically weakly conformal (see \textit{e.g.}
\cite{BairdWood:03});\label{Page:HarmonicMapsFromAndToSpheresAndTheirConformalityOrRealIsotropy}\index{Harmonic~map!from~the~$2$-sphere}\index{Conformal~map!from~the~$2$-sphere}\index{Harmonic~map!into~the~$n$-sphere}\index{Real~isotropic~map!from~the~$2$-sphere}\index{Harmonic~map!into~$\cn{P}^n$}\index{Real~isotropic~map!from~the~$2$-sphere}\index{Real~isotropic~map!into~the~$n$-sphere}\index{Space~form!}\index{Complex~space~form}\index{Space~form!complex}
in particular, it defines a minimal submanifold at regular points. This result can be improved when $N^n$ is also a sphere (more generally, a space form) or the complex projective space (more generally, a complex space form): in that case, every harmonic map is real isotropic\footnote{More precisely, when the codomain $N$ is a complex space form, any harmonic map $\varphi:S^2\to N$ is complex isotropic (and therefore, also real isotropic).} (\textit{e.g.} \cite{Salamon:85}, p. 190 for the first case and \cite{EellsWood:83} for the second).


\subsection{Pluriharmonic and $(1,1)$-geodesic maps}\label{Subsection:SectionIn:Chapter:HarmonicAndPluriharmonicMapsHarmonicMorphisms:PluriharmonicAnd(11)GeodesicMaps}


We now turn our attention to a particular class of harmonic maps:

\begin{definition}[\emph{$(1,1)$-geodesic map}]\label{Definition:11GeodesicMap}\index{$(1,1)$-geodesic~map!}
\textup{(see \cite{BairdWood:03}, p. 254)} Let $\varphi:(M,J^M)\to{N}$ be a smooth map from an almost Hermitian manifold into a Riemannian manifold. Then, $\varphi$ is said to be \emph{$\mathit{(1,1)}$-geodesic} if
\begin{equation}\label{Equation:EquationIn:Definition:11GeodesicMap}
(\nabla\,\d\varphi)(Y^{10},Z^{01})=0,\quad\forall\,Y^{10}\in{T}^{10}M,\,Z^{01}\in{T}^{01}M.
\end{equation}
\end{definition}\label{Page:(11)GeodesicMapsAreHarmonic}\index{$(1,1)$-geodesic~map!is~harmonic}\index{Pluriharmonic~map!}

$(1,1)$-geodesic maps are harmonic as, for any $X\in{TM}$
\[0=\nabla\d\varphi(X-iJX,X+iJX)=\nabla\d\varphi(X,X)+\nabla\d\varphi(JX,JX)\]
so that on choosing an orthonormal frame $\{X_i,JX_i\}_{i=1,...,m}$ for $M^{2m}$ we see that $\trace \nabla\d\varphi=0$.

When $(M,J)$ is a complex manifold (not necessarily endowed with a metric), it also makes sense to speak of \emph{pluriharmonic maps} as those for which their restriction to any complex curve on $M$ is harmonic (see comments to Proposition \ref{Proposition:HarmonicMapsFromRiemannSurfaces}). When $(M,g,J)$ is Kähler, the notions of pluriharmonic and $(1,1)$-geodesic maps coincide (see Proposition \ref{Proposition:PluriharmonicAnd(1,1)GeodesicsMapsFromKahlerManifolds} and Corollary \ref{Corollary:PluriharmonicMapsAre(1,1)GeodesicMaps}) and imply harmonicity.\label{Page:(1,1)GeodesicAndPluriharmonicMaps}\index{Pluriharmonic~map!and~$(1,1)$-geodesic~maps}\index{$(1,1)$-geodesic~map!and~pluriharmonic~maps}\index{Pluriharmonic~map!and~harmonic~maps}\index{Superminimal~submanifold}\index{Submanifold!superminimal}\index{Minimal~submanifold!and~superminimal~submanifold}\index{Submanifold!pluriminimal}\index{Pluriminimal~submanifold!}\index{Minimal~submanifold!and~pluriminimal~submanifolds}

Let $(N,g,J)$ be an almost Hermitian manifold. An almost complex submanifold $M$ of $N$ is said \emph{superminimal} if $J$ is parallel along $M$ (\cite{BairdWood:03}, p. 223):
\begin{equation}\label{Equation:EquationIn:Definition:SuperminimalSubmanifolds}\index{$\iota$,~inclusion~map}\index{Inclusion~map}
\nabla^{N}_X J=0,\quad\forall\,X\in{T}M\text{.}
\end{equation}
Superminimal submanifolds are minimal (see \cite{BairdWood:03}, p. 224): if $\iota:M\to N$ denotes the inclusion map, $\nabla\d\iota(X,JY)=J\nabla\d \iota (X,Y)$ from which we can deduce $\trace\nabla \d \iota=0$. Notice that equation
\eqref{Equation:EquationIn:Definition:SuperminimalSubmanifolds} is stronger then imposing that $(M,g,J)$ be a Kähler
submanifold of $(N,g,J)$, as we require the total covariant derivative of $J$ along $M$, $\nabla_{TM}^N J$, vanish and not just $\nabla^M_{TM}J$. Of course, superminimal submanifolds are Kähler for the induced metric and the inclusion map becomes pluriharmonic (equivalently, $(1,1)$-geodesic): as a matter of fact, from \eqref{Equation:EquationIn:Definition:SuperminimalSubmanifolds}, we deduce $J\nabla^N_{T^{01}M}(X-iJX)=i\nabla_{T^{01}M}(X-iJX)$ for every $X\in{T}M$ so that $\nabla^N_{T^{01}M}T^{10}M\subseteq{T}^{10}N$; hence, $\nabla\d \iota(T^{01}M,T^{10}M)\subseteq T^{10}N$ (see also Lemma \ref{Lemma:AlternativeConditionFor(11)GeodesicMaps}). We call to submanifolds satisfying this weaker condition \emph{pluriminimal}: \textit{i.e.}, given an almost Hermitian manifold $(N,g,J)$, a pluriminimal submanifold $M$ is a \emph{Kähler submanifold} for which the inclusion map is pluriharmonic (\cite{BurstallEschenburgFerreiraTribuzy:04}, \cite{EschenburgTribuzy:95}, \cite{EschenburgTribuzy:98}). These are still minimal submanifolds, as the inclusion map is $(1,1)$-geodesic and therefore harmonic. The next three results were obtained in joint work with M.~Svensson (\cite{SimoesSvensson:06}). Firstly, when the codimension is two, the concepts of pluriminimal and superminimal coincide:

\begin{proposition}[\textup{Pluriminimal~and~superminimal~submanifolds~of~codimension~two}]\label{Proposition:PluriminimalAndSuperminimalInCodimension2}\index{Superminimal~submanifold!}\index{Superminimal~submanifold!and~pluriminimal~submanifolds}
Let $(N,g,J)$ be a Hermitian manifold and $M$ a pluriminimal submanifold of codimension two. Then, $M$ is superminimal.
\end{proposition}
\begin{proof}
Fix a local frame $\{V,JV\}$ for $TM^\perp$, and denote by $X,Y,Z$ general tangent vectors to $M$. Under the isomorphism $\tilde{g}$ in \eqref{Equation:IsomorphismTildegGivenByTheMetric}, $\nabla_XJ$ is skew-symmetric; on the other hand, it satisfies $(\nabla_XJ)J=-J(\nabla_XJ)$. It follows that $h((\nabla_XJ)V,V)=h((\nabla_XJ)V,JV)=0$ and, since $M$ is Kähler, \[h((\nabla_XJ)Y,Z)=h((\nabla^M_XJ)Y,Z)+h((\nabla^\perp_XJ)Y,Z)=0\text{.}\]
It remains to show that $h((\nabla_XJ)Y,V)=0$; by antisymmetry, then
also $h((\nabla_XJ)V,Y)$ vanishes. Denote by $\iota:M\to N$ the
inclusion map. As $M$ is pluriminimal, we have $\nabla\d\iota(X,JY)=\nabla\d\iota(JX,Y)$ and consequently
$h((\nabla_XJY,V)=h(\nabla_{JX}Y,V)$. Thus,\addtolength{\arraycolsep}{-1.0mm}
\[\begin{array}{lllll}
h((\nabla_XJ)Y,V)&=&h(\nabla_XJY,V)-h(J\nabla_XY,V)&=&h(\nabla_{JX}Y,V)-h(J\nabla_YX,V)\\
~&=&h(\nabla_YJX,V)-h(J\nabla_YX,V)&=&h((\nabla_YJ)X,V).
\end{array}\]\addtolength{\arraycolsep}{1.0mm}\noindent
Since $J$ is integrable, we have $\nabla_{JX}J=J\nabla_XJ$ (see Corollary \ref{Proposition:CorollaryToSalamonTheorem3.2ComplexAnd12SymplecticVersusJ1AndJ2Holomorphicity} and equation \eqref{Equation:FirstEquationAfter:Lemma:VerticalJaHolomorphicityNadNullityOfda}).\addtolength{\arraycolsep}{-1.0mm}
\[\begin{array}{lllll}
h((\nabla_XJ)Y,V)&=&-h(J^2(\nabla_XJ)Y,V)&=&-h(J(\nabla_{JX}J)Y,V)\\
~&=&-h(J(\nabla_YJ)JX,V)&=&h(J^2(\nabla_YJ)X,V)=-h((\nabla_XJ)Y,V).
\end{array}\]\addtolength{\arraycolsep}{1.0mm}\noindent
Hence, $\nabla_XJ=0$ so that $M$ is superminimal.
\end{proof}

In the same way that minimal submanifolds are closely related to harmonic maps (Proposition \ref{Proposition:HarmonicMapsAndMinimalSubmanifolds} and comments), pluriminimal submanifolds will be closely related
with pluriharmonic maps. Indeed, we have the following:

\begin{theorem}[\textup{Pluriharmonic~maps~and~pluriminimal~submanifolds}]\label{Theorem:PluriminimalSubmanifoldsAsImageOfPluriharmonicAndPluriconformalMaps}\index{Pluriminimal~submanifold!and~pluriharmonic~maps}
Let $\varphi:(M,J^M)\to(N^n,g,J^N)$ be a pluriharmonic holomorphic map from a complex manifold $(M,J)$ into an almost Hermitian manifold $(N,g,J^N)$. Then, $\varphi(M)$ is (locally) a pluriminimal submanifold.
\end{theorem}

Therefore, pluriminimal submanifolds can be (locally) constructed as the image of a pluriharmonic holomorphic map $\varphi$ from a complex manifold into any almost Hermitian manifold. We prove Theorem \ref{Theorem:PluriminimalSubmanifoldsAsImageOfPluriharmonicAndPluriconformalMaps}
after the next theorem (which is given in \cite{EschenburgTribuzy:95} for the particular case when $M$ is Kähler and \cite{ArezzoPirolaSolci:04} when the codomain is the Euclidean space); see Corollary \ref{Corollary:MinimalSubmanifoldsAndTwistorMaps} for a ``twistorial" version of this fact):

\begin{theorem}\label{Theorem:PullBackMetricThroughAPluriharmonicAndPluriconformalMapIsKahler}\index{Metric!pull~back}\index{Pluriharmonic~map!and~pull-back~metric}\index{Compatible!metric~and~complex~structure}\index{$\nabla^{\varphi^{-1}}$,~pull-back~connection}\index{$\varphi^\ast{h}$,~pull-back~metric}
Let $\varphi:(M,J)\to (N^n,h)$ be a pluriharmonic and pluriconformal map from a complex manifold into a Riemannian manifold $(N,h)$. Then, the pull-back metric $g=\varphi^\ast h$ on $M$ is Kähler.
\end{theorem}
\begin{proof} As $\d\varphi(T^{1,0}M)$ is isotropic, it follows easily that $g$ is \emph{compatible} with $J$, \textit{i.e.}, $(M,g,J)$ is a Hermitian manifold. Let $\omega=g(J\cdot,\cdot)$ be the corresponding Kähler form. As $J$ is integrable, $\d\omega$ has no $(3,0)$-part; to show that $\omega$ is closed it is therefore enough to show that the $(1,2)$-part of $\d\omega$ vanishes. To this end, let $(z_{1},...,z_{m})$ be local holomorphic coordinates on $M$; they determine locally a Kähler metric and a Levi-Civita connection $\widetilde{\nabla}$ (see Lemma \ref{Lemma:ExistenceOfLocallyDefinedKahlerMetricsForComplexManifolds}). As $\varphi$ is pluriharmonic, we get
\[0=\widetilde{\nabla}\d\varphi(\partial_{z_k},\partial_{\bar{z}_m})=\nabla^{\varphi^{-1}}_{\partial_{z_k}}\d\varphi(\partial_{\bar{z}_m})-\d\varphi(\widetilde{\nabla}_{\partial_{z_m}}\partial_{\bar{z}_m})=\nabla^{\varphi^{-1}}_{\partial_{z_k}}\d\varphi(\partial_{\bar{z}_m})\text{,}\]
where $\nabla^{\varphi^{-1}}$ is the pull-back connection on $\varphi^{-1}TN$. Hence,\addtolength{\arraycolsep}{-1.0mm}
\[\begin{array}{lll}
\d\omega(\partial_{z_k},\partial_{z_l},\partial_{\bar{z}_m})&=&\partial_{z_k}\omega(\partial_{z_l},\partial_{\bar{z}_m})-\partial_{z_l}\omega(\partial_{z_k},\partial_{\bar{z}_m})\\
~&=&~ih(\nabla^{\varphi^{-1}}_{\partial_{z_k}}\d\varphi(\partial_{z_l})-\nabla^{\varphi^{-1}}_{\partial_{z_l}}\d\varphi(\partial_{z_k}),\partial_{\bar{z}_m})\\
~&=&ih(\d\varphi([\partial_{z_k},\partial_{z_l}]),\partial_{\bar{z}_m})=0\text{.}
\end{array}\]\addtolength{\arraycolsep}{1.0mm}\noindent
This proves that $\d\omega=0$ so that $M$ is Kähler.
\end{proof}

Thus we can now easily prove Theorem \ref{Theorem:PluriminimalSubmanifoldsAsImageOfPluriharmonicAndPluriconformalMaps}: as $\varphi$ is holomorphic, it is pluriconformal (Remark \ref{Remark:ConformalAndPluriconformalMapsFromARiemannSurface}). Hence, for the pull-back metric, $(M,g,J)$ is Kähler and so we have an isometric pluriharmonic immersion, which shows that its image is pluriminimal.


\section{Harmonic~morphisms}\label{Section:SectionIn:Chapter:HarmonicAndPluriharmonicMapsHarmonicMorphisms:HarmonicMorphisms}


Harmonic morphisms are smooth maps that \textit{preserve} Laplace's equation \eqref{Equation:LaplaceEquationForFunctionsOnARiemannianManifold}; more precisely:

\begin{definition}\label{Definition:HarmonicMorphism}\index{Harmonic~morphism!}
Let $M$ and $N$ be two Riemannian manifolds and $\varphi:M\to N$ a smooth map. Then, $\varphi$ is said to be a \emph{harmonic morphism} if it pulls back (locally defined) harmonic functions on $N$ to (locally defined) harmonic functions on $M$. In other words, if $\V$ is an open set in $N$ with $\varphi^{-1}(\V)$ non-empty and $f:\V\to\rn$ a harmonic function, then $f\circ\varphi:\varphi^{-1}(\V)\to\rn$ is a harmonic function.
\end{definition}

To give a geometric characterization of harmonic morphisms, we need the following notion (see \textit{e.g.} \cite{BairdWood:03}):

\begin{definition}[\emph{Horizontally~(weakly)~conformal~map}]\label{Definition:HorizontallyWeaklyConformalMap}\index{Horizontally~(weakly)~conformal~map!}\index{Conformal~map!horizontally~(weakly)|see{Horizontally\\(weakly)~conformal~map}}\index{Horizontal~space|see{Harmonic~morphism}}\index{Harmonic~morphism!horizontal~space}
Let $\varphi:M\to N$ be a smooth map between two Riemannian manifolds $(M,g)$ and $(N,h)$. Then, $\varphi$ is called \emph{horizontally weakly conformal} if for all $x\in{M}$ either:

(i) $\d\varphi_x=0$, or

(ii) $\d\varphi_x:T_xM\to T_{\varphi(x)}N$ maps the \textit{horizontal space} $\H_x=(\ker\d\varphi_x)^\bot$ conformally onto $T_{\varphi(x)}N$; \textit{i.e.}, $\d\varphi_x$ is surjective and there is $\Lambda_x\neq 0$ with
\begin{equation}\label{Equation:EquationIn:Definition:HorizontallyWeaklyConformalMap}
h(\d\varphi_x X,\d\varphi_xY)=\Lambda_xg(X,Y),\quad\forall\,X,Y\in\H_x\text{.}
\end{equation}
\end{definition}
Of course, if we define $\Lambda$ to be zero at points $x\in{M}$ for which $\d\varphi_x=0$, we get (\textit{cf.} Definition \ref{Definition:ConformalMap}) a smooth non-negative function $\Lambda_x$ defined on $M$, and we can consider its square root $\lambda_x$, to which is usually called the \emph{dilation} (of $\varphi$ at $x$). We can now give the following characterization of harmonic morphisms (\cite{Fuglede:81}, \cite{Ishihara:79}, see also
\cite{BairdWood:03}):\index{Dilation~of~a~horizontally~weakly\\~conformal~map}\index{Horizontally~(weakly)~conformal~map!dilation~of}

\begin{theorem}[\textup{Geometric~characterization~of~harmonic~morphisms}]\label{Theorem:GeometricCharacterizationOfHarmonicMorphisms}\index{Harmonic~morphism!geometric~characterization}
A smooth map $\varphi:M\to N$ between Riemannian manifolds is a harmonic morphism if and only if it is simultaneously harmonic and horizontally weakly conformal.
\end{theorem}

The following theorem will be crucial in our subsequent work (\cite{BairdEells:81}, see also \cite{BairdWood:03}, p. 122):

\begin{theorem}[\textup{Harmonic~morphisms~and~mean~curvature~of~the~fibres}]\label{Theorem:HarmonicMorphismsAndMeanCurvatureOfFibresSpecialCaseOfRiemannSurfaces}\index{Harmonic~morphism!mean~curvature~of~the~fibres}
Let $\varphi:M^m\to N^n$ be a smooth non-constant horizontally weakly conformal map between Riemannian manifolds. Then, $\varphi$ is harmonic, and so a harmonic morphism, if and only if, at every regular point, the mean curvature vector field $\mu^\V$ of the fibres and the gradient of the dilation $\lambda$ of $\varphi$ are related by
\begin{equation}\label{Equation:EquationIn:Theorem:HarmonicMorphismsAndMeanCurvatureOfFibresSpecialCaseOfRiemannSurfaces}\index{Harmonic~morphism!minimal~submanifolds}\index{Minimal~submanifold!and~harmonic~morphisms}
(n-2)\H(\grad \log\lambda)+(m-n)\mu^\V=0\text{.}
\end{equation}
In particular, when $n=2$, $\varphi$ is harmonic, and so a harmonic morphism, if and only if at every regular point the fibres of $\varphi$ are minimal.
\end{theorem}
Of course, we immediately get from \eqref{Equation:EquationIn:Theorem:HarmonicMorphismsAndMeanCurvatureOfFibresSpecialCaseOfRiemannSurfaces} that a Riemannian submersion $\varphi:M^m\to N^n$ is a harmonic morphism if and only if it has minimal fibres (\cite{Fuglede:81}). In the particular case where $n=2$, harmonic morphisms do not depend on the conformal class of $N^2$ (\cite{BairdWood:03}; compare with \ref{Proposition:HarmonicMapsFromRiemannSurfaces}) as horizontal (weak) conformality remains unchanged by conformal changes of the metric on $N^2$ and the fact that the fibres are minimal does not depend on the metric on $N^2$ at all. In particular, the concept of harmonic morphism to a Riemann surface is well-defined.

\label{Page:TheIdeaToConstructHarmonicMorphismsAfter:Theorem:HarmonicMorphismsAndMeanCurvatureOfFibresSpecialCaseOfRiemannSurfaces}\index{Harmonic~morphism!and~holomorphic~maps}\index{Harmonic~morphism!to~a~Riemann~surface}\index{Riemann~surface!valued~harmonic~morphism}
If $(M,g,J)$ is an almost Hermitian manifold and $\varphi:M\to N^2$ is a holomorphic map, $\varphi$ is automatically horizontally weakly conformal, as it is easy to check. Therefore, in such a case, we have a harmonic morphism if and only if its fibres are minimal (at regular points). As we shall see in Chapter \ref{Chapter:HarmonicMorphismsAndTwistorSpaces}, this will be the fundamental idea for constructing harmonic morphisms using twistor methods.


\chapter{Harmonic maps and twistor spaces}\label{Chapter:HarmonicMapsAndTwistorSpaces}



In \cite{Salamon:85} (Corollary 4.2), it is shown that conformal and harmonic immersions \linebreak$\varphi:M^2\to N^{2n}$ always arise as the projection $\varphi=\pi\circ\psi$ of suitable holomorphic maps $\psi:M^2\to\Sigma^+N$:

\corollarywithoutnumber{\ref{Corollary:SalamonCorollary4.2}}{Let $M^2$ be a Riemann surface, $N^{2n}$ an oriented even-dimensional manifold. Consider $\varphi:M^2\to N^{2n}$ an immersion. Then, $\varphi$ is a conformal and harmonic map if and only if $\varphi$ is (locally) the projection of a $(J^M,\J^2)$-holomorphic map $\psi:M^2\to\Sigma^+N$.}

The fact that projections of $\J^2$-holomorphic maps $\psi:M^2\to\Sigma^+N$ are harmonic is a consequence of the
following theorem (\cite{Salamon:85}, Theorem 3.5):
\theoremwithoutnumber{\ref{Theorem:SalamonTheorem3.5}}{Let $(M^{2m},J^M,g)$ be a cosymplectic manifold and
$N^{2n}$ an oriented Riemannian manifold. Take $\psi:M^{2m}\to\Sigma^+ N^{2n}$ a $(J^M,\J^2)$-holomorphic map.
Then, the projected map $\varphi:M^{2m}\to N^{2n}$, $\varphi=\pi\circ\psi$, is harmonic.}

\noindent This is essentially a consequence of Proposition \ref{Proposition:LichnerowiczProposition} (\cite{Lichnerowicz:70}, see \cite{Salamon:85} p. 175):

\propositionwithoutnumber{\ref{Proposition:LichnerowiczProposition}}{Let $M^{2m}$ be a cosymplectic manifold and $N^{2n}$ a $(1,2)$-symplectic manifold. Then, any holomorphic map $\varphi:M^2\to N^{2n}$ is harmonic.}

\label{Page:IntroductionTo:Chapter:HarmonicMapsAndTwistorSpaces}What really matters in proving Proposition
\ref{Proposition:LichnerowiczProposition} is that there is indeed an almost complex structure on $N$ which renders $\varphi$ holomorphic and is ``$(1,2)$-symplectic along $\varphi$": this is the idea behind Theorem \ref{Theorem:SalamonTheorem3.5}. As for the converse in Corollary \ref{Corollary:SalamonCorollary4.2}, it follows from Theorem \ref{Theorem:SalamonsTheorem4.1}.

The main purpose of this chapter is to give more general versions of these results and, at the same time, using the tools developed in Chapter \ref{Chapter:TwistorSpaces}, provide unified proofs for the above stated and for the new results. More precisely, we prove (\cite{SimoesSvensson:06}, see also Remark 5.7 of \cite{Rawnsley:85})

\theoremwithoutnumber{\ref{Theorem:NewTheorem35}}{Let $M$ be a $(1,2)$-symplectic manifold and $N^{2n}$ an oriented
even-dimensional Riemannian manifold. Consider a $(J^M,\J^2)$-holomorphic map $\psi:M^{2m}\to\Sigma^+N^{2n}$. Then,
the projected map $\varphi:M^{2m}\to N^{2n}$, $\varphi=\pi\circ\psi$, is pluriconformal and $(1,1)$-geodesic.}

In the same way that Theorem \ref{Theorem:SalamonTheorem3.5} is a generalization of Proposition \ref{Proposition:LichnerowiczProposition}, we show that the above theorem can be seen as a generalization of the following result (see \cite{BairdWood:03}, Lemma 8.2.1):

\propositionwithoutnumber{\ref{Proposition:821Wood}}{If $\varphi:M\to{N}$ is a holomorphic map between $(1,2)$-symplectic manifolds $(M,g,J^M)$ and $(N,h,J^N)$, then $\varphi$ is $(1,1)$-geodesic.}

We also establish a converse to Theorem \ref{Theorem:NewTheorem35}: namely, we see under what circumstances we can guarantee the existence of a twistor lift for a given pluriconformal $(1,1)$-geodesic map (Theorem \ref{Theorem:NewTheorem4.1} and Corollary \ref{Corollary:NewSalamonCorollary4.2}).

We can already derive a consequence from Theorem \ref{Theorem:NewTheorem35}, that will be the starting point for our
construction of harmonic morphisms in the next chapter (\cite{SimoesSvensson:06}):

\begin{corollarywithoutsection}\label{Corollary:MinimalSubmanifoldsAndTwistorMaps}
\index{Minimal~submanifold!}\index{Pluriminimal~submanifold!and~$\J^2$-holomorphic~maps}
If $(M,g,J)$ is a Kähler manifold, $\psi:M\to\Sigma^+N$ a $\J^2$-holomorphic map and $\varphi=\pi\circ\psi$ an immersion\footnote{Equivalently, $\psi$ is nowhere vertical.}, then $\varphi(M)\subseteq N$ is an immersed (pluri)minimal submanifold of $N$.
\end{corollarywithoutsection}
\begin{proof} As $M$ is $(1,2)$-symplectic, $\varphi$ is $(1,1)$-geodesic. But then it is also pluriharmonic as $M$ is Kähler. On the other hand, its pluriconformality is obvious from its holomorphicity relatively to the almost Hermitian structure induced by $\psi$. Therefore, the induced pull-back metric on $M$ is still Kähler (Theorem \ref{Theorem:PullBackMetricThroughAPluriharmonicAndPluriconformalMapIsKahler}). But pluriharmonicity is not influenced by change of compatible metrics so that $\varphi$ is still pluriharmonic for this new metric that makes it a Riemannian immersion. In particular, it is $(1,1)$-geodesic and therefore harmonic. Hence, $\varphi(M)\subseteq{N}$ is minimal.
\end{proof}
\begin{remarkwithoutsection}\label{Remark:RemarkTo:Corollary:MinimalSubmanifoldsAndTwistorMaps}
Notice the importance of $M$ being Kähler in Corollary \ref{Corollary:MinimalSubmanifoldsAndTwistorMaps}: we need $M$ to be $(1,2)$-symplectic; without this, we can not guarantee that the projected map $\varphi:M\to N$ is $(1,1)$-geodesic (Theorem \ref{Theorem:NewTheorem35}). On the other hand, we use the fact that $M$ is complex to make sure that the pull-back metric makes $M$ a $(1,2)$-symplectic manifold for which $\varphi$ is still $(1,1)$-geodesic (and therefore harmonic).
\end{remarkwithoutsection}


\section{Fundamental lemma}\label{Section:SectionIn:TwistorSpacesAndHarmonicMaps:FundamentalLemma}


Let $(M,g,J^M)$ be an almost Hermitian manifold and $(N,h)$ an oriented even-dimensional Riemannian
manifold. Let $\psi:M\to\Sigma^+N$, $\psi(x)=(\pi\circ\psi(x),J_\psi(x))$, be a map to the twistor
space $\Sigma^+N$ of $N$ and define $\varphi:M\to N$ by $\varphi=\pi\circ\psi$. We shall say that \emph{$\varphi$ is $(J^M,J_{\psi})$-holomorphic (at $x_0$)} if
\[\d\varphi_{x_0}(J^M X_{x_0})=J_{\psi}(x_0)(\d\varphi_{x_0} X_{x_0}),\quad\forall\,X_{x_0}\in{T}_{x_0}M\text{.}\]
Consider the vector bundle $\varphi^{-1}TN$ over $M$ with the orientation, metric and connection inherited from those on $TN$ and let $\Sigma^+\varphi^{-1}TN$ denote its twistor bundle. Let $\sigma_\psi$ denote the section \emph{associated} to $\psi$:\addtolength{\arraycolsep}{-1.2mm}
\[\begin{array}{llll}
\sigma_\psi:&M&\to&\Sigma^+\varphi^{-1}TN \\
~&x&\to& (x,J_{\psi}(x))\text{.}
\end{array}\]\addtolength{\arraycolsep}{1.2mm}\noindent
For the section $\sigma_\psi$, $(\varphi^{-1}TN,\varphi^{-1}h,\nabla^{\varphi^{-1}},\sigma_\psi)$ becomes an almost Hermitian vector bundle. Let $\omega_{\sigma}$ be the fundamental $2$-form associated with $\sigma_\psi$ and define $\d_a\omega_{\sigma}$ as in Chapter \ref{Chapter:TwistorSpaces}.

In the opposite direction, let $\varphi:M\to N$ be any smooth map and $\sigma$ a section of $\Sigma^+\varphi^{-1}TN$,
$\sigma:M\to\Sigma^+\varphi^{-1}TN$ with $\sigma(x)=(x,J_\sigma(x))$, $J_\sigma(x)\in\Sigma^+T_{\varphi(x)}N$. Then, $\sigma$ defines a map $\psi_\sigma:M\to\Sigma^+N$ by\addtolength{\arraycolsep}{-1.2mm}
\begin{equation}\label{Equation:SecondEquationIn:Lemma:FundamentalLemmaToProve3.5}
\begin{array}{llll}
\psi: & M&\to&\Sigma^+N \\
~ & x&\to&\big(\varphi(x),J_{\sigma}(x)\big)\text{.}
\end{array}
\end{equation}\addtolength{\arraycolsep}{1.2mm}\noindent

With this terminology, we can state the following result:

\begin{lemma}[\textup{Fundamental lemma}]\label{Lemma:FundamentalLemmaToProve3.5}
Let $\psi:M\to\Sigma^+N$, $\varphi=\pi\circ\psi$ and $\omega_\sigma$ be the fundamental $2$-form associated with $\sigma_\psi$. For any fixed $x_0\in{M}$, we have
\begin{equation}\label{Equation:FirstEquationIn:Lemma:FundamentalLemmaToProve3.5}
\psi\text{ is}\,(J^M,\J^a)\text{-holomorphic at~}x_0\impl\left\{
\begin{array}{l}
\d_a\omega_{\sigma}(x_0)=0 \\
\varphi\text{ is }(J^M,J_{\psi})\text{-holomorphic at~}x_0\
\end{array}\right.(a=1,2)
\end{equation}
Conversely, let $\varphi:M\to N$ be any smooth map, $\sigma$ a section of $\Sigma^+\varphi^{-1}TN$ and $\psi_\sigma$ as in \eqref{Equation:SecondEquationIn:Lemma:FundamentalLemmaToProve3.5}. Then, for any fixed $x_0\in{M}$, we have
\begin{equation}\label{Equation:ThirdEquationIn:Lemma:FundamentalLemmaToProve3.5}
\left\{\begin{array}{l}
\varphi\text{ is }(J^M,J_{\sigma})\text{-holomorphic at }x_0 \\
\d_a\omega_{\sigma}(x_0)=0 \
\end{array}\right.\,\impl\,\psi\text{ is~}(J^M,\J^a)\text{-holomorphic~at~}x_0\text{.}
\end{equation}
\end{lemma}

\begin{proof}
Let $\psi:M\to\Sigma^+N$ be $(J^M,\J^a)$-holomorphic and $\varphi=\pi\circ\psi$.

(i) It is easy to check that $\varphi$ is $(J^M,\J_\psi)$-holomorphic: in fact,
\[\varphi=\pi\circ\psi\,\impl\,\d\varphi(J^MX)=\d\pi(d\psi(J^MX))=\d\pi(\J^a\d\psi{X})=J_\psi\d\pi(\d\psi{X})=J_\psi\d\varphi{X}\]
since $\psi$ is $(J^M,\J^a)$-holomorphic and $\pi$ is ``holomorphic" in the sense of equation
\eqref{Equation:HolomorphicityOfTheProjectionMap}.

(ii) As in equation \eqref{Equation:DecompositionIntoHorizontalAndVerticalPartsOfTheDerivativeOfASection}, we can show that the decomposition of $\d\psi X$ into horizontal and vertical parts is given by
\begin{equation}\label{Equation:FirstEquationInTheProofOf:Lemma:FundamentalLemmaToProve3.5}\index{Decomposition!of~a~map~into~$\Sigma^+N$}
\{d\psi{X}-\nabla^{\varphi^{-1}}_XJ_{\psi}\}\oplus\nabla^{\varphi^{-1}}_XJ_{\psi}\in\H\oplus\V\text{.}
\end{equation}

(iii) We can now finish the proof of the first implication: we want to show that $\d_a\omega_{\sigma}=0$. Indeed, we can write
\[\psi\,\text{ is }(J^M,\J^a)\text{-holomorphic }\impl\,(\d\psi(X^{10}))^\V\in\V^{10,a}=\left\{\begin{array}{c}
T^{02^\star}\,(a=1)\text{,} \\
T^{20^\star}\,(a=2)\text{.}
\end{array}\right.\]
Hence, since $(d\psi(X))^\V=\nabla_X J_{\psi}$ we deduce, for $a=1$,
\[<\big(\nabla_{X^{10}}J_{\psi}\big)Y^{10}_{J_{\psi}},Z^{10}_{J_{\psi}}>=0\,\equi\,\big(\nabla_{X^{10}}\omega_{\sigma}\big)(Y^{10}_{J_{\psi}},Z^{10}_{J_{\psi}})=0\,\equi\,\d_1\omega_\sigma=0\text{,}\]
whereas for $a=2$ we have
\[<\big(\nabla_{X^{10}}J_{\psi}\big)Y^{01}_{J_{\psi}},Z^{01}_{J_{\psi}}>=0\,\equi\,\big(\nabla_{X^{10}}\omega_{\sigma}\big)(Y^{01}_{J_{\psi}},Z^{01}_{J_{\psi}})=0\,\equi\,\d_2\omega_\sigma=0\text{,}\]
as required.

Conversely, suppose now that we have a map $\varphi:M\to N$ and a section\addtolength{\arraycolsep}{-1.2mm}
\[\begin{array}{llll}
\sigma: & M&\to&\Sigma^+\varphi^{-1}TN \\
~&x&\to&(x,J_\sigma(x))\text{,~where~}J_\sigma(x):T_{\varphi(x)}N\to{T}_{\varphi(x)}N
\end{array}\]\addtolength{\arraycolsep}{1.2mm}\noindent
such that $\varphi$ is $(J^M,J_{\sigma})$-holomorphic and $\d_a\omega_{\sigma}=0$. We want to show that the map\addtolength{\arraycolsep}{-1.2mm}
\[\begin{array}{llll}
\psi: & M&\to&\Sigma^+TN \\
~& x&\to&(\varphi(x),J_\sigma(x))
\end{array}\]\addtolength{\arraycolsep}{1.2mm}\noindent
is $(J^M,\J^a)$-holomorphic. Take the decomposition of $T\Sigma^+\varphi^{-1}TN$ into its horizontal and vertical parts.

(iv) For the horizontal part, we have\addtolength{\arraycolsep}{-1.0mm}
\[\begin{array}{lll}\big(\d\psi(J^MX)\big)^\H&=&\d\pi|_\H^{-1}\circ\d\pi|_\H\big(\d\psi(J^MX)\big)^\H=\d\pi|_\H^{-1}\d\pi\big(\d\psi(J^MX)\big)\\
~&=&\d\pi|_\H^{-1}\d\varphi(J^MX)=\d\pi|_\H^{-1}J_{\sigma}\d\varphi{X}=\d\pi|_\H^{-1}J_{\sigma}\d\pi|_\H(\d\psi{X})^\H\\
~&=&\J^\H(\d\psi{X})^\H\text{.}\end{array}\]\addtolength{\arraycolsep}{1.0mm}\noindent

(v) For the vertical part, all we have to show is that $(\d\psi(X^{10})\big)^\V\in\V^{10}_a$ ($a=1,2$) and this follows
from $\d_a\omega_{\sigma}=0$, just as in (iii), concluding our proof.
\end{proof}

\begin{remark}\label{Remark:RemarkToLemma:Lemma:FundamentalLemmaToProve3.5:TheProblemItSolvesAsWeDoNotHaveAlmostComplexStructures}
Notice that the main problem this lemma solves is that already pointed out in Remark \ref{Remark:TheCaseOfGeneralVectorBundles}: when we are considering the bundle $\Sigma^+N$ we have two natural almost
complex structures: namely, $\J^1$ and $\J^2$. When we take instead the bundle $\Sigma^+V$, conditions $\d_a\omega=0$ arise more naturally then that of defining $\J^a$ on the latter bundle (in the way we pointed out in Remark
\ref{Remark:TheCaseOfGeneralVectorBundles})
and then speak of holomorphicity relatively to these structures. Choosing this terminology, $\J^a$-holomorphicity on a
bundle $\Sigma^+{N}$ is replaced by the conditions $\d_a\omega=0$ on $\Sigma^+V$ and $\varphi$ holomorphic. Comparing with Theorem \ref{Theorem:SalamonsTheorem3.2}, now we have to impose that $\varphi$ be holomorphic, which will give the horizontal holomorphicity of the lift, that was automatically guaranteed in the case of sections of
$\Sigma^+N$ (Remark \ref{Remark:TheImportanceOfBeingDealingWithSectionsInSalamonTheorem3.2}).
\end{remark}

\begin{corollary}\label{Corollary:AnotherVersionOfTheFundamentalLemmaUsingT10AndT01Spaces}
As before, let $\psi:M\to\Sigma^+ N$ be a smooth map from the almost Hermitian manifold $(M,g,J^M)$ and $\varphi:M\to N$,
$\varphi=\pi\circ\psi$ the projected map. Then, for each $x_0\in{M}$, $\psi$ is $\J^{1}$-holomorphic at $x_0$ if and only if $\varphi$ is $(J^M,J_\psi)$-holomorphic at $x_0$ and $\nabla_{T^{10}_{x_0}M}T^{10}_{J_\psi}N\subseteq{T}^{10}_{J_\psi(x_0)}N$; equivalently, if and only if $\varphi$ is $(J^M,J_\psi)$-holomorphic (at $x_0$) and
\begin{equation}\label{Equation:FirstEquationIn:Corollary:AnotherVersionOfTheFundamentalLemmaUsingT10AndT01Spaces}
\forall\,X^{10}\in{T}^{10}_{x_0}M,Y^{10}\in\Gamma(\varphi^{-1}T^{10}_{J_\psi}N)\,\exists\,Z_{x_0}^{10}\in\varphi^{-1}T^{10}_{J_\psi(x_0)}N:\,\nabla_{X^{10}_{x_0}}Y^{10}=Z_{x_0}^{10}\text{.}
\end{equation}
Analogously, $\psi$ is $\J^2$-holomorphic at $x_0$ if and only if $\varphi$ is $(J^M,J_\psi)$-holomorphic at $x_0$ and
$\nabla_{T_{x_0}^{01}M}T^{10}_{J_\psi}N\subseteq{T}^{10}_{J_\psi(x_0)}N$. Equivalently, if and only if $\varphi$ is
$(J^M,J_\psi)$-holomorphic and
\begin{equation}\label{Equation:FourthEquationIn:Corollary:AnotherVersionOfTheFundamentalLemmaUsingT10AndT01Spaces}
\forall\,X_{x_0}^{01}\in{T}_{x_0}^{01}M,\,Y^{10}\in\Gamma(\varphi^{-1}T^{10}_{J_\psi}N)\,\exists\,Z_{x_0}^{10}\in\varphi^{-1}T^{10}_{J_\psi(x_0)}N:\,\nabla_{X_{x_0}^{01}}Y^{10}=Z_{x_0}^{10}\text{.}
\end{equation}
\end{corollary}
\begin{proof}
Using Lemma \ref{Lemma:FundamentalLemmaToProve3.5}, we see that $\psi$ is $\J^a$-holomorphic if and only if $\varphi$ is
$(J^M,J_\psi)$-holomorphic and $\d_a\omega_\sigma=0$, where $\omega_\sigma$ is the fundamental $2$-form associated to $\psi$. From Lemma \ref{Lemma:Salamon1.2Generalized}, we have
\[\d_a\omega_\psi=0\,\equi\,\left\{\begin{array}{l}
\nabla_{T^{10}M}T^{10}_{J\psi}N\subseteq T^{10}_{J_\psi}N\,(a=1)\\
\nabla_{T^{01}M}T^{10}_{J\psi}N\subseteq T^{10}_{J_\psi}N\,(a=2)
\end{array}\right.\text{,}\]
which proves our assertion.
\end{proof}


\section{Harmonic maps and twistor projections}



\subsection{Proposition \ref{Proposition:LichnerowiczProposition}}\label{Subsection:ProofOfLichnerowiczProposition}


Recall the following result (\cite{Lichnerowicz:70}, see \cite{Salamon:85} p. 175):
\begin{proposition}\label{Proposition:LichnerowiczProposition}\index{Lichnerowicz~Proposition!}
Let $(M,g,J^M)$ be a cosymplectic manifold and $(N,h,J^N)$ a $(1,2)$-symplectic manifold. Then, any holomorphic map $\varphi:M\to{N}$ is harmonic.
\end{proposition}

As we want to establish the similarity between the different results in the introduction of this chapter, we now present a new proof of Proposition \ref{Proposition:LichnerowiczProposition}, done in several steps:

\noindent\textit{1$^{\text{st}}$ step. The $(0,1)$-part of $J^N(\trace_{g}\nabla^{\varphi^{-1}} J^N)$} (see also the proof of Lemma 8.1.5 in \cite{BairdWood:03}, p. 251).

If $\varphi$ is a holomorphic map between two almost Hermitian manifolds $M$ and $N$, then
\begin{equation}\label{Equation:EquationInTheFirstStepToProveLichnerowiczProposition}
\{J^N(\trace_{g}\nabla^{\varphi^{-1}}J^N)\}^{01}=4\{\sum_j\nabla^{\varphi^{-1}}_{\overline{Z}_j}\d\varphi{Z}_j\}^{01}
\end{equation}
where $Z_j$, $\overline{Z}_j$ are the vectors $\frac{1}{2}(X_j-iJX_j)$, $\frac{1}{2}(X_j+iJX_j)$, respectively, with \{$X_j,JX_j\}$ an orthonormal frame and
\[\trace_{g}\nabla^{\varphi^{-1}}J^N:=\sum_i(\nabla^N_{\d\varphi{X}_i}J^N)\d\varphi{X}_i,\,\{X_i\}\text{~orthonormal~for~the~metric~on~}M\text{.}\]

\noindent\textit{Proof of the 1$^{\text{st}}$ step.} The proof follows from noticing that $\frac{1}{2}(\Id+iJ)Jv=(Jv)^{01}=-iv$ and decomposing both sides in equation
\eqref{Equation:EquationInTheFirstStepToProveLichnerowiczProposition}.

\noindent \textit{2$^{\text{nd}}$ step. Decomposing the tension field} (see Lemma 8.1.5., \cite{BairdWood:03}, p. 251).

Let $\varphi:(M,g,J)\to(N,g,J)$ be a holomorphic map between two almost Hermitian manifolds. Then,
\begin{equation}\label{Equation:EquationInTheSecondStepToProveLichnerowiczProposition}
\tau(\varphi)=J^N\big(\trace_{g}\nabla^{\varphi^{-1}}J^N\big)-\d\varphi\big(J^M\trace\nabla^M J^M\big)\text{.}
\end{equation}
\noindent \textit{Proof of the 2$^{\text{nd}}$ step.} Since $\nabla\d\varphi$ is symmetric, we have
\begin{equation}\label{Equation:FirstEquationIn:EquationInTheProofOf:TheSecondStepToProveLichnerowiczProposition}\index{Tension~field!of~a~map~from~a~Hermitian~manifold}
\tau(\varphi)=4\sum_j\big(\nabla\d\varphi\big)(\overline{Z}_j,Z_j)\text{.}
\end{equation}
Hence,
\[\tau(\varphi)=4\sum_j\nabla_{\overline{Z}_j}^{\varphi^{-1}TN}\d\varphi{Z}_j-4\d\varphi(\sum_j\nabla^M_{\overline{Z}_j}Z_j)\text{.}\]
If $\varphi$ is a holomorphic map between two almost Hermitian manifolds, it preserves the $(1,0)$ and $(0,1)$-parts. So, since\addtolength{\arraycolsep}{-1.0mm}
\[\begin{array}{lll}4\big\{\d\varphi\sum_j\nabla_{\overline{Z}_j}^MZ_j\}^{01}&=&\d\varphi\big(J^M\trace\nabla{J}\big)^{01}\text{~~and}\\
4\sum_j\nabla_{\overline{Z}_j}^{\varphi^{-1}TN}\d\varphi{Z}_j&=&\big(J^N\trace_{g}\nabla^{\varphi^{-1}}J^N\big)^{01}\text{,~by~\eqref{Equation:EquationInTheFirstStepToProveLichnerowiczProposition},}
\end{array}\]\addtolength{\arraycolsep}{1.0mm}\noindent
we deduce
\[\big(\tau(\varphi)\big)^{01}=\big\{J^N\big(\trace_{g}\nabla^{\varphi^{-1}}J^N\big)-\d\varphi\big(J^M\trace\nabla^MJ^M\big)\big\}^{01}\text{.}\]
As $\tau(\varphi)(x)\in{T}_{\varphi(x)}N$ is real, we conclude that equation
\eqref{Equation:EquationInTheSecondStepToProveLichnerowiczProposition} is satisfied, as desired. \qed

\noindent\textit{3$^{\text{rd}}$ step. (Fundamental step) $\d_2\omega_{\sigma}=0$.}

We show that, if $\varphi:M\to N$ is a holomorphic map, with $M$ almost Hermitian and $N$ $(1,2)$-symplectic, then $\d_2\omega_\sigma=0$, where $\omega_\sigma$ is the fundamental $2$-form associated to $\sigma$ and $\sigma$ is the section of $\Sigma^+\varphi^{-1}TN$ defined in the natural way,\addtolength{\arraycolsep}{-1.2mm}
\[\begin{array}{llll}
\sigma: & M&\to&\Sigma^+\varphi^{-1}TN \\
~& x&\to&\big(x,J^N_{\varphi(x)}\big)\text{.}
\end{array}\]\addtolength{\arraycolsep}{1.2mm}\noindent

\noindent\textit{Proof of the 3$^{\text{rd}}$ step.} Let us start by noticing that $(\varphi^{-1}TN,\varphi^{-1}h,\nabla^{\varphi^{-1}},\sigma)$ is an almost Hermitian vector bundle and there is a natural identification between $(\varphi^{-1}TN)_x^{10}$ (respectively, $(\varphi^{-1}TN)_x^{01}$) and $T_{\varphi(x)}^{10}N$ (respectively, $T_{\varphi(x)}^{01}N$).

Now, from Lemma \ref{Lemma:Salamon1.2Generalized}, $\d_2\omega_{\sigma}=0$ if an only if, for all $X^{10}\in{T}^{10}M$, $Y^{01},Z^{01}\in(\varphi^{-1}TN)^{01}$
\[(\nabla_{X^{10}}^{\varphi^{-1}TN}\omega_{\sigma})(Y^{01},Z^{01})=0\,\equi\,\big(\nabla^{N}_{\d\varphi{X}^{10}}\omega_{J^N}\big)(\tilde{Y}^{01},\tilde{Z}^{01})=0,\quad\forall\,\tilde{Y}^{01},\tilde{Z}^{01}\in{T}^{01}N\text{.}\]
Since $\varphi$ is holomorphic, $\d\varphi (X^{10})\in{T}^{10}N$ so that the above equation is trivially true as $N$ is $(1,2)$-symplectic and so $\d_2\omega_{J^N}=0$ (Proposition \ref{Proposition:SalamonsLemma1.3}).\qed

\noindent\textit{4$^{\text{th}}$ step. Conclusion of the proof of Proposition \ref{Proposition:LichnerowiczProposition}.}

Using the second step and the fact that $M$ is cosymplectic,
\[\tau(\varphi)=J^N\big(\trace_{g}\nabla^{\varphi^{-1}}J^N\big)-\d\varphi(J^M\trace\nabla^MJ^M)=J^N\big(\trace_{g}\nabla^{\varphi^{-1}}J^N\big)\text{.}\]
From the first step,
\[\big\{J^N\big(\trace_{g}\nabla^{\varphi^{-1}}J^N\big)\big\}^{01}=4\sum_j\big\{\nabla_{\overline{Z}_j}^{\varphi^{-1}TN}\d\varphi{Z}_j\big\}^{01}\text{.}\]
Now, since $\overline{Z}_j\in{T}^{01}M$, $\d\varphi(Z_j)\in(\varphi^{-1}TN)^{10}$ and $\d_2\omega_{\sigma}=0$
(from the third step), we can use equation \eqref{Equation:FirstEquationIn:Lemma:Salamon1.2Generalized} in
Lemma \ref{Lemma:Salamon1.2Generalized} to deduce that $\nabla_{\overline{Z}_j}^{\varphi^{-1}TN}\d\varphi Z_j$ lies in
$(\varphi^{-1}TN)^{10}$ (or, equivalently, in ${T}^{10}N$. Therefore, $\nabla_{\overline{Z}_j}^{\varphi^{-1}TN}\d\varphi Z_j$ has vanishing $(0,1)$-part and $\big(\tau(\varphi)\big)^{01}=0$. Hence, by reality, $\tau(\varphi)=0$ and $\varphi$ is harmonic, as desired. \qed


\subsection{Twistor projections}\label{Subsection:ProofOfSalamonsTheorem3.5}


We have the following twistorial generalization of Proposition \ref{Proposition:LichnerowiczProposition} (\cite{Salamon:85}, Theorem 3.5):
\begin{theorem}\label{Theorem:SalamonTheorem3.5}\index{$\J^2$-holomorphic!and~harmonic~maps}\index{$\J^2$-holomorphic!projection}
Let $(M^{2m},J^M,g)$ be a cosymplectic manifold and $N^{2n}$ an oriented Riemannian manifold. Take
$\psi:M^{2m}\to\Sigma^+ N^{2n}$ a $(J^M,\J^2)$-holomorphic map. Then, the projected map $\varphi:M^{2m}\to N^{2n}$,
$\varphi=\pi\circ\psi$, is harmonic.
\end{theorem}

Notice that we could have considered $\Sigma^-N$ instead of $\Sigma^+N$ in this theorem. In fact, $\Sigma^-N=\Sigma^+\tilde{N}$, where $\tilde{N}$ denotes the manifold $N$ with the \emph{opposite} orientation. Hence, if $\psi:M^{2m}\to\Sigma^-N$ is a $\J^2$-holomorphic map as above, $\varphi:M\to\tilde{N}$ is harmonic. But harmonicity does not depend on the orientation, so that $\varphi$ is still harmonic as a map to $N$.

We now present a new proof for this result, compatible with that for Proposition \ref{Proposition:LichnerowiczProposition}.
We divide it into the same steps as the proof of that proposition:

\noindent\textit{1$^{\text{st}}$ step. The $(0,1)$-part of $J^N(\trace_{g}\nabla^{\varphi^{-1}} J^N)$.}

Just as before, we can prove that, with $\psi$ and $\varphi$ as in Theorem \ref{Theorem:SalamonTheorem3.5},
\begin{equation}\label{Equation:EquationInTheFirstStepToProveSalamonTheorem3.5}
\big(J_{\psi}\trace_{g}\nabla^{\varphi^{-1}}J_{\psi}\big)^{01}=4\sum_j\big(\nabla_{\overline{Z}_j}\d\varphi{Z}_j\big)^{01}
\end{equation}
where, for any orthonormal frame $\{X_i\}$ of $TM$,\addtolength{\arraycolsep}{-1.0mm}
\[\trace_{g}\nabla^{\varphi^{-1}}J_{\psi}=\sum_i(\nabla^{\varphi^{-1}}J_{\psi})(\d\varphi{X}_i,\d\varphi{X}_i)\,\big(=\sum_i\nabla_{X_i}^{\varphi^{-1}TN}J_{\psi}\d\varphi{X}_i-J_{\psi}(\nabla^M_{X_i}X_i)\big)\text{.}\]\addtolength{\arraycolsep}{1.0mm}\noindent

\noindent\textit{Proof of the 1$^{\text{st}}$ step.}
As before, the proof follows from noticing that $\frac{1}{2}(\Id+iJ)Jv=(Jv)^{01}=-iv$ and decomposing both sides in equation \eqref{Equation:EquationInTheFirstStepToProveSalamonTheorem3.5}, noticing that $\varphi$ is $(J^M,J_{\psi})$-holomorphic (Lemma \ref{Lemma:FundamentalLemmaToProve3.5}).

\noindent\textit{2$^{\text{nd}} step$. Decomposing the tension field.}

Let $\varphi$ and $\psi$ be in the conditions of Theorem \ref{Theorem:SalamonTheorem3.5}. Then, we have
\begin{equation}\label{Equation:EquationInTheSecondStepToProveSalamonTheorem3.5}
\tau(\varphi)=J_{\psi}(\trace_{g}\nabla^{\varphi^{-1}}J_{\psi})-\d\varphi(J^M\trace\nabla^MJ^M)\text{.}
\end{equation}

\noindent\textit{Proof of the 2$^{\text{nd}}$ step.} Once again, the proof goes exactly as before, as $\varphi$ becomes a $(J^M,J_\psi)$-holomorphic map.

\noindent\textit{3$^{\text{rd}}$ step. (Fundamental step) $\d_2\omega_{\sigma}=0$.}

Let $\varphi$ and $\psi$ be as given and consider the section $\omega_\sigma$ of $\Sigma^+\varphi^{-1}TN$ as in Lemma
\ref{Lemma:FundamentalLemmaToProve3.5}; then, $\d_2\omega_{\sigma}=0$.

\noindent\textit{Proof of the 3$^{\text{rd}}$ step.} The proof is precisely Lemma \ref{Lemma:FundamentalLemmaToProve3.5}(equation \eqref{Equation:FirstEquationIn:Lemma:FundamentalLemmaToProve3.5})
for the case $a=2$.

\noindent\textit{4$^{\text{th}}$ step. Conclusion of the proof of Theorem \ref{Theorem:SalamonTheorem3.5}.}

We repeat the fourth step of the previous proof, with $J^N$ replaced by $J_\psi$ throughout. \qed


\section{$(1,1)$-geodesic maps and twistor projections}


\subsection{Proposition~\ref{Proposition:821Wood}}\label{Subsection:ProofOfWoodsLemma8.2.1}


In the spirit of Proposition \ref{Proposition:LichnerowiczProposition}, we have the following result (see \cite{BairdWood:03}, Lemma 8.2.1):
\begin{proposition}\label{Proposition:821Wood}
If $\varphi:M\to{N}$ is a holomorphic map between $(1,2)$-symplectic manifolds $(M,g,J^M)$ and $(N,h,J^N)$, then $\varphi$ is $(1,1)$-geodesic.
\end{proposition}

In the same way that we divided the proof of Proposition \ref{Proposition:LichnerowiczProposition} into several steps to show the similarity between that result and Theorem \ref{Theorem:SalamonTheorem3.5}, we shall now divide the proof of Proposition \ref{Proposition:821Wood} to exhibit its relation with Theorem \ref{Theorem:NewTheorem35} and with the two results from the previous section:

\noindent\textit{1$^{\text{st}}$ step. Alternative characterizations of $(1,1)$-geodesic maps.}

\begin{proposition}[\textup{$(1,1)$-geodesic~maps:~alternative~condition}]\label{Lemma:AlternativeConditionFor(11)GeodesicMaps}\index{$(1,1)$-geodesic~map!alternative~condition}
Let $\varphi:M\to N$ be a map from a Hermitian manifold $(M,g,J)$ to a Riemannian manifold $N$. Then, $\varphi$ is $(1,1)$-geodesic if and only if for every point $x\in{M}$ there is an almost Hermitian structure $J_{\varphi_{x}}$ on $T_{\varphi_x}N$ such that\footnote{Notice that we do not require any compatibility between $\varphi$ and the almost Hermitian structure; \textit{i.e.}, $\varphi$ may or may not be holomorphic with respect to this structure; see the next corollary for further details.}
\begin{equation}\label{Equation:FirstEquationIn:Lemma:AlternativeConditionFor(11)GeodesicMaps}
(\nabla\,\d\varphi)_{x}(Y^{10},Z^{01})\in{T}^{10}_{J_{\varphi_{x}}}N,\quad\forall\,Y^{10}\in{T}^{10}_{x}M,\,Z^{01}\in{T}^{01}_{x}M\text{.}
\end{equation}
\end{proposition}
\begin{proof}
If $\varphi$ is $(1,1)$-geodesic, take any almost Hermitian structure on $T_{\varphi_{x}}N$. Then, since
$(\nabla\,\d\varphi)_{x}(Z^{01},Y^{10})=0$ (Definition \ref{Definition:11GeodesicMap}),
\eqref{Equation:FirstEquationIn:Lemma:AlternativeConditionFor(11)GeodesicMaps}
obviously holds. Conversely, suppose that
\eqref{Equation:FirstEquationIn:Lemma:AlternativeConditionFor(11)GeodesicMaps}
holds for some $J_{\varphi_{x}}$. Then, taking $Z^{10}=X-iJX$, $Z^{01}=X+iJX$, $X\in{T}_{x}M$, we deduce
\[\nabla\,\d\varphi(Z^{01},Z^{10})\in{T}^{10}_{J_{\varphi_{x}}}N\text{.}\]
But since
\[\nabla\,\d\varphi(Z^{01},Z^{10})=\nabla\,\d\varphi(X,X)+\nabla\,\d\varphi(JX,JX)\]
is real, we deduce $\nabla\,\d\varphi(Z^{01},Z^{10})=0$. Now, using a polarization argument,\addtolength{\arraycolsep}{-1.0mm}
\[\begin{array}{ll}~&\nabla\,\d\varphi(\overline{Z},Z)=0,\quad\forall\,Z\in{T}^{10}M\,\impl\,\nabla\,\d\varphi(\overline{Z}+\overline{Y},Z+Y)=0\\
\impl&\nabla\,\d\varphi(\overline{Z},Z)+\nabla\,\d\varphi(\overline{Y},Y)+\nabla\,\d\varphi(\overline{Z},Y)+\nabla\,\d\varphi(\overline{Y},Z)=0\\
\impl&\Real\big(\nabla\,\d\varphi(Y,\overline{Z})\big)=0\text{.}\end{array}\]\addtolength{\arraycolsep}{1.0mm}\noindent
But then $\nabla\,\d\varphi(Y,\overline{Z})\in{T}^{10}_{J_{\varphi_{x}}}N$ and has vanishing real part, which implies that it must be zero. Hence, $\varphi$ is $(1,1)$-geodesic.
\end{proof}\label{Page:CompatibleTwistorLift}\index{Almost~Hermitian~structure!(strictly)~compatible}\index{Compatible!almost~Hermitian~structure}\index{Pluriconformal~map!and~(strictly)~compatible~structures}\index{Compatible!twistor~lift|see{Compatible~twistor~lift}}\index{Compatible~twistor~lift!}\index{Twistor~lift!compatible}\index{Twistor~lift!strictly~compatible}\index{Compatible~twistor~lift!strictly}

Given a smooth map $\varphi:(M,g,J)\to N$ obtained as the projection $\varphi=\pi\circ\psi$ of a map $\psi$ to $\Sigma^+N$, $\varphi=\pi\circ\psi$, without requiring further conditions \textit{a priori} on $\psi$, nothing guarantees that $\varphi$ is holomorphic relatively to the induced almost Hermitian structure $J_\psi$. However, if it is, we shall say that the structure $J_\psi$ is \emph{strictly compatible} with $\varphi$ (or that the map $\psi$ is a \emph{strictly compatible twistor lift} of $\varphi$). Of course, such a structure $J_{\psi}$ can exist if and only if $\d\varphi(T^{10}M)\subseteq T^{10}_{J_\psi}N$ is isotropic: in other words, if and only if $\varphi$ is (weakly)
pluriconformal\footnote{In the case $M^2$ is a Riemann surface, if and only if $\varphi$ is (weakly) conformal.}. If $J_\psi$ preserves $\d\varphi(TM)$ but does not necessarily render $\varphi$ holomorphic, we shall say that $J_\psi$ (or the map $\psi$) is \emph{compatible} with $\varphi$.

\begin{lemma}[\textup{More~on~$(1,1)$-geodesic~maps}]\label{Lemma:About(11)GeodesicMapsAndOneVersusAllAlmostHermitianStructures}\index{$(1,1)$-geodesic~map!and~almost~Hermitian~structures}

The following conditions are equivalent:

\textup{(i)} $\varphi$ is $(1,1)$-geodesic.

\textup{(ii)} For each $x\in{M}$ there is a Hermitian structure $J_{\varphi_{x}}$ in $T_{\varphi_{x}}N$ such that
\eqref{Equation:FirstEquationIn:Lemma:AlternativeConditionFor(11)GeodesicMaps} holds.

\textup{(iii)} For each $x\in{M}$ and for all almost Hermitian structures $J_{\varphi_{x}}$ at $T_{\varphi_{x}}N$,
\eqref{Equation:FirstEquationIn:Lemma:AlternativeConditionFor(11)GeodesicMaps} holds.

Moreover, if $\varphi$ is a (weakly) pluriconformal map, then the following are equivalent:

\textup{(i')} $\varphi$ is $(1,1)$-geodesic.

\textup{(ii')} For each $x\in{M}$ there is a compatible Hermitian structure $J_{\varphi_{x}}$ in $T_{\varphi_{x}}N$ such that \eqref{Equation:FirstEquationIn:Lemma:AlternativeConditionFor(11)GeodesicMaps} holds.

\textup{(iii')} For each $x\in{M}$ and for all compatible Hermitian structures $J_{\varphi_{x}}$ at $T_{\varphi_{x}}N$, \eqref{Equation:FirstEquationIn:Lemma:AlternativeConditionFor(11)GeodesicMaps} holds.
\end{lemma}
\begin{proof}
We have seen in the preceding Lemma \ref{Lemma:AlternativeConditionFor(11)GeodesicMaps} that conditions (i) and (ii) are equivalent. On the other hand, (iii) implies (ii). Finally, if (i) holds then $\nabla\d\varphi(T^{10}M,T^{10}M)=0$ and consequently (iii) also holds.

For the second part, if (iii') holds so does (ii'), and consequently so does (i') with the same proof as in Lemma \ref{Lemma:AlternativeConditionFor(11)GeodesicMaps}. If (ii') is true, then $\nabla\d\varphi(Y^{10},Z^{01})=0$ and so (iii') holds. Finally, if (i') is satisfied, since $\varphi$ is (weakly) pluriconformal, we can find a Hermitian structure on $T_{\varphi_{x}}N$ for which $\varphi$ is holomorphic at $x$ and, of course, for this structure, \eqref{Equation:FirstEquationIn:Lemma:AlternativeConditionFor(11)GeodesicMaps} is verified, so that (ii') is implied by (i'), concluding our proof.
\end{proof}

We now return to the proof of Proposition \ref{Proposition:821Wood}:

\noindent\textit{2$^{\text{nd}}$ step. Decomposing $\nabla \d\varphi(Y^{10},Z^{01})$.}

We have
\begin{equation}\label{Equation:FirstEquationInTheSecondStepOf:Subsection:ProofOfWoodsLemma8.2.1}
\nabla\,\d\varphi(Y^{10},Z^{01})=\nabla\d\varphi(Z^{01},Y^{10})=\nabla^{\varphi^{-1}}_{Z^{01}}\d\varphi(Y^{10})-\d\varphi\big(\nabla^M_{Z^{01}}Y^{10}\big)\text{.}
\end{equation}

\noindent\textit{Proof of the 2$^{\text{nd}}$ step.} Immediate.

\noindent\textit{3$^{\text{rd}}$ step. (Fundamental step) $\d_2\omega_{\sigma}=0$.}

As in the third step of the proof of Proposition \ref{Proposition:LichnerowiczProposition}, $\d_2\omega_{\sigma}=0$, where $\omega_{\sigma}$ is the fundamental $2$-form associated with $\sigma$,\addtolength{\arraycolsep}{-1.2mm}
\[\begin{array}{llll}
\sigma: & M&\to&\Sigma^+\varphi^{-1}TN \\
~& x&\to&\big(x,J^N_{\varphi(x)}\big)\text{.}
\end{array}\]\addtolength{\arraycolsep}{1.2mm}\noindent

\noindent\textit{Proof of the 3$^{\text{rd}}$ step.} The proof goes precisely as in the proof of Proposition \ref{Proposition:LichnerowiczProposition}.

\noindent\textit{4$^{\text{th}}$ step. Conclusion of the proof of Proposition \ref{Proposition:821Wood}.}

Using the second step, \eqref{Equation:FirstEquationInTheSecondStepOf:Subsection:ProofOfWoodsLemma8.2.1} holds. Since $\varphi$ is holomorphic, $\d\varphi(Y^{10})\in{T}^{10}N$; hence, using the third step and Lemma
\ref{Lemma:Salamon1.2Generalized} (equation \eqref{Equation:SecondEquationIn:Lemma:Salamon1.2Generalized}) we deduce that $\nabla_{Z^{01}}\d\varphi (Y^{10})$ lies in $T^{10}N$. As for the second term in the right-hand side of \eqref{Equation:FirstEquationInTheSecondStepOf:Subsection:ProofOfWoodsLemma8.2.1}, since $M$ is $(1,2)$-symplectic, using Proposition \ref{Proposition:SalamonsLemma1.3}, we deduce that $\nabla_{Z^{01}}Y^{10}$ lies in ${T}^{10}M$ and using the holomorphicity of $\varphi$ we can therefore write
\[\nabla\d\varphi(Y^{10},Z^{01})\in{T}^{10}_{J_\psi}N\text{.}\]
Lemma \ref{Lemma:AlternativeConditionFor(11)GeodesicMaps} finishes the proof by showing that $\varphi$ is $(1,1)$-geodesic.\qed


\subsection{Twistor projections}\label{Subsection:ProofOfNewTheorem3.5}


We can now generalize Proposition \ref{Proposition:821Wood} (\cite{SimoesSvensson:06}, see also Remark 5.7 of
\cite{Rawnsley:85}):
\begin{theorem}\label{Theorem:NewTheorem35}\index{$\J^2$-holomorphic!projection}\index{$\J^2$-holomorphic!and~$(1,1)$-geodesic~maps}
Let $M$ be a $(1,2)$-symplectic manifold and $N^{2n}$ an oriented even-dimensional Riemannian manifold. Consider
a $(J^M,\J^2)$-holomorphic map $\psi:M^{2m}\to\Sigma^+N^{2n}$. Then, the projected map $\varphi:M^{2m}\to N^{2n}$, $\varphi=\pi\circ\psi$, is $(1,1)$-geodesic\footnote{and pluriconformal (see Remark \ref{Remark:ConformalAndPluriconformalMapsFromARiemannSurface}). We could also establish a similar result replacing $\Sigma^+N$ with $\Sigma^-N$, with a justification similar to that after Theorem \ref{Theorem:SalamonTheorem3.5}.}.
\end{theorem}

\begin{proof}

The proof goes just as before: the first and second steps are similar to the preceding ones. As for the remaining steps:

\textit{3$^{\text{rd}}$ step. (Fundamental step) $\d_2\omega_{\sigma}=0$.}

Let $\psi$ and $\varphi$ be as given and consider the section of $\Sigma^+\varphi^{-1}TN$ as in Lemma \ref{Lemma:FundamentalLemmaToProve3.5}. Then, $\d_2\omega_{\sigma}=0$.

\noindent\textit{Proof of the 3$^{\text{rd}}$ step.} The proof of this step is precisely Lemma \ref{Lemma:FundamentalLemmaToProve3.5}.

\textit{4$^{\text{th}}$ step. Conclusion of the proof of Theorem \ref{Theorem:NewTheorem35}.}

Using the second step, equation \eqref{Equation:FirstEquationInTheSecondStepOf:Subsection:ProofOfWoodsLemma8.2.1}:
\[\nabla\,\d\varphi(Y^{10},Z^{01})=\nabla^{\varphi^{-1}}_{Z^{01}}\d\varphi{Y}^{10}-\d\varphi(\nabla^M_{Z^{01}}Y^{10})\]
holds. Since $\varphi$ is $(J^M,J_\psi)$-holomorphic, $\d\varphi({Y}^{10})$ lies in $(\varphi^{-1}TN)^{10}$; since $\d_2\omega_\psi=0$, using Lemma \ref{Lemma:Salamon1.2Generalized} we can deduce that $\nabla_{Z^{01}}(\d\varphi Y^{10})$ belongs to $(\varphi^{-1}TN)^{10}_{J\psi}$. As for the second term on the right-hand side of
\eqref{Equation:FirstEquationInTheSecondStepOf:Subsection:ProofOfWoodsLemma8.2.1}, since $M$ is $(1,2)$-symplectic and $\varphi$ $(J^M,J_\psi)$-holomorphic, we again deduce $\d\varphi(\nabla_{Z^{01}}Y^{10})\in{T}^{10}_{J_\psi}N$. Hence, using Lemma \ref{Lemma:AlternativeConditionFor(11)GeodesicMaps} we conclude that $\varphi$ is $(1,1)$-geodesic.
\end{proof}


\section{Lifts of harmonic maps}\label{Section:LiftsOfHarmonicMaps}


We already know that given a $\J^2$-holomorphic map $\psi:M\to\Sigma^+N^{2n}$, the projected map $\varphi=\pi\circ\psi$ is harmonic (if $M$ is cosymplectic) or $(1,1)$-geodesic (if $M$ is $(1,2)$-symplectic). Moreover, $\varphi$ is also (weakly) pluriconformal as it is $(J^M,J_{\psi})$-holomorphic, so it is (weakly) conformal when $M^2$ is a Riemann surface (Remark \ref{Remark:ConformalAndPluriconformalMapsFromARiemannSurface}). The idea is now to look for converses of these results. In \cite{Salamon:85}, a converse is obtained in the case of Riemann surfaces (see Theorem \ref{Theorem:SalamonsTheorem4.1}) and we generalize this to higher dimensions (Theorem \ref{Theorem:NewTheorem4.1}). In order to clarify the relation between these two last results, we reformulate the first, using the more general setting developed in Chapter \ref{Chapter:TwistorSpaces}.

Thus, in the next sections we shall be looking for ways of constructing a $\J^2$-holomorphic map $\psi:M\to\Sigma^+N$ with $\pi\circ\psi=\varphi$, where $\varphi:M^{2m}\to N^{2n}$ is a given map.


\subsection{The~case~of~Riemann~surfaces}\label{Subsection:SubsectionOfSection:LiftsOfHarmonicMaps:TheRiemannSurfacesCase}


Let $M^2$ be a Riemann surface and $\varphi:M^2\to N$ a conformal immersion, where $(N^{2n},h)$ is an oriented even-dimensional Riemannian manifold. We are searching for $\J^2$-holomorphic maps $\psi:M^2\to\Sigma^+N$ with $\varphi=\pi\circ\psi$, where $\pi:\Sigma^+N\to N$ is the canonical projection. In particular, if such a map exists, $\varphi$ is $(J^M,J_{\psi})$-holomorphic and, therefore, $\psi$ is a strictly compatible twistor lift for $\varphi$. Hence, we are only interested in almost Hermitian structures for $N$ defined at points $\varphi(x)$ (\textit{i.e.},
\emph{along} $\varphi$) which render $\varphi$ holomorphic.\index{Almost~Hermitian~structure!along a map}

Consider then the vector bundle $V$ over $M^2$ given at each point $x\in{M}$ as the orthogonal complement (for the $N$-metric) of $\d\varphi_x(T_xM)$, which is an even-dimensional vector subbundle of $\varphi^{-1}TN\to M^2$. Consider $V$ equipped with the metric and the connection induced by those on $\varphi^{-1}TN$:\addtolength{\arraycolsep}{-1.0mm}
\[\begin{array}{rll}<X_x,Y_x>_{V_x}&=&<X_x,Y_x>_N,\quad\forall\,X_x,Y_x\in{V}_x,\\
\nabla_{X}^VY&=&\proj_{V}(\nabla^{\varphi^{-1}}_XY),\quad\forall\,X\in\Gamma(TM),\,Y\in\Gamma(V)
\end{array}\]\addtolength{\arraycolsep}{1.0mm}\noindent
where $\proj_V$ is the orthogonal projection onto $V$. Since $\nabla^V g_V=0$, we have a Riemannian vector bundle. Finally, define an orientation on $V$ in the following way: given $x$ on $M^2$, on $\d\varphi_x(T_xM)$ use the orientation transported from $T_xM$\footnote{\textit{I.e.}, if $\{X_1,X_2\}$ is a positive basis for $T_xM$, $\{\d\varphi_xX_1,\d\varphi_xX_2\}$ is a positive basis for $\d\varphi_xT_xM$.} and define an orientation on $V_x$ by the condition that $\{Y_1,...,Y_{2n-2}\}$ is positive for $V_x$ if and only if $\{Y_1,...,Y_{2n-2},X_1,X_2\}$ is a positive basis for $T_{\varphi(x)}N$. We are therefore left with an oriented even-dimensional Riemannian vector bundle on $M^2$ as in Section \ref{Section:TheBundleSigmaPlusV} so that we can consider the associated bundle $\Sigma^+V\to M^2$.

To each positive almost Hermitian structure on $V_x$ corresponds one and only one positive almost Hermitian structure on
$T_{\varphi(x)}N$ such that $\varphi$ is holomorphic at $x$: if $J_{V_x}$ is defined on $V_x$, $J_{\varphi_x}$ is defined on $T_{\varphi(x)}N$ by
\begin{equation}\label{Equation:DefinitionOfTheHermitianStructureOnTvarphixNFromTheOneInVx}
\begin{array}{l}
J_{\varphi_x}:T_{\varphi(x)}N^{2n}\to T_{\varphi(x)}N^{2n}, \\
J_{\varphi_x}=\left\{\begin{array}{l}
J_{V_x},\text{ on }V_x\text{,}\\
J^M\text{~transported~from~$T_xM$~\textit{via}~}\varphi=\text{~rotation~by~}+\frac{\pi}{2}\text{~on~}\d\varphi_x(T_xM)\text{.}
\end{array}\right.
\end{array}
\end{equation}\index{Koszul-Malgrange!Theorem!and~Riemann~surfaces}\index{Riemann~surface!vanishing~of~$R^{02}_V$}\index{Riemann~surface!Koszul-Malgrange~Theorem}\noindent
This means that the set of positive almost Hermitian structures along $\varphi$ which render $\varphi$ holomorphic is precisely $\Sigma^+V$. Notice the importance of $\varphi$ being conformal: if this was not so, we could not make $J_{\varphi_x}$ into a metric-preserving complex structure on $T_{\varphi(x)}N$. Since $M^2$ is a Riemann surface, the $(0,2)$-part of the curvature of $V$ vanishes. Hence, by Theorem \ref{Theorem:KoszulMalgrangeComplexStructureOnSigmaPlusV}, we know that $\Sigma^+V$ is a holomorphic vector bundle over our Riemann surface $M^2$ with the Koszul-Malgrange complex structure $\J^{KM}$.

We consider the map\addtolength{\arraycolsep}{-1.2mm}
\begin{equation}\label{Equation:TheMapFromSigmaPlusVIntoSigmaPlusNOnTheSalamonTheorem4.1Proof}
\begin{array}{lcll}
\eta: & \Sigma^+V&\to&\Sigma^+N \\
~& (x,J_{V_x})&\to& (\varphi(x),J_{\varphi_x})\text{.}
\end{array}
\end{equation}\addtolength{\arraycolsep}{1.2mm}\noindent
Then, we have the following theorem (see also \cite{Salamon:85}, Theorem 4.1):

\begin{theorem}\label{Theorem:SalamonsTheorem4.1}\index{$\J^2$-holomorphic!lift}
Let $\varphi:M^2\to N^{2n}$ be a conformal immersion as above, where $M^2$ is a Riemann surface and $N^{2n}$ an oriented
even-dimensional Riemannian manifold. Then, the following conditions are equivalent:

\textup{(i)} The map $\varphi$ is harmonic.

\textup{(ii)} The map $\eta$ is $(\J^{KM},\J^2)$-holomorphic.

\textup{(iii)} (The image of) $\eta$ is $\J^2$-stable.
\end{theorem}

This result follows from a more general result, which we now give.


\subsection{Lifts of $(1,1)$-geodesic maps}


We already know that the projection $\varphi=\pi\circ\psi$ of a $(J^M,\J^2)$-holomorphic map $\psi:M\to\Sigma^+N^{2n}$ from a $(1,2)$-symplectic manifold $M^{2m}$ is a $(1,1)$-geodesic and pluriconformal map. We now look for a converse of this result, \textit{i.e.}, given a $(1,1)$-geodesic map from a $(1,2)$-symplectic manifold to an oriented even-dimensional Riemannian manifold $N^{2n}$, when is it the projection of a $(J^M,\J^2)$-holomorphic map? For reasons that will become clear during the proof, this is the case whenever $M$ is Kähler and the normal bundle $V=(\d\varphi TM)^\bot$ with the induced connection and metric from $\varphi^{-1}TN$ has vanishing $(0,2)$-part of its curvature:
\begin{equation}\label{Equation:(02)PartOfTheCurvature}
R_V(T^{10}M,T^{10}M)=0.
\end{equation}

So, let $\varphi:M^{2m}\to N^{2n}$ be a pluriconformal map and set $V=\big(\d\varphi(TM)\big)^\bot$. Consider the twistor bundle $\Sigma^+V$ of $V$: this bundle has the Koszul-Malgrange holomorphic structure provided that the $(0,2)$-part of the curvature of $V$ vanishes. For each $x\,\in\,M$, define
\begin{equation}\label{Equation:DefinitionOfTheHermitianStructureOnTvarphixNFromTheOneInVxInTheGeneralCase}
\begin{array}{l}
J_{\varphi_x}:T_{\varphi(x)}N^{2n}\to T_{\varphi(x)}N^{2n}, \\
J_{\varphi_x}=\left\{
 \begin{array}{l}
 J_{V_x}\text{,~on~}V_x\text{,} \\
 \tilde{J}_x\text{,~on~}\d\varphi_{x}(T_xM),
 \end{array}\right.
\end{array}
\end{equation}
where $\tilde{J}_x$ is the almost Hermitian structure induced\footnote{Notice that $\tilde{J}_x$ is defined by declaring its $(1,0)$-subspace to be $\d\varphi(T^{10}M)$, which is possible since we are imposing that $\varphi$ be pluriconformal (see p.\ \pageref{Page:kDimensionalIsotropicSubspacesAndHermitianStructures})} on $\d\varphi(T_xM)$ from that on $T_xM$ \textit{via} $\varphi$.

Consider the map $\eta:\Sigma^+V\to\Sigma^+N$, defined as before:\addtolength{\arraycolsep}{-1.2mm}
\begin{equation}\label{Equation:TheMapFromSigmaPlusVIntoSigmaPlusNOnTheSalamonNewTheorem4.1}
\begin{array}{lcll}
\eta: & \Sigma^+V&\to&\Sigma^+N \\
~& (x,J_{V_x})&\to& (\varphi(x),J_{\varphi_x})\text{.}
\end{array}
\end{equation}\addtolength{\arraycolsep}{1.2mm}\noindent

The following result was obtained in joint work with M.~Svensson (\cite{SimoesSvensson:06}):

\begin{theorem}\label{Theorem:NewTheorem4.1}\index{$\J^2$-holomorphic!lift}
Let $\varphi:M^{2m}\to N^{2n}$ be a pluriconformal map, $M^{2m}$ a Kähler manifold and $N^{2n}$ an oriented Riemannian manifold. Set $V=\big(\d\varphi(TM)\big)^\bot$ and suppose that $R^{02}_V=0$ (\textit{i.e.}, the $(0,2)$-part of the curvature of the normal bundle $V$ vanishes). Then, the following conditions are equivalent:

\textup{(i)} The map $\varphi:M\to N$ is $(1,1)$-geodesic.

\textup{(ii)} The map $\eta$ is $(\J^{KM},\J^2)$-holomorphic.

\textup{(iii)} (The image of) $\eta$ is $\J^2$-stable.
\end{theorem}

Notice that we cannot expect fewer requirements on the manifold $M$: it should be $(1,2)$-symplectic if one wants to get $(1,1)$-geodesic maps and in order to get a holomorphic bundle (which will be what guarantees the existence of holomorphic maps to it!) we must also impose that $M$ be Hermitian. Hence, $M$ must be a Kähler manifold and $V$ must have $R^{02}_V=0$ so that we can introduce the complex structure $J^{KM}$.

In the case where $M^2$ is a Riemann surface, the condition on the curvature is automatically satisfied and pluriconformality reduces to conformality (Remark \ref{Remark:ConformalAndPluriconformalMapsFromARiemannSurface}). In particular, Theorem \ref{Theorem:SalamonsTheorem4.1} is just a particular case of Theorem \ref{Theorem:NewTheorem4.1}.

\noindent\textit{Proof of Theorem \ref{Theorem:NewTheorem4.1}.}

\noindent\textit{1$^{\text{st}}$ step. The complex structure on $\Sigma^+V$.}

We can put the Koszul-Malgrange holomorphic structure on $\Sigma^+V$. Moreover, identifying each almost Hermitian $J_V$ structure with its associated $(1,0)$-subspace, $s^{10}_V$,
\begin{equation}\label{Equation:FirstEquationInTheFirstStepOf:Theorem:NewTheorem4.1}
J_V\text{~is~a~holomorphic~section~if~and~only~if~}\nabla^V_{T^{01}M}s^{10}_V\subseteq{s}^{10}_V\text{.}
\end{equation}

\noindent\textit{Proof of the 1$^{\text{st}}$ step.} The fact that we can put the Koszul-Malgrange holomorphic structure on $\Sigma^+V$ is guaranteed by the assumption $R^{02}_V=0$. The fact that holomorphic sections are characterized by equation \eqref{Equation:FirstEquationInTheFirstStepOf:Theorem:NewTheorem4.1} is a direct consequence of \eqref{Equation:ConditionForHolomorphicityOfSectionsOfTheBundleSoVcnWithItsKoszulMalgrangeStructure}.

\noindent\textit{2$^{\text{nd}}$ step. (Fundamental step) $\d_2\omega_{J_\varphi}=0$.}

We prove that, if $\varphi$ is a $(1,1)$-geodesic pluriconformal map from a Kähler manifold to a Riemannian manifold with $R^{02}_V=0$, then for every $\J^{KM}$-holomorphic section $s^{10}_V$ of $\Sigma^+V$, the associated section $s^{10}_{TN}\simeq J_{\varphi}$ of $\Sigma^+\varphi^{-1}TN$ has $\d_2\omega_{J_\varphi}=0$.

\noindent\textit{Proof of the 2$^{\text{nd}}$ step.} We have to show that
\[\nabla^{\varphi^{-1}}_{T^{01}M}s^{10}_{TN}\subseteq s^{10}_{TN}\]
where $s^{10}_{TN}=s^{10}_V\oplus \d\varphi(T^{10}M)$. Now
\begin{equation}\label{Equation:SecondEquationOnTheProofOfTheSecondStepOfTheProofOf:Theorem:NewTheorem4.1}
\nabla^{\varphi^{-1}}_{T^{01}M}(s^{10}_V\oplus\d\varphi(T^{10}M))\subseteq\nabla^V_{T^{01}M}s^{10}_V+\nabla_{T^{01}M}^{TM}s^{10}_V+\nabla^{\varphi^{-1}}_{T^{01}M}\d\varphi(T^{10}M)\text{.}
\end{equation}
Since $s^{10}_V$ is $\J^{KM}$-holomorphic, $\nabla^V_{T^{01}M}s^{10}_V\subseteq{s}^{10}_V\subseteq s^{10}_{TN}$. For the third term on the right-hand side of \eqref{Equation:SecondEquationOnTheProofOfTheSecondStepOfTheProofOf:Theorem:NewTheorem4.1}, we have
\[\nabla^{\varphi^{-1}}_{T^{01}M}\d\varphi(T^{10}M)=\nabla\d\varphi(T^{01}M,T^{10}M)+\d\varphi(\nabla_{T^{01}M}^{M}T^{10}M)=\d\varphi(\nabla_{T^{01}M}^{M}T^{10}M)\text{,}\]
since $\varphi$ is $(1,1)$-geodesic. On the other hand, as $M$ is $(1,2)$-symplectic, using Proposition \ref{Proposition:SalamonsLemma1.3}, we deduce $\nabla_{T^{01}M}T^{10}M\subseteq T^{10}M$ so that $\nabla^{\varphi^{-1}}_{T^{01}M}\d\varphi(T^{10}M)\subseteq\d\varphi(T^{10}M)\subseteq{s}^{10}_{TN}$. Finally, for the second term,
\[\nabla^M_{T^{01}}s^{10}_V\,\subseteq\,\,\d\varphi(T^{01}M)\oplus \d\varphi(T^{10}M)\text{;}\]
this term will lie in $\d\varphi(T^{10}M)$ if and only if $<\nabla^M_{T^{01}M}s^{10}_V,\d\varphi(T^{10}M)>$ vanishes.
But\addtolength{\arraycolsep}{-1.0mm}
\[\begin{array}{lll}
<\nabla^M_{X^{01}}s^{10}_V,\d\varphi(T^{10}M)>&=&<\nabla^{\varphi^{-1}}_{X^{01}}s^{10}_V,\d\varphi(T^{10}M)>\\
~&=&X^{01}<s^{10}_V,\d\varphi(T^{10}M)>-<s^{10}_V,\nabla^{\varphi^{-1}}_{X^{01}}\d\varphi(T^{10}M)>\text{,}
\end{array}\]\addtolength{\arraycolsep}{1.0mm}\noindent
which vanishes as $s^{10}_V$ and $\d\varphi(T^{10}M)$ are orthogonal and $\nabla_{T^{01}M}\d\varphi(T^{10}M)\subseteq\d\varphi(T^{10}M)$ since $M$ is $(1,2)$-symplectic.

Once again, notice the importance of imposing that $\d_2 \omega_{J^M}=0$ (\textit{i.e.}, imposing that $M$ be $(1,2)$-symplectic) and $\d_1\omega_{J^M}=0$ (\textit{i.e.}, requiring $M$ to be Hermitian so that $\Sigma^+V$ has the Koszul-Malgrange holomorphic structure) in the preceding proof.\qed

\noindent\textit{3$^{\text{rd}}$ step. Holomorphicity of $\eta$.}

We prove that if composition with $\eta$ transforms each Koszul-Malgrange holomorphic section of $\Sigma^+V$ into a $(J^M,\J^2)$-holomorphic map, then $\eta$ is $(\J^{KM},\J^2)$-holomorphic.

\noindent\textit{Proof of the 3$^{\text{rd}}$ step.} Since $\Sigma^+V$ is a holomorphic vector bundle (for the complex
structure $\J^{KM}$) over the complex manifold $M$, for each point $(x,s_x)\in\Sigma^+V$ and each vector $v\in{T}_{(x,s_x)}\Sigma^+V$ there is a (locally defined) $(J^M,\J^{KM})$-holomorphic section $\sigma:M\to \Sigma^+V$ with $\sigma(x)=(x,s_x)$ and $\d\sigma_x(u)=v$ for some $u\in{T}_xM$. So, take $v\in{T}_{(x,s_x)}\Sigma^+V$ and $\sigma$ a $\J^{KM}$-holomorphic section of $\Sigma^+V$ with $d\sigma_x(u)=v$, for some $u\in{T}_xM$. Then, $\eta\circ\sigma$ is a $(J^M,\J^2)$-holomorphic map from $M$ to $\Sigma^+N$ and we get\addtolength{\arraycolsep}{-1.0mm}
\[\begin{array}{rll}\d\eta_{(x,s_x)}(\J^{KM}v)&=&\d\eta_{(x,s_x)}(\J^{KM}(\d\sigma_xu))=\d\eta_{(x,s_x)}(\d\sigma_x(J^Mu))\\
~&=&\J^2\big(\d\eta_{(x,s_x)}(\d\sigma_xu)\big)=\J^2(\d\eta_{(x,s_x)}v),\end{array}\]\addtolength{\arraycolsep}{1.0mm}\noindent
as required.\qed

\noindent\textit{4$^{\text{th}}$ step. Composition with $\eta$.}

We show that if $J_V$ is a $\J^{KM}$-holomorphic section of $\Sigma^+V$, then $\eta\circ{J}_V:M\to\Sigma^+N$ is $(J^M,\J^2)$-holomorphic.

\noindent\textit{Proof of the 4$^\text{th}$ step.} This is just our Lemma \ref{Lemma:FundamentalLemmaToProve3.5}: we know that $\varphi$ is $(J^M,J_{\varphi})$-holomorphic by the very definition of $J_{\varphi}$, and the latter is a section of $\Sigma^+\varphi^{-1}TN$ with $\d_2\omega_{J_\varphi}=0$, by the second step. Therefore, our map $\eta\circ J_V$ is $(J^M,\J^2)$-holomorphic.\qed

\noindent\textit{5$^{\text{th}}$ step. Conclusion of the implication (i)$\impl$(ii).}

We know from the third step that $\eta$ will be holomorphic if it transforms each $\J^{KM}$-holomorphic section into a $(J^M,\J^2)$-holomorphic map $M\to \Sigma^+N$. Take $J_V$ to be any $\J^{KM}$-holomorphic section. Then, from the fourth step, $\eta\circ J_V$ is $(J^M,\J^2)$-holomorphic and the proof is complete.\qed

\noindent\textit{6$^{\text{th}}$ step. Conclusion of the proof of Theorem \ref{Theorem:NewTheorem4.1}.}

We have shown that (i)$\impl$(ii). It is obvious that (ii)$\impl$(iii) and hence we are left with showing that (iii)$\impl$(i).

Since $\eta$ is $\J^2$-stable, we may choose a section $J_V$ such that $\eta\circ J_V=J_\varphi$ is $(J^M,\J^2)$-holomorphic at a point $x\in{M}$ (see Lemma \ref{Lemma:LemmaInTheSixthStepOfTheProofOf:Theorem:SalamonsTheorem4.1} below). Hence, $J_\varphi$ has $\d_2\omega_{J_\varphi}(x)=0$ (Lemma \ref{Lemma:FundamentalLemmaToProve3.5}). Using \eqref{Equation:FourthEquationIn:Lemma:Salamon1.2Generalized}, we have
\[\nabla^{\varphi^{-1}}_{T^{01}_xM}s^{10}_{TN}\subseteq{s}^{10}_{TN}\text{.}\]
But $s^{10}_{TN}$ is spanned by $s^{10}_V$ and $\d\varphi(T^{10}M)$. This means that, for every smooth vector field $f(z).V^{10}_z\oplus{g}(z).\d\varphi_z(X^{10}_z)$, with $V^{10}_z\in{s}^{10}_{V}(z)$,
$X^{10}_z\in{T}^{10}_zM$ and for any $Y^{10}_x\in{T}^{01}_xM$,
\[\nabla_{Y^{01}_z}^{\varphi^{-1}}\big(f(z)V^{10}+g(z)\d\varphi(X^{10})\big)\in{s}_V^{10}(x)\oplus\d\varphi(T_x^{10}M)=s^{10}_{TN}(x)\simeq{T}^{10}_{J_\varphi(x)}N\text{.}\]
In particular, since $V^{10},\d\varphi(X^{10})\in{s}^{10}_V\oplus\d\varphi(T^{10}_xM)$, we have\addtolength{\arraycolsep}{-1.0mm}
\[\begin{array}{ll}~&Y^{01}_x(f).V^{10}_x+f(x).\nabla_{Y^{01}_x}V^{10}+Y^{01}_x(g).\d\varphi_x(X^{10}_x)+g(x).\nabla_{Y^{01}_x}\d\varphi(X^{10})\in{T}^{10}_{J_\varphi(x)}N\\
\impl&\forall\,f,g\quad{f}(x).\nabla_{Y^{01}_x}V^{10}+g(x).\nabla_{Y^{01}_x}\d\varphi(X^{10})\in{T}^{10}_{J_\varphi(x)}N\\
\impl&\nabla^{\varphi^{-1}}_{Y^{01}_x}\d\varphi(X^{10})\in{T}^{10}_{J_\varphi(x)}N\text{.}
\end{array}\]\addtolength{\arraycolsep}{1.0mm}\noindent
Hence,
\[\nabla_{T^{01}_xM}\d\varphi(T^{10}M)\subseteq{T}^{10}_{J_\varphi}N.\]
Since $M$ is $(1,2)$-symplectic, we easily conclude that $(\nabla\d\varphi)_x(T^{01}M,T^{10}M)\subseteq{T}^{10}_{J_\varphi(x)}N$; varying $x$ on $M$ and using Lemma \ref{Lemma:AlternativeConditionFor(11)GeodesicMaps} shows that
$\varphi$ is $(1,1)$-geodesic, as we wanted.\qed

Our proof of Theorem \ref{Theorem:NewTheorem4.1} will be completed by the following technical result:

\begin{lemma}\label{Lemma:LemmaInTheSixthStepOfTheProofOf:Theorem:SalamonsTheorem4.1}
Let $M$ be a Hermitian manifold, $N^{2n}$ an oriented even-dimensional Riemannian manifold and $\varphi:M\to N$ a
pluriconformal map with $R^{02}_V=0$, $V=(\d\varphi(TM))^\bot$. Take the Koszul-Malgrange holomorphic structure on $\Sigma^+V$ and $\eta:\Sigma^+V\to\Sigma^+N$ as in \eqref{Equation:TheMapFromSigmaPlusVIntoSigmaPlusNOnTheSalamonNewTheorem4.1}. If $\eta$ is $\J^a$-stable ($a=1,2$) then, at any point $x\in{M}$ and for each $J_{V_x}\in\Sigma^+V_x$ there is a smooth section $s^{10}_V$ of $\Sigma^+V$ with $s^{10}_V(x)=J_{V_x}$ defined around $x$ such that $\eta\circ s$ is $(J^M,\J^a)$-holomorphic at the point $x$.
\end{lemma}
\begin{proof}
Let us start by emphasizing that we want a section of the bundle $\Sigma^+V$ and not just a map (in which case our problem would be just the existence of smooth maps which are holomorphic at a point). The idea of this proof is not very complicated: since $\eta$ is $\J^a$-stable, we can transfer to $\Sigma^+V$ the almost complex structure $\J^a$ \textit{via} the map $\eta$ to obtain $\tilde{\J}^a$. Hence, we are reduced to showing that there is a
section $s^{10}_V$ of $\Sigma^+V$ which is $(J^M,\tilde{\J}^a)$-holomorphic at a point $x$ with the initial
condition $J_{V_x}$. Also using $\eta$, we can induce a splitting of $T\Sigma^+V$ into horizontal and vertical parts, induced from that splitting on $\Sigma^+N$; such a splitting is preserved by $\tilde{\J}^a$. It is easy to check that any section $\sigma$ of $\Sigma^+V$ is horizontally holomorphic for this splitting. On the other hand, $\tilde{\J}^a$ preserves $\ker \d\pi_{\Sigma^+V}$, where $\pi_{\Sigma^+V}:\Sigma^+V\to M$ is just the canonical projection map\footnote{Equivalently, $\eta$ maps $\ker \d\pi_{\Sigma^+V}$ into the vertical part of $\Sigma^+N$.}. Since $\pi_{\Sigma^+V}$ is a submersion, we can choose charts $(\V,\nu)$ of $\Sigma^+V$ and $(\openU,\mu)$ of $M$ such that $\mu\circ \pi_{\Sigma^+V}\circ\nu^{-1}$ is given by $(x,y)\to x$. We can now construct a section which is holomorphic at $x$; the only part we have to be concerned about is the vertical part and, at a point $x$, it is easy to construct a
smooth map which is holomorphic at $x$. More precisely, we know that $\mu\circ \pi_{\Sigma^+V}\circ \nu^{-1}:A\times{B}\subseteq\rn^{2m}\times\rn^{2k-2m}\to\rn^{2m}$ is just the projection map and the kernel of $\d\pi_{\Sigma^+V}$ at $s_0$ is the image of the map $\d\big(y\to\nu^{-1}(0,y)\big)(0)$. We can transfer the structure
$\tilde{\J}^a_{s_0}$ to $T_0(\rn^{2m}\times\rn^{2k-2m})$, to get $\tilde{\J}^a_0$. Since $\tilde{\J}^a_{s_0}$ preserves
$\ker\d\pi_{\Sigma^+V}(s_0)$, $\tilde{\J}^a_0$ preserves $\rn^{2k-2m}$. Similarly, we can transport the complex structure of $M$ at $x$ \textit{via} $\mu$ to get an almost complex structure $J_0$ at $\rn^{2m}$. Take any $f:\rn^{2m}\to \rn^{2k-2m}$ $(J_0,\tilde{\J}_0^a)$-holomorphic at the origin and consider the section that in these local coordinates is written as $x\to(x,f(x))$. Then, it is obviously a section and is holomorphic. Conveniently choosing $f$ so that at zero it gives $J_{V_x}$ in this local coordinates, we conclude the proof.
\end{proof}


\section{Corollaries}


We can now use the results in the previous sections to relate pluriconformal, $(1,1)$-geodesic maps $\varphi:M\to N$ with holomorphic maps $\psi:M\to\Sigma^+N$ (\cite{SimoesSvensson:06}):

\begin{corollary}\label{Corollary:NewSalamonCorollary4.2}
Let $\varphi:M^{2m}\to N^{2n}$ be a smooth immersion from a Kähler manifold $M^{2m}$ into an oriented
even-dimensional manifold $N^{2n}$, and assume that its normal bundle $V$ has $R^{02}_V=0$. Then, $\varphi$ is a pluriconformal and $(1,1)$-geodesic map if and only if it is (locally) the projection of a $(J^M,\J^2)$-holomorphic map $\psi:M^{2m}\to\Sigma^+N^{2n}$.
\end{corollary}
\begin{proof}
If $\varphi=\pi\circ\psi$ is the projection of a $\J^2$-holomorphic map $\psi:M^{2m}\to\Sigma^+N$, $\varphi$ is
$(1,1)$-geodesic and pluriconformal, by Theorem \ref{Theorem:NewTheorem35}.

Conversely, if $\varphi$ is a pluriconformal and $(1,1)$-geodesic map with $R^{02}_\bot=0$, Theorem \ref{Theorem:NewTheorem4.1} guarantees that the holomorphic bundle $\Sigma^+\varphi{TM}^\bot=\Sigma^+V$ is holomorphically mapped into $\Sigma^+N^{2n}$. Therefore, take any holomorphic section $s^{10}_V:M^{2m}\to\Sigma^+V$ of this holomorphic
bundle and consider the map $\psi=\eta\circ{s}^{10}_V:M^{2m}\to\Sigma^+N$. This is a $(J^M,\J^2)$-holomorphic map as it is the composition of holomorphic maps. Note that $\eta\circ{s}^{10}_V(x)=(\varphi(x),\eta_2(x))$, so that $\varphi=\pi_{\Sigma^+N}\circ\psi$ by the very definition of $\eta$, and the proof is complete.
\end{proof}

Once again, notice that we can change $\Sigma^+N$ with $\Sigma^-N\simeq\Sigma^+\tilde{N}$, where $\tilde{N}$ denotes the manifold $N$ with the opposite orientation. In fact, if $\psi:M\to\Sigma^-N$ is $\J^2$-holomorphic as above, the map $\varphi:M\to\tilde{N}$ is \emph{pluriconformal} and \emph{$(1,1)$-geodesic}. Hence, it has these same two properties as a map to $N$. Conversely, if $\varphi:M\to N$ verifies the conditions on Corollary \ref{Corollary:NewSalamonCorollary4.2}, then it does as a map to $\tilde{N}$. Hence, it also admits a twistor lift to $\Sigma^+\tilde{N}\simeq\Sigma^-N$.

When $m=1$ the condition on the curvature is automatically satisfied and pluriconformality reduces to conformality. Hence, we get (\cite{Salamon:85}, Corollary 4.2):

\begin{corollary}\label{Corollary:SalamonCorollary4.2}
Let $M^2$ be a Riemann surface, $N^{2n}$ an oriented even-dimensional manifold. Consider $\varphi:M^2\to N^{2n}$ an immersion. Then, $\varphi$ is a conformal and harmonic map if and only if $\varphi$ is (locally) the projection of a $(J^M,\J^2)$-holomorphic map $\psi:M^2\to\Sigma^+N$.
\end{corollary}

Using Theorem \ref{Theorem:PullBackMetricThroughAPluriharmonicAndPluriconformalMapIsKahler}, we know that a pluriharmonic and pluriconformal map $\varphi:M\to N^{2n}$ defined on a complex manifold $(M,J)$ induces a Kähler metric on the domain manifold and, for this metric, $\varphi$ is still $(1,1)$-geodesic. Hence, if the $(0,2)$-part of the curvature of the normal bundle $V$ vanishes, using Corollary \ref{Corollary:NewSalamonCorollary4.2}, $\varphi$ is locally the projection of a $\J^2$-holomorphic map to $\Sigma^+N$. In some cases, we can guarantee the vanishing of $R^{02}_V$ for a pluriharmonic, pluriconformal map (\cite{SimoesSvensson:06}):

\begin{theorem}[\textup{The case of symmetric spaces}]\label{Theorem:TheVanishingOfR20InTheCaseOfSymmetricSpaces}
Let $N$ be a Riemannian symmetric space of Euclidean, compact or non-compact type and $\varphi: M\to N$ a pluriconformal, pluriharmonic immersion. Then the $(0,2)$-part of the curvature of its normal bundle vanishes. Consequently, $\varphi$ is locally the projection of a $\J^2$-holomorphic map to $\Sigma^+N$.
\end{theorem}

Together with Corollary \ref{Corollary:NewSalamonCorollary4.2}, given a smooth immersion $\varphi:M\to N^{2n}$, with $M$ Kähler and $N^{2n}$ a symmetric space of Euclidean, compact or non-compact case, this last result implies that $\varphi$ is pluriharmonic and pluriconformal if and only if it is (locally) the projection of a $\J^2$-holomorphic map to $\Sigma^+N$. In some particular cases, pluriconformality follows from pluriharmonicity and we can therefore guarantee that pluriharmonic maps are precisely the projections of $\J^2$-holomorphic maps. That is what happens, for instance, when $M$ is a compact Kähler manifold with quasi-positive Ricci curvature\footnote{The Ricci curvature is called quasi-positive if it is positive semi-definite everywhere, and positive definite at least at one point (\cite{Wu:88}).}. For more details and for a proof of Theorem \ref{Theorem:TheVanishingOfR20InTheCaseOfSymmetricSpaces}, see \cite{SimoesSvensson:06}, Theorem 3.8. and the examples thereafter.\index{Pluriharmonic~map!and~symmetric~spaces}\index{Symmetric!space!and~pluriharmonic~maps}


\section{Real~isotropic~and~totally~umbilic~maps}\label{Section:SectionInChapter:TwistorSpacesAndHarmonicMaps:TheHAndJ1CasesRealIsotropyAndTotallyUmbilicMaps}


We have seen in Theorems \ref{Theorem:SalamonTheorem3.5} and \ref{Theorem:NewTheorem35} that the projection $\varphi$ of a $\J^2$-holomorphic map $\psi:M^{2m}\to \Sigma^+N$ is a harmonic (when $M^{2m}$ is cosymplectic) or $(1,1)$-geodesic map (when $M^{2m}$ is Kähler). In Theorem \ref{Theorem:NewTheorem4.1} we obtained a partial converse to this result. We would now like to see what can we say about $\J^1$-holomorphic maps. We shall start by studying $\H$-holomorphic maps $\psi:M^{2m}\to\Sigma^+N$ .


\subsection{$\H$-holomorphic maps}


Recall Definition \ref{Definition:HolomorphicityInCertainSubbundles} for $\H$-holomorphic map: given two almost complex manifolds $(M,J^M)$ and $(Z,J^Z)$ with $TZ=\H\oplus\V$ a stable decomposition, $\psi:(M,J^M)\to(Z,J^Z)$ is $\H$-holomorphic if equation \eqref{Equation:EquationIn:Definition:HolomorphicityInCertainSubbundles} is verified:
\[(\d\psi(J^MX))^\H=J^Z(\d\psi{X})^\H,\quad\forall\,X\in{T}M\text{.}\]

If $N^{2n}$ is an oriented even-dimensional Riemannian manifold, we can take $Z=\Sigma^+N$ and consider the decomposition of $T\Sigma^+N$ into its horizontal and vertical spaces. This decomposition is stable for both the almost complex structures $\J^1$ and $\J^2$. Moreover, since $\J^1$ and $\J^2$ coincide in the horizontal subbundle, there is no ambiguity when we refer to $\H$-holomorphicity.

\begin{theorem}\label{Theorem:ProjectionsOfHHolomorphicMaps}\index{$\H$-holomorphic!projection}\index{Pluriconformal~map!$\H$-holomorphic~projections}
Let $\psi:M\to \Sigma^+N$ be an $\H$-holomorphic map defined on an almost Hermitian manifold $M^{2m}$. Then, the projected map $\varphi:M\to N$ is weakly pluriconformal.
\end{theorem}
\begin{proof}If $\varphi=\pi\circ\psi$, then $\varphi$ is $(J^M,J_\psi)$-holomorphic:\addtolength{\arraycolsep}{-1.0mm}
\[\begin{array}{lll}\d\varphi(J^MX)&=&\d\pi\circ\psi(J^MX)=\d\pi((\d\psi{J}^MX)^\H+(\d\psi{J}^MX)^\V)=\d\pi((\d\psi{J}^MX)^\H)\\
~&=&\d\pi(\J^H(\d\psi{X})^\H)=J_\psi\d\pi((\d\psi{X})^\H)=J_\psi\d\pi(\d\psi{X})=J_\psi\d\varphi{X}\text{.}\end{array}\]\addtolength{\arraycolsep}{1.0mm}\noindent
Therefore, $\varphi$ must be pluriconformal (Remark \ref{Remark:ConformalAndPluriconformalMapsFromARiemannSurface}).
\end{proof}

\begin{corollary}\label{Corollary:CorollaryTo:Theorem:ProjectionsOfHHolomorphicMaps}\index{$\H$-holomorphic!map!from~a~Riemann~surface} If $\psi:M^2\to\Sigma^+N$ is an $\H$-holomorphic map defined in a Riemann surface, the projected map $\varphi=\pi\circ\psi$ is a weakly conformal map.
\end{corollary}
\begin{proof}
Immediate from the preceding theorem and Remark \ref{Remark:ConformalAndPluriconformalMapsFromARiemannSurface}.
\end{proof}

Recall from p.\ \pageref{Page:CompatibleTwistorLift} that a twistor lift $\psi:M\to\Sigma^+N$ is strictly compatible with $\varphi=\pi\circ\psi$ if $\varphi$ is holomorphic with respect to the almost Hermitian structure $J_\psi$ defined by $\psi$.

\begin{proposition}\label{Proposition:LocalHHolomorphicLifts}\index{$\H$-holomorphic!lift}\index{Conformal~map!and~$\H$-holomorphic~lifts}\index{Pluriconformal~map!and~$\H$-holomorphic~lifts}\index{Compatible~twistor~lift!and~$\H$-holomorphic~maps}
Let $\varphi:M^{2m}\to N^{2n}$ be a smooth map on an almost Hermitian manifold $M^{2m}$. Then, any strictly compatible lift $\psi$ of $\varphi$ is $\H$-holomorphic\footnote{Notice that we are not claiming that such a lift exists; following the discussion on p.\ \pageref{Page:CompatibleTwistorLift}, a necessary condition for the existence of such a lift is weakly pluriconformality of $\varphi$.}.
\end{proposition}
\begin{proof}For any lift $\psi$ of $\varphi$, we have\addtolength{\arraycolsep}{-1.0mm}
\[\begin{array}{ll}~&\J^\H(\d\psi{X})^\H=(\d\psi(JX))^\H\,\equi\,\d\pi(\J^\H(\d\psi{X})^\H)=\d\pi((\d\psi(JX))^\H)\\
\equi&{J}_\psi\d\varphi{X}=\d\varphi{JX}\text{.}\end{array}\]\addtolength{\arraycolsep}{1.0mm}\noindent
Hence, if $\psi$ is a strictly compatible lift of $\varphi$, $\psi$ is $\H$-holomorphic.
\end{proof}

We can therefore conclude that $\H$-holomorphicity of a twistor lift depends only on whether the projected map $\varphi$ is (or not) $(J^M,J_\psi)$-holomorphic.

\begin{proposition}\label{Proposition:LocalLiftsOfPluriconformalLifts}\index{Pluriconformal~map!and~$\H$-holomorphic~lifts}\index{$\H$-holomorphic!lift}
Let $\varphi:M\to N^{2n}$ be a pluriconformal map from an almost Hermitian manifold $M^{2m}$ and let $x\in{M}$. Then, there is some open neighbourhood $\openU\subseteq M$ of $x$ and an $\H$-holomorphic lift $\psi:\openU\subseteq{M}\to\Sigma^+N$.
\end{proposition}
\begin{proof}
From the preceding Proposition \ref{Proposition:LocalHHolomorphicLifts}, it suffices to construct a strictly compatible lift $\psi$ around the point $x\in{M}^{2m}$. As $\varphi$ is pluriconformal, we have a smooth map $\eta:\Sigma^+V\to\Sigma^+\varphi^{-1}TN$ as in Section \ref{Section:LiftsOfHarmonicMaps}. Take any smooth section $s^{10}_V$ of the bundle $\Sigma^+V$; then, the map $\psi$ in $\Sigma^+ N$ obtained by composing with $\eta$ is a strictly compatible lift of $\varphi$.
\end{proof}


\subsection{$\J^1$-holomorphic maps}


The following reduces to \cite{Salamon:85}, Proposition 4.4 in the case $m=1$:

\begin{theorem}[\textup{Projections~of~$\J^1$-holomorphic~maps}]\label{Theorem:GeneralizedTheorem4.4ProjectionsOfJ1HolomorphicMapsInTheHigherDimensionalCase}\index{$\J^1$-holomorphic!projection}\index{Real~isotropic~map!as~projection~of~a~$\J^1$-holomorphic~map}\index{$\J^1$-holomorphic!and~real~isotropic}
Let $\psi:M^{2m}\to \Sigma^+N$ be a $\J^1$-holomorphic map on a Hermitian manifold $M^{2m}$. Then the projected map
$\varphi=\pi\circ\psi$ is real isotropic.
\end{theorem}
\begin{proof}
Take $\psi$ to be a $\J^1$-holomorphic map to the twistor space $\Sigma^+N$. As in Section \ref{Section:SectionIn:TwistorSpacesAndHarmonicMaps:FundamentalLemma}, consider the almost Hermitian vector bundle $(\varphi^{-1}TN,\varphi^{-1}h,\nabla^{\varphi^{-1}},\sigma_\psi)$, where $\varphi=\pi\circ\psi$. Then $\varphi$ is $(J^M,J_\psi)$-holomorphic (Lemma \ref{Lemma:FundamentalLemmaToProve3.5}) so that $\d\varphi(T^{10}M)\subseteq{T}^{10}_{J_\psi}N\simeq(\varphi^{-1}TN)^{10}$.
Since $\psi$ is $\J^1$-holomorphic, using Corollary \ref{Corollary:AnotherVersionOfTheFundamentalLemmaUsingT10AndT01Spaces} we deduce $\nabla_{T^{10}M}\d\varphi(T^{10}M)\in(\varphi^{-1}TN)^{10}$. We then see by induction that
\[\nabla^{r-1}_{T^{10}M}\d\varphi(T^{10}M)\in(\varphi^{-1}TN)^{10}\simeq\varphi^{-1}(T^{10}_{J_\psi}N),\quad\forall\,r\geq1\text{.}\]
Thus,
\[<\nabla^{r-1}_{T^{10}M}\d\varphi(T^{10}M),\nabla^{s-1}_{T^{10}M}\d\varphi(T^{10}M)>=0,\quad\forall\,r,s\geq1\text{,}\]
showing that equation \eqref{Equation:FirstEquationIn:Definition:RealIsotropicMapFromAHermitianManifold} is satisfied and so $\varphi$ is real isotropic.
\end{proof}

We look for a converse of this result, starting with the case where the domain is a Riemann surface. As in the ``$\J^2$ case" (Section \ref{Section:LiftsOfHarmonicMaps}), given a conformal map $\varphi:M^2\to N^{2n}$, consider the normal bundle $V=(\d\varphi TM)^\bot\subseteq \varphi^{-1}TN$. But now take on $M^2$ the \emph{conjugate} Hermitian structure $(-J^M)$ on $M$ and note that $R_V^{02}$ still vanishes. Therefore we still have a Koszul-Malgrange holomorphic structure on the bundle $\Sigma^+V$, but now characterized by
\begin{equation}\label{Equation:TheKoszulMalgrangeHolomorphicStructureForTheConjugateBaseManifold}\index{Complex~structure!conjugate}\index{Koszul-Malgrange!complex~structure~on~$\Sigma^+V$!$\J^1$~case}
s^{10}_V\text{~is~a~holomorphic~section~if~and~only~if~}\nabla^V_{T^{10}M}s^{10}_V\subseteq{s}^{10}_V\text{.}
\end{equation}
As before, we have a map $\eta:\Sigma^+V\to\Sigma^+N$ defined as in \eqref{Equation:TheMapFromSigmaPlusVIntoSigmaPlusNOnTheSalamonTheorem4.1Proof}. Then (see \cite{Salamon:85}, Theorem 4.3), we have

\begin{theorem}\label{Theorem:SalamonTheorem4.3J1LiftsOfRealIsotropicMapsFromRiemannSurfaces}
Let $\varphi:M^2\to N^{2n}$ ($n\geq3$) be a totally umbilic conformal immersion into an oriented even-dimensional manifold. Then, there is a (locally defined) $\J^1$-holomorphic lift of $\varphi$ to $\Sigma^+N$. Further, the following conditions are equivalent:

\textup{(i)} $\varphi$ is totally umbilic.

\textup{(ii)} $\eta$ is $(\J^{KM},\J^1)$-holomorphic.

\textup{(iii)} $\eta$ is $\J^1$-stable.
\end{theorem}

This result is a particular case of Theorem \ref{Theorem:NewTheorem4.3:J1HolomorphicLiftsOfTotallyUmbilicMaps} below.

Notice that the class of maps given by projections (the class of real-isotropic maps) is not the same as the class of maps (the totally-umbilic maps) for which we can guarantee the existence of the lift, in contrast to the harmonic case, see also Remark \ref{Remark:RemarkTo:Proposition:PluriconformalMapsAndAlmostHermitianStructuresAtAPoint}.

In the proof of Theorem \ref{Theorem:NewTheorem4.1}, we used the alternative characterization of $(1,1)$-geodesic maps given in Lemma \ref{Lemma:About(11)GeodesicMapsAndOneVersusAllAlmostHermitianStructures}. In a similar fashion, totally umbilic maps can be characterized in the following way:

\begin{proposition}\label{Proposition:PluriconformalMapsAndAlmostHermitianStructuresAtAPoint}\index{Pluriconformal~map!}\index{Totally~umbilic~map!}\index{Totally~umbilic~map!and~Hermitian~structures}\index{Pluriconformal~map!and~Hermitian~structures}
Let $\varphi:M^{2m}\to N^{2n}$ be a pluriconformal map defined on a Hermitian manifold $M^{2m}$, $n-m\geq2$. Then, the following conditions are equivalent:

\textup{(i)} $\varphi$ is totally umbilic at $x$.

\textup{(ii)} For all $J_{\varphi}\in\Sigma^{+}T_{\varphi_{x}}N$ that render $\varphi$ holomorphic at $x$, \begin{equation}\label{Equation:EquationIn:Proposition:PluriconformalMapsAndAlmostHermitianStructuresAtAPoint}
\nabla\d\varphi(T^{10}M,T^{10}M)\subseteq{T}^{10}_{J_\varphi}N.
\end{equation}
\end{proposition}
The case $n=m$ is also trivially true if we do not require the positivity of $J_\varphi$ on (ii). The proof of this result will be an immediate consequence of the following lemma:

\begin{lemma}\label{Lemma:IsotropicSubspacesAndAlmostHermitianStructuresOnTheUsualRealEuclideanVectorSpace}
Let $F^{10}$ be a isotropic subspace in $\rn^{2n}\otimes\cn=\cn^{2n}$ with complex dimension $m<n$. Consider the set
$\Sigma^{+}_{F^{10}}\rn^{2n}$ of all positive Hermitian structures $J$ on $\rn^{2n}$ for which $F^{10}\,\subseteq\,T^{10}_J\rn^{2n}$.
Then the set
\begin{equation}\label{Equation:EquationIn:Lemma:IsotropicSubspacesAndAlmostHermitianStructuresOnTheUsualRealEuclideanVectorSpace}
Q_{F^{10}}=\{u\in\cn^{2n}:u\in{T}^{10}_J\rn^{2n},\quad\forall\,J\in\Sigma^{+}_{F^{10}}\rn^{2n}\}
\end{equation}
coincides with $F^{10}$ if and only if $n-m\geq2$.
\end{lemma}
\begin{proof}In order to see what is going on, we start by checking the case $m=1$. Without loss of generality, assume that $F^{10}$ is spanned (as a $\cn$-subspace) by $e_1-ie_2$. If $n=3$, we deduce that $Q_{F^{10}}$ is the set of vectors in $\rn^{6}\simeq\cn^3$ that belong to $T^{10}_J\rn^6$ for all $J\in\Sigma^+\rn^6$ for which $Je_1=e_2$. Take $J_0$ to be the usual Hermitian structure on $\rn^6$ and $J_1$ defined by $J_1e_1=e_2$, $J_1 e_3=e_5$, $J_1e_6=e_4$ (notice that a basis for $\rn^{6}$ is $\{e_1,J_1 e_1, e_3, J_1 e_3,e_4,J_1 e_4\}=\{e_1,e_2,e_3,e_5,e_4,-e_6\}$ which is clearly
positive). Then, taking $u\in{Q}_{F^{10}}$, since $u$ belongs to $T^{10}\rn^{2n}$ for both these structures, we deduce that there exist $a_{ij}\in\cn$ with\addtolength{\arraycolsep}{-1.0mm}
\[\begin{array}{ll}~&\left\{\begin{array}{l}
u=a_{11} (e_1-i e_2)+a_{12}(e_3-ie_4)+a_{13}(e_5-ie_6)\\
u=a_{21} (e_1-i e_2)+a_{22}(e_3-ie_5)+a_{23}(e_6-ie_4)
\end{array}\right.\\\noalign{\medskip}
\impl&\left[\begin{array}{cccc}
1 & 0 & -1 & 0\\
-i & 0 & 0 & +i\\
0 & 1 & i & 0\\
0 & -i & 0 & -1
\end{array}\right]\left[\begin{array}{l}
a_{12}\\
a_{13}\\
a_{22}\\
a_{23}\end{array}\right]=0\end{array}\]\addtolength{\arraycolsep}{1.0mm}\noindent
As the above matrix has nonzero determinant, we have $a_{12}=a_{13}=a_{22}=a_{23}=0$, which shows that $u\in\wordspan\{e_1-ie_2\}$. Moreover, notice that if we had taken $n=2$ (corresponding to $\rn^{4}$), then there would have been only one positive Hermitian structure (namely, $J_0$) for which $e_1-ie_2$ lies in its $(1,0)$-subspace; therefore, in that case, we would have $Q_{F^{10}}=\wordspan\{e_1-ie_2,e_3-ie_4\}\neq F^{10}$.

For the general case, proceed by induction on $j=n-m$: for $j=2$, $n=m+2$ and we proceed exactly as above: assume that $F^{10}=\wordspan\{e_1-ie_2,...,e_{2m-1}-ie_{2m}\}$, take the two positive Hermitian structures $J_0$ and $J_1$ on $\rn^{2m+4}$ with $J_{1}e_{2i-1}=e_{2i}$, $1\leq i\leq m$, $J_1 e_{2m+1}=e_{2m+3}$, $J_1 e_{2m+4}=e_{2m+2}$; both these structures preserve $F^{10}$ and, proceeding as before, if $u$ lies in the $(1,0)$-subspace for both $J_0$ and $J_1$ we deduce the existence of $a_{ij}$ with\addtolength{\arraycolsep}{-1.0mm}
\[\begin{array}{ll}~&\left\{\begin{array}{ll}
u = & a_{11}(e_1-ie_2)+...+a_{1m}(e_{2m-1}-ie_{2m})\\
~& +a_{1m+1}(e_{2m+1}-ie_{2m+2})+a_{1m+2}(e_{2m+3}-ie_{2m+4})\\
u=&a_{21}(e_1-ie_2)+...+a_{2m}(e_{2m-1}-ie_{2m})\\
~ &+a_{2m+1}(e_{2m+1}-ie_{2m+3})+a_{2m+2}(e_{2m+4}-ie_{2m+2})
\end{array}\right.\\\noalign{\medskip}
\impl&\left[\begin{array}{cccc}
1 & 0 & -1 & 0\\
-i & 0 & 0 & +i\\
0 & 1 & i & 0\\
0 & -i & 0 & -1
\end{array}\right]\left[\begin{array}{l}
a_{1m+1}\\
a_{1m+2}\\
a_{2m+1}\\
a_{2m+2}\end{array}\right]=0\end{array}\]\addtolength{\arraycolsep}{1.0mm}\noindent
and, therefore, $u\in{F}^{10}$.

Assume now that the result is true for $j=k\geq2$. Once again, assume that $F^{10}$ is spanned by $\{e_1-ie_2,...,e_{2m-1}-ie_{2m}\}$ and take $u\in{Q}_{F^{10}}$. Since $u$ belongs to the $(1,0)$-subspace of any positive Hermitian structure $J$ with $J e_1=e_2,...,Je_{2m-1}=e_{2m}$ we deduce that $u$ belongs to $(1,0)$-subspace of both the structures $J_0$ (canonical Hermitian structure in $\rn^{2m+2k+2}$) and $J_1$, where $J_1$ is defined as $J_0$ for all $\{e_i\}_{i=1,..,2m+2k-2}$ and $J_1e_{2m+2k-1}=e_{2m+2k+1}$, $J_1 e_{2m+2k+2}=e_{2m+2k}$ (positive with the same argument as before) so that there exist $a_{ij}$ with\addtolength{\arraycolsep}{-1.0mm}
\[\left\{\begin{array}{lll}
u&=&a_{11} (e_1-i e_2)+...+a_{1m+k-1}(e_{2m+2k-3}-ie_{2m+2k-2})\\
~&~&+ a_{1m+k}(e_{2m+2k-1}-ie_{2m+2k})+a_{1m+k+1}(e_{2m+2k+1}-ie_{2m+2k+2})\\
u&=&a_{21} (e_1-i e_2)+...+a_{2m+k-1}(e_{2m+2k-3}-ie_{2m+2k-2})\\
~&~&+ a_{2m+k}(e_{2m+2k-1}-ie_{2m+2k+1})+a_{2m+k+1}(e_{2m+2k+2}-ie_{2m+2k})\text{.}\\
\end{array}\right.\]\addtolength{\arraycolsep}{1.0mm}\noindent
Hence, as before, $a_{1m+k}=a_{1m+k+1}=0$ and $u\in\rn^{2m+2k}_0\subseteq\rn^{2m+2}$. We deduce that such a vector $u$ lies in $\rn^{2m+2k}_0\,\subseteq\,\rn^{2m+2k+2}$, and it must lie in the $(1,0)$-subspace for any Hermitian structure on $\rn^{2m+2k}$ with $\{Je_{2i-1}=e_{2i}\}_{i=1,...,m}$ since any such structure induces a positive Hermitian structure on
the all $\rn^{2m+2k+2}$ (with the additional condition $\tilde{J}e_{2m+2k+1}=e_{2m+2k+2}$) and $u$ lies in $(1,0)$-subspace of that Hermitian structure. By the induction hypothesis, the ``if" part of our result follows. The ``only if" part is a consequence from the observation that when $n=m+1$, then $Q$ is the $(1,0)$-subspace of the only positive Hermitian structure $J$ for which $F^{10}\subseteq{T}^{10}_J\rn^{2n}$ and $Q\neq F^{10}$.
\end{proof}

\noindent \textit{Proof of Proposition \ref{Proposition:PluriconformalMapsAndAlmostHermitianStructuresAtAPoint}.} Let $\varphi$ be a pluriconformal map. If $\varphi$ is totally umbilic, consider the isotropic subspace $\d\varphi(T^{10}M)$. Then, $\nabla\d\varphi(T^{10}M,T^{10}M)\subseteq\d\varphi(T^{10}M)\subseteq{T}^{10}_{J_\varphi}N$ for any $J_{\varphi}$ strictly compatible with $\varphi$ so that (ii) is satisfied. Conversely, if (ii) holds, then $\nabla\d\varphi(T^{10}M,T^{10}M)\subseteq T^{10}_{J_\varphi}N$ for all $J_{\varphi}$ that render $\varphi$ holomorphic at $x$ and therefore for all $J_{\varphi}$ that preserve $\d\varphi(T^{10}M)$; in the notation of Lemma \ref{Lemma:IsotropicSubspacesAndAlmostHermitianStructuresOnTheUsualRealEuclideanVectorSpace},
$\nabla\d\varphi(T^{10}M,T^{10}M)\subseteq Q_{\d\varphi(T^{10}M)}=\d\varphi(T^{10}M)$ and $\varphi$ is totally
umbilic, as desired. \qed

\begin{remark}\label{Remark:RemarkTo:Proposition:PluriconformalMapsAndAlmostHermitianStructuresAtAPoint}
In Lemma \ref{Lemma:About(11)GeodesicMapsAndOneVersusAllAlmostHermitianStructures}, we saw that $\varphi$ is $(1,1)$-geodesic if and only if, for each $x\in{M}$, $\nabla\d\varphi_{x}(T^{10}M,T^{01}M)$ lies in the $(1,0)$-subspace of \emph{some} or \emph{all} the almost Hermitian structures on $T_{\varphi(x)}N$: the \textit{some} justifies the fact that maps obtained as projections of $\J^2$-holomorphic maps are $(1,1)$-geodesic, whereas the \textit{all} guarantees the existence of a lift for a given $(1,1)$-geodesic map (plus, of course, the conditions stated in Theorem \ref{Theorem:NewTheorem4.1}). Now, for real isotropic and totally umbilic maps, the situation is different, since whether equation
\eqref{Equation:EquationIn:Proposition:PluriconformalMapsAndAlmostHermitianStructuresAtAPoint} is satisfied for \emph{some} or for \emph{all} the the almost Hermitian structures on $T_{\varphi(x)}N$ results in different conditions. As a matter of fact, as we shall see, the latter condition is necessary to guarantee the existence of a lift to the twistor space, whereas the first (plus the integrability of the complex structure on the domain manifold) gives real-isotropy of the map.
\end{remark}

\begin{theorem}[\textup{$\J^1$-holomorphic~lifts~of~pluriconformal~maps}]\label{Theorem:NewTheorem4.3:J1HolomorphicLiftsOfTotallyUmbilicMaps}\index{Pluriconformal~map!}\index{Totally~umbilic~map!}\index{$\J^1$-holomorphic!lift}
Let $\varphi:M\to N$ be a pluriconformal map from a Hermitian manifold $M^{2m}$ to an oriented even-dimensional Riemannian manifold $N^{2n}$, $n-m\geq2$. Define $\eta$ by
\eqref{Equation:TheMapFromSigmaPlusVIntoSigmaPlusNOnTheSalamonNewTheorem4.1}. Then, the following conditions are equivalent:

\textup{(i)} $\varphi$ is totally umbilic.

\textup{(ii)} $\eta$ is $(\J^{KM},\J^1)$-holomorphic.

\textup{(iii)} $\eta$ is $\J^1$-stable.
\end{theorem}

\begin{proof}Following the same ideas as in the proof of Theorem \ref{Theorem:NewTheorem4.1}, we shall show that (i) $\impl$ (ii) $\impl$ (iii) $\impl$ (i). Now, $\eta$ will be $(\J^{KM},\J^1)$-holomorphic if it maps $\J^{KM}$-holomorphic
sections $s^{10}_V$ into sections $s^{10}_{TN}\,\simeq\,J_\varphi$ of $\Sigma^+\varphi^{-1}TN$ with $\d_1\omega_{J_\varphi}=0$. Using Lemma \ref{Lemma:Salamon1.2Generalized}, $\d_1\omega_{J_\varphi}=0$ if and only if $\nabla_{T^{10}M}s^{10}_{TN}\subseteq\,s^{10}_{TN}$. So, we must show that, if $s^{10}_V$ satisfies $\nabla^V_{T^{10}M}s^{10}_V\subseteq\,s^{10}_V$, then $\nabla_{T^{10}M}s^{10}_{TN}\subseteq\,s^{10}_{TN}$, where
$s^{10}_{TN}=s^{10}_V\oplus\,\d\varphi(T^{10}M)$. Now, taking $X^{10},Y^{10}\in{T}^{10}M$ and $Z^{10}\in{s}^{10}_V$,
\[\nabla_{X^{10}}(Z^{10}+{Y}^{10})=\nabla^V_{X^{10}}Z^{10}+\nabla^M_{X^{10}}Z^{10}+\d\varphi(Y^{10})+\nabla_{X^{10}}\d\varphi(Y^{10})\text{.}\]
The first term on the right-hand side of the above expression is not problematic, as $\nabla^V_{\partial_z}s^{10}_V\subseteq s^{10}_V$ from the very definition of our $\J^{KM}$-holomorphic structure (the third term is also non-problematic). Hence, we must make sure that, for all $\J^{KM}$-holomorphic sections $s^{10}_V$, $\nabla_{X^{10}}\d\varphi(Y^{10})\in{s}^{10}_V\oplus\d\varphi(T^{10}M)$ and $\nabla^M_{T^{10}M}s^{10}_V\subseteq s^{10}_V\oplus\d\varphi(T^{10}M)$. The first condition is satisfied
if and only if $\nabla_{X^{10}}\d\varphi(Y^{10})$ belongs to $T^{10}_{J_\varphi}N$ for every $J_\varphi$ that renders $\varphi$ holomorphic; the second is then automatic since
\[<\nabla^M_{X^{10}}Z^{10},\d\varphi(Y^{10})>=X^{10}<Z^{10},\d\varphi(Y^{10})>-<Z^{10},\nabla_{X^{10}}\d\varphi(Y^{10})>=0\text{.}\]
Since $\varphi$ is totally umbilic and $M$ is Hermitian we deduce $\nabla_{X^{10}}\d\varphi(Y^{10})\in\d\varphi(T^{10}M)$\footnote{Notice, moreover, that we do not really need $\varphi$ to be totally umbilic: we only need $\nabla\d\varphi(X^{10},Y^{10})$ to belong to the $(1,0)$-subspace of every strictly compatible Hermitian structure: see also Proposition \ref{Proposition:PluriconformalMapsAndAlmostHermitianStructuresAtAPoint}.} so that the first part of our proof is done: we proved that (i) implies (ii). That (ii) implies (iii) is trivial. To show that (iii) implies (i): since $\eta$ is $\J^1$-stable, using Lemma \ref{Lemma:LemmaInTheSixthStepOfTheProofOf:Theorem:SalamonsTheorem4.1} ($a=1$) for each point $x\in{M}^2$ and each $J_{V_x}\simeq s^{10}_{V_x}\in\Sigma^+V_x$ we can find a smooth section $s^{10}_V\simeq J_{V}$ of $\Sigma^+V$ with $\eta\circ s^{10}_V=s^{10}_{TN}\simeq J^N$ $(J^M,\J^1)$-holomorphic at $x$ and $s^{10}_V(x)=s^{10}_{V_x}$. Then, using Lemma \ref{Lemma:FundamentalLemmaToProve3.5} we can deduce
$\d_1\omega_{J^{N}}(x)=0$ so that using Lemma \ref{Lemma:Salamon1.2Generalized} (equation
\eqref{Equation:ThirdEquationIn:Lemma:Salamon1.2Generalized}) we obtain
\begin{equation}\label{Equation:EquationInTheProofOf:Theorem:SalamonTheorem4.3J1LiftsOfRealIsotropicMapsFromRiemannSurfaces}
\nabla^{\varphi^{-1}}_{T^{10}_xM}s^{10}_{TN}\subseteq s^{10}_{TN}(x)\text{.}
\end{equation}
We can repeat this argument for all $s^{10}_{TN}(x)$ obtained from some $s^{10}_V(x)$: $s^{10}_{TN}=s^{10}_V\oplus\d\varphi(T^{10}M)$; \textit{i.e.}, for all strictly compatible positive Hermitian structures on $T_{\varphi(x)}N$, equation \eqref{Equation:EquationInTheProofOf:Theorem:SalamonTheorem4.3J1LiftsOfRealIsotropicMapsFromRiemannSurfaces} holds. Hence, proceeding as in the sixth step in the proof of Theorem \ref{Theorem:SalamonsTheorem4.1} we deduce
\[\nabla\d\varphi_x(X^{10},Y^{10})\in{s}^{10}_{V_x}\oplus\d\varphi_x(T^{10}_xM)\]
for all $s^{10}_{V_x}\in\Sigma^+V_x$. As $n-m\geq2$ Proposition
\ref{Proposition:PluriconformalMapsAndAlmostHermitianStructuresAtAPoint} implies that $\varphi$ is totally umbilic, finishing the proof .
\end{proof}

Notice that for the implications (i) $\impl$ (ii) $\impl$ (iii) we only used the fact that $\varphi$ is a pluriconformal totally umbilic map, the fact $m-n\geq2$ being irrelevant. In particular, a pluriconformal totally umbilic map admits a $\J^1$-holomorphic lift even in the case $\dim N-\dim M=2$.

Comparing with Theorem \ref{Theorem:NewTheorem4.1}, we see that we do not need to require the domain to be Kähler, but only Hermitian. In fact, we need the integrability of the structure to use Koszul-Malgrange Theorem (and we need this in both the $\J^1$ and $\J^2$ cases) but in the present result we used $\d_1\omega_J=0$ (whereas before we needed $\d_2\omega_J=0$).

Finally, if $\varphi$ is a totally umbilic conformal immersion then there is (locally) a $(J^M,\J^1)$-holomorphic map $\psi$ with $\varphi=\pi\circ\psi$: as $\Sigma^+V$ is a holomorphic bundle over $M^2$ for the $\J^{KM}$-holomorphic structure introduced, take any holomorphic section $\sigma$ of this bundle; then $\eta\circ\sigma=\psi$ is a $(J^M,\J^1)$-holomorphic map with $\varphi=\pi\circ\psi$.


\section{The $4$-dimensional case in detail}


In this section, we describe what is special to the $4$-dimensional case. Firstly, given a conformal map $\varphi:M^2\to N^4$, with $N^4$ an oriented Riemannian manifold, there is one and only one strictly compatible lift $\psi^+$ to $\Sigma^+N$ (and another, $\psi^-$ to $\Sigma^-N$). Therefore, if $\varphi$ is also harmonic, from Corollary \ref{Corollary:SalamonCorollary4.2} it follows that $\psi^+$ is $\J^2$-holomorphic. From the remarks to Theorem
\ref{Theorem:NewTheorem4.3:J1HolomorphicLiftsOfTotallyUmbilicMaps}, if $\varphi$ is also totally umbilic, it also admits a $\J^1$-holomorphic lift to $\Sigma^+ N$ so that by uniqueness we conclude the existence of a twistor lift to $\Sigma^+N$ which is simultaneously $\J^1$ and $\J^2$-holomorphic. However, as we have seen after Theorem \ref{Theorem:NewTheorem4.3:J1HolomorphicLiftsOfTotallyUmbilicMaps}, what is really important to guarantee that $\varphi$ has a $\J^1$-holomorphic lift is that $\partial^2_z\varphi\in{T}^{10}_{J_{\varphi}}N$ for all strictly compatible $J_\varphi$. If the map $\varphi$ is real isotropic and $\partial_z\varphi,\partial^2_z\varphi$ are linearly independent, then, as they span an isotropic subspace, we can take $J_\varphi$ to be the only (now not necessarily positive) almost Hermitian structure whose $(1,0)$-subspace is spanned by those vectors: if $J_\varphi$ is positive, by uniqueness of the lift and because $\partial^2_z\varphi\in{T}^{10}_{J_\varphi}N$, it must be $\J^1$-holomorphic. Otherwise, $J_\varphi$ is negative but still $\J^1$-holomorphic as $\partial^2_z\varphi$ belongs to its $(1,0)$-subspace. In conclusion, we can state the following result (see also \cite{Salamon:85}):

\begin{theorem}\label{Theorem:J1HolomorphicLiftsInTheFourDimensionalCase}\index{Real~isotropic~map!in~the~$4$-dimensional~case}\index{$\J^1$-holomorphic!lift}
Let $N^4$ be an oriented Riemannian manifold and $\Sigma^{+}N$, $\Sigma^{-}N$ its positive and negative twistor spaces,
respectively. Consider a real isotropic map $\varphi:M^2\to N^4$. Then, at any point $z_0$ for which $\partial_z\varphi\neq 0$, there are (unique) strictly compatible lifts $\psi^{+}$ and $\psi^{-}$ to $\Sigma^{+}N$ and $\Sigma^{-}N$ respectively of $\varphi$. Further:

\textup{(i)} If $\partial_{z}^{2} \varphi(z_{0})$ and $\partial_z\varphi (z_0)$ are linearly independent, \emph{either} $\psi^+$ \emph{or} $\psi^-$ is $\J^1$-holomorphic around $z_0$.

\textup{(ii)} If $\partial_{z}^{2}\varphi$ and $\partial_z\varphi$ are linearly dependent around $z_0$ \emph{both} $\psi^+$ \emph{and} $\psi^-$ are $\J^1$-holomorphic.
\end{theorem}
\noindent\textit{First proof.} Since $\varphi$ is real isotropic, the space $F$ spanned by $\{\partial_z\varphi,\partial^{2}_{z}\varphi\}$ is isotropic; therefore, if these two vectors are linearly independent, there is one and only one almost Hermitian structure $J_{\psi}$ for which the $(1,0)$ tangent space is $F$; $J_{\psi}$ defines either $\psi^{+}$ or $\psi^{-}$, depending on whether $J_\psi$ is positive or negative (Definition \ref{Definition:PositiveHermitianStructureOnAEuclideanVectorSpace}). The map $\varphi$ is holomorphic with respect to $J_\psi$ since $\d\varphi(T^{10}M)\subseteq T^{10}_{J_\psi}N$ and $\psi$ becomes $\J^1$-holomorphic (as a map to $\Sigma^+ N$ or $\Sigma^-N$ accordingly) since from real isotropy of $\varphi$ it follows that $\nabla_{\partial_z}T^{10}_{J\psi}N\subseteq T^{10}_{J_\psi}N$.

If $\partial_{z}^{2}\varphi$ and $\partial_z\varphi$ are linearly dependent around $z_0$, we deduce $\partial^2_z\varphi\in{T}^{10}_{J_\psi^{\pm}}N$ so that both $J_{\psi^+}$ and $J_{\psi^-}$ are $\J^1$-holomorphic lifts of $\varphi$ with the same proof as in Theorem \ref{Theorem:NewTheorem4.3:J1HolomorphicLiftsOfTotallyUmbilicMaps} (and remarks thereafter). \qed

\smallskip We shall now give a second proof of (i) that will be useful in the sequel. We shall need the following:

\begin{lemma}\label{Lemma:AdaptedAlmostHermitianStructuresToRealIsotropicMaps}
Let $\varphi:M^2\to N^4$ a real isotropic map with $\{\partial_z\varphi,\partial_{z}^{2}\varphi\}$ linearly independent
in an open set $\U\subseteq M^2$. Consider the real and imaginary part of $\partial_{z}^{2}\varphi=u+iv$, where \[u=\nabla_{\partial_x}\partial_x\varphi-\nabla_{\partial_y}\partial_y\varphi\text{~and~}v=-\nabla_{\partial_x}\partial_y\varphi-\nabla_{\partial_y}\partial_x\varphi\]
for $z=x+iy$. Then, any almost Hermitian structure strictly compatible with $\varphi$ is given by\addtolength{\arraycolsep}{-1.0mm}
\begin{equation}\label{Equation:FirstEquationIn:Lemma:AdaptedAlmostHermitianStructuresToRealIsotropicMaps}
\begin{array}{ccc}
J(\partial_x\varphi)&=&\partial_y\varphi;\\
J(u)&=&\pm v\text{.}
\end{array}
\end{equation}\addtolength{\arraycolsep}{1.0mm}\noindent
\end{lemma}
\begin{proof}We shall do the proof in several steps:

(i) One always has
\begin{equation}\label{Equation:FirstEquationInProofOf:Lemma:AdaptedAlmostHermitianStructuresToRealIsotropicMaps}
\begin{array}{ll}\|\partial_x\varphi\|=\|\partial_y\varphi\|; &
<\partial_x\varphi,\partial_y\varphi>=0;
\\
\|u\|=\|v\|; &<u,v>=0;\\
<u,\partial_x\varphi>=-<v,\partial_y\varphi> ; &
<u,\partial_y\varphi>=<v,\partial_x\varphi>.
\end{array}
\end{equation}
In fact, conformality implies the first two equations and the other equations follow from
$<\partial_{z}^{2}\varphi,\partial_{z}^2\varphi>=<\partial_{z}\varphi,\partial_{z}^2\varphi>=0$.

(ii) \textit{$u\in\d\varphi(TM)$ if and only if $v\in\d\varphi(TM)$:}

If $u\in\d\varphi(TM)$, then $u=a\partial_x\varphi+b\partial_y\varphi$ where $<u,\partial_x\varphi>=a\|\partial_x\varphi\|^2$ and $<u,\partial_y\varphi>=b\|\partial_y\varphi\|^2$. Therefore,
$\|u\|^2=(a^2+b^2)\|\partial_x\varphi\|^2$. On the other hand, using
\eqref{Equation:FirstEquationInProofOf:Lemma:AdaptedAlmostHermitianStructuresToRealIsotropicMaps},
$<v,\partial_x\varphi>=b\|\partial_y\varphi\|^2$ and $<v,\partial_y\varphi>=-a\|\partial_x\varphi\|^2$ so that $\|\proj_{\d\varphi(TM)}v\|^2=(a^2+b^2)\|\partial_x\varphi\|^2$. Since $\|u\|=\|v\|$ from
\eqref{Equation:FirstEquationInProofOf:Lemma:AdaptedAlmostHermitianStructuresToRealIsotropicMaps},
we have $\|v\|=\|\proj_{\d\varphi(TM)}v\|$ so that $v\in\d\varphi(TM)$. Interchanging the roles of $u$ and $v$ gives (ii).

(iii) \textit{$u\in\d\varphi(TM)$ if and only if $\partial^{2}_{z}\varphi=\lambda\partial_z\varphi$:}

If $u\in\d\varphi(TM)$ (so that also $v\in\d\varphi(TM)$) then, with the same argument as above, $u=a\partial_x\varphi+b\partial_y\varphi$ and $v=b\partial_x\varphi-a\partial_y\varphi$ so that \[\partial_{z}^{2}\varphi=u+iv=(a+ib)\partial_x\varphi+(b-ia)\partial_y\varphi=(a+ib)(\partial_x\varphi-i\partial_y\varphi)=\lambda\partial_z\varphi\text{.}\]
Conversely, if $\partial_{z}^{2}\varphi=\lambda\partial_z\varphi$, we can deduce the existence of $a,b\in\rn$ with
\[u+iv=(a+ib)(\partial_x\varphi-i\partial_y\varphi)=a\partial_x\varphi+b\partial_y\varphi+i(b\partial_x\varphi-a\partial_y\varphi)\,\impl\,u,v\in\d\varphi(TM)\text{.}\]

(iv) Let us now finally prove that any strictly compatible almost Hermitian structure is given by
\eqref{Equation:FirstEquationIn:Lemma:AdaptedAlmostHermitianStructuresToRealIsotropicMaps}.
Such a structure must have $J(\partial_x\varphi)=\partial_y\varphi$ as it is compatible with $\varphi$. On the other hand, since we are assuming that $\partial_z\varphi$ and $\partial_{z}^{2}\varphi$ are linearly independent, we know that $u$ and $v$ do not belong to $\d\varphi(TM)$. Hence, $J u$ cannot lie in $\d\varphi(TM)$ and equations \eqref{Equation:FirstEquationIn:Lemma:AdaptedAlmostHermitianStructuresToRealIsotropicMaps} completely determine $J$. Take $\tilde{u}=u-\proj_{\d\varphi(TM)}u$. We have $<v,\tilde{u}>=0$ using
\eqref{Equation:FirstEquationInProofOf:Lemma:AdaptedAlmostHermitianStructuresToRealIsotropicMaps}\footnote{More explicitly, we use \eqref{Equation:FirstEquationInProofOf:Lemma:AdaptedAlmostHermitianStructuresToRealIsotropicMaps}
and $\proj_{\d\varphi(TM)}u=<u,\partial_x\varphi>\|\partial_x\varphi\|^{-2}\partial_x\varphi+<u,\partial_y\varphi>\|\partial_y\varphi\|^{-2}\partial_y\varphi$.} so that $\tilde{v}=v-proj_{\d\varphi(TM)}v$ lies in the orthogonal complement of $\wordspan\{\tilde{u},\d\varphi(TM)\}$. Finally, that $\|\tilde{u}\|=\|\tilde{v}\|$ is easy to verify. Therefore, $\{\partial_x\varphi,\partial_y\varphi,\tilde{u},\tilde{v}\}$ form an orthogonal basis for $TN$ with
$\|\tilde{u}\|=\|\tilde{v}\|$. As $J$ is almost Hermitian and preserves $\d\varphi(TM)$, it must map $\tilde{u}$ to $\pm\tilde{v}$. Hence, $J\tilde{u}=\pm \tilde{v}$; this implies that $J(u)=\pm v$.
\end{proof}

\noindent\textit{Second proof of Theorem \ref{Theorem:J1HolomorphicLiftsInTheFourDimensionalCase} (i).} As before, consider the real and imaginary part of $\partial_{z}^{2}\varphi=u+iv$. Then, any almost Hermitian structure strictly compatible with $\varphi$ verifies $J(u)=\pm v$. Let $J_\psi$ be the structure given by $J_{\psi}(u)=-v$\footnote{Therefore, $J_\psi\in\Sigma^+ N$ or $\Sigma^- N$ according as
$\{\partial_x\varphi,\partial_y\varphi,u,-v\}$ is a positively or negatively oriented basis for $TN$.}. It follows that $\partial_{z}^{2}\varphi=u+iv=u-iJ_\psi u\in{T}^{10}_{J_\psi}N$ and the map $\psi$ is $\J^1$-holomorphic since $T^{10}_{J_\psi}N$ is spanned by $\{\partial_z\varphi,\partial_{z}^{2}\varphi\}$ and $\nabla_{\partial_z}T^{10}_{J_\psi}N\subseteq T^{10}_{J_\psi}N$ as in the first proof, from $\varphi$ isotropy. \qed


\subsection{Examples of totally umbilic and real isotropic maps}


\begin{example}[\emph{Holomorphic~maps~are~real~isotropic}]\label{Example:EveryHolomorphicMapBetweenHermitianManifoldsIsRealIsotropic}
If $\varphi:M\to N$ is a holomorphic map between Hermitian manifolds, then $\varphi$ is real isotropic. Compare this result with Proposition \ref{Proposition:LichnerowiczProposition} and Proposition \ref{Proposition:821Wood}. In fact,
\[\d\varphi(T^{10}M)\subseteq{T}^{10}N,\,\nabla_{T^{10}M}\d\varphi(T^{10}M)\subseteq\nabla_{T^{10}M}T^{10}N\subseteq{T}^{10}N,..\text{.}\]
so that all $\nabla^j_{T^{10}M}\d\varphi(T^{10}M)\subseteq\,T^{10}N$; by isotropy of $T^{10}N$, the result follows.
\end{example}

\begin{example}[\emph{Pluriconformal,~totally-umbilic~and~real-isotropic~maps}]\label{Example:EveryTotallyUmbilicMapIsRealIsotropic}
If $\varphi:M\to N$, is a pluriconformal totally umbilic map from a Hermitian manifold $M$, then $\varphi$ is real isotropic. In fact, $\nabla\d\varphi(T^{10}M,T^{10}M)\subseteq\d\varphi(T^{10}M)$. Hence, since $M$ is Hermitian, $\nabla_{T^{10}M}\nabla^{T^{10}M}\d\varphi(T^{10}M)\subseteq\nabla_{T^{10}M}\d\varphi(T^{10}M)\subseteq\d\varphi(T^{10}M)$ and, inductively, $\nabla^j_{T^{10}M}\d\varphi(T^{10}M)\subseteq\d\varphi(T^{10}M)$. From pluriconformality of $\varphi$, it follows that \eqref{Equation:FirstEquationIn:Definition:RealIsotropicMapFromAHermitianManifold} holds and, therefore, $\varphi$ is real isotropic.
\end{example}

\begin{example}[\emph{A~non-holomorphic~example~of~a~real-isotropic (and totally umbilic) map}]\label{Example:ANonHolomorphicExampleOfRealIsotropicMap}
Let $\varphi:\cn\to\cn^2$ be the map defined by\addtolength{\arraycolsep}{-1.2mm}
\begin{equation}\label{Equation:EquationInThe:Example:ANonHolomorphicExampleOfRealIsotropicMap}
\begin{array}{llll}
\varphi:&\cn&\to&\cn^2\\
~& z&\to& (z^2+\bar{z},z^2+\bar{z})\text{.}
\end{array}
\end{equation}\addtolength{\arraycolsep}{1.2mm}\noindent
Then, we have
\[\frac{\partial\varphi}{\partial{z}}=(2z,2z)\text{~and~}\frac{\partial^2\varphi}{\partial{z}^2}=(2,2)=z.\frac{\partial\varphi}{\partial{z}}\text{.}\]
\end{example}

\begin{example}[\emph{Holomorphic~submersions~are~totally~umbilic}]\label{Example:EveryHolomorphicSubmersionBetweenHermitianManifoldsIsTotallyUmbilic}
If $\varphi:M\to N$ is a holomorphic submersion between two Hermitian manifolds, then it is totally umbilic. In fact, we have
\[\nabla\d\varphi(T^{10}M,T^{10}M)\subseteq\underbrace{\nabla_{T^{10}M}\underbrace{\d\varphi(T^{10}M)}_{\text{\tiny{$\subseteq\,\d\varphi(T^{10}M)=T^{10}N$}}}}_{\text{\tiny{$\subseteq{T}^{10}N=\d\varphi(T^{10}M)$}}}+\underbrace{\d\varphi(\underbrace{\nabla_{T^{10}M}T^{10}M)}_{\text{\tiny{$\subseteq{T}^{10}M$}}}}_{\text{\tiny{$\subseteq\,\d\varphi(T^{10}M)$}}}\,\subseteq\,\d\varphi(T^{10}M)\text{.}\]
\end{example}


\begin{example}[\emph{A~holomorphic~totally~umbilic~map}]\label{Example:ExampleOfATotallyUmbilicHolomorphicMapNotFromCnAndNotASubmersion}
The map\addtolength{\arraycolsep}{-1.2mm}
\begin{equation}\label{Equation:EquationIn:Example:ExampleOfATotallyUmbilicHolomorphicMapNotFromCnAndNotASubmersion}
\begin{array}{lcll}
\varphi:&\cn^2&\to&\cn^3 \\
~& (z,w)&\to& (z^2+w^2,z^2,w^2)
\end{array}
\end{equation}\addtolength{\arraycolsep}{1.2mm}\noindent
is a totally umbilic map, since
\[\begin{array}{l}\frac{\partial\varphi}{\partial z}=(2z,2z,0),\quad\frac{\partial\varphi}{\partial w}=(2w,2w,0)\,\text{~and}\\
\frac{\partial^2\varphi}{\partial{z}^2}=(2,2,0),\quad\frac{\partial^2\varphi}{\partial{w}^2}=(2,2,0),\quad\frac{\partial^2\varphi}{\partial{z}\partial{w}}=(0,0,0)\text{.}\end{array}\]
\end{example}

\begin{example}[\emph{A~totally~umbilic~non-holomorphic~map~and~its~lift}]\label{Example:ExampleOfATotallyUmbilicNonHolomorphicMapAndItsLift}
Consider the map\addtolength{\arraycolsep}{-1.2mm}
\begin{equation}\label{Equation:FirstEquationIn:Example:ExampleOfATotallyUmbilicNonHolomorphicMapAndItsLift}
\begin{array}{llll}
\varphi:&\cn&\to&\cn^2 \\
\text{ }& z&\to& \frac{1}{2}(z+\bar{z},z-\bar{z})\text{.}
\end{array}
\end{equation}\addtolength{\arraycolsep}{1.2mm}\noindent
This a totally umbilic map. We find a $\J^1$-holomorphic lift: take $z=x+iy$ and define $J(x,y)$ as the (constant!) Hermitian structure given by $J(x,y)(e_1)=e_3$,
$J(x,y)=e_4$, where $\cn^2\simeq\rn^4$. Consider the map\addtolength{\arraycolsep}{-1.2mm}
\begin{equation}\label{Equation:SecondEquationIn:Example:ExampleOfATotallyUmbilicNonHolomorphicMapAndItsLift}
\begin{array}{lcll}
\psi:&\cn\simeq\rn^2&\to&\Sigma^+(\rn^4) \\
~& (x,y)&\to& (f(x,y),J(x,y))\text{;}
\end{array}
\end{equation}\addtolength{\arraycolsep}{1.2mm}\noindent
it is easy to check that $\psi$ is a $\J^1$-holomorphic lift of
$\varphi$.

\end{example}



\section{Summary}


We summarize the main results of this chapter as follows\label{Page:PictureOfResume}

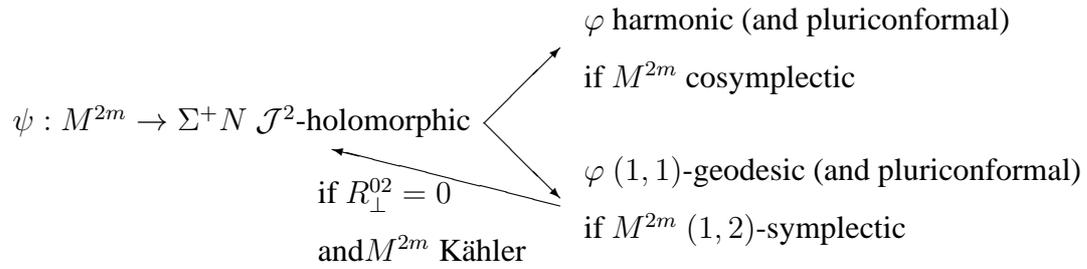
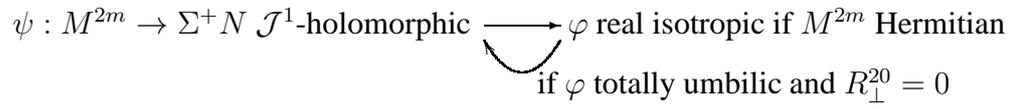
\setlength{\unitlength}{1cm}
\begin{picture}(20,10)(0,0)
\put(0,8){$\psi:M^{2m}\to\Sigma^+N$ $\J^2$-holomorphic}
\put(6.2,8.1){\vector(1,1){1}} \put(6.2,8.1){\vector(1,-1){1}}
\put(7.3,9){\begin{tabular}{l}$\varphi$ harmonic (and pluriconformal)\\ if $M^{2m}$ cosymplectic\end{tabular}} \put(7.3,7){\begin{tabular}{l}$\varphi$~$(1,1)$-geodesic~(and~pluriconformal)\\if~$M^{2m}$~$(1,2)$-symplectic\end{tabular}}
\put(7.2,7){\vector(-4,1){3}}\put(3.8,6.7){\begin{tabular}{l}if $R^{02}_\bot=0$\\and$ M^{2m}$ Kähler\end{tabular}}
\put(0,5){$\psi:M^{2m}\to\Sigma^+N$ $\J^1$-holomorphic}
\put(6.2,5.1){\vector(1,0){1}} \put(7.3,5){$\varphi$ real isotropic
if $M^{2m}$ Hermitian}
\put(6.2,4.8){$ \xy  {\ar@/^1pc/(10,0)*{};(0,0)*{}};\endxy$}
\put(6.9,4.2){if $\varphi$ totally umbilic and $R^{20}_{\bot}=0$}
\put(0,3){$\psi:M^{2m}\to\Sigma^+N$ $\H$-holomorphic}
\put(6.2,3.1){\vector(1,0){1}}\put(7.2,3.1){\vector(-1,0){1}}
\put(7.3,3){$\varphi$ $(J^M,J_\psi)$-holomorphic}
\end{picture}


\chapter{Harmonic morphisms and twistor spaces}\label{Chapter:HarmonicMorphismsAndTwistorSpaces}

\index{Harmonic~morphism!with~superminimal~fibres} Following the idea on p.\
\pageref{Page:TheIdeaToConstructHarmonicMorphismsAfter:Theorem:HarmonicMorphismsAndMeanCurvatureOfFibresSpecialCaseOfRiemannSurfaces},
we shall construct harmonic morphisms $\varphi:M^{2m}\to\cn$ by finding maps that are holomorphic with respect to some almost Hermitian structure on $M^{2m}$ (which guarantees that $\varphi$ is horizontally weakly conformal). These maps will have pluriminimal (hence, minimal) fibres. On the other hand, pluriminimal submanifolds can be constructed as the image of a pluriharmonic holomorphic map from a complex manifold into any almost Hermitian manifold (Theorem \ref{Theorem:PluriminimalSubmanifoldsAsImageOfPluriharmonicAndPluriconformalMaps}). Finally, Theorem \ref{Theorem:NewTheorem35} tells us how to construct maps $\varphi$ which are $(1,1)$-geodesic (pluriharmonic, when the domain is Kähler) as projections of $\J^2$-holomorphic maps $\psi$ into the twistor space. Hence, to obtain a harmonic morphism, we need to find suitable $\J^1$-holomorphic variations $H$ of those $\J^2$-holomorphic maps. Notice that we could think of choosing $H$ to be $\J^2$-holomorphic. However, some immediate problems would arise: first of all, $\J^2$ is never integrable and it would therefore be difficult to produce such holomorphic maps. On the other hand, as we look for maps from the \emph{complex manifold} $N\times{P}$, the fact that $h=\pi\circ H$ is a holomorphic (local) diffeomorphism would imply that the induced almost complex structure $J$ on $M^{2m}$ is integrable: but together with $\J^2$-holomorphicity we would conclude that $(M^{2m},g,J)$ is Kähler and, therefore, few harmonic morphisms could be constructed this way. Finally, we could try to construct $\J^2$-holomorphic maps whose domain is not necessarily a complex manifold; however, in such a case, we would have the problem noticed in Remark \ref{Remark:RemarkTo:Corollary:MinimalSubmanifoldsAndTwistorMaps}: the fibres may not be minimal! These observations lead us to the search for $\J^1$-holomorphic maps which are simultaneously $\J^2$-holomorphic in some variables (\textit{i.e.}, \emph{horizontal} in those variables), as we shall see below.

Note that when the codomain is a complex manifold of higher dimension, our construction will give holomorphic maps with superminimal fibres, and hence holomorphic families of superminimal (and so, minimal) submanifolds.


\section{Harmonic morphisms with superminimal fibres}


As motivated above, we have the following result, obtained in joint work with M.~Svensson (\cite{SimoesSvensson:06}):

\begin{theorem}\label{Theorem:ConstructionOfHarmonicMorphismsWithSuperminimalFibresBruSve}\index{Harmonic~morphism!holomorphic~parametrizations!}
Let $N^{2n}$ and $P^{2p}$ be complex manifolds and $M^{2(n+p)}$ an oriented Riemannian manifold. Denote by $\pi_1:N\times P\to N$ the projection onto the first factor. Assume that we have a $\J^1$-holomorphic map
\[H:W\subseteq N\times P\to\Sigma^+M,\quad(z,\xi)\mapsto H(z,\xi)\text{,}\]
defined on some open subset $W$, such that

\textup{(i)} for each $z$, the map $P\ni\xi\to{H}(z,\xi)\in\Sigma^+M$ is horizontal on its domain;

\textup{(ii)} the map $h=\pi\circ H$ is a diffeomorphism onto its image.

\noindent Then, the map $\pi_1\circ h^{-1}:h(W)\subseteq M\to N$ is holomorphic and has superminimal fibres with respect to the Hermitian structure on $h(W)$ defined by the section $H\circ h^{-1}$. In particular, when $n=1$, $\pi_1\circ h^{-1}$ is a harmonic morphism.
\end{theorem}

\begin{proof} Let us write $\varphi=\pi_1\circ h^{-1}$. The map $\varphi$ is holomorphic with respect to the Hermitian
structure on $M$ defined by the section $H\circ h^{-1}$ of $\Sigma^+M$. The fact that $H$ is horizontal in its second argument implies that this complex structure is parallel along the fibres of $\varphi$, \textit{i.e.}, $\varphi$ has superminimal fibres.
\end{proof}
\begin{remark} When $M$ is not orientable a similar result is also clearly true; we must of course replace $\Sigma^+M$ with $\Sigma~M$, the bundle of all orthogonal complex structures on $M$. On the other hand, we may replace $\Sigma^+M$ by any smaller subbundle invariant under $\J^1$ and $\J^2$.
\end{remark}

\begin{remark} One could ask whether there is a similar construction to this yielding holomorphic maps to some Kähler manifold, the fibres of which are pluriminimal (therefore, still minimal) but not necessarily superminimal. In the case $n=1$ that produces harmonic morphisms the answer is negative, as pluriminimal and superminimal coincide by Proposition
\ref{Proposition:PluriminimalAndSuperminimalInCodimension2}.
\end{remark}

We also have a converse of this result (\cite{SimoesSvensson:06}):

\begin{theorem}\label{Theorem:ConverseToThe:Theorem:ConstructionOfHarmonicMorphismsWithSuperminimalFibresBruSve}
Assume that $(M^{2(n+p)},J)$ is a Hermitian manifold, $N^{2n}$ a complex manifold and
$\varphi:M\rightarrow N$ a holomorphic submersion with superminimal fibres. Then, around any point $x\in M$ there exists a neighbourhood $W$ of $(\varphi(x),0)\in N\times\cn^{p}$ and a holomorphic map
\[H:W\subseteq N\times\cn^{p}\rightarrow \Sigma^+M\text{,}\]
horizontal in the second argument, such that $\pi\circ H$ is a diffeomorphism onto a neighbourhood of $x$ and
\[\varphi(\pi\circ H(z,\xi))=z\qquad((z,\xi)\in W)\text{.}\]
\end{theorem}

\begin{proof} From Corollary \ref{Proposition:CorollaryToSalamonTheorem3.2ComplexAnd12SymplecticVersusJ1AndJ2Holomorphicity}, integrability of $J$ is equivalent to the fact that the map $\sigma_J:M\rightarrow\Sigma^+M$ induced by $J$ is $\J^1$-holomorphic. As the fibres of $\varphi$ are superminimal, $\sigma_J$ will map these into the horizontal space of $\Sigma^+M$.

Fix a point $x\in M$ and a holomorphic chart
\[\eta:W\subseteq{N}\times\cn^p\to U\subseteq M\]
with $\varphi(\eta(z,\xi))=z$; this is possible as $\varphi$ is submersive. The map $H=\sigma_J\circ\eta$ is holomorphic as it is the composition of holomorphic maps. On the other hand, $\sigma_J$ is horizontal along the fibres, forcing $H$ to be horizontal in its second argument.
\end{proof}

\begin{remark}\label{Remark:RemarkTo:Theorem:ConverseToThe:Theorem:ConstructionOfHarmonicMorphismsWithSuperminimalFibresBruSve}
When $n=1$, \textit{i.e.}, $N$ is a Riemann surface, the condition that $J$ be integrable on $M$ is automatic from the fact that $\varphi$ has superminimal fibres, see \cite{BairdWood:03}, Theorem 7.9.1, p. 228.
\end{remark}


\section{Examples}


The following constructions were obtained in \cite{SimoesSvensson:06}.


\subsection{Maps from Euclidean spaces}


We construct local harmonic morphisms $\varphi:\mathbb{R}^{2n}\rightarrow\mathbb{C}$ using Theorem \ref{Theorem:ConstructionOfHarmonicMorphismsWithSuperminimalFibresBruSve}. An approach similar to this has been used in \cite{BairdWood:95} (see also \cite{BairdWood:03}), and we follow the notation set out there.

It is well known that $(\Sigma^{+}(\rn^{2n}),\J^1)$ is a complex manifold, and following \cite{BairdWood:95} we can easily construct a chart. For each $q\in\cn^n\cong\rn^{2n}$ and each $\mu=(\mu_{1},...,\mu_{n(n-1)/2})\in\cn^{n(n-1)/2}$, let
$M=M(\mu)\in\so(n,\cn)$ be the matrix with entries
\[M^i_{\overline{j}}(\mu)=\begin{pmatrix}
0 & \mu_{1} & \mu_{2} & ... & \mu_{m-1} \\
-\mu_{1} & 0 & \mu_{m} & ... & \mu_{2m-3} \\
-\mu_{2} & -\mu_{m} & 0 & ... & \mu_{3m-6} \\
... & ... & ... & ... & ... \\
-\mu_{m-1} & -\mu_{2m-3} & -\mu_{3m-6} & ... & 0
\end{pmatrix}\text{.}\]

Then $\mu$ determines the positive almost Hermitian structure $J(\mu)$ on $T_q\rn^{2n}$ whose $(1,0)$-cotangent space is spanned by $dq^i-M^i_{\overline{j}}\d\overline{q}^j$. For each $(q,\mu)$ we define $w\in\cn^n$ by $w^i=q^i-M^i_{\overline{j}}\overline{q}^j$, and this gives a holomorphic chart on an open dense subset of
$\Sigma^+(\rn^{2n})$, defined by
\[\psi:(q,J(\mu))\mapsto(w,\mu)\in\cn^n\times\cn^{n(n-1)/2}\text{.}\]

A \emph{holomorphic} map $f:\cn^{m}\rightarrow\Sigma^{+}(\rn^{2n})$ is \emph{horizontal} if and only if $\mu(f(z))$ \emph{is constant}. Hence, any holomorphic map $H:\cn^{k}\times\cn^{n-k}\rightarrow\Sigma^{+}(\rn^{2n})$ such that
\begin{itemize}
\item[(i)] $h=\pi\circ H:\cn^n\rightarrow\cn^n$ is a local
diffeomorphism, where $\pi:\Sigma^{+}\rightarrow \rn^{2n}$ is the
natural projection, and
\item[(ii)] $\psi\circ H(z,\xi)=(\omega(z,\xi),\mu(z))$ is holomorphic,
\end{itemize}
will define a holomorphic map $\pi_1\circ h^{-1}$ with superminimal fibres. When $k=1$, this will be a harmonic morphism.

\begin{example}
Take $n=3$, $k=1$ and choose $H:\cn\times\cn^{2}\rightarrow\Sigma^+(\rn^6)$ defined by \[H(z,\xi_1,\xi_2)=\big(h(z,\xi_1,\xi_2),J(\mu(z))\big)\text{,}\]
where $h(z,\xi_1,\xi_2)=\left(\dfrac{\xi_1+z\overline{\xi}_2}{1+\|z\|^2},\dfrac{\xi_2-z\overline{\xi_1}}{1+\|z\|^2},f(z)-\xi_1-\xi_2\right)$,
$f$ is any holomorphic function and
\[M(\mu(z))=\begin{pmatrix}
0 & z & 0 \\
-z & 0 & 0 \\
0 & 0 & 0
\end{pmatrix}\text{.}\]
As $\psi(H(z,\xi))=\big((\xi_1,\xi_2,f(z)-\xi_1-\xi_2),(z,0,0)\big)$,
we deduce that $H$ is horizontal in $\xi$; the resulting harmonic morphism $\varphi(q)=z$ will be a regular solution to the equation
\[f(z)+q^1+q^2+z(\overline{q}^1-\overline{q}^2)-q^3=0\text{.}\]
Choosing for example $f(z)=z$, gives the map \[\varphi(q^1,q^2,q^3)=\frac{q^3-q^1-q^2}{1+\overline{q}^1-\overline{q}^2}\text{.}\]
\end{example}


\subsection{Maps from complex projective spaces}


Let us denote by $G_p(\cn^{p+q})$ the Grassmannian of $p$-spaces in $\cn^{p+q}$. Our first result shows that there are only a few cases which allow non-trivial constructions of harmonic morphisms to surfaces (\cite{SimoesSvensson:06}).

\begin{proposition}\label{Proposition:PropositionThatShowsThatHarmonicMorphismsWithSuperminimalFibresFromCPNWithNGeq3AreCanonicallyHolomorphic}
Let $\openU\subseteq G_p(\cn^{p+q})$ be open and assume that $J$ is an almost Hermitian structure on $\openU$. If $p+q\geq4$, then any holomorphic map $\varphi:(\openU,J)\to N^2$ to a Riemann surface with superminimal fibres is $\pm$holomorphic with respect to the usual complex structure on $G_p(\cn^{p+q})$.
\end{proposition}

In particular, when $n\geq3$, any map from some open subset of $\cn P^n$ to a surface, which is holomorphic with respect to some almost Hermitian structure on this subset, is necessarily $\pm$holomorphic with respect to the usual complex structure on $\cn P^n$.

%


\subsubsection{Maps from $\cn P^3$.}


We construct a holomorphic map from an open dense subset of $\cn P^3$. Let $Z$ be the flag manifold consisting of all orthogonal decompositions
\[\cn^4=E_1\oplus E_2\oplus E_3\text{,}\]
where
\[\dim E_1=2\ \text{ and }\ \dim E_2=\dim E_3=1\text{.}\]
This is a complex manifold; its complex structure is induced by the embedding
\[Z\ni(E_1,E_2,E_3)\mapsto (E_1,E_1\oplus E_2)\in G_2(\cn^4)\times G_3(\cn^4)\text{.}\]

The space $Z$ is fibred over $\cn P^3$ by the mapping
\[ \pi:Z\to\cn P^3,\quad \pi(E_1,E_2,E_3)=E_2\text{.}\]
With each such decomposition, we obtain a positive orthogonal complex structure $J$ on $T_{E_2}\cn P^3$ by identifying
\[T^\cn_{E_2}\cn P^3\cong\Hom(E_2,E_2^\perp)\oplus\Hom(E_2^\perp,E_2)\]
in the usual way (see, \textit{e.g.}, \cite{EschenburgTribuzy:99}), and defining the $(1,0)$-subspace of $J$ to be
\[\Hom(E_1,E_2)\oplus\Hom(E_2,E_3).\]
This embeds $Z$ as a subbundle of $\Sigma^+(\cn P^3)$; the complex structure on $Z$ is easily seen to be the restriction of $\J^1$ to $Z$.

Alternatively, as a bundle over $\cn P^3$, $Z$ may be identified with the Grassmannian $G_2(T^{1,0}\cn P^3)$ of $2$-dimensional subspaces of $T^{1,0}\cn P^3$, see \textit{e.g.}, \cite{DavidovSergeev:93}, page 59.

As a homogeneous space, we have
\[Z\cong\frac{\U(4)}{\U(2)\times\U(1)\times\U(1)}\text{.}\]
As it is well-known (see, \textit{e.g.}, \cite{EschenburgTribuzy:99}), $Z$ carries a $\U(4)$-invariant distribution, the \emph{horizontal distribution}, which we denote by $\H$. This distribution is transversal to the fibres of $\pi$ and has the property that, under the embedding of $Z$ into $\Sigma^+(\cn P^3)$, it is mapped onto the horizontal distribution of $\Sigma^+(\cn P^3)$.

Following the strategy laid out in the previous section, we aim to construct a holomorphic map $(z,\xi)\mapsto H(z,\xi)$, horizontal in $\xi$ with the property that the induced map $h=\pi\circ H$ is a diffeomorphism onto its image:
\begin{equation}\label{Equation:SchemeToProduceMapsFromCP3ToC2WithSuperminimalFibres}
\xymatrix{ \cn^2\times\cn\supseteq
U\ar[rr]^H\ar[rrd]^h\ar[d]^{\pi_1} & & Z
\ar[d]^\pi \\
\cn^2 &  & \cn P^3  }
\end{equation}
Then $\pi_1\circ h^{-1}$ will be holomorphic and have superminimal fibres with respect to some (integrable) Hermitian structure on the image of $h$.

The map $H$ is given by two holomorphic maps $$f:U\to G_2(\cn^4),\quad s:U\to G_3(\cn^4),$$ with the property that
$f\subseteq s$ at every point. The condition that $H$ is horizontal in $\xi$ can be expressed by the condition \[\partial_\xi f\subseteq s,\]
see, \textit{e.g.}, \cite{ErdemWood:83}.

To find such $f$, recall that we have an inclusion mapping with image on an open, dense subset
\[\Hom(\cn^2,\cn^2)\to G_2(\cn^4),\]
obtained by associating to a linear map $\phi:\cn^2\to\cn^2=(\cn^2)^\perp\subseteq\cn^4$ its
graph, regarded as an element in $G_2(\cn^4)$. Thus, assuming that $f$
takes values in $\Hom(\cn^2,\cn^2)$, we may write
\[f=\begin{pmatrix} \alpha & \beta \\   \gamma & \delta\end{pmatrix}\text{,}\]
where $\alpha$, $\beta$, $\gamma$ and $\delta$ are holomorphic,
complex valued functions on $\cn^2\times\cn$. This means that
\[f=\wordspan_\cn\{e_1+\alpha e_3+\gamma e_4,e_2+\beta e_3+\delta e_4\},\]
where $\{e_i\}_{i=1}^4$ is the standard basis for $\cn^4$.

In a similar fashion, we have $\Hom(\cn^3,\cn)\subseteq G_3(\cn^4)$, and we may thus write
\[s=\begin{pmatrix} u & v & w\end{pmatrix}\text{,}\]
for some holomorphic functions $u$, $v$ and $w$ on $U$, so that
\[s=\wordspan_\cn\{e_1+ue_4,e_2+ve_4,e_3+we_4\}\text{.}\]
In general, a vector $ae_1+be_2+ce_3+de_4$ belongs to $s$ if and only if
\[0=\det\begin{pmatrix} a & 1 & 0 & 0 \\ b & 0 & 1 & 0 \\ c & 0 & 0 & 1 \\ d & u & v & w \end{pmatrix}=au+bv+cw-d\text{.}\]
The requirement that $(z,\xi)\mapsto H(z,\xi)$ takes its values in
$Z$ and is horizontal in $\xi$ is thus expressed in the following
equations:\addtolength{\arraycolsep}{-1.0mm}
\begin{equation}\label{Equation:FirstEquationIn:Example:TheComplexProjectiveSpace}
\left.\begin{array}{lll}
u+\alpha w-\gamma&=&0\\
v+\beta w-\delta&=&0\\
w\partial_\xi\alpha-\partial_\xi\gamma&=&0\\
w\partial_\xi\beta-\partial_\xi\delta&=&0\text{.}
\end{array}\qquad\right\}
\end{equation}\addtolength{\arraycolsep}{1.0mm}\noindent
We can thus choose $w$, $\alpha$ and $\beta$ arbitrarily; this will determine $\gamma$ and $\delta$ up to additive functions of $z$, and from these, $u$ and $v$ will be determined.

The induced map $h$ is given by $f^\perp\cap s=[x_1,x_2,x_3,x_4]\in\cn P^3$; hence
\begin{equation*}\label{Equation:SecondEquationIn:Example:TheComplexProjectiveSpace}
\begin{array}{ll}
x_1u+x_2v+x_3w-x_4=0\text{,}
\\ x_1+x_3\overline\alpha+x_4\overline\gamma=0\text{,}\\
x_2+x_3\overline\beta+x_4\overline\delta=0\text{.}
\end{array}
\end{equation*}
Solving these equations using the two first equations of \eqref{Equation:FirstEquationIn:Example:TheComplexProjectiveSpace}, we get
\begin{equation}\label{Equation:ThirdEquationIn:Example:TheComplexProjectiveSpace}
\left.\begin{array}{lll}
x_1 & = \overline\gamma(\overline\beta\delta-|\beta|^2w-w)+\overline\alpha(\overline\delta\beta w-|\delta|^2-1)\\
x_2 & = \overline\delta(\overline\alpha\gamma-|\alpha|^2w-w)+\overline\beta(\overline\gamma\alpha w-|\gamma|^2-1)\\
x_3 & = 1+|\gamma|^2+|\delta|^2-w(\overline\gamma\alpha+\overline\delta\beta)\\
x_4 & = w(1+|\alpha|^2+|\beta|^2)-(\overline\alpha\gamma+\overline\beta\delta).
\end{array}\qquad\right\}
\end{equation}
It is now easy to construct maps $H$ satisfying the conditions of Theorem
\ref{Theorem:ConstructionOfHarmonicMorphismsWithSuperminimalFibresBruSve}.
\begin{example} We may choose our data as
\begin{equation*}
\alpha=\xi+z_1,\ \beta=\xi+z_2,\ \omega=2\xi,\ \gamma=\xi^2+z_1,\
\delta=\xi^2+z_2.
\end{equation*}
The last two lines of \eqref{Equation:FirstEquationIn:Example:TheComplexProjectiveSpace} are satisfied, and a lengthy calculation shows that the map $h$ is given by $h(z_1,z_2,\xi)=[x_1,x_2,x_3,x_4]$, with $x_1,x_2,x_3,x_4$ defined
by
\eqref{Equation:ThirdEquationIn:Example:TheComplexProjectiveSpace}
is a diffeomorphism in a neighbourhood of the origin. Thus it
induces a map from a subset of $\cn P^n$ to $\cn^2$; this map is
holomorphic with respect to some Hermitian structure and has
superminimal fibres.
\end{example}

\begin{example}[\emph{Harmonic morphisms with superminimal fibres}]
We next construct a harmonic morphism $\varphi:\cn P^3\to\cn$ with superminimal fibres. For that, we consider a scheme similar to \eqref{Equation:SchemeToProduceMapsFromCP3ToC2WithSuperminimalFibres}, now changing the roles of $\cn$ and $\cn^2$:
\begin{equation}
\xymatrix{ \cn\times\cn^2\supseteq
U\ar[rr]^H\ar[rrd]^h\ar[d]^{\pi_1} & & Z
\ar[d]^\pi \\
\cn &  & \cn P^3  }
\end{equation}
As before, we look for maps $f$ and $s$,
\[f=\begin{pmatrix}\alpha&\beta \\ \gamma & \delta\end{pmatrix}\text{,}\qquad s=\begin{pmatrix}u & v & w\end{pmatrix}\]
with $f\subseteq s$ and $\partial_{\xi_i}f\subseteq s$ ($i=1,2$). These conditions are equivalent to\addtolength{\arraycolsep}{-1.0mm}
\begin{equation}\label{Equation:FirstEquationIn:Example:HarmonicMorphismsWithSuperminimalFibres}
\left.\begin{array}{lll}
u+\alpha w-\gamma&=&0\\
v+\beta w-\delta&=&0\\
w\partial_{\xi_i}\alpha-\partial_{\xi_i}\gamma&=&0,\quad i=1,2\\
w\partial_{\xi_i}\beta-\partial_{\xi_i}\delta&=&0,\quad i=1,2\text{.}
\end{array}\qquad\right\}
\end{equation}\addtolength{\arraycolsep}{1.0mm}\noindent
Choose $w=P(z)$, $\alpha=Q(\xi_1)$, $\beta=i.R(\xi_2)$, $\gamma=w.\alpha$ and $\delta=w.\beta$, where $P,Q$ and $R$ are arbitrary holomorphic functions (for instance, polynomials) with $P(0)=Q(0)=R(0)=0$ and first derivatives at the origin given by nonzero real numbers $p,q$ and $r$. The last two equations in \eqref{Equation:FirstEquationIn:Example:HarmonicMorphismsWithSuperminimalFibres} are satisfied since
\[\partial_{\xi_i}\gamma=\omega.\partial_{\xi_i}\alpha\quad\text{and}\quad\partial_{\xi_i}\delta=\omega.\partial_{\xi_i}\beta,\quad{i}=1,2\]
and equations \eqref{Equation:ThirdEquationIn:Example:TheComplexProjectiveSpace} simplify to

\[\begin{split}
x_1&=-\overline{\alpha}(1+|w|^{2})\\
x_2 & =-\overline{\beta}(1+|w|^{2})\\
x_3 & =1\\
x_4 &=w\text{.}
\end{split}\]
Hence, the map $H$ we obtain to $\cn P^3$ will be a local diffeomorphism around the origin if and only if it is a diffeomorphism the map defined by $\tilde{H}=(x_1,x_2,x_4)$. We have
\begin{equation*}
\partial_{(z,\xi_1,\xi_2)}\tilde{H} (0)=
\left[\begin{array}{ccc}
0 & 0 & 0 \\
0 & 0 & 0 \\
p & 0 & 0 \
\end{array}\right]\quad\text{and}\quad\partial_{(\bar{z},\bar{\xi}_1,\bar{\xi}_2)}\tilde{H}(0)=
\left[\begin{array}{ccc}
0 & -q & 0 \\
0 & 0 & ir \\
0 & 0 & 0 \
\end{array}\right]\text{.}
\end{equation*}
So, as a real map $\tilde{H}:\rn^6\to\rn^6$, $\tilde{H}$ has derivative at the origin given by
\[\left[\begin{array}{cccccc}
0 & -q & 0 & 0 & 0 & 0 \\
0 & 0 & 0 & 0 & 0 & r \\
p & 0 & 0 & 0 & 0 & 0 \\
0 & 0 & 0 & 0 & q & 0 \\
0 & 0 & r & 0 & 0 & 0 \\
0 & 0 & 0 & p & 0 & 0
\end{array}\right]\]
with nonzero determinant as we assumed that $p$, $q$ and $r$ are different from zero. Hence, $H$ is a local diffeomorphism and it defines a (local) harmonic morphism $\varphi:\cn P^3\to\cn$ with superminimal fibres. Notice that $\varphi$ is holomorphic with respect to the complex structure on $\cn P^3$ defined by $H$ (Theorem \ref{Theorem:ConstructionOfHarmonicMorphismsWithSuperminimalFibresBruSve}) but also with respect to the canonical complex structure on $\cn P^3$ (Proposition \ref{Proposition:PropositionThatShowsThatHarmonicMorphismsWithSuperminimalFibresFromCPNWithNGeq3AreCanonicallyHolomorphic}).
\end{example}


\chapter{Jacobi fields and twistor spaces}\label{Chapter:JacobiVectorFieldsAndTwistorSpaces}



The infinitesimal deformations of harmonic maps are called \emph{Jacobi fields}. They satisfy a system of partial differential equations given by the linearization of those for harmonic maps. The use of twistor methods in the study of Jacobi fields has proved quite fruitful, leading to a series of results (\cite{LemaireWood:96}, \cite{MontielUrbano:97}, \cite{Wood:02}). In \cite{LemaireWood:02} and \cite{LemaireWood:07} several properties of Jacobi fields along harmonic maps from the $2$-sphere to the complex projective plane and to the $4$-sphere are obtained by carefully studying the twistorial construction of those harmonic maps. In particular, relating the infinitesimal deformations of the harmonic maps to those of the holomorphic data describing them. In this chapter, we show how results in Chapter \ref{Chapter:HarmonicMapsAndTwistorSpaces} extend to ``first order parametric versions", providing this way a unified twistorial framework for the results in \cite{LemaireWood:02} and \cite{LemaireWood:07}.


\section{Jacobi~vector~fields~along~harmonic~maps}\label{Section:SectionOf:Chapter:JacobiVectorFieldsAndTwistorSpaces:JacobiVectorFieldsAlongHarmonicMaps}


\index{Jacobi!equation}\index{Jacobi!vector~field|see{Jacobi~field}}\index{Vector~field!Jacobi|see{Jacobi~field}}\index{Jacobi!operator}
Given a harmonic map $\varphi:M\to N$ and a (smooth) vector field $v\in\Gamma(\varphi^{-1}TN)$ along it, $v$ is said to be a \textit{Jacobi field} (along $\varphi$) (see \textit{e.g.}, \cite{EellsLemaire:78}, p.11, \cite{Wood:06}) if it satisfies the linear \emph{Jacobi equation} $J_{\varphi}(v)=0$, where the \emph{Jacobi operator} $J_{\varphi}$ is defined by
\begin{equation}\label{Equation:JacobiEquationForVectorFields}
J_{\varphi}(v)=\triangle v-\trace R^N(\d\varphi,v)\d\varphi\text{.}
\end{equation}
Here, $\triangle$ is the Laplacian on $\varphi^{-1}TN$:
\begin{equation}\label{Equation:LaplacianOnAVectorAlongAMap}
\triangle{v}=-\sum_i(\nabla^{\varphi^{-1}}_{X_i}\nabla^{\varphi^{-1}}_{X_i}v-\nabla^{\varphi^{-1}}_{\nabla^M_{X_i}X_i}v)\text{~($X_i$~orthonormal~(local)~frame~for~$TM$)}
\end{equation}
and
\begin{equation}\label{Equation:TheTraceOfTheCurvatureOnTheJacobiEquation}
\trace{R}^N(\d\varphi,v)\d\varphi=\sum_iR^N(\d\varphi{X}_i,v)(\d\varphi{X}_i)\text{.}
\end{equation}

Jacobi fields are characterized as lying in the kernel of the second variation of the energy functional \eqref{Equation:EnergyFunctional}. Indeed, if $\varphi_{t,s}$ is a two-parameter variation of a harmonic map $\varphi_{(0,0)}$, then, writing $v=\left.\frac{\partial\varphi}{\partial t}\right|_{(0,0)}$ and $w=\left.\frac{\partial\varphi}{\partial s}\right|_{(0,0)}$, the \emph{Hessian} $H_{\varphi}$ of $\varphi$ is the bilinear operator on $\Gamma(\varphi^{-1}TN)$ given by
\begin{equation}\label{Equation:HessianOfAHArmonicMap}\index{Harmonic~map!Hessian}\index{$H_\varphi$,~Hessian~of~a~harmonic~map}\index{Jacobi~field!and~the~Hessian~$H_\varphi$}\index{Integrable~Jacobi~field}\index{Jacobi~field!integrable!}
H_\varphi(v,w):=\left.\frac{\partial^2E(\varphi_{t,s})}{\partial{t}\partial{s}}\right|_{(0,0)}=\int_M<J_\varphi(v),w>(x)\d\mu_M(x)
\end{equation}
so that a Jacobi field $v$ (along $\varphi$) is characterized by the condition
\[H_\varphi(v,w)=0,\quad\forall\,w\text{.}\]

As we stated in the introduction, the main idea is to generalize the results from the previous chapters, where no parameter $t$ is involved (and to which we shall refer as \emph{non-parametric}), to \emph{first order parametric} versions.

Given a (family of) map(s) $\varphi:I\times M\to N$, $(t,x)\to\varphi_t(x)$, we say that $\varphi$ \emph{is harmonic to first order} if
\begin{equation}\label{Equation:DefinitionOfHarmonicToFirstOrder}\index{Harmonic~map!to~first~order}\index{To~first~order!harmonic~map|see{Harmonic~map}}
\begin{array}{l}
\varphi_0\text{ is harmonic and }\ddtatzero\tau(\varphi_t)=0
\end{array}
\end{equation}
where $\ddtatzero\tau(\varphi_t)=\nabla_{\ddtatzero}^{\varphi^{-1}TN}\tau(\varphi_t)$ and $\tau(\varphi_t)=\trace\nabla\d\varphi_t\in\varphi^{-1}TN$. Note that, writing $\tau(\varphi_t)$ as $\sum\tau(\varphi_t)_j\partial_{\psi_j}$ where $(\psi_1,...,\psi_n)$ are local coordinates for $N$, condition \eqref{Equation:DefinitionOfHarmonicToFirstOrder} is equivalent to the requirement that $\tau(\varphi_0)_j=\ddtatzero(\tau(\varphi_t)_j)=0$ for all $1\leq j\leq n$ (\cite{Wood:06}).

Equivalently, $\varphi$ is harmonic to first order if and only if $\tau(\varphi_t)$ is $o(t)$.

Let $\varphi_0:M\to N$ be a harmonic map between two manifolds $M$ and $N$. Let $v\in\Gamma(\varphi_0^{-1}TN)$ be a vector field along $\varphi_0$ and let $\varphi:I\times M\to N$ a one-parameter variation of $\varphi_0$. We say that $\varphi$ is \emph{tangent} to $v$ if $v=\left.\frac{\partial\varphi_t}{\partial t}\right|_{t=0}$.

The following result is a key ingredient in what follows (\cite{LemaireWood:02}):

\begin{proposition}\label{Proposition:JacobiVectorFieldsInduceVariationsHarmonicToFirstOrder}\index{Harmonic~map!and~Jacobi~fields}\index{Jacobi~field!and~harmonic~to~first~order}
Let $\varphi_0:M\to N$ be a harmonic map between compact manifolds $M$ and $N$. Let $v\in\Gamma(\varphi_0^{-1}TN)$ be a vector field along $\varphi_0$ and $\varphi:I\times M\to N$ a one-parameter variation of $\varphi_0$ tangent to $v$. Then,
\begin{equation}\label{Equation:EquationIn:Proposition:JacobiVectorFieldsInduceVariationsHarmonicToFirstOrder}\index{Vector~field!tangent~to~a~one-parameter~variation}
\ddtatzero\tau(\varphi_t)=-J_\varphi(v).
\end{equation}
In particular, $v$ is Jacobi if and only if any tangent one-parameter variation is harmonic to first order.
\end{proposition}


We have seen in Corollaries \ref{Corollary:SalamonCorollary4.2} and \ref{Corollary:NewSalamonCorollary4.2} that harmonicity was not enough to establish a relation with possible twistor lifts of a map: conformality (or pluriconformality) was also a key ingredient, as maps obtained as projections of twistorial maps must be holomorphic
with respect to some almost Hermitian structure along the map. On the other hand, when the domain is the $2$-sphere, harmonicity implies (weak) conformality or even real isotropy, the last case occurring if the target manifold is itself also a sphere or the complex projective space (p.\ \pageref{Page:HarmonicMapsFromAndToSpheresAndTheirConformalityOrRealIsotropy} and references therein).

Let $M^2$ be a Riemann surface and $\phi:I\times M\to N$ a smooth map. The map $\phi$ is said to be \textit{conformal to first order} if (see \cite{Wood:02}; compare with the condition of conformal map from a Riemann surface \eqref{Equation:EquationAfter:Definition:RealIsotropicMapFromARiemannSurface:ConformalityConditionOnARiemannSurface})
\begin{equation}\label{Equation:ConformalToFirstOrderFromARiemannSurface}\index{Conformal~map!to~first~order}\index{To~first~order!conformal~map|see{Conformal~map}}
\phi_0\text{~is~conformal~and~}\ddtatzero<\partial_z\phi_t,\partial_z\phi_t>=0\text{.}
\end{equation}
Equivalently, conformality to first order is the same as requiring $<\partial_z\phi_t,\partial_z\phi_t>$ to be $o(t)$. Analogously, $\phi$ is said to be \textit{real isotropic to first order} if
(see \cite{Wood:02}; compare with Definition
\ref{Definition:RealIsotropicMapFromARiemannSurface})\index{Real~isotropic~map!to~first~order}\index{To~first~order!real~isotropic~map}
\begin{equation}\label{Equation:Definition:RealIsotropicMapsToFirstOrder}
\phi_0\text{~is~real~isotropic~and~}\ddtatzero<\partial^r_z\phi_t,\partial^s_z\phi_t>=0,\quad\forall\,r,s\geq1\text{.}
\end{equation}
In the same spirit, $\phi$ is said to be \textit{complex isotropic to first order} (see \cite{Wood:02}) if
\begin{equation}\label{Equation:ComplexIsotropicToFirstOrder}\index{Complex~isotropic~map!to~first~order}\index{To~first~order!complex~isotropic~map}
\phi_0\text{~is~complex~isotropic~and~}\ddtatzero<\nabla^{r-1}_{\partial^{r-1}_z}\partial^{10}_z\phi,\nabla^{s-1}_{\partial^{s-1}_{\bar{z}}}\partial^{10}_{\bar{z}}\phi_t>_{Herm}=0,\quad\forall\,r,s\geq1\text{.}
\end{equation}

As in the non-parametric case (p.\ \pageref{Page:HarmonicMapsFromAndToSpheresAndTheirConformalityOrRealIsotropy}
and references therein), harmonicity to first order implies conformality to first order for maps defined on $S^2$ and even real isotropy when the codomain is itself a real or complex space form. More precisely, we have the following (\cite{Wood:02}, Propositions 3.2, 3.4 and 4.4):

\begin{proposition}\label{Proposition:ConformalityAndRealIsotropyToFirstOrderFromHarmonicityToFirstOrderAndSpheres}\index{Real~isotropic~map!to~first~order,~from~the~$2$-sphere}\index{Real~isotropic~map!to~first~order,~to~the~sphere}\index{Conformal~map!to~first~order,~from~the~sphere}\index{Conformal~map!to~first~order,~from~the~sphere}\index{Space~form!and~real~isotropy~to~first~order}\index{Real~isotropic~map!to~first~order~and~space~forms}\index{To~first~order!isotropy~and~harmonicity}
Let $\varphi_0:S^2\to N^n$ be a harmonic map from the $2$-sphere and $v$ a Jacobi field along $\varphi$. Then, any smooth variation of $\varphi_0$ tangent to $v$ is conformal to first order. Moreover, if $N^n$ is a space form (respectively, a complex space form), any such variation is real isotropic (respectively, complex isotropic) to first order\footnote{In particular, when $N^n$ is a complex space form, $\varphi$ is also real isotropic to first order; see Section \ref{Section:SectionIn:Chapter:Addendums:RealAndComplexIsotropicMaps} for more details.}.
\end{proposition}


\section{Twistorial constructions}


As we have seen, Jacobi fields induce variations that are harmonic (and, in some cases, conformal or even real isotropic) to first order. On the other hand, in Chapter \ref{Chapter:HarmonicMapsAndTwistorSpaces} we have seen that
conformality, harmonicity and real isotropy of the map $\varphi$ corresponds to $\H$, $\J^2$ and $\J^1$-holomorphicity of the twistor lift $\psi$. The question that naturally arises is whether these results have a to first order version and this is what we shall answer in the following.

First of all, we need the notion of map \textit{holomorphic to first order}:

\begin{definition}\label{Definition:HolomorphicToFirstOrder}\index{Holomorphic!to~first~order}\index{Orthogonal~decomposition}\index{Stable!decomposition}\index{$\H$-holomorphic!to~first~order}
Let $(M,J)$ be an almost complex manifold and $(Z,h,\J)$ an almost Hermitian manifold. Given a smooth map $\psi:I\times M\to Z$, we say that $\psi$ is \textit{holomorphic to first order} if
\begin{equation}\label{Equation:FirstEquationInDefinition:HolomorphicToFirstOrder}
\psi_0:M\to Z\text{ is holomorphic and}
\end{equation}
\begin{equation}\label{Equation:SecondEquationInDefinition:HolomorphicToFirstOrder}
\nabla^{\psi^{-1}}_{\ddtatzero}\{\d\psi_tJX-\J\d\psi_tX\}=0\quad\forall\,X\in{TM}
\end{equation}
where $\nabla$ is the Levi-Civita connection on $Z$ induced by the metric $h$. Moreover, if $TZ=\H\oplus\V$ is a $\J$-stable decomposition of $TZ$, orthogonal with respect to $h$, we shall say that $\psi$ is \emph{$\H$-holomorphic to first order} if (compare with Definition \ref{Definition:HolomorphicityInCertainSubbundles})
\begin{equation}\label{Equation:ThirdEquationInDefinition:HolomorphicToFirstOrder}
\big(\d\psi_0 JX\big)^{\H}=J\big(\d\psi_0 X\big)^{\H}\text{~(\textit{i.e.}, $\psi_0$ is $\H$-holomorphic) and}
\end{equation}
\begin{equation}\label{Equation:FourthEquationInDefinition:HolomorphicToFirstOrder}
\nabla^{\psi^{-1}}_{\ddtatzero}\{(\d\psi_t(JX))^{\H}-\J^\H(\d\psi_tX)^{\H}\}=0\quad\forall\,X\in{T}M
\end{equation}
where $\J^\H$ is the restriction of $\J$ to $\H$. Changing $\H$ to $\V$ gives the definition of \emph{$\V$-holomorphicity to first order}.
\end{definition}

In contrast with the non-parametric case, it is not obvious that $\J$-holomorphicity to first order implies $\H$-holomorphicity to first order. As a matter of fact, from \eqref{Equation:SecondEquationInDefinition:HolomorphicToFirstOrder}, it follows that $(\nabla_{\ddtatzero}\{\d\psi_t(JX)-\J\d\psi_tX\})^\H=0$ but this is not \eqref{Equation:FourthEquationInDefinition:HolomorphicToFirstOrder}. However, we do have the following:

\begin{lemma}\label{Lemma:JHolomorphicityToFirstOrderIffBothHAndVHolomorphicityToFirstOrderHold}
Let $\psi:(M,J)\to(Z,h,\J)$ be a smooth map and let $TZ=\H\oplus\V$ be a $\J$-stable decomposition of $Z$, orthogonal with respect to $h$. Then, $\psi$ is holomorphic to first order if and only if $\psi$ is both $\H$ and $\V$-holomorphic to first order.
\end{lemma}
\begin{proof}
Assume that $\psi$ is holomorphic to first order, so that \eqref{Equation:FirstEquationInDefinition:HolomorphicToFirstOrder} and
\eqref{Equation:SecondEquationInDefinition:HolomorphicToFirstOrder} hold. Then, \eqref{Equation:ThirdEquationInDefinition:HolomorphicToFirstOrder} is satisfied. As for \eqref{Equation:FourthEquationInDefinition:HolomorphicToFirstOrder} we have\addtolength{\arraycolsep}{-1.0mm}
\[\begin{array}{ll}
~&\eqref{Equation:FourthEquationInDefinition:HolomorphicToFirstOrder}\,\equi\,<\nabla_{\ddtatzero}\{(\d\psi_tJX)^\H-\J^\H(\d\psi_tX)^\H\},Y>=0,\quad\forall\,Y\in{T}Z\\
\equi&\ddtatzero\hspace{-2mm}<\hspace{-1mm}(\d\psi_tJX)^\H\hspace{-1mm}-\hspace{-1mm}\J^\H(\d\psi_tX)^\H\hspace{-1mm},Y\hspace{-1mm}>\hspace{-1mm}-\hspace{-1mm}<\hspace{-1mm}(\d\psi_0JX)^\H\hspace{-1mm}-\hspace{-1mm}\J^\H(\d\psi_0X)^\H\hspace{-1mm},\nabla_{\ddtatzero}\hspace{-2mm}Y\hspace{-1mm}>=0\\
\equi&\text{(since~\eqref{Equation:ThirdEquationInDefinition:HolomorphicToFirstOrder}~holds)~}\ddtatzero<\d\psi_tJX-\J\d\psi_tX,Y^H>=0\\
\equi&<\nabla_{\ddtatzero}\{\d\psi_tJX-\J\d\psi_tX\},Y^\H>+<\d\psi_0JX-\J\d\psi_0X,\nabla_{\ddtatzero}Y^\H>=0\\
\equi&\text{($\psi_0$~is~holomorphic)~}\big(\nabla_{\ddtatzero}\{\d\psi_tJX-\J\d\psi_tX\}\big)^\H=0\text{,}\end{array}\]
\addtolength{\arraycolsep}{1.0mm}\noindent which is true since $\psi$ is $\J$-holomorphic to first order. Hence, $\psi$ is $\H$-holomorphic to first order. Changing $\H$ to $\V$ shows that $\psi$ is $\V$-holomorphic to first order. As for the converse, assume that $\psi$ is both $\H$ and $\V$-holomorphic to first order. Then, since $\psi_0$ is both $\H$ and $\V$-holomorphic, we immediately deduce that $\psi_0$ is holomorphic. As for equation \eqref{Equation:SecondEquationInDefinition:HolomorphicToFirstOrder}: using the above equivalences, from $\H$-holomorphicity to first order and $\psi_0$ holomorphicity we can deduce $\big(\nabla_{\ddtatzero}\{\d\psi_tJX-\J\d\psi_t X\}\big)^\H=0$; similarly, $\big(\nabla_{\ddtatzero}\{\d\psi_tJX-\J\d\psi_tX\}\big)^\V=0$. Adding both these identities gives $\nabla_{\ddtatzero}\{\d\psi_tJX-\J\d\psi_t X\}=0$, showing that $\psi$ is holomorphic to first order and concluding our proof.
\end{proof}

\begin{remark}\label{Remark:FirstRemarkTo:Lemma:JHolomorphicityToFirstOrderIffBothHAndVHolomorphicityToFirstOrderHold}\index{Levi-Civita~connection!and~first~order~holomorphicity}
The importance of choosing the Levi-Civita connection on $Z$ is \textit{illusory}. In particular, we can define the concept of \emph{holomorphicity to first order} (\emph{or $\H$, $\V$-holomorphicity to first order}) for maps defined between almost complex manifolds, not necessarily equipped with any metric. In fact, if $\H\oplus\V$ is any $\J$-stable decomposition of $TZ$ and $\psi:I\times(M,J)\to(Z,h,\J)$ a smooth map, $\psi$ is holomorphic to first order (respectively, $\H$ or $\V$-holomorphic to first order) with respect to the pull back of the Levi-Civita connection $\nabla$ on $(Z,h,\J)$ if and only if $\psi$ is holomorphic to first order (respectively, $\H$ or $\V$-holomorphic to first order) with respect to the pull back of any connection $\widetilde{\nabla}$ on $Z$. Indeed, letting $\{Y_j\}$ denote a (local) frame for $TZ$, if $\psi$ is holomorphic to first order with respect to $\nabla$ then \eqref{Equation:SecondEquationInDefinition:HolomorphicToFirstOrder} holds. Now,
\[\nabla^{\psi^{-1}}_{\ddtatzero}(\d\psi_tJX-\J\d\psi_tX)=0\,\equi\sum_j\ddtatzero\big((\d\psi_tJX-\J\d\psi_tX)_jY_j\big)=0\text{,}\]
equivalently,
\begin{equation}\label{Equation:FirstEquationIn:Remark:FirstRemarkTo:Lemma:JHolomorphicityToFirstOrderIffBothHAndVHolomorphicityToFirstOrderHold}
\sum_j\{\ddtatzero\big((\d\psi_tJX-\J\d\psi_tX)_j\big).Y_j+(\d\psi_0JX-\J\d\psi_0X)_j.\nabla^{\psi^{-1}}_{\ddtatzero}Y_j\}=0\text{.}
\end{equation}
Since $\psi$ is holomorphic to first order, it further satisfies \eqref{Equation:FirstEquationInDefinition:HolomorphicToFirstOrder}, so that \eqref{Equation:FirstEquationIn:Remark:FirstRemarkTo:Lemma:JHolomorphicityToFirstOrderIffBothHAndVHolomorphicityToFirstOrderHold} is equivalent to
\begin{equation}\label{Equation:SecondEquationIn:Remark:FirstRemarkTo:Lemma:JHolomorphicityToFirstOrderIffBothHAndVHolomorphicityToFirstOrderHold}
\ddtatzero\big((\d\psi_tJX-\J\d\psi_tX)_j\big)=0,\quad\forall\,j\text{.}
\end{equation}
Now, since equation \eqref{Equation:FirstEquationInDefinition:HolomorphicToFirstOrder} does not depend on the chosen connection, we can deduce that holomorphicity with respect to $\widetilde{\nabla}$ reduces to the same condition \eqref{Equation:SecondEquationIn:Remark:FirstRemarkTo:Lemma:JHolomorphicityToFirstOrderIffBothHAndVHolomorphicityToFirstOrderHold}. Thus, $\psi$ being holomorphic to first order does not depend on the chosen connection. For $\H$ (respectively, $\V$) holomorphicity to first order we use similar arguments, replacing $\{Y_j\}$ for a horizontal (respectively, vertical) frame.
\end{remark}

\begin{remark}
If one is looking for maps to $\Sigma^+N$ which are holomorphic to first order, the first thing we need is a metric on the twistor space and the Levi-Civita connection associated with it. We can define such a metric $h$ in a natural way: let $(x,J)\in\Sigma^+N$ and consider the tangent space at this point, $T_{(x,J)}\Sigma^+N=\H\oplus\V$. We know that we have the identifications $\H\simeq T_xN$ and $\V\simeq \m_J(T_xN)$. To get a metric on $\H$, transport the metric from that on $T_xN$; \textit{i.e.},\index{Metric!on~$\Sigma^+M$}\index{Twistor~space!metric}\index{$\Sigma^+M$!metric}
\begin{equation}\label{Equation:FirstEquationOnTheDefinitionOfTheMetricOnSigmaPLusM}
h(X,Y)=<\d\pi_{(x,J)}X,\d\pi_{(x,J)}Y>,\quad\forall\,X,Y\in\H,
\end{equation}
where $<,>$ denotes the metric on $N$ at $x=\pi(x,J)$. For the vertical space $\V\simeq \m_J(T_xN)\subseteq\L(T_xN,T_xN)$, we can consider the restriction of the metric on the space $\L(T_xN,T_xN)$, the latter being defined, as usual, by \begin{equation}\label{Equation:SecondEquationOnTheDefinitionOfTheMetricOnSigmaPLusM}
h_{\L}(V,W)=\sum_{ij}<Ve_i,e_j><We_i,e_j>,\,V,W\in\L(T_xN,T_xN),\,\{e_i\}\text{~orthonormal.}
\end{equation}
Hence,
\begin{equation}\label{Equation:ThirdEquationOnTheDefinitionOfTheMetricOnSigmaPLusM}
h(V,W)=h_{\L}(V,W),\quad\forall\,V,W\in\V\text{.}
\end{equation}
Finally, we declare $\H$ and $\V$ to be orthogonal under the metric $h$; \textit{i.e.},
\begin{equation}\label{Equation:FourthEquationOnTheDefinitionOfTheMetricOnSigmaPLusM}
h(X,V)=0,\quad\forall\,X\in\H,\,V\in\V\text{.}
\end{equation}
For this metric, the decomposition $\H\oplus\V$ is orthogonal and $\J^a$-stable and the projection map $\pi:\Sigma^+N\to
N$ is a Riemannian submersion. It is also easy to verify that $(\Sigma^+N,h,\J^a)$ ($a=1,2$) are almost Hermitian manifolds (\textit{i.e.}, $h(\J^a X,\J^a Y)=h(X,Y)$ for all $X,\,Y\in{T}\Sigma^+N$). Thus, considering the Levi-Civita
connection associated with this metric, it makes sense to speak of maps being $\J^1$ (or $\J^2$) holomorphic to first order, as well as of maps $\H$ (or $\V$) holomorphic to first order\footnote{Notice that, since $\J^1$ and $\J^2$ coincide on $\H$ we do not have to specify whether we are considering $\J^1$ or $\J^2$ $\H$-holomorphicity; the same does not hold for the vertical part, case where we do have to specify whether $\V$-holomorphicity is with respect to with $\J^1$ or $\J^2$.}.
\end{remark}


\subsection{The $\H$-holomorphic case}


Recall from p.\ \pageref{Page:CompatibleTwistorLift} that a twistor lift $\psi$ of the map $\varphi$ is called \emph{strictly compatible} with $\varphi$ if $\varphi$ is holomorphic with respect to $J_\psi$. If $J_\psi$ preserves $\d\varphi(TM)$ but does not necessarily render $\varphi$ holomorphic, the twistor lift is called \emph{compatible} with $\varphi$. In the non-parametric case, given a conformal map $\varphi:M^2\to{N}^{2n}$, we can always find lifts $\psi:M^2\to\Sigma^+N$ such that $\varphi$ is holomorphic with respect to $J_\psi$. In other words, (locally defined) strictly compatible lifts always exist. In general, this lift may not be $\J^1$ or $\J^2$-holomorphic but it is $\H$-holomorphic. Let $\varphi_t$ be a
variation conformal to first order of the map $\varphi$. Then, if a lift $\psi_t$ to the twistor space that makes $\varphi_t$ holomorphic for all small $t$ exists, $\varphi_t$ is necessarily conformal for all small $t$, which may not be the case. So, even if conformality is preserved to first order, there might be no strictly compatible twistor lift for all $t$; hence, we should relax the condition on conformality. We shall say that a twistor lift $\psi$ of a conformal to first order map $\varphi$ is \emph{compatible to first order} (with $\varphi$) if:
\begin{equation}\label{Equation:TwistorLiftCompatibleWithTheMapInTheVariationalCase}\index{Compatible~twistor~lift!to~first~order}\index{To~first~order!compatible~twistor~lift}
\psi_0\text{~is~strictly~compatible~with~}\varphi_0\text{~and~}\psi_t\text{~is~compatible~with~}\varphi_t,\quad\forall\,t\text{.}
\end{equation}

We start with a few lemmas which will be important in the sequel:

\begin{lemma}\label{Lemma:LiftsOfConformalToFirstOrderDifferFromHolomorphyByJustAnO(t)Vector}
Let $\varphi:I\times M^2\to N^{2n}$ be a map conformal to first order, let $z_0\in{M}^2$, and suppose that $\partial_{z_0}\varphi_0\neq0$. Let $\psi$ be a twistor lift around $z_0$ which compatible to first order with $\varphi$. Then for all $X\in\Gamma(TM)$ there is a function $a_t^X$ and a vector $v_t^X\in\varphi_t^{-1}(TN)$ with $a^X_0=1$, $v_0^X=0$ and $\left.\frac{\partial{a_t^X}}{\partial{t}}\right|_{t=0}=0$, $\nabla_{\ddtatzero}v_t^X=0$ such that
\begin{equation}\label{Equation:FirstEquationInLemma:LiftsOfConformalToFirstOrderDifferFromHolomorphyByJustAnO(t)Vector}
J_{\psi_t}\d\varphi_t X=a^X_t \d\varphi_t JX+v_t^X.
\end{equation}
In particular, $\varphi$ is \emph{holomorphic to first order} with respect to $J_{\psi}$ in the sense that
\begin{equation}\label{Equation:SecondEquationIn:Lemma:LiftsOfConformalToFirstOrderDifferFromHolomorphyByJustAnO(t)Vector}\index{Holomorphic!to~first~order}
\varphi_0\text{~is~holomorphic~with~respect~to~$J_{\psi_0}$~and~}\nabla_{\ddtatzero}\d\varphi_tJX=\nabla_{\ddtatzero}J_{\psi_t}\d\varphi_tX\text{.}
\end{equation}
\end{lemma}
\begin{proof}
Since $\psi$ is compatible to first order, $J_{\psi_t}$ preserves $\d\varphi_t(TM)$ for all $t$. Hence, there are $a_t^X$ and $b_t^X$ such that
\begin{equation}\label{Equation:FirstEquationInTheProofOf:Lemma:LiftsOfConformalToFirstOrderDifferFromHolomorphyByJustAnO(t)Vector}
J_{\psi_t}\d\varphi_t X=a_t^X \d\varphi_t JX + b_t^X \d\varphi_t X\text{.}
\end{equation}
Take $v_t^X=b_t^X\d\varphi_t X$. Since, at $t=0$, $J_{\psi_0}\d\varphi_0 X=\d\varphi_0 JX$ we deduce $v_0^X=0$ and $a_0^X=1$. Now, since $\d\varphi_t X$ and $J_{\psi_t}\d\varphi_tX$ form an orthogonal basis for $\d\varphi_t(TM)$ and
$<\d\varphi_tJX,J_{\psi_t}\d\varphi_tX>,\|J_{\psi_t}\d\varphi_tX\|^{2}$ are not zero for all small $t$, we have\addtolength{\arraycolsep}{-1.0mm}
\begin{equation*}\begin{array}{lcll}
~&\d\varphi_tJX&=&\frac{<\d\varphi_tJX,\d\varphi_tX>}{\|\d\varphi_tX\|^{2}}\d\varphi_tX+\frac{<\d\varphi_tJX,J_{\psi_t}\d\varphi_tX>}{\|J_{\psi_t}\d\varphi_tX\|^{2}}J_{\psi_t}\d\varphi_tX\\
\impl&J_{\psi_t}\d\varphi_tX&=&\frac{\d\varphi_tJX-<\d\varphi_tJX,\d\varphi_tX>\|\d\varphi_tX\|^{-2}\d\varphi_tX}{<\d\varphi_tJX,J_{\psi_t}\d\varphi_tX>\|\d\varphi_tX\|^{-2}}\\
\impl&<J_{\psi_t}\d\varphi_tX,\d\varphi_tJX>&=&\frac{\|\d\varphi_tJX\|^2-<\d\varphi_tJX,\d\varphi_tX>^2\|\d\varphi_tX\|^{-2}}{<\d\varphi_tJX,J_{\psi_t}\d\varphi_tX>\|\d\varphi_tX\|^{-2}}\\
\impl&<J_{\psi_t}\d\varphi_tX,\d\varphi_tJX>^2&=&\|\d\varphi_tJX\|^2\|\d\varphi_tX\|^{2}-<\d\varphi_tJX,\d\varphi_tX>^2\text{.}
\end{array}
\end{equation*}\addtolength{\arraycolsep}{1.0mm}\noindent
Differentiating with respect to $t$ at the point $t=0$ the above identity yields
\begin{equation}\label{Equation:SecondEquationInTheProofOf:Lemma:LiftsOfConformalToFirstOrderDifferFromHolomorphyByJustAnO(t)Vector}
\ddtatzero<J_{\psi_t}\d\varphi_t X,\d\varphi_t JX>=\ddtatzero<\d\varphi_t X,\d\varphi_t X>\text{.}
\end{equation}
On the other hand, computing the inner product of \eqref{Equation:FirstEquationInTheProofOf:Lemma:LiftsOfConformalToFirstOrderDifferFromHolomorphyByJustAnO(t)Vector} with $J_{\psi_t}\d\varphi_t X$ and using the fact that $<\d\varphi_t X,J_{\psi_t}\d\varphi_t X>$ $=0$ for all $t$, we get
$<\d\varphi_t X,\d\varphi_t X>=a_t^X <\d\varphi_tJX,J_{\psi_t}\d\varphi_t X>$. So,\addtolength{\arraycolsep}{-1.0mm}
\[\begin{array}{lll}\ddtatzero<\d\varphi_tX,\d\varphi_tX>&=&\left.\frac{\partial{a}_t^X}{\partial{t}}\right|_{t=0}.<\d\varphi_0JX,J_{\psi_0}\d\varphi_0X>\\~&~&+a^X_0.\ddtatzero<\d\varphi_tJX,J_{\psi_t}\d\varphi_tX>\end{array}\]\addtolength{\arraycolsep}{1.0mm}\noindent
and we deduce $\ddtatzero{a}_t^X=0$, as $a_0^X=1$ and \eqref{Equation:SecondEquationInTheProofOf:Lemma:LiftsOfConformalToFirstOrderDifferFromHolomorphyByJustAnO(t)Vector} hold. Using \eqref{Equation:FirstEquationInTheProofOf:Lemma:LiftsOfConformalToFirstOrderDifferFromHolomorphyByJustAnO(t)Vector} again we can now write
\[\nabla_{\ddtatzero}v_t^X=\nabla_{\ddtatzero}J_{\psi_t}\d\varphi_tX-\left.\frac{\partial{a}_t^X}{\partial{t}}\right|_{t=0}.\d\varphi_0JX-1.\nabla_{\ddtatzero}\d\varphi_tJX\]
so that\addtolength{\arraycolsep}{-1.0mm}
\[\begin{array}{lll}<\nabla_{\ddtatzero}v_t^X,\d\varphi_0X>&=&\ddtatzero<J_{\psi_t}\d\varphi_tX,\d\varphi_tX>-<J_{\psi_0}\d\varphi_0X,\nabla_{\ddtatzero}\d\varphi_tX>\\
~&~&-\ddtatzero<\d\varphi_tJX,\d\varphi_tX>+<\d\varphi_0JX,\nabla_{\ddtatzero}\d\varphi_tX>,\end{array}\]\addtolength{\arraycolsep}{1.0mm}\noindent
which vanishes since $<J_{\psi_t}\d\varphi_tX,\d\varphi_t X>=0$ for all $t$, the second and last terms cancelling as $\varphi_0$ is $J_{\psi_0}$-holomorphic and $\varphi$ is conformal to first order. Analogously,\addtolength{\arraycolsep}{-1.0mm}
\[\begin{array}{lll}<\nabla_{\ddtatzero}v_t^X,J_{\psi_0}\d\varphi_0X>&=&\frac{1}{2}\ddtatzero\|J_{\psi_t}\d\varphi_tX\|^2-<\nabla_{\ddtatzero}\d\varphi_tJX,J_{\psi_0}\d\varphi_0X>\\
~&=&\frac{1}{2}\ddtatzero\|\d\varphi_tX\|^2-\frac{1}{2}\ddtatzero<\d\varphi_tJX,\d\varphi_tJX>=0\end{array}\]\addtolength{\arraycolsep}{1.0mm}\noindent
so that $<\nabla_{\ddtatzero} v_t^X,\d\varphi_0(TM)>=0$. For the orthogonal part, taking $r_t\in(\d\varphi_tTM)^\perp$,
\[<\nabla_{\ddtatzero}v_t^X,r_0>=\ddtatzero<v_t^X,r_t>-<v_0^X,\nabla_{\ddtatzero}r_t>=0\text{,}\]
showing that $\nabla_{\ddtatzero}v^X_t=0$ and concluding the proof.
\end{proof}

\begin{lemma}\label{Lemma:HHolomorphicityToFirstOrderImpliesHolomorphicityToFirstOrderOfTheProjectedMap}
Let $\psi:I\times M^2\to \Sigma^+N$ be $\H$-holomorphic to first order. Then, $\varphi=\pi\circ\psi$ is
$(J^M,J_{\psi})$-holomorphic to first order (in the sense of \eqref{Equation:SecondEquationIn:Lemma:LiftsOfConformalToFirstOrderDifferFromHolomorphyByJustAnO(t)Vector}).
\end{lemma}
\begin{proof}
In fact, since $\psi$ is $\H$-holomorphic to first order, $\nabla_{\ddtatzero}\{(\d\psi_tJX)^{\H}-\J^\H(\d\psi_tX)\}^{\H}=0$. Therefore, for all
$Y\in\H$,\addtolength{\arraycolsep}{-1.0mm}
\[\begin{array}{ll}~&0=\,<\nabla_{\ddtatzero}\{(\d\psi_t JX)^{\H}-\J^\H(\d\psi_tX)^{\H}\},Y>\\
=&\ddtatzero\hspace{-2mm}<(\d\psi_tJX)^{\H}-\J^\H(\d\psi_tX)^{\H},Y>\hspace{-1mm}-\hspace{-1mm}<(\d\psi_0JX)^{\H}\hspace{-1mm}-\hspace{-1mm}\J^\H(\d\psi_0X)^{\H},\nabla_{\ddtatzero}\hspace{-2mm}Y>\\
=&\ddtatzero<(\d\psi_tJX)^{\H}-\J^\H(\d\psi_tX)^{\H},Y>\text{,~as~$\psi_0$~is~$\H$-holomorphic.}
\end{array}\]\addtolength{\arraycolsep}{1.0mm}\noindent
Since $\pi$ is a Riemannian submersion and $(\J^a,J_\psi)$-holomorphic (see \eqref{Equation:HolomorphicityOfTheProjectionMap}), the above equation can be written as
\[0=\ddtatzero<\d\pi(\d\psi_tJX)-J_{\psi_t}\d\pi(\d\psi_t X),\d\pi(Y)>\text{.}\]
Hence, for all $\tilde{Y}\in{T}N$,\addtolength{\arraycolsep}{-1.0mm}
\[\begin{array}{lll}
0&=&\ddtatzero<\d\varphi_t JX-J_{\psi_t}\d\varphi_t X,\tilde{Y}>\\
~&=&<\nabla_{\ddtatzero}\{\d\varphi_tJX-J_{\psi_t}\d\varphi_tX\},\tilde{Y}>-<\d\varphi_0JX-J_{\psi_0}\d\varphi_0X,\nabla_{\ddtatzero}\tilde{Y}>\\
~&=&<\nabla_{\ddtatzero}\{\d\varphi_tJX-J_{\psi_t}\d\varphi_tX\},\tilde{Y}>\text{~,~since~$\varphi_0$~is~$J_{\psi_0}$~holomorphic,}\end{array}\]\addtolength{\arraycolsep}{1.0mm}\noindent
showing that $\nabla_{\ddtatzero}\d\varphi_tJX=\nabla_{\ddtatzero}J_{\psi_t}\d\varphi{X}$ and concluding our proof (compare with the non-parametric proof of Theorem
\ref{Theorem:ProjectionsOfHHolomorphicMaps}).
\end{proof}

\begin{proposition}[\textup{$\H$-holomorphicity~and~conformality~to~first~order}]\label{Proposition:ProjectionsAndLiftsOfConformalToFirstOrder}\index{$\H$-holomorphic!to~first~order!projection}
Let $I\subseteq\rn$ be an interval around $0$ and let $\psi:I\times M^2\to\Sigma^+ N$ be $\H$-holomorphic to first order. Then, the projected map $\varphi=\pi\circ\psi:I\times M^2\to N^{2n}$ is conformal to first order. Conversely, let $\varphi:I\times M^2\to N$ be a map conformal to first order and assume that $z_0\in{M}^2$ is such that $\partial_z\varphi_0(z_0)\neq 0$. Then, reducing $I$ if necessary, there is an open set $\V$ containing $z_0$ and a map $\psi:I\times\V\to\Sigma^+ N$ $\H$-holomorphic to first order which is compatible to first order with $\varphi$.
\end{proposition}
\begin{proof}
Take $\psi:I\times M^2\to\Sigma^+ N$ $\H$-holomorphic to first order. Since $\varphi_0$ is holomorphic with respect to
$J_{\psi_0}$ (see proof of Theorem \ref{Theorem:ProjectionsOfHHolomorphicMaps}), we know that $\varphi_0$ is conformal (Proposition \ref{Proposition:LocalHHolomorphicLifts}). As for the first order variation, using the preceding Lemma
\ref{Lemma:HHolomorphicityToFirstOrderImpliesHolomorphicityToFirstOrderOfTheProjectedMap},\addtolength{\arraycolsep}{-1.0mm}
\[\begin{array}{lll}\ddtatzero\|\d\varphi_tJX\|^2&=&2<\nabla_{\ddtatzero}\d\varphi_t JX,\d\varphi_0 JX>\\
~&=&2<\nabla_{\ddtatzero}J_{\psi_t}\d\varphi_t X,J_{\psi_0}\d\varphi_{0}X>\\
~&=&\ddtatzero<J_{\psi_t}\d\varphi_tX,J_{\psi_t}\d\varphi_tX>=\ddtatzero<\d\varphi_tX,\d\varphi_tX>\text{.}\end{array}\]\addtolength{\arraycolsep}{1.0mm}\noindent
Using similar arguments we can show that
\[\ddtatzero<\d\varphi_t JX,\d\varphi_t X>=0\text{,}\]
concluding the first part of the proof.

For the converse, we shall make use of Lemma \ref{Lemma:LiftsOfConformalToFirstOrderDifferFromHolomorphyByJustAnO(t)Vector}.
Let $\psi$ be any twistor lift of $\varphi$ compatible to first order. Since we are assuming that $\partial_{z_0}\varphi\neq 0$, we can use Lemma \ref{Lemma:LiftsOfConformalToFirstOrderDifferFromHolomorphyByJustAnO(t)Vector} to deduce that there is a function $a_t^X$ and a vector field $v_t^X$ with
\[J_{\psi_t}\d\varphi_t X=a_t^X \d\varphi_t JX +v_t^X\]
where $a_0^X=1$, $v_0^X=0$, $\left.\frac{\partial{a_t^X}}{\partial{t}}\right|_{t=0}=0$ and $\nabla_{\ddtatzero}v_t^X=0$. Now, $\psi$ is $\H$-holomorphic to first order if and only if $\psi_0$ is $\H$-holomorphic (which is true from Proposition \ref{Proposition:LocalHHolomorphicLifts}) and equation \eqref{Equation:SecondEquationInDefinition:HolomorphicToFirstOrder} holds. Using the same argument as in Lemma \ref{Lemma:JHolomorphicityToFirstOrderIffBothHAndVHolomorphicityToFirstOrderHold}, \eqref{Equation:SecondEquationInDefinition:HolomorphicToFirstOrder} is equivalent to
\[0=\ddtatzero<(\d\psi_t JX)^{\H}-\J^\H(\d\psi_tX)^{\H},Y^\H>,\quad\forall\,Y^\H\in\H\text{,}\]
equivalently,\addtolength{\arraycolsep}{-1.0mm}
\[\begin{array}{ll}~&0=\ddtatzero<\d\varphi_tJX-J_{\psi_t}\d\varphi_tX,\tilde{Y}>\\
=&<\nabla_{\ddtatzero}\hspace{-2mm}\{\d\varphi_tJX\hspace{-1mm}-\hspace{-1mm}a_t^X\d\varphi_tJX\hspace{-1mm}-\hspace{-1mm}v_t^X\},\tilde{Y}>\hspace{-1mm}-\hspace{-1mm}<\underbrace{\d\varphi_0JX\hspace{-1mm}-\hspace{-1mm}a_0^X\d\varphi_0JX\hspace{-1mm}-\hspace{-1mm}v_0^X}_{\text{\tiny{$=0$}}},\nabla_{\ddtatzero}\hspace{-4mm}\tilde{Y}>\end{array}\]\addtolength{\arraycolsep}{1.0mm}\noindent
which is equivalent to
\[<\nabla_{\ddtatzero}\d\varphi_tJX-a_0^X\nabla_{\ddtatzero}\d\varphi_tJX-\left.\frac{\partial{a_t^X}}{\partial{t}}\right|_{t=0}\d\varphi_0JX-\nabla_{\ddtatzero}v_t^X,\tilde{Y}>=0\text{.}\]
Since this is clearly true from the given conditions on $a_t^X$ and $v_t^X$, we have established \eqref{Equation:SecondEquationInDefinition:HolomorphicToFirstOrder} and concluded the proof.
\end{proof}

\begin{lemma}[\textup{Condition~for~$\V$-holomorphicity~to~first~order}]\label{Lemma:ConditionForVHolomorphicity}
Let $\psi:I\times M^2\to\Sigma^+ N$ be a map $\J^a$ ($a=1$ \emph{or} $a=2$) $\V$-holomorphic to first order. Then
\begin{equation}\label{Equation:FirstEquationIn:Lemma:ConditionForVHolomorphicity}
\nabla^{\psi^{-1}\L(TN,TN)}_{\ddtatzero}\{\nabla_{JX}^{\psi^{-1}\L(TN,TN)}J_{\psi_t}+(-1)^{a}{J}_{\psi_t}\nabla_{X}^{\psi^{-1}\L(TN,TN)}J_{\psi_t}\}=0\text{.}
\end{equation}
Conversely, if $\psi_0$ is $\J^a$ $\V$-holomorphic and \eqref{Equation:FirstEquationIn:Lemma:ConditionForVHolomorphicity}
is satisfied then $\psi$ is $\J^a$ $\V$-holomorphic to first order.
\end{lemma}
\begin{proof}
Recall from \eqref{Equation:FirstEquationInTheProofOf:Lemma:FundamentalLemmaToProve3.5} and \eqref{Equation:DefinitionOfTheAlmostComplexStructuresJ1AndJ2OnSigmaPlusM} that if $\psi:M\to\Sigma^+ N$ is any smooth map then, for $a=1,2$,
\[(\d\psi X)^\V=\nabla^{\psi^{-1}\L(TN,TN)}_X J_\psi\text{ and }\J^{a}(\d\psi X)^\V=(-1)^{a+1} J_\psi\nabla^{\psi^{-1}\L(TN,TN}_X J_\psi\]
Thus, we can rephrase equation \eqref{Equation:FirstEquationIn:Lemma:ConditionForVHolomorphicity}
as
\begin{equation}\label{Equation:FirstEquationInTheProofOf:Lemma:ConditionForVHolomorphicity}
\nabla^\V_{\ddtatzero}\{\d\psi_t(JX)-\J^{a}\d\psi_t X\}^\V=0\text{.}
\end{equation}
Hence, if $\psi$ is $\J^{a}$ $\V$-holomorphic to first order, then
\[\nabla_{\ddtatzero}\{\d\psi_t(JX)-\J^{a}\d\psi_tX\}^\V=0\]
which implies \eqref{Equation:FirstEquationInTheProofOf:Lemma:ConditionForVHolomorphicity}. Conversely, if
\eqref{Equation:FirstEquationInTheProofOf:Lemma:ConditionForVHolomorphicity} holds and $\psi_0$ is $\J^{a}$ $\V$-holomorphic, we have\addtolength{\arraycolsep}{-1.0mm}
\[\begin{array}{ll}~&<\nabla_{\ddtatzero}\{\d\psi_t (JX)-\J^{a}\d\psi_tX\}^\V,Y>\\
=&\ddtatzero<\{\d\psi_t(JX)-\J^{a}\d\psi_tX\}^\V,Y>-<\{\d\psi_0(JX)-\J^{a}\d\psi_0X\}^\V,\nabla_{\ddtatzero}\hspace{-2mm}Y>\\
=&\text{~(as~$\psi_0$~is~$\J^{a}$~$\V$-holomorphic)~}\ddtatzero<\{\d\psi_t(JX)-\J^{a}\d\psi_tX\}^\V,Y^\V>\\
=&<\nabla_{\ddtatzero}\hspace{-2mm}\{\d\psi_t(JX)-\J^{a}\d\psi_tX\}^\V\hspace{-1mm},Y^\V>\hspace{-1mm}+\hspace{-1mm}<\{\d\psi_0(JX)-\J^{a}\d\psi_0X\}^\V\hspace{-1mm},\nabla_{\ddtatzero}\hspace{-2mm}Y^\V>\\
=&\text{~(as~$\psi_0$~is~$\J^{a}$~$\V$-holomorphic)~}<\nabla^\V_{\ddtatzero}\{\d\psi_t(JX)-\J^{a}\d\psi_tX\}^\V,Y>=0\text{,}\end{array}\]\addtolength{\arraycolsep}{1.0mm}\noindent
showing that $\psi$ is $\J^{a}$ $\V$-holomorphic to first order and concluding the proof.
\end{proof}

\begin{corollary}\label{Corollary:HorizontalAndVerticalHolomorphicityToFirstOrderImplyHolomorphicity}
Let $\psi:I\times M^2\to\Sigma^+ N$ be a map such that:

\textup{(i)} $\psi_0$ is holomorphic.

\textup{(ii)} $\psi$ is $\H$-holomorphic to first order.

\textup{(iii)} $\psi$ satisfies equation \eqref{Equation:FirstEquationIn:Lemma:ConditionForVHolomorphicity} (for $a=1$ or $a=2$).

Then, $\psi$ is $\J^1$ (respectively, $\J^2$) holomorphic to first order.
\end{corollary}
\begin{proof}
Immediate from the previous lemma and Lemma
\ref{Lemma:JHolomorphicityToFirstOrderIffBothHAndVHolomorphicityToFirstOrderHold}
\end{proof}


\subsection{The $\J^1$-holomorphic case}\label{Subsection:J1CaseToFirstOrder}


Next, we give a useful characterization for maps to be $\J^a$-holomorphic to first order ($a=1$ or $2$):

\begin{lemma}\label{Lemma:AnotherCharacterizationOfJ1AndJ2HolomorphicMapsToFirstOrder}\index{$\J^1$-holomorphic!to~first~order!characterization~of}
Let $\psi:I\times M^2\to\Sigma^+ N$ be a smooth map. Then, $\psi$ is $\J^a$-holomorphic to first order ($a=1$ or $2$)  if and only if
\begin{equation}\label{Equation:FirstEquationIn:Lemma:AnotherCharacterizationOfJ1AndJ2HolomorphicMapsToFirstOrder}
\varphi\text{ is $(J^M,J_{\psi})$-holomorphic to first order and}
\end{equation}
\begin{equation}\label{Equation:SecondEquationIn:Lemma:AnotherCharacterizationOfJ1AndJ2HolomorphicMapsToFirstOrder}
\left.\begin{array}{l}\forall\,Y_t^{10}\in\varphi_t^{-1}(T^{10}_{J_{\psi_t}}N)\quad\exists\,Z_t^{10}\in\varphi_t^{-1}(T^{10}_{J_{\psi_t}}N)\quad\text{such~that}\\
\quad\nabla_{\ddtatzero}\nabla_{\partial_z}Y_t^{10}=\nabla_{\ddtatzero}Z_t^{10}\quad{and}\quad\nabla_{\partial_z}Y_0^{10}=Z_0^{10}
\end{array}\quad\right\}\text{~(a=1)}
\end{equation}
or
\begin{equation}\label{Equation:ThirdEquationIn:Lemma:AnotherCharacterizationOfJ1AndJ2HolomorphicMapsToFirstOrder}
\left.\begin{array}{l}\forall\,Y_t^{10}\in\varphi_t^{-1}(T^{10}_{J_{\psi_t}}N)\quad\exists\,Z_t^{10}\in\varphi_t^{-1}(T^{10}_{J_{\psi_t}}N)\quad\text{such~that}\\
\quad\nabla_{\ddtatzero}\nabla_{\partial_{\bar{z}}}Y_t^{10}=\nabla_{\ddtatzero}Z_t^{10}\quad{and}\quad\nabla_{\partial_{\bar{z}}}Y_0^{10}=Z_0^{10}
\end{array}\quad\right\}\text{~(a=2)}
\end{equation}
(compare with \eqref{Equation:FirstEquationIn:Corollary:AnotherVersionOfTheFundamentalLemmaUsingT10AndT01Spaces}).
\end{lemma}
\begin{proof}
We shall do the proof only for the $\J^1$ case, the $\J^2$ case being similar.

Assume that $\psi$ is $\J^1$-holomorphic to first order. Then, using Lemmas
\ref{Lemma:JHolomorphicityToFirstOrderIffBothHAndVHolomorphicityToFirstOrderHold} and \ref{Lemma:HHolomorphicityToFirstOrderImpliesHolomorphicityToFirstOrderOfTheProjectedMap}, $\varphi$ is $(J^M,J_{\psi})$-holomorphic to first order. On the other hand, $\psi$ satisfies equation
\eqref{Equation:FirstEquationIn:Lemma:ConditionForVHolomorphicity},
\[\nabla^{\L(TN,TN)}_{\ddtatzero}\{\nabla_{JX}J_{\psi_t}-J_{\psi_t}\nabla^{\L(TN,TN)}_{X}J_{\psi_t}\}=0\]
which implies that
\begin{equation}\label{Equation:FirstEquationInTheProofOf:Lemma:AnotherCharacterizationOfJ1AndJ2HolomorphicMapsToFirstOrder}
\nabla_{\ddtatzero}\{\nabla_{JX}(J_{\psi_t}Y_t)-\nabla_XY_t\}=\nabla_{\ddtatzero}\{J_{\psi_t}\big(\nabla_{JX}Y_t+\nabla_X(J_{\psi_t}Y_t)\big)\}\text{.}
\end{equation}
Take $Y^{10}_t$ in $T^{10}_{J_{\psi_t}}N$; then, $Y^{10}_t=\frac{1}{2}(Y_t-iJ_{\psi_t}Y_t)$ for some $Y_t$ and we can
deduce\addtolength{\arraycolsep}{-1.0mm}
\[\begin{array}{ll}~&\nabla_{\ddtatzero}\nabla_{\partial_z}Y^{10}_t=\frac{1}{4}\nabla_{\ddtatzero}\{\nabla_XY_t-\nabla_{JX}(J_{\psi_t}Y_t)-i\big(\nabla_X(J_{\psi_t}Y_t)+\nabla_{JX}Y_t\big)\}\\
=&\text{(using~\eqref{Equation:FirstEquationInTheProofOf:Lemma:AnotherCharacterizationOfJ1AndJ2HolomorphicMapsToFirstOrder})~}\frac{1}{4}\nabla_{\ddtatzero}\{-J_{\psi_t}\big(\nabla_{JX}Y_t+\nabla_X(J_{\psi_t}Y_t)\big)-i\big(\nabla_X(J_{\psi_t}Y_t)+\nabla_{JX}Y_t\big)\}\text{.}\end{array}\]\addtolength{\arraycolsep}{1.0mm}\noindent
Let $2iZ_t=\nabla_{JX}Y_t+\nabla_X(J_{\psi_t}Y_t)$ so that\addtolength{\arraycolsep}{-1.0mm}
\[\begin{array}{lll}\nabla_{\ddtatzero}\nabla_{\partial_z}Y^{10}_t&=&\frac{1}{4}\nabla_{\ddtatzero}\{-2iJ_{\psi_t}Z_t+2Z_t\}\\
~&=&\nabla_{\ddtatzero}\frac{1}{2}\{Z_t-iJ_{\psi_t}Z_t\}=\nabla_{\ddtatzero}Z^{10}_t\end{array}\]\addtolength{\arraycolsep}{1.0mm}\noindent
Moreover, since $\psi_0$ is holomorphic, $\nabla_XY_0-\nabla_{JX}(J_{\psi_0}Y_0)=-J_{\psi_0}\big(\nabla_{JX}Y_0+\nabla_X(J_{\psi_0}Y_0)\big)$ and\addtolength{\arraycolsep}{-1.0mm}
\[\begin{array}{lll}\nabla_{\partial_z}Y^{10}_0&=&\frac{1}{4}\{\nabla_XY_0-\nabla_{JX}(J_{\psi_0}Y_0)-i\big(\nabla_X(J_{\psi_0}Y_0)+\nabla_{JX}Y_0\big)\}\\
~&=&\frac{1}{4}\{-2iJ_{\psi_0}Z_0+2Z_0\}=Z^{10}_0\text{,}\end{array}\]\addtolength{\arraycolsep}{1.0mm}\noindent
finishing the ``only if" part of our proof.

For the converse: suppose now that
\eqref{Equation:FirstEquationIn:Lemma:AnotherCharacterizationOfJ1AndJ2HolomorphicMapsToFirstOrder} and \eqref{Equation:SecondEquationIn:Lemma:AnotherCharacterizationOfJ1AndJ2HolomorphicMapsToFirstOrder} hold. Then, $\psi_0$ is $\J^1$-holomorphic, using Corollary \ref{Corollary:AnotherVersionOfTheFundamentalLemmaUsingT10AndT01Spaces}. Take $Y^{10}_t\in{T}^{10}_{J_{\psi_t}}N$. As \eqref{Equation:SecondEquationIn:Lemma:AnotherCharacterizationOfJ1AndJ2HolomorphicMapsToFirstOrder} holds, there is $Z^{10}_t$ with $\nabla_{\partial_z}Y^{10}_0=Z^{10}_0$ and $\nabla_{\ddtatzero}\{\nabla_{\partial_z}Y^{10}_t-Z^{10}_t\}=0$, which implies that
\[\nabla_{\ddtatzero}\frac{1}{2}\{\nabla_XY_t-\nabla_{JX}(J_{\psi_t}Y_t)-i\big(\nabla_X(J_{\psi_t}Y_t)+\nabla_{JX}Y_t\big)\}-Z_t+iJ_{\psi_t}Z_t\}=0\]
so that
\[\left\{\begin{array}{l}\nabla_{\ddtatzero}\{\nabla_X Y_t-\nabla_{JX}(J_{\psi_t}Y_t)\}=2\nabla_{\ddtatzero}Z_t\\
\nabla_{\ddtatzero}\{\nabla_{X}(J_{\psi_t}Y_t)+\nabla_{JX}Y_t\}=2\nabla_{\ddtatzero}(J_{\psi_t}Z_t)\text{.}\end{array}\right.\]
In particular, the right-hand side of \eqref{Equation:FirstEquationInTheProofOf:Lemma:AnotherCharacterizationOfJ1AndJ2HolomorphicMapsToFirstOrder} is given by\addtolength{\arraycolsep}{-1.0mm}\noindent
\[\begin{array}{ll} ~&\nabla_{\ddtatzero}\{J_{\psi_t}\big(\nabla_{JX}Y_t+\nabla_X(J_{\psi_t}Y_t)\big)\}=\big(\nabla_{\ddtatzero}J_{\psi_t}\big)\big(\nabla_{JX}Y_0+\\
~&+\nabla_X(J_{\psi_0}Y_0)\big)+J_{\psi_0}\nabla_{\ddtatzero}\{\nabla_{JX}Y_t+\nabla_X(J_{\psi_t}Y_t)\}\\
=&\big(\nabla_{\ddtatzero}\hspace{-0.4mm}J_{\psi_t}\big)\big(\nabla_{JX}Y_0+\nabla_X(J_{\psi_0}Y_0)\big)+2J_{\psi_0}\nabla_{\ddtatzero}\hspace{-0.4mm}(J_{\psi_t}Z_t)\\
=&\text{~(since~$\nabla_{JX}Y_0+\nabla_X(J_{\psi_0}Y_0)=2J_{\psi_0}Z_0$)~}-2\nabla_{\ddtatzero}Z_t-2J_{\psi_0}\nabla_{\ddtatzero}(J_{\psi_t}Z_t)\\
~&+2J_{\psi_0}\nabla_{\ddtatzero}(J_{\psi_t}Z_t)=-2\nabla_{\ddtatzero}Z_t\end{array}\]\addtolength{\arraycolsep}{1.0mm}\noindent
which is the left-hand side of \eqref{Equation:FirstEquationInTheProofOf:Lemma:AnotherCharacterizationOfJ1AndJ2HolomorphicMapsToFirstOrder}. Hence, \eqref{Equation:FirstEquationInTheProofOf:Lemma:AnotherCharacterizationOfJ1AndJ2HolomorphicMapsToFirstOrder} is satisfied. Together with the fact that $\psi_0$ is $\J^1$-holomorphic, we can conclude that equation \eqref{Equation:FirstEquationIn:Lemma:ConditionForVHolomorphicity}
is verified. As for the horizontal part, we have that condition
\[\nabla_{\ddtatzero}(\d\psi_tJX-\J\d\psi_tX)^\H=0\]
is equivalent to
\begin{equation}\label{Equation:SecondEquationInTheProofOf:Lemma:AnotherCharacterizationOfJ1AndJ2HolomorphicMapsToFirstOrder}
\ddtatzero<\d\varphi_tJX-J_{\psi_t}\d\varphi_tX,Y>=0,\quad\forall\,Y\text{,}
\end{equation}
as $\psi_0$ is holomorphic. Now, we know that \eqref{Equation:FirstEquationIn:Lemma:AnotherCharacterizationOfJ1AndJ2HolomorphicMapsToFirstOrder}
holds so that \eqref{Equation:SecondEquationInTheProofOf:Lemma:AnotherCharacterizationOfJ1AndJ2HolomorphicMapsToFirstOrder} is trivially satisfied. Therefore, we are under the conditions of Corollary
\ref{Corollary:HorizontalAndVerticalHolomorphicityToFirstOrderImplyHolomorphicity} and can conclude that our map is $\J^1$-holomorphic to first order, as desired.
\end{proof}

\begin{remark}Recall that, in the non-parametric case, $\J^1$-holomorphicity is characterized by (Corollary \ref{Corollary:AnotherVersionOfTheFundamentalLemmaUsingT10AndT01Spaces})
\[\nabla_{\partial_z}T^{10}_{J_{\psi_0}}N\subseteq{T}^{10}_{J_{\psi_0}}N\text{.}\]
If we try to get an analogue for this condition but with a first order derivative, we immediately get problems, as
$\nabla_{\ddtatzero}T^{10}$ could mean $a.\nabla_{\ddtatzero}Z^{10}+\ddtatzero a.Z^{10}$ (for each vector
$Z^{10}$) or just $\nabla_{\partial_t}Z^{10}$. Hence, the above lemma gives a clean way of stating the analogue of
$\nabla_{\partial_z}T^{10}\subseteq T^{10}$ for the parametric case.
\end{remark}

From the preceding lemma we can also deduce the following:
\begin{lemma}\label{Lemma:FundamentalTrickWithTheCurvatureAndFirstOrderVariations}\index{Curvature~tensor~and~first~order~variations}
Let $\psi:I\times M^2\to\Sigma^+N$ be a map $\J^1$-holomorphic to first order and consider the projected map $\varphi=\pi\circ\psi$. Then for all $r\geq1$ there is $Z^{10}_{t}\in\varphi_t^{-1}(T^{10}_{J_{\psi_t}}N)$ with
\begin{equation}\label{Equation:EquationIn:Lemma:FundamentalTrickWithTheCurvatureAndFirstOrderVariations}
\partial^r_z\varphi_0=Z^{10}_{0}\text{~and~}\nabla_{\ddtatzero}\partial^r_z\varphi_t=\nabla_{\ddtatzero}Z^{10}_{t}\text{.}
\end{equation}
\end{lemma}
\begin{proof}We shall do the proof by induction on $r$. For $r=1$, we have $\partial_z\varphi_t=\d\varphi_tX-i\d\varphi_t JX$ so that using
\eqref{Equation:FirstEquationIn:Lemma:AnotherCharacterizationOfJ1AndJ2HolomorphicMapsToFirstOrder} we have
\[\nabla_{\ddtatzero}\partial_z\varphi_t=\nabla_{\ddtatzero}\{\d\varphi_tX-i\d\varphi_tJX\}=\nabla_{\ddtatzero}\{\d\varphi_tX-iJ_{\psi_t}\d\varphi_tX\}\text{.}\]
Taking $Z^{10}_t=\d\varphi_t X-iJ_{\psi_t}\d\varphi_tX\in\varphi_t^{-1}(T^{10}_{J_{\psi_t}}N)$ we obtain the desired result as $\varphi_0$ $(J^M,J_{\psi_0})$-holomorphic implies
$Z^{10}_0=\d\varphi_0 X-iJ_{\psi_0}\d\varphi_0 X=\d\varphi_0(X-iJ^M
X)=\partial_z\varphi_0$. Assume now that the result is valid for $r=k$; \textit{i.e.}, there is $Z^{10,k}_t$ such that
\[\nabla_{\ddtatzero}\partial^k_z\varphi_t=\nabla_{\ddtatzero}Z^{10,k}_t\text{~and~}\partial^k_z\varphi_0=Z^{10,k}_0\text{.}\]
Let us show that it also holds when $r=k+1$. Since $\partial^{k+1}_z\varphi_t=\partial_z\partial^k_z\varphi_t$,\addtolength{\arraycolsep}{-1.0mm}
\[\begin{array}{lll}\nabla_{\ddtatzero}\partial^{k+1}_z\varphi&=&\nabla_{\ddtatzero}\nabla_{\partial_z}\partial^k_z\varphi_t=R(\left.\frac{\partial\varphi_t}{\partial{t}}\right|_{t=0},\partial_z\varphi_t)\partial^k_z\varphi_t+\nabla_{\partial_z}\nabla_{\ddtatzero}\partial^k_z\varphi_t\\
~&~&+\nabla_{[\left.\frac{\partial\varphi_t}{\partial{t}}\right|_{t=0},\partial_z\varphi_t]}\partial^k_z\varphi_t\text{.}\end{array}\]\addtolength{\arraycolsep}{1.0mm}\noindent
Using the fact that
$[\left.\frac{\partial\varphi_t}{\partial_t}\right|_{t=0},\partial_z\varphi_t]=\d\varphi_t[\partial_z,\ddtatzero]=0$,
that $R$ is tensorial and the induction hypothesis, the latter expression becomes\addtolength{\arraycolsep}{-1.0mm}
\[\begin{array}{ll}~&R(\left.\frac{\partial\varphi_t}{\partial{t}}\right|_{t=0},\partial_z\varphi_0)Z^{10,k}_0+\nabla_{\partial_z}\nabla_{\ddtatzero}Z^{10,k}_t\\
=&R(\left.\frac{\partial\varphi_t}{\partial{t}}\right|_{t=0},\partial_z\varphi_0)Z^{10,k}_0\hspace{-1mm}+\hspace{-1mm}\nabla_{\ddtatzero}\hspace{-2mm}\nabla_{\partial_z}Z^{10,k}_t\hspace{-1mm}+\hspace{-1mm}R(\partial_z\varphi_0,\left.\frac{\partial\varphi_t}{\partial{t}}\right|_{t=0})Z^{10,k}_0=\nabla_{\ddtatzero}\hspace{-2mm}\nabla_{\partial_z}Z^{10,k}_t\text{,}\end{array}\]\addtolength{\arraycolsep}{1.5mm}\noindent
as $R$ is antisymmetric on the first two arguments. Now, since $\psi$ is $\J^1$-holomorphic,
\eqref{Equation:SecondEquationIn:Lemma:AnotherCharacterizationOfJ1AndJ2HolomorphicMapsToFirstOrder}
holds so that there is $Z^{10,k+1}_t$ such that
\[\nabla_{\ddtatzero}\nabla_{\partial_z}Z^{10,k}_t=\nabla_{\ddtatzero}Z^{10,k+1}_t\text{~and~}\nabla_{\partial_z}Z^{10,k}_0=Z^{10,k+1}_0\text{.}\]
But the second condition gives
$\partial_z^{k+1}\varphi_0=\partial_z\partial^k_z\varphi_0=\nabla_{\partial_z}Z^{10,k}_0=Z^{10,k+1}_0$ whereas the first holds precisely that
$\nabla_{\ddtatzero}\partial_z^{k+1}\varphi_t=\nabla_{\ddtatzero}\nabla_{\partial_z}Z^{10,k}_t=\nabla_{\ddtatzero}Z^{10,k+1}_t$,
as we wanted to show.
\end{proof}

\begin{remark}\label{Remark:FundamentalTrickWithTheCurvatureTensor}
Notice the ``fundamental trick" with the curvature tensor: from
$\H$-holomorphicity to first order, equation
\eqref{Equation:SecondEquationIn:Lemma:LiftsOfConformalToFirstOrderDifferFromHolomorphyByJustAnO(t)Vector}
is verified. In particular, we can
write\addtolength{\arraycolsep}{-1.0mm}
\[\begin{array}{ll}
~&\nabla_{\ddtatzero}\nabla_{Y}J_{\psi_t}\d\varphi_tX=R(\ddtatzero,Y)(J_{\psi_0}\d\varphi_0X)+\nabla_Y\nabla_{\ddtatzero}J_{\psi_t}\d\varphi_tX\\
=&R(\ddtatzero,Y)(\d\varphi_0JX)+\nabla_Y\nabla_{\ddtatzero}\d\varphi_tJX\\
=&R(\ddtatzero,Y)(\d\varphi_0JX)+R(Y,\ddtatzero)(\d\varphi_0JX)+\nabla_{\ddtatzero}\nabla_Y\d\varphi_tJX\end{array}\]\addtolength{\arraycolsep}{1.0mm}\noindent
so that from $R$ antisymmetry we conclude
\begin{equation}\label{Equation:EquationIn:Remark:FundamentalTrickWithTheCurvatureTensor:TheFundamentalTrickWithTheCurvatureTensor}
\nabla_Y\nabla_{\ddtatzero}J_{\psi_t}\d\varphi_tX=\nabla_Y\nabla_{\ddtatzero}\d\varphi_tJX,\quad\forall\,X,Y\in{T}M\text{.}
\end{equation}
\end{remark}

\begin{proposition}[\textup{Projections~of~maps~$\J^1$-holomorphic~to~first~order}]\label{Proposition:ProjectionsOfJ1HolomorphicToFirstOrderMapsAndRealIsotropy}\index{$\J^1$-holomorphic!to~first~order!projection}\index{Real~isotropic~map!to~first~order}\index{Curvature~tensor~and~first~order~variations}
Let $\psi:I\times{M}^2\to\Sigma^+N$ be a map $\J^1$-holomorphic to first order, where $M^2$ is any Riemann surface. Then, the projection map $\varphi=\pi\circ\psi$ is real isotropic to first order.
\end{proposition}

Notice that we could replace $\Sigma^+N$ with $\Sigma^-N$, as real isotropy (to first order) does not depend on the fixed orientation on $N$.

\begin{proof}
That $\varphi_0$ is real isotropic follows from the non-parametric case for projections of maps from the twistor space (Theorem \ref{Theorem:GeneralizedTheorem4.4ProjectionsOfJ1HolomorphicMapsInTheHigherDimensionalCase}). Therefore, we are left with proving that
\[\ddtatzero<\partial_z^r\varphi,\partial_z^r\varphi>=0,\quad\forall\,r\geq 1\text{.}\]
Using Lemma \ref{Lemma:FundamentalTrickWithTheCurvatureAndFirstOrderVariations}, for fixed $r\geq1$ choose $Z^{10}_t\in\varphi_t^{-1}(T^{10}_{J_{\psi_t}}N)$ with $\partial^r_z\varphi_0=Z^{10}_0$ and $\nabla_{\ddtatzero}\partial^r_z\varphi_t=\nabla_{\ddtatzero}Z^{10}_t$. Then\addtolength{\arraycolsep}{-1.0mm}
\[\begin{array}{lll}\ddtatzero<\partial_z^r\varphi,\partial_z^r\varphi>&=&2<\nabla_{\ddtatzero}\partial_z^r\varphi_t,\partial_z^r\varphi_0>=2<\nabla_{\ddtatzero}Z^{10}_t,Z^{10}_0>\\
~&=&\ddtatzero<Z^{10}_t,Z^{10}_t>=0\end{array}\]\addtolength{\arraycolsep}{1.0mm}\noindent
as $Z^{10}_t\in\varphi^{-1}(T^{10}_{J_{\psi_t}}N)$ for all $t$.
\end{proof}

We now turn our attention to the existence of lifts $\J^1$-holomorphic to first order for a given map $\varphi:I\times M^2\to N^4$ into an oriented $4$-manifold, with real isotropic to first order. Recall
that in the non-parametric case such lift exists (see Theorem \ref{Theorem:J1HolomorphicLiftsInTheFourDimensionalCase}). Moreover, as we have seen in the second proof of this result, the lift was defined by $J(\d\varphi X)=\d\varphi JX$ and $J(u)=-v$ where $\partial^2_z\varphi=u+iv$. An analogue for the parametric case is as follows:

\begin{theorem}\label{Theorem:J1HolomorphicToFirstOrderLiftsInTheFourDimensionalCase}\index{$\J^1$-holomorphic!to~first~order!lift}
Let $\varphi:I\times M^2\to N^4$ be a map real isotropic to first order. Let $z_0\,\in\,M^2$ and suppose that $\partial_z\varphi_0 (z_0)$ and $\partial^2_z\varphi_0(z_0)$ are linearly independent. Then, reducing
$I$ if necessary, there is an open set $\openU$ around $z_0$ and either a map $\psi^+:I\times\openU\to\Sigma^+N^4$ or a map $\psi^-:I\times\openU\to \Sigma^-N^4$ which is $\J^1$-holomorphic to first order and compatible to first order with
$\varphi$.
\end{theorem}

Before proving Theorem \ref{Theorem:J1HolomorphicToFirstOrderLiftsInTheFourDimensionalCase}, we give a couple of lemmas:

\begin{lemma}\label{Lemma:FirstAuxiliarLemmaToProve:Theorem:J1HolomorphicToFirstOrderLiftsInTheFourDimensionalCase}
Let $\varphi$ be as in the preceding theorem. Consider
\begin{equation}
\label{Equation:FirstEquationIn:Lemma:FirstAuxiliarLemmaToProve:Theorem:J1HolomorphicToFirstOrderLiftsInTheFourDimensionalCase}
u_t^X=\nabla_X\d\varphi_tX-\nabla_{JX}\d\varphi_tJX\quad\text{and}\quad{v}_t^X=-\nabla_X\d\varphi_tJX-\nabla_{JX}\d\varphi_tX\text{.}
\end{equation}
Suppose that the $\J^1$-holomorphic lift of $\varphi_0$ is $\psi_0^+\in\Sigma^+ N$ (respectively, $\psi_0^-\in\Sigma^-N$). Take $J_{\psi_t}$ the unique positive (respectively, negative) almost Hermitian structure on $T_{\varphi_t}N$ compatible with $\varphi_t$.
Then,
\begin{equation}\label{Equation:SecondEquationIn:Lemma:FirstAuxiliarLemmaToProve:Theorem:J1HolomorphicToFirstOrderLiftsInTheFourDimensionalCase}
\nabla_{\ddtatzero}J_{\psi_t}u_t^X=-\nabla_{\ddtatzero}v_t^X\text{.}
\end{equation}
\end{lemma}
\begin{proof}
Since $\varphi$ is real isotropic to first order, $\ddtatzero<\partial^2_z\varphi_t,\partial_z\varphi_t>=0$, equivalently, $\ddtatzero<u_t^X+iv_t^X,\d\varphi_tX-i\d\varphi_t JX>$. Thus, we have
\begin{align}
\notag\ddtatzero<u_t^X,\d\varphi{X}>&=-\ddtatzero<v_t^X,\d\varphi_tJX>\text{~and}\\
\label{Equation:SecondEquationInTheProofOf:Lemma:FirstAuxiliarLemmaToProve:Theorem:J1HolomorphicToFirstOrderLiftsInTheFourDimensionalCase}\ddtatzero<u_t^X,\d\varphi_tJX>&=\ddtatzero<v_t^X,\d\varphi_tX>\text{.}
\end{align}
Similarly, $\ddtatzero<\partial^2_z\varphi_t,\partial^2_z\varphi_t>=0$ is equivalent to $\ddtatzero<u_t^X+iv_t^X,u_t^X+iv_t^X>=0$ and implies
\begin{align}
\notag\ddtatzero<u_t^X,u_t^X>&=\ddtatzero<v_t^X,v_t^X>\text{~and}\\
\label{Equation:FourthEquationInTheProofOf:Lemma:FirstAuxiliarLemmaToProve:Theorem:J1HolomorphicToFirstOrderLiftsInTheFourDimensionalCase}\ddtatzero<u_t^X,v_t^X>&=0\text{.}
\end{align}
As $\psi$ is compatible with $\varphi$, using Lemma \ref{Lemma:LiftsOfConformalToFirstOrderDifferFromHolomorphyByJustAnO(t)Vector}, we know that $\varphi$ is holomorphic to first order with respect to $J_\psi$. On the other hand, since $\partial_z\varphi_0$ and $\partial^2_z\varphi_0$ are linearly independent, we deduce that $\d\varphi_0 X, \d\varphi_0 JX$, $u_0^X$ and $v_0^X$ form a basis for $T_{\varphi_0}N$ (see \ref{Lemma:AdaptedAlmostHermitianStructuresToRealIsotropicMaps}). Hence,
\eqref{Equation:SecondEquationIn:Lemma:FirstAuxiliarLemmaToProve:Theorem:J1HolomorphicToFirstOrderLiftsInTheFourDimensionalCase}
will be satisfied if and only if the following four points are verified:

(i) $<\nabla_{\ddtatzero}J_{\psi_t}u_t^X,\d\varphi_0X>=-<\nabla_{\ddtatzero}v_t^X,\d\varphi_0X>$.

From the second step in the proof of Theorem \ref{Theorem:J1HolomorphicLiftsInTheFourDimensionalCase}, we know that $J_{\psi_0}u_0^X=-v_0^X$. Thus, we have\addtolength{\arraycolsep}{-1.0mm}
\[\begin{array}{rll}<\nabla_{\ddtatzero}J_{\psi_t}u_t^X,\d\varphi_0X>&=&\ddtatzero\hspace{-0.5mm}<J_{\psi_t}u_t^X,\d\varphi_tX>-<J_{\psi_0}u_0^X,\nabla_{\ddtatzero}\hspace{-0.5mm}\d\varphi_tX>\\
~&=&-\ddtatzero<u_t^X,J_{\psi_t}\d\varphi_tX>+<v_0^X,\nabla_{\ddtatzero}\d\varphi_tX>\\
~&=&-\ddtatzero<u_t^X,\d\varphi_tJX>+<v_0^X,\nabla_{\ddtatzero}\d\varphi_tX>\\
~&=&\hspace{-0.5mm}\text{(using~\eqref{Equation:SecondEquationInTheProofOf:Lemma:FirstAuxiliarLemmaToProve:Theorem:J1HolomorphicToFirstOrderLiftsInTheFourDimensionalCase})~}\hspace{-1.5mm}-\hspace{-1.5mm}\ddtatzero\hspace{-2mm}<v_t^X\hspace{-1mm},\d\varphi_tX\hspace{-1mm}>\hspace{-1mm}+\hspace{-1mm}<\hspace{-1mm}v_0^X,\hspace{-1mm}\nabla_{\hspace{-1mm}\ddtatzero}\hspace{-4.5mm}\d\varphi_tX\hspace{-1mm}>\\
~&=&<\nabla_{\ddtatzero}v_t^X,\d\varphi_0X>.\end{array}\]\addtolength{\arraycolsep}{1.0mm}\noindent

(ii) $<\nabla_{\ddtatzero}J_{\psi_t}u_t^X,\d\varphi_0JX>=-<\nabla_{\ddtatzero}v_t^X,\d\varphi_0 JX>$.

The argument is similar to the one in (i).

(iii)
$<\nabla_{\ddtatzero}J_{\psi_t}u_t^X,u_0^X>=-<\nabla_{\ddtatzero}v_t^X,u_0^X>$.\addtolength{\arraycolsep}{-1.0mm}

In fact,
\[\begin{array}{rll}
<\nabla_{\ddtatzero}J_{\psi_t}u_t^X,u_0^X>&=&\ddtatzero<J_{\psi_t}u_t^X,u_t^X>-<J_{\psi_0}u_0^X,\nabla_{\ddtatzero}u_t^X>\\
~&=&\ddtatzero<v_t^X,u_t^X>-<\nabla_{\ddtatzero}v_t^X,u_0^X>\\
~&=&\text{~(using~\eqref{Equation:FourthEquationInTheProofOf:Lemma:FirstAuxiliarLemmaToProve:Theorem:J1HolomorphicToFirstOrderLiftsInTheFourDimensionalCase})~}-<\nabla_{\ddtatzero}v_t^X,u_0^X>\text{.}
\end{array}\]\addtolength{\arraycolsep}{1.0mm}\noindent

(iv) Finally, let us prove $<\nabla_{\ddtatzero}J_{\psi_t}u_t^X,v_0^X>=-<\nabla_{\ddtatzero}v_t^X,v_0^X>$.\addtolength{\arraycolsep}{-1.0mm}

Indeed,
\[\begin{array}{lll}
<\nabla_{\ddtatzero}J_{\psi_t}u_t^X,v_0^X>&=&-\frac{1}{2}\ddtatzero<J_{\psi_t}u_t^X,J_{\psi_t}u_t^X>=-\frac{1}{2}\ddtatzero<u_t^X,u_t^X>\\
~&=&-\frac{1}{2}\ddtatzero<v_t^X,v_t^X>=-<\nabla_{\ddtatzero}v_t^X,v_0^X>,
\end{array}\]\addtolength{\arraycolsep}{1.0mm}\noindent
concluding our proof.
\end{proof}

\begin{lemma}\label{Lemma:SecondAuxiliarLemmaToProve:Theorem:J1HolomorphicToFirstOrderLiftsInTheFourDimensionalCase}
Let $\varphi$ be as in the preceding Theorem \ref{Theorem:J1HolomorphicToFirstOrderLiftsInTheFourDimensionalCase}. Then
\begin{equation}
\label{Equation:EquationIn:Lemma:SecondAuxiliarLemmaToProve:Theorem:J1HolomorphicToFirstOrderLiftsInTheFourDimensionalCase}
\nabla_{\ddtatzero}\partial_z^3\varphi_t=\nabla_{\ddtatzero}\{a_t\partial_z\varphi_t+b_t\partial^2_z\varphi_t\}
\end{equation}
for some $a_t$, $b_t$.
\end{lemma}
\begin{proof}
We know that $\partial_z\varphi_t$, $\partial^2_z\varphi_t$, $\overline{\partial_z\varphi_t}$, $\overline{\partial^2_z\varphi_t}$ span $T^{\cn}N^4$. Hence, there are $a_t,b_t,c_t$ and $d_t$ with
\[\partial_z^3\varphi_t=a_t\partial_z\varphi_t+b_t\partial^2_z\varphi_t+c_t\overline{\partial_z\varphi_t}+d_t\overline{\partial^2_z\varphi_t}\]
where $c_0=d_0=0$ since $\partial_z^3\varphi_0\in\wordspan\{\partial_z\varphi_0,\partial^2_z\varphi_0\}=T^{10}_{J_{\psi_0}}N$.
Therefore,\addtolength{\arraycolsep}{-1.0mm}
\[\begin{array}{lll}
\nabla_{\ddtatzero}\partial_z^3\varphi_t&=&a_0\nabla_{\ddtatzero}\partial_z\varphi_t+\left.\frac{\partial{a_t^X}}{\partial{t}}\right|_{t=0}\partial_z\varphi_0+b_0\nabla_{\ddtatzero}\partial^2_z\varphi_t+\left.\frac{\partial{b_t^X}}{\partial{t}}\right|_{t=0}\partial^2_z\varphi_0+\\
~&~&+\left.\frac{\partial{c_t^X}}{\partial{t}}\right|_{t=0}\overline{\partial_z\varphi_0}+\left.\frac{\partial{a_t^X}}{\partial{t}}\right|_{t=0}\overline{\partial^2_z\varphi_0}\text{.}\end{array}\]\addtolength{\arraycolsep}{1.0mm}\noindent
Now, using the fact that $\varphi$ is real isotropic to first order,
we have
\[<\nabla_{\ddtatzero}\partial_z^3\varphi_t,\partial_z\varphi_0>=-<\partial_z^3\varphi_0,\nabla_{\ddtatzero}\partial_z\varphi_t>\]
which implies\addtolength{\arraycolsep}{-1.0mm}
\[\begin{array}{ll}~&a_0<\nabla_{\ddtatzero}\hspace{-1mm}\partial_z\varphi_t,\partial_z\varphi_0>+\left.\frac{\partial{a_t^X}}{\partial{t}}\right|_{t=0}\hspace{-1mm}<\partial_z\varphi_0,\partial_z\varphi_0>+b_0<\nabla_{\ddtatzero}\partial^2_z\varphi_t,\partial_z\varphi_0>+\\
~&+\left.\frac{\partial{b_t^X}}{\partial{t}}\right|_{t=0}<\partial^2_z\varphi_0,\partial_z\varphi_0>+\left.\frac{\partial{c_t^X}}{\partial{t}}\right|_{t=0}<\overline{\partial_z\varphi_0},\partial_z\varphi_0>+\left.\frac{\partial{d_t^X}}{\partial{t}}\right|_{t=0}<\overline{\partial^2_z\varphi_0},\partial_z\varphi_0>\\
=&-<a_0\partial_z\varphi_0+b_0\partial^2_z\varphi_0,\nabla_{\ddtatzero}\partial_z\varphi_t>\text{.}\end{array}\]\addtolength{\arraycolsep}{1.0mm}\noindent
Since \[<\nabla_{\ddtatzero}\partial_z\varphi_t,\partial_z\varphi_0>=<\partial^2_z\varphi_0,\partial_z\varphi_0>=<\partial_z\varphi_0,\partial_z\varphi_0>=0\]
and
\[<\nabla_{\ddtatzero}\partial^2_z\varphi_t,\partial_z\varphi_0>=-<\partial^2_z\varphi_0,\nabla_{\ddtatzero}\partial_z\varphi_t>\text{,}\]
we deduce
\begin{equation}\label{Equation:FirstEquationInTheProofOf:Lemma:SecondAuxiliarLemmaToProve:Theorem:J1HolomorphicToFirstOrderLiftsInTheFourDimensionalCase}
\left.\frac{\partial{c_t^X}}{\partial{t}}\right|_{t=0}\|\partial_z\varphi_0\|^2+\left.\frac{\partial{d_t^X}}{\partial{t}}\right|_{t=0}<\overline{\partial^2_z\varphi_0},\partial_z\varphi_0>=0\text{.}
\end{equation}
Similarly, from
\[<\nabla_{\ddtatzero}\partial_z^3\varphi_t,\partial^2_z\varphi_0>=-<\partial_z^3\varphi_0,\nabla_{\ddtatzero}\partial^2_z\varphi_t>\]
we have
\begin{equation}\label{Equation:SecondEquationInTheProofOf:Lemma:SecondAuxiliarLemmaToProve:Theorem:J1HolomorphicToFirstOrderLiftsInTheFourDimensionalCase}
\left.\frac{\partial{c_t^X}}{\partial{t}}\right|_{t=0}<\overline{\partial_z\varphi_0},\partial^2_z\varphi_0>+\left.\frac{\partial{d_t^X}}{\partial{t}}\right|_{t=0}\|\partial^2_z\varphi_0\|^2=0\text{.}
\end{equation}
Writing $\lambda=\left.\frac{\partial{c_t^X}}{\partial{t}}\right|_{t=0}$,
$\beta=\left.\frac{\partial{d_t^X}}{\partial{t}}\right|_{t=0}$ and
$r=<\partial_z\varphi_0,\overline{\partial^2_z\varphi_0}>$, \eqref{Equation:FirstEquationInTheProofOf:Lemma:SecondAuxiliarLemmaToProve:Theorem:J1HolomorphicToFirstOrderLiftsInTheFourDimensionalCase} and
\eqref{Equation:SecondEquationInTheProofOf:Lemma:SecondAuxiliarLemmaToProve:Theorem:J1HolomorphicToFirstOrderLiftsInTheFourDimensionalCase} imply that
\[\left\{\begin{array}{l}\lambda\|\partial_z\varphi_0\|^2+\beta r=0\\
\lambda \bar{r}+\beta\|\partial^2_z\varphi_0\|^2=0\end{array}\right.\text{which~give~}
\left\{\begin{array}{l}\lambda\|\partial_z\varphi_0\|^2\bar{r}+\beta\|r\|^2=0\\
\lambda\|\partial_z\varphi_0\|^2\bar{r}+\beta\|\partial^2_z\varphi_0\|^2\|\partial_z\varphi_0\|^2=0\end{array}\right.\]
and imply
$\beta\big(\|r\|^2-\|\partial_z\varphi_0\|^2\|\partial^2_z\varphi_0\|^2\big)=0$; consequently, $\beta=0\quad\text{or}\quad\|r\|^2=\|\partial_z\varphi_0\|^2\|\partial^2_z\varphi_0\|^2$.
If $\|r\|^2=\|\partial_z\varphi_0\|^2\|\partial^2_z\varphi_0\|^2$ then, from
$\|<\partial_z\varphi_0,\overline{\partial^2_z\varphi_0}>\|=\|\partial_z\varphi_0\|\|\partial^2_z\varphi_0\|$,
we could deduce that $\partial^2_z\varphi_0$ lies in
$\wordspan\{\partial_z\varphi_0\}$, which we are assuming as
false. Therefore, $\beta=\left.\frac{\partial{d}^X_t}{\partial{t}}\right|_{t=0}=0$ and,
consequently also $\lambda=0$ (from the first equation above).
Thus, \eqref{Equation:EquationIn:Lemma:SecondAuxiliarLemmaToProve:Theorem:J1HolomorphicToFirstOrderLiftsInTheFourDimensionalCase}
holds, as wanted.
\end{proof}

We are finally ready to prove Theorem \ref{Theorem:J1HolomorphicToFirstOrderLiftsInTheFourDimensionalCase}:

\noindent\textit{Proof of Theorem \ref{Theorem:J1HolomorphicToFirstOrderLiftsInTheFourDimensionalCase}.}
As before, take $\psi^+_0$ or $\psi^-_0$ the $\J^1$-holomorphic lift of $\varphi_0$. Assume, without loss of generality, that it is $\psi^+_0$. Then, at each $t$ take $J_{\psi_t}$ the unique positive almost Hermitian structure compatible with $\varphi_t$ and let us prove that this map $\psi$ is $\J^1$-holomorphic to first order. Using Lemma
\ref{Lemma:LiftsOfConformalToFirstOrderDifferFromHolomorphyByJustAnO(t)Vector} $\varphi$ is $J_{\psi}$ holomorphic to first order and we are left with proving that \eqref{Equation:SecondEquationIn:Lemma:AnotherCharacterizationOfJ1AndJ2HolomorphicMapsToFirstOrder}
holds. It is enough to prove that there is a basis $\{Y^{10}_{1_t},Y^{10}_{2_t}\}$ of $\varphi^{-1}_t(T^{10}_{J_{\psi_t}}N)$ for which \eqref{Equation:SecondEquationIn:Lemma:AnotherCharacterizationOfJ1AndJ2HolomorphicMapsToFirstOrder} holds. Now, take $Y^{10}_{1_t}=\d\varphi_t X-iJ_{\psi_t}\d\varphi_t X$ and
$Y^{10}_{2_t}=u_t^X-iJ_{\psi_t}u_t^X$ where $u_t^X$ is as in
\eqref{Equation:FirstEquationIn:Lemma:FirstAuxiliarLemmaToProve:Theorem:J1HolomorphicToFirstOrderLiftsInTheFourDimensionalCase}.
Then,\addtolength{\arraycolsep}{-1.0mm}
\[\begin{array}{rll}
\nabla_{\ddtatzero}\nabla_{\partial_z}Y^{10}_{1_t}&=&R(\left.\frac{\partial\varphi_t}{\partial{t}}\right|_{t=0},\partial_z\varphi_t)Y^{10}_{1_0}+\nabla_{\partial_z}\hspace{-0.5mm}\nabla_{\ddtatzero}\hspace{-1mm}(\d\varphi_tX-iJ_{\psi_t}\d\varphi_tX)\\
~&=&\nabla_{\ddtatzero}\nabla_{\partial_z}\big(\d\varphi_tX-i\d\varphi{JX}\big)=\nabla_{\ddtatzero}\big(u_t^X+iv_t^X\big)\\
~&=&\nabla_{\ddtatzero}(u_t^X-iJ_{\psi_t}u_t^X)=\nabla_{\ddtatzero}Y^{10}_{2_t}\text{.}\end{array}\]\addtolength{\arraycolsep}{1.0mm}\noindent
Analogously,\addtolength{\arraycolsep}{-1.0mm}
\[\begin{array}{rll}~&\nabla_{\ddtatzero}\nabla_{\partial_z}Y^{10}_{2_t}=R(\left.\frac{\partial\varphi_t}{\partial{t}}\right|_{t=0},\partial_z\varphi_0)Y^{10}_{2_0}+\nabla_{\partial_z}\nabla_{\ddtatzero}\big(u_t^X-iJ_{\psi_t}u_t^X\big)\\
=&R(\left.\frac{\partial\varphi_t}{\partial{t}}\right|_{t=0},\partial_z\varphi_0)\partial^2_z\varphi_0+R(\partial_z\varphi_0,\left.\frac{\partial\varphi_t}{\partial{t}}\right|_{t=0})(u_0^X+iv_0^X)+\nabla_{\ddtatzero}\nabla_{\partial_z}(u_t^X+iv_t^X)\\
=&\nabla_{\ddtatzero}(a\partial_z\varphi+b\partial^2_z\varphi_t)=\frac{\partial{a}}{\partial{t}}Y^{10}_{1_0}+\frac{\partial{b}}{\partial{t}}Y^{10}_{2_0}+a\nabla_{\ddtatzero}\partial_z\varphi_t+\nabla_{\ddtatzero}\partial^2_z\varphi_t\\
=&\frac{\partial{a}}{\partial{t}}Y^{10}_{1_0}+\frac{\partial{b}}{\partial{t}}Y^{10}_{2_0}+a\nabla_{\ddtatzero}Y^{10}_{1_t}+\nabla_{\ddtatzero}Y^{10}_{2_t}=\nabla_{\ddtatzero}\big(aY^{10}_{1_t}+bY^{10}_{2_t}\big)\text{,}\end{array}\]\addtolength{\arraycolsep}{1.0mm}\noindent
where we have used $Y^{10}_{2_0}=\partial^2_z\varphi_0$, $\nabla_{\ddtatzero}J_{\psi_t}u_t^X=\nabla_{\ddtatzero}v^X_t$, $u_0^X+iv_0^X=\partial^2_z\varphi_0$ and $\nabla_{\partial_z}(u_t^X+iv_t^X)=\partial_z^3\varphi_t$. Hence, $Y^{10}_{1_t}$ and $Y^{10}_{2_t}$ satisfy equation \eqref{Equation:SecondEquationIn:Lemma:AnotherCharacterizationOfJ1AndJ2HolomorphicMapsToFirstOrder}, concluding our proof.\qed


\subsection{The $\J^2$-holomorphic case}


We prove the following:

\begin{theorem}\label{Theorem:ProjectionsOfJ2HolomorphicToFirstOrderMaps}\index{$\J^2$-holomorphic!to~first~order!projection} Let $\psi:I\times M^2\to\Sigma^+ N$ be a map $\J^2$-holomorphic to first order. Then, $\varphi=\pi\circ\psi:I\times M^2\to N$ is harmonic to first order\footnote{and conformal to first order from Lemma \ref{Lemma:JHolomorphicityToFirstOrderIffBothHAndVHolomorphicityToFirstOrderHold} and Proposition
\ref{Proposition:ProjectionsAndLiftsOfConformalToFirstOrder}.}.
\end{theorem}

Since harmonicity (to first order) does not depend on the orientation on $N$, we could replace $\Sigma^+N$ by $\Sigma^-N$.

\begin{proof}That $\varphi_0$ is harmonic follows from Theorem \ref{Theorem:SalamonTheorem3.5}. Hence, we are left with proving that $\ddtatzero\tau(\varphi_t)=0$. Since $\psi$ is $\J^2$-holomorphic to first order, we deduce that $\psi$ is both ($\J^2$) $\H$ and $\V$-holomorphic to first order (Lemma \ref{Lemma:JHolomorphicityToFirstOrderIffBothHAndVHolomorphicityToFirstOrderHold}). From vertical holomorphicity have (see \eqref{Equation:FirstEquationIn:Lemma:ConditionForVHolomorphicity}) \[\nabla_{\ddtatzero}\big(\nabla_{JX}J_{\psi_t}+J_{\psi_t}\nabla_X J_{\psi_t}\big)=0\]
so that\addtolength{\arraycolsep}{-1.0mm}
\[\begin{array}{ll}~&\{\nabla_{\ddtatzero}\big(\nabla_{JX}J_{\psi_t}+J_{\psi_t}\nabla_XJ_{\psi_t}\big)\}(\d\varphi_0X)=0\\
\impl&\nabla_{\ddtatzero}\hspace{-1mm}\{\nabla_{JX}(J_{\psi_t}\d\varphi_tX)-J_{\psi_t}\nabla_{JX}\d\varphi_tX+J_{\psi_t}\nabla_{X}(J_{\psi_t}\d\varphi_tX)+\nabla_X\d\varphi_tX\}=0\\
\impl&\nabla_{\ddtatzero}\hspace{-1mm}\{\underbrace{\nabla_{JX}\d\varphi_tJX+\nabla_X\d\varphi_tX}_{\text{\tiny{$=\tau(\varphi_t)$}}}\}\hspace{-0.2mm}=\hspace{-0.2mm}\nabla_{\ddtatzero}\hspace{-1mm}\{J_{\psi_t}\big(\nabla_{JX}\d\varphi_tX-\nabla_X(J_{\psi_t}\d\varphi_tX)\big)\}\\
\impl&\nabla_{\ddtatzero}\tau(\varphi_t)=\nabla_{\ddtatzero}\{J_{\psi_t}\big(\nabla_{JX}\d\varphi_tX-\nabla_X(J_{\psi_t}\d\varphi_tX)\big)\}\text{.}\end{array}\]\addtolength{\arraycolsep}{1.0mm}\noindent
Using Lemma
\ref{Lemma:HHolomorphicityToFirstOrderImpliesHolomorphicityToFirstOrderOfTheProjectedMap}
and equation
\eqref{Equation:SecondEquationIn:Lemma:LiftsOfConformalToFirstOrderDifferFromHolomorphyByJustAnO(t)Vector}
together with symmetry of the second fundamental form of $\varphi_0$, the right-hand side of the above identity becomes\addtolength{\arraycolsep}{-1.0mm}
\[\begin{array}{ll}~&(\nabla_{\ddtatzero}\hspace{-2.6mm}J_{\psi_t})(\nabla_{JX}\d\varphi_0X\hspace{-0.3mm}-\hspace{-0.3mm}\nabla_XJ_{\psi_0}\d\varphi_0X)\hspace{-0.3mm}+\hspace{-0.3mm}J_{\psi_0}\{\nabla_{\ddtatzero}\hspace{-1mm}(\nabla_{JX}\d\varphi_tX\hspace{-0.3mm}-\hspace{-0.3mm}\nabla_XJ_{\psi_t}\d\varphi_tX)\}\\
=&\nabla_{\ddtatzero}J_{\psi_t}\{\nabla\d\varphi_0(JX,X)-\nabla\d\varphi_0(X,JX)+\d\varphi_0(\nabla_{JX}X-\nabla_XJX)\}\\
~&+J_{\psi_0}\{\nabla_{\ddtatzero}\big(\nabla_{JX}\d\varphi_tX-\nabla_X\d\varphi_tJX\big)\}\\
=&J_{\psi_0}\{\nabla_{\ddtatzero}\big(\nabla\d\varphi_t(JX,X)-\nabla\d\varphi_t(X,JX)+\d\varphi_t(\nabla_{JX}X-\nabla_XJX)\big)\}=0\text{,}\end{array}\]\addtolength{\arraycolsep}{1.0mm}\noindent
so that $\ddtatzero\tau(\varphi_t)=0$, concluding the proof.
\end{proof}

\begin{theorem}\label{Theorem:J2HolomorphicLiftsToFirstOrder}\index{$\J^2$-holomorphic!to~first~order!lift}
Let $\varphi:I\times M^2\to N^{2n}$ be a map harmonic and conformal to first order and let $z_0\,\in\,M^2$. Assume that $\partial_z\varphi_0(z_0)\neq 0$. Then, reducing $I$ if necessary, there is an open set $\openU$ around $z_0$ and a map $\psi:I\times \openU\to\Sigma^+N$ which is $\J^2$-holomorphic to first order and with $\varphi=\pi\circ\psi$.
\end{theorem}
Once again, since harmonicity (to first order) does not depend on the orientation of $N$, we could replace $\Sigma^+N$ with $\Sigma^-N$.

\begin{proof}
For each $t$ consider $V_t=\d\varphi_t(TM)^{\perp}\subseteq\varphi_t^{-1}(TN)$, bundle over $M^2$. Since $M^2$ is a Riemann surface, $R^{20}_{V_t}=0$ and we can conclude that for each $t$ there is a Koszul-Malgrange holomorphic structure on $\Sigma^+V_t$. Moreover, Theorem \ref{Theorem:KoszulMalgrangeTheorem:ParametricVersion} guarantees the existence of a smooth section $s^{10}_V$ with $s^{10}_{V_t}$ a Koszul-Malgrange holomorphic section of
$\Sigma^+V_t$: $\nabla^\perp_{\partial_{\bar{z}}}
s^{10}_{V_t}\subseteq s^{10}_{V_t}$. So,
\[J_{\psi_t}(\nabla^\perp_{X+iJX}(v_t-iJ_{\psi_t}v_t))=i\nabla^\perp_{X+iJX}(v_t-iJ_{\psi_t}v_t),\quad\forall\,v_t\in\d\varphi_t(TM)^\perp\text{,}\]
equivalently,
\begin{equation}\label{Equation:FirstEquationInTheProofOf:Theorem:J2HolomorphicLiftsToFirstOrder}
\left\{\begin{array}{l}J_{\psi_t}(\nabla^{\perp}_Xv_t+\nabla^\perp_{JX}J_{\psi_t}v_t)=-\nabla^\perp_{JX}v_t+\nabla^\perp_XJ_{\psi_t}v_t\text{,}\\
J_{\psi_t}(\nabla^\perp_{JX}v_t-\nabla^\perp_XJ_{\psi_t}v_t)=\nabla^\perp_Xv_t+\nabla^\perp_{JX}J_{\psi_t}v_t\text{.}\end{array}\right.
\end{equation}
Take $s^{10}_{TN}=s^{10}_{V}\oplus T^{10,\top}$ where $T^{10,\top}_t$ is the $(1,0)$-part on $\d\varphi_t(TM)^{\cn}$ determined by $J^{\top}_t=\text{rotation~by~}+\frac{\pi}{2}$ on $\d\varphi_t(TM)$\footnote{Notice that as $\varphi_t$ is not conformal we might not get a Hermitian structure by setting $T^{10,\top}=\d\varphi_t(T^{10}M)$; on the other hand, positive rotation by $\pi/2$ comes from the natural orientation on $\d\varphi_t(TM)$ imported from $TM$ \textit{via} $\d\varphi_t$}. Then $s^{10}_{TN}$ defines a compatible (in the sense of \eqref{Equation:TwistorLiftCompatibleWithTheMapInTheVariationalCase})
twistor lift of $\varphi$. Let us check that $\psi$ is $\J^2$-holomorphic to first order. That $\psi_0$ is holomorphic is
immediate from the proof of Theorem \ref{Corollary:SalamonCorollary4.2}. From the proof of Proposition \ref{Proposition:ProjectionsAndLiftsOfConformalToFirstOrder}, we deduce that $\psi$ is $\H$-holomorphic to first order as it is compatible to first order with $\varphi$ and the latter is conformal to first order. Hence, using Corollary
\ref{Corollary:HorizontalAndVerticalHolomorphicityToFirstOrderImplyHolomorphicity},
we are left with proving that \eqref{Equation:FirstEquationIn:Lemma:ConditionForVHolomorphicity}:
\[\nabla_{\ddtatzero}\nabla_{JX}J_{\psi_t}=-\nabla_{\ddtatzero}J_{\psi_t}\nabla_XJ_{\psi_t}\]
holds. We shall establish this equation by showing that both sides agree when applied to any vector $v\in{TN}$. For that, we consider, in turn, the three cases $v=\d\varphi_0 X$, $\d\varphi_0JX$ and $v\in\d\varphi_0(TM)^\perp$.

(i) $v=\d\varphi_0 X$.

From $\psi_0$ holomorphicity, we have $\nabla_{JX}J_{\psi_0}=-J_{\psi_0}\nabla_X J_{\psi_0}$. On the other
hand, as $\psi$ is $\H$-holomorphic to first order, equations
\eqref{Equation:SecondEquationIn:Lemma:LiftsOfConformalToFirstOrderDifferFromHolomorphyByJustAnO(t)Vector}
and
\eqref{Equation:EquationIn:Remark:FundamentalTrickWithTheCurvatureTensor:TheFundamentalTrickWithTheCurvatureTensor}
are satisfied. Finally, for all $t$,
\[\nabla_{JX}\d\varphi_tX-\nabla_X\d\varphi_tJX=\nabla\d\varphi_t(JX,X)-\nabla\d\varphi_t(X,JX)+\d\varphi_t(\nabla_{JX}X-\nabla_XJX)=0\text{.}\]
Thus,
\[\begin{array}{ll}~&(\nabla_{\ddtatzero}\nabla_{JX}J_{\psi_t})\d\varphi_0X=-(\nabla_{\ddtatzero}J_{\psi_t}\nabla_XJ_{\psi_t})\d\varphi_0X\\
\equi&\nabla_{\ddtatzero}((\nabla_{JX}J_{\psi_t})\d\varphi_tX)-(\nabla_{JX}J_{\psi_0})(\nabla_{\ddtatzero}\d\varphi_tX)\\
~&=-\nabla_{\ddtatzero}(J_{\psi_t}(\nabla_XJ_{\psi_t})\d\varphi_tX)+J_{\psi_0}(\nabla_XJ_{\psi_0})(\nabla_{\ddtatzero}\d\varphi_tX)\\
\equi&\nabla_{\ddtatzero}\nabla_{JX}(J_{\psi_t}\d\varphi_tX)-\nabla_{\ddtatzero}(J_{\psi_t}\nabla_{JX}\d\varphi_tX)\\
~&=-\nabla_{\ddtatzero}(J_{\psi_t}\nabla_X(J_{\psi_t}\d\varphi_tX)-\nabla_{\ddtatzero}\nabla_X\d\varphi_tX\\
\equi&\nabla_{\ddtatzero}\nabla_{JX}\d\varphi_tJX+\nabla_{\ddtatzero}\nabla_X\d\varphi_tX\\
~&=(\nabla_{\ddtatzero}J_{\psi_t})(\nabla_{JX}\d\varphi_0X)+J_{\psi_0}(\nabla_{\ddtatzero}\nabla_{JX}\d\varphi_tX)\\
~&\quad-(\nabla_{\ddtatzero}J_{\psi_t})(\nabla_X(J_{\psi_0}\d\varphi_0X))-J_{\psi_0}(\nabla_{\ddtatzero}\nabla_X(J_{\psi_t}\d\varphi_tX)\\
~&=(\nabla_{\ddtatzero}J_{\psi_t})(\nabla_{JX}\d\varphi_0X-\nabla_X\d\varphi_0JX)\\
~&\quad+J_{\psi_0}(\nabla_{\ddtatzero}\nabla_{JX}\d\varphi_tX-\nabla_{\ddtatzero}\nabla_X\d\varphi_tJX)\\
\equi&\ddtatzero\tau(\varphi_t)=J_{\psi_0}\nabla_{\ddtatzero}(\nabla_{JX}\d\varphi_tX-\nabla_{X}\d\varphi_tJX)\,\equi\,\ddtatzero\tau(\varphi_t)=0\text{,}\end{array}\]
which is true from harmonicity to first order of $\varphi$.

(ii) $v=\d\varphi_0 JX$.

The argument is very similar to the preceding one.

(iii) $v\in\d\varphi_0(TM)^\perp$.

We have\addtolength{\arraycolsep}{-1.0mm}
\[\begin{array}{ll}~&(\nabla_{\ddtatzero}\nabla_{JX}J_{\psi_t})v=-(\nabla_{\ddtatzero}J_{\psi_t}\nabla_XJ_{\psi_t})v\\
\equi&\nabla_{\ddtatzero}\big((\nabla_{JX}J_{\psi_t})v\big)+(\nabla_{JX}J_{\psi_0})(\nabla_{\ddtatzero}v)\\
~&=-\nabla_{\ddtatzero}\big((J_{\psi_t}\nabla_XJ_{\psi_t})v\big)-(J_{\psi_0}\nabla_XJ_{\psi_0})(\nabla_{\ddtatzero}v)\\
\equi&\nabla_{\ddtatzero}\hspace{-1mm}\nabla_{JX}(J_{\psi_t}v)-\nabla_{\ddtatzero}\hspace{-1mm}(J_{\psi_t}\nabla_{JX}v)=-\nabla_{\ddtatzero}\hspace{-1mm}\big(J_{\psi_t}\nabla_X(J_{\psi_t}v)\big)-\nabla_{\ddtatzero}\hspace{-1mm}\nabla_Xv\text{.}\end{array}\]\addtolength{\arraycolsep}{1.0mm}\noindent
This equation is satisfied if the following two equations hold:
\begin{align}
\label{Equation:FirstEquationInTheThirdStepOfTheProofOf:Theorem:J2HolomorphicLiftsToFirstOrder}\nabla_{\ddtatzero}\hspace{-0.5mm}\nabla^\perp_{JX}(J_{\psi_t}v)-\nabla_{\ddtatzero}\hspace{-0.5mm}(J_{\psi_t}\nabla^\perp_{JX}v)+\nabla_{\ddtatzero}\hspace{-0.5mm}\big(J_{\psi_t}\nabla^\perp_X(J_{\psi_t}v)\big)+\nabla_{\ddtatzero}\hspace{-0.5mm}\nabla^\perp_Xv=0\text{,}\\
\label{Equation:SecondEquationInTheThirdStepOfTheProofOf:Theorem:J2HolomorphicLiftsToFirstOrder}\nabla_{\ddtatzero}\hspace{-0.5mm}\nabla^{\top}_{JX}(J_{\psi_t}v)-\nabla_{\ddtatzero}\hspace{-0.5mm}(J_{\psi_t}\nabla^{\top}_{JX}v)+\nabla_{\ddtatzero}\hspace{-0.5mm}\big(J_{\psi_t}\nabla^{\top}_X(J_{\psi_t}v)\big)+\nabla_{\ddtatzero}\hspace{-0.5mm}\nabla^{\top}_Xv=0\text{.}
\end{align}
Now, \eqref{Equation:FirstEquationInTheThirdStepOfTheProofOf:Theorem:J2HolomorphicLiftsToFirstOrder} is easy to check using the fact that $s^{10}$ is Koszul-Malgrange holomorphic for each $t$ so that \eqref{Equation:FirstEquationInTheProofOf:Theorem:J2HolomorphicLiftsToFirstOrder} holds. As for \eqref{Equation:SecondEquationInTheThirdStepOfTheProofOf:Theorem:J2HolomorphicLiftsToFirstOrder}, letting $Q(v)$ denote its left-hand side, we shall prove $<Q(v),w>=0$ for all $w\in{T}N$. Again, we do this by establishing the three cases $w=\d\varphi_0 X$, $w=\d\varphi_0 JX$ and $w\in\d\varphi_0(TM)^\perp$.

(iii$_\text{a}$) When $w=\d\varphi_0X$, we have\addtolength{\arraycolsep}{-1.0mm}
\[\begin{array}{lll}Q(v,\d\varphi_0X)&=&\ddtatzero<\nabla^{\top}_{JX}(J_{\psi_t}v_t)-J_{\psi_t}\nabla^{\top}_{JX}v_t+J_{\psi_t}\nabla^{\top}_X(J_{\psi_t}v_t)+\nabla^{\top}_Xv_t,\d\varphi_tX>\\
~&~&-<\nabla^{\top}_{JX}(J_{\psi_0}v)-J_{\psi_0}\nabla^{\top}_{JX}v+J_{\psi_t}\nabla^{\top}_X(J_{\psi_0}v)+\nabla^{\top}_Xv,\nabla_{\ddtatzero}\d\varphi_tX>\\
~&=&\ddtatzero(-<J_{\psi_t}v_t,\nabla_{JX}\d\varphi_tX>-<v_t,\nabla_X\d\varphi_tX>)\\
~&~&+\ddtatzero(-<v_t,\nabla_{JX}\d\varphi_t JX>+<J_{\psi_t}v_t,\nabla_X\d\varphi_tJX>)\\
~&=&-\ddtatzero<v_t,\nabla_X\d\varphi_tX+\nabla_{JX}\d\varphi_tJX>\\
~&~&+\ddtatzero<J_{\psi_t}v_t,\nabla_X\d\varphi_tJX-\nabla_{JX}\d\varphi_tX>\\
~&=&-<\nabla_{\ddtatzero}v_t,\tau(\varphi_0)>-<v_0,\nabla_{\ddtatzero}\tau(\varphi_t)>=0\text{,}\end{array}\]\addtolength{\arraycolsep}{1.0mm}\noindent
as required.

(iii$_\text{b}$) For $\d\varphi_0 (JX)$ the argument is similar to
the preceding one.

(iii$_\text{c}$) Let $w_t\in\d\varphi_t(TM)^\perp$. Then,\addtolength{\arraycolsep}{-1.0mm}
\[\begin{array}{lll}Q(v,w)&=&\ddtatzero<\nabla^{\top}_{JX}(J_{\psi_t}v)-J_{\psi_t}\nabla^{\top}_{JX}v+J_{\psi_t}\nabla^{\top}_X(J_{\psi_t}v)+\nabla^{\top}_Xv,w>\\
~&~&-<\nabla^{\top}_{JX}(J_{\psi_0}v)-J_{\psi_0}\nabla^{\top}_{JX}v+J_{\psi_0}\nabla^{\top}_X(J_{\psi_0}v)+\nabla^{\top}_Xv,\nabla_{\ddtatzero}w>\text{.}\end{array}\]\addtolength{\arraycolsep}{1.0mm}\noindent
The first term on the right side of the above equation vanishes as $w_t$ lies in $\d\varphi_t(TM)^\perp$, whereas the second is zero from $\psi_0$-holomorphicity, concluding our proof.
\end{proof}


\subsection{The $4$-dimensional case}


\begin{theorem}\label{Theorem:HorizontalHolomorphicLiftsUpToFirstOrderInThe4DimensionalCase}
Let $\varphi:I\times M^2\to N^4$ be harmonic and real isotropic to first order and let $z_0\in{M}^2$. Assume that $\partial_z\varphi_0(z_0)$ and $\partial^2_z\varphi_0(z_0)$ are linearly independent. Then, reducing $I$ if necessary, there is an open set $\openU$ around $z_0$ and either a map $\psi^+:I\times\U \to\Sigma^+N$ or a map $\psi^-:I\times\openU\to\Sigma^-N$ which is simultaneously $\J^1$ and $\J^2$-holomorphic to first order
and has $\varphi=\pi\circ\psi$. Conversely, if $\psi:I\times M^2\to\Sigma^{+}N^4$ (or $\psi:I\times M^2\to\Sigma^-N^4$) is $\J^1$ and $\J^2$-holomorphic to first order, the projected map $\varphi=\pi\circ\psi:I\times M^2\to N^4$ is harmonic and real isotropic to first order.
\end{theorem}

\begin{proof}
The converse is obvious from Theorems \ref{Proposition:ProjectionsOfJ1HolomorphicToFirstOrderMapsAndRealIsotropy}
and \ref{Theorem:ProjectionsOfJ2HolomorphicToFirstOrderMaps}. As for the first part, in Theorem
\ref{Theorem:J1HolomorphicToFirstOrderLiftsInTheFourDimensionalCase} we saw that we can lift the map $\varphi$ to a map $\J^1$-holomorphic to first order. Moreover, this
lift could be defined as the unique positive or negative almost complex structure compatible with $\varphi$. On the other hand, in Theorem \ref{Theorem:J2HolomorphicLiftsToFirstOrder} we have seen that there is a map $\J^2$-holomorphic to first order with $\varphi=\pi\circ\psi$ and for which $\varphi$ is compatible. From the comment after Theorem \ref{Theorem:J2HolomorphicLiftsToFirstOrder}, there is also a twistor lift of $\varphi$ into $\Sigma^-N$. Therefore, from the dimension of $N$, we conclude that the lifts constructed in both cited results are the same and, therefore, simultaneously $\J^1$ and $\J^2$-holomorphic to first order.
\end{proof}

We would now like to guarantee the \textit{uniqueness to first order} of our twistor lift. Before stating such a result we start with a lemma:

\begin{lemma}\label{Lemma:FirstAuxiliarLemmaToProve:Proposition:UnicityToFirstOrderOfHorizontalHolomorphicLifts}
Let $\psi:I\times{M}^2\to\Sigma^+ N$ a map $\J^1$-holomorphic to first order. Consider the twistor projection $\varphi_t=\pi\circ\psi$ and the vectors
\[\partial^2_z\varphi_t=u_t+iv_t\]
so that
\[\left\{\begin{array}{l}
u_t=\nabla_X\d\varphi_t X-\nabla_{JX}\d\varphi_t JX\text{,}\\
v_t=-\nabla_X\d\varphi_t JX-\nabla_{JX}\d\varphi_t X\text{.}
\end{array}\right.\]
Then, for all $z_0$ for which $\partial_z\varphi_0(z_0)$ and $\partial^2_z\varphi_0(z_0)$ are linearly independent, the following equations are satisfied
\begin{align}
\label{Equation:FirstEquationIn:Lemma:FirstAuxiliarLemmaToProve:Proposition:UnicityToFirstOrderOfHorizontalHolomorphicLifts}\nabla_{\ddtatzero}\d\varphi_tJX&=\nabla_{\ddtatzero}J_{\psi_t}\d\varphi_tX\text{,}\\
\label{Equation:SecondEquationIn:Lemma:FirstAuxiliarLemmaToProve:Proposition:UnicityToFirstOrderOfHorizontalHolomorphicLifts}\nabla_{\ddtatzero}J_{\psi_t}u_t&=\nabla_{\ddtatzero}-v_t\text{,}\\
\label{Equation:ThirdEquationIn:Lemma:FirstAuxiliarLemmaToProve:Proposition:UnicityToFirstOrderOfHorizontalHolomorphicLifts}\nabla_{\ddtatzero}J_{\psi_t}v_t&=\nabla_{\ddtatzero}u_t\text{.}
\end{align}
\end{lemma}
Notice the similarity with Lemma \ref{Lemma:FirstAuxiliarLemmaToProve:Theorem:J1HolomorphicToFirstOrderLiftsInTheFourDimensionalCase}:
the main difference is that in that lemma, we were given $\varphi$ and \emph{defined} the twistor lift as the unique lift
compatible with $\varphi$. Now, we are given the twistor map $\psi$ but nothing guarantees that projecting the map to $\varphi_t$ makes $\varphi$ compatible; \textit{i.e.}, $J_{\psi_t}$ may not preserve $\d\varphi_t(TM)$.

\begin{proof}
That $\nabla_{\ddtatzero}\d\varphi_t JX=\nabla_{\ddtatzero}J_{\psi_t}\d\varphi_t X$ follows from the proof of
\ref{Proposition:ProjectionsAndLiftsOfConformalToFirstOrder}. Since $\partial_z\varphi_0(z_0)$ and $\partial^2_z\varphi_0(z_0)$ are linearly independent vectors, we can deduce that $\d\varphi_t X$, $\d\varphi_t JX$, $u_t$ and $v_t$ form a basis for $T_{\varphi_t(z)}N$ for $(t,z)$ is a neighbourhood of $(0,z_0)$. On the other hand, as $\varphi$ is the projection of a map $\J^1$-holomorphic to first order, we know that it must be real isotropic to first order from Proposition\ref{Proposition:ProjectionsOfJ1HolomorphicToFirstOrderMapsAndRealIsotropy}. Hence,
\[\ddtatzero<\partial^2_z\varphi_t,\partial_z\varphi_t>=0\text{,}\]
equivalently, $\ddtatzero<u_t+iv_t,\d\varphi_t X-i\d\varphi_t JX>=0$
and therefore
\begin{align}
\notag\ddtatzero<u_t,\d\varphi{X}>&=-\ddtatzero<v_t,\d\varphi_tJX>\text{,}\\
\label{Equation:SecondEquationInTheProofOf:Lemma:FirstAuxiliarLemmaToProve:Proposition:UnicityToFirstOrderOfHorizontalHolomorphicLifts}\ddtatzero<u_t,\d\varphi_tJX>&=\ddtatzero<v_t,\d\varphi_tX>\text{.}
\end{align}
Similarly, $\ddtatzero<\partial^2_z\varphi_t,\partial^2_z\varphi_t>=0$ is equivalent to $\ddtatzero<u_t+iv_t,u_t+iv_t>=0$
and implies
\begin{align}
\notag\ddtatzero<u_t,u_t>&=\ddtatzero<v_t,v_t>\text{,}\\
\label{Equation:FourthEquationInTheProofOf:Lemma:FirstAuxiliarLemmaToProve:Proposition:UnicityToFirstOrderOfHorizontalHolomorphicLifts}\ddtatzero<u_t,v_t>&=0
\end{align}
The argument to establish \eqref{Equation:SecondEquationIn:Lemma:FirstAuxiliarLemmaToProve:Proposition:UnicityToFirstOrderOfHorizontalHolomorphicLifts}, will now be similar to the one in Lemma \ref{Lemma:FirstAuxiliarLemmaToProve:Theorem:J1HolomorphicToFirstOrderLiftsInTheFourDimensionalCase}:

(i) We start by proving that $<\nabla_{\ddtatzero}J_{\psi_t}u_t,\d\varphi_0X>=-<\nabla_{\ddtatzero}v_t,\d\varphi_0X>$. Indeed,\addtolength{\arraycolsep}{-1.0mm}
\[\begin{array}{ll}
~&<\nabla_{\ddtatzero}J_{\psi_t}u_t,\d\varphi_0X>=\ddtatzero<J_{\psi_t}u_t,\d\varphi_tX>-<J_{\psi_0}u_0,\nabla_{\ddtatzero}\d\varphi_tX>\\
=&\text{~($J_{\psi_0}u_0=-v_0$,~Th.~\ref{Theorem:J1HolomorphicLiftsInTheFourDimensionalCase})~}-\ddtatzero<u_t,J_{\psi_t}\d\varphi_tX>+<v_0^X,\nabla_{\ddtatzero}\d\varphi_tX>\\
=&-<\nabla_{\ddtatzero}\hspace{-1mm}u_t,J_{\psi_0}\d\varphi_0X>-<u_0,\nabla_{\ddtatzero}\hspace{-1mm}J_{\psi_t}\d\varphi_tX>+<v_0,\nabla_{\ddtatzero}\d\varphi_tX>\\
=&-\ddtatzero<u_t,\d\varphi_tJX>+<v_0,\nabla_{\ddtatzero}\d\varphi_tX>\\
=&\text{~(using~\eqref{Equation:SecondEquationInTheProofOf:Lemma:FirstAuxiliarLemmaToProve:Proposition:UnicityToFirstOrderOfHorizontalHolomorphicLifts})~}-\ddtatzero<v_t,\d\varphi_tX>+<v_0,\nabla_{\ddtatzero}\d\varphi_tX>\\
=&-<\nabla_{\ddtatzero}v_t,\d\varphi_0X>\text{.}\end{array}\]\addtolength{\arraycolsep}{1.0mm}\noindent

(ii) Replacing $\d\varphi_0X$ by $\d\varphi_0JX$ and using similar arguments, we can show that $<\nabla_{\ddtatzero}J_{\psi_t}u_t,\d\varphi_0JX>=-<\nabla_{\ddtatzero}v_t,\d\varphi_0JX>$.

(iii) Next, we prove that $<\nabla_{\ddtatzero}J_{\psi_t}u_t,u_0>=-<\nabla_{\ddtatzero}v_t,u_0>$. In fact,\addtolength{\arraycolsep}{-1.0mm}
\[\begin{array}{ll}~&<\nabla_{\ddtatzero}J_{\psi_t}u_t,u_0>=\ddtatzero<J_{\psi_t}u_t,u_t>-<J_{\psi_0}u_0,\nabla_{\ddtatzero}u_t>\\
=&\text{~($J_{\psi_0}u_0=-v_0$)~}\ddtatzero<v_t,u_t>-<\nabla_{\ddtatzero}v_t,u_0>\\
=&\text{~(using~\eqref{Equation:FourthEquationInTheProofOf:Lemma:FirstAuxiliarLemmaToProve:Proposition:UnicityToFirstOrderOfHorizontalHolomorphicLifts})~}-<\nabla_{\ddtatzero}v_t,u_0>\text{.}\end{array}\]\addtolength{\arraycolsep}{1.0mm}\noindent

(iv) Finally, we are left with proving
$<\nabla_{\ddtatzero}J_{\psi_t}u_t,v_0>=-<\nabla_{\ddtatzero}v_t,v_0>$. Since $v_0=-J_{\psi_0}u_0$, we have\addtolength{\arraycolsep}{-1.0mm}
\[\begin{array}{lll}<\nabla_{\ddtatzero}J_{\psi_t}u_t,v_0>&=&-\frac{1}{2}\ddtatzero<J_{\psi_t}u_t,J_{\psi_t}u_t>=-\frac{1}{2}\ddtatzero<u_t,u_t>\\
~&=&-\frac{1}{2}\ddtatzero<v_t,v_t>=-<\nabla_{\ddtatzero}v_t,v_0>.\end{array}\]\addtolength{\arraycolsep}{1.0mm}\noindent

This establishes \eqref{Equation:SecondEquationIn:Lemma:FirstAuxiliarLemmaToProve:Proposition:UnicityToFirstOrderOfHorizontalHolomorphicLifts}. The last equation
\eqref{Equation:ThirdEquationIn:Lemma:FirstAuxiliarLemmaToProve:Proposition:UnicityToFirstOrderOfHorizontalHolomorphicLifts}
has a similar proof.
\end{proof}

\begin{proposition}\label{Proposition:UnicityToFirstOrderOfHorizontalHolomorphicLifts}
Let $\psi^1,\psi^2:I\times M^2\to\Sigma^+ N^4$ be two maps $\J^1$-holomorphic to first order such
that $\psi_0^1=\psi_0^2$ and the variational vector fields induced on $N^4$ are the same; \textit{i.e.}, writing  $a_i:=\ddtatzero(\pi\circ\psi^i)$, $i=1,2$, we have $a_1=a_2$. Then, at all points $z_0$ for which $\partial_z\varphi_0(z_0)$ and $\partial^2_z\varphi_0(z_0)$ are linearly independent, writing $w_i=\ddtatzero\psi_t^i$, $i=1,2$, we have
\begin{equation}
\label{Equation:FirstEquationIn:Proposition:UnicityToFirstOrderOfHorizontalHolomorphicLifts}
w_1=w_2\text{.}
\end{equation}
\end{proposition}
\begin{proof}
Let $\varphi^i=\pi\circ\psi_i$ ($i=1,2$) denote the projection maps. From our hypothesis, it follows that $w_1^\H=w_2^\H$. Hence, the only thing left is to prove that the vertical parts coincide. Now, from the proof of Lemma \ref{Lemma:ConditionForVHolomorphicity}, $(\ddtatzero\psi^i_t)^\V=\nabla^{\L(TN,TN)}_{\ddtatzero}J_{\psi^i_t}$ so that our result follows if $\nabla_{\ddtatzero}J_{\psi^1_t}=\nabla_{\ddtatzero}J_{\psi^2_t}$. We prove this identity showing that $Q(Y)=0$ for all $Y$, where
\[Q(Y)=(\nabla_{\ddtatzero}J_{\psi^1_t})Y-(\nabla_{\ddtatzero}J_{\psi^2_t})Y\text{.}\]
We consider the four possible cases for $Y$; namely, when $Y$ is equal to $\d\varphi_0 X$, $\d\varphi_0 JX$, $u_0$ or $v_0$, where $u_0$ and $v_0$ are as in the preceding lemma (notice that, since $\psi^1_0=\psi^2_0$, then $u^1_0=u^2_0$ and $v^1_0=v^2_0$).

(i) When $Y=\d\varphi_0X$, we have\addtolength{\arraycolsep}{-1.0mm}
\[\begin{array}{lll}Q(\d\varphi_0X)&=&\nabla_{\ddtatzero}(J_{\psi^1_t}\d\varphi^1_t)-J_{\psi^1_0}(\nabla_{\ddtatzero}\d\varphi_tX)\\
~&~&-\nabla_{\ddtatzero}(J_{\psi^2_t}\d\varphi^2_t)+J_{\psi^2_0}(\nabla_{\ddtatzero}\d\varphi^2_tX)\\
~&=&\nabla_{\ddtatzero}\d\varphi^1_tJX-\nabla_{\ddtatzero}\d\varphi^2_tJX-J_{\psi_0}(\nabla_Xa_1-\nabla_Xa_2)\\
~&=&\text{~(since~$a_1=a_2$)~}\nabla_{JX}a_1-\nabla_{JX}a_2=0\text{,}\end{array}\]\addtolength{\arraycolsep}{1.0mm}\noindent
where we have used Lemma
\ref{Lemma:FirstAuxiliarLemmaToProve:Proposition:UnicityToFirstOrderOfHorizontalHolomorphicLifts},
as well as the fact that $J_{\psi^2_0}=J_{\psi^1_0}$ and
$\nabla_{\ddtatzero}\d\varphi^i_tX=\nabla_X\left.\frac{\partial\varphi^i_t}{\partial{t}}\right|_{t=0}=\nabla_Xa^i$.

(ii) For $Y=\d\varphi_0JX$ we use similar arguments.

(iii) Taking $Y=u_0$, we have
\[Q(u_0)=\nabla_{\ddtatzero}(J_{\psi^1_t}u^1_t)-J_{\psi^1_0}(\nabla_{\ddtatzero}u^1_t)-\nabla_{\ddtatzero}(J_{\psi^2_t}u^2_t)+J_{\psi^2_0}(\nabla_{\ddtatzero}u^2_t)\text{,}\]
equivalently,
\begin{equation}\label{Equation:FirstEquationInTheProofOf:Proposition:UnicityToFirstOrderOfHorizontalHolomorphicLifts}
Q(u_0)=-\nabla_{\ddtatzero}v^1_t+\nabla_{\ddtatzero}v^2_t-J_{\psi_0}(\nabla_{\ddtatzero}u^1_t-\nabla_{\ddtatzero}u^2_t)\text{.}
\end{equation}
But\addtolength{\arraycolsep}{-1.0mm}
\[\begin{array}{lll}\nabla_{\ddtatzero}v^1_t&=&-\nabla_{\ddtatzero}\nabla_X\d\varphi^1_tJX-\nabla_{\ddtatzero}\nabla_{JX}\d\varphi^1_tX\\
~&=&-R(\left.\frac{\partial\varphi^1_t}{\partial{t}}\right|_{t=0},\d\varphi^1_tX)\d\varphi^1_tJX+\nabla_X\nabla_{\ddtatzero}\d\varphi^1_tJX\\
~&~&+R(\left.\frac{\partial\varphi^1_t}{\partial{t}}\right|_{t=0},\d\varphi^1_tJX)\d\varphi^1_tX+\nabla_{JX}\nabla_{\ddtatzero}\d\varphi^1_tX\\
~&=&\hspace{-0.4mm}-R(a_1,\d\varphi_0X)\d\varphi_0JX\hspace{-0.3mm}+\hspace{-0.3mm}\nabla_X\hspace{-0.1mm}\nabla_{JX}a_1\hspace{-0.3mm}+\hspace{-0.3mm}R(a_1,\d\varphi_0JX)\d\varphi_0X\hspace{-0.3mm}+\hspace{-0.3mm}\nabla_{JX}\nabla_{X}a_1\text{.}\end{array}\]\addtolength{\arraycolsep}{1.0mm}\noindent
As $a_1=a_2$ we deduce $\nabla_{\ddtatzero}v^1_t=\nabla_{\ddtatzero}v^2_t$; with similar
reasoning $\nabla_{\ddtatzero}u^1_t=\nabla_{\ddtatzero}u^2_t$ so that the right-hand side of \eqref{Equation:FirstEquationInTheProofOf:Proposition:UnicityToFirstOrderOfHorizontalHolomorphicLifts} vanishes and, consequently, $Q(u_0)=0$.

(iv) For $Y=v_0$ we use similar arguments to the above, concluding our proof.
\end{proof}

Hence, the twistor lifts constructed in Theorem \ref{Theorem:HorizontalHolomorphicLiftsUpToFirstOrderInThe4DimensionalCase} are \emph{unique to first order}, in the sense that the vector field $w$ induced on $\Sigma^+N^4$ (or $\Sigma^-N^4$) by the map $\psi$, $w=\ddtatzero\psi_t$ depends only on the initial projected map $\varphi_0$ and on the Jacobi field $v$ along $\varphi_0$.


\section{Summary}


We have found the first order analogues of the results in Chapter \ref{Chapter:HarmonicMapsAndTwistorSpaces}. Namely, we established the following correspondences:

\setlength{\unitlength}{1cm}
\begin{picture}(20,10)(0,0)
\put(0,8){$\begin{array}{ll}\psi:M^{2}\to\Sigma^+N&\J^2\text{-holomorphic}\\
~&\text{to~first~order}\end{array}$}
\put(6.2,8.1){\vector(1,0){1}} \put(7.2,8.1){\vector(-1,0){1}}
\put(7.3,8){$\begin{array}{l}\varphi\text{~harmonic (and conformal)}\\\text{to first order}\end{array}$}
\put(0,6){$\psi:M^{2}\to\Sigma^+N$ $\J^1$-holomorphic}
\put(6.2,6.1){\vector(1,0){1}} \put(7.3,6){$\varphi$ real isotropic to first order}
\put(6.2,5.8){$ \xy   {\ar@/^1pc/(10,0)*{};(0,0)*{}};\endxy$}
\put(6.9,5.2){if $\dim  N=4$, possibly to $\Sigma^-N$}
\put(0,3.5){$\begin{array}{ll}\psi:M^{2}\to\Sigma^+N&\H\text{-holomorphic}\\~&\text{to first order}\end{array}$}
\put(6.2,3.6){\vector(1,0){1}}\put(7.2,3.6){\vector(-1,0){1}}
\put(7.3,3.5){$\begin{array}{ll}\varphi&(J^M,J_\psi)\text{-holomorphic}\\~&\text{to first order}\end{array}$}
\end{picture}

Moreover, taking $N^4$ the $4$-sphere or the complex projective plane, letting $\varphi:M^2\to N^4$ be a harmonic map and $v\in\varphi^{-1}(TN)$ is a Jacobi field, real isotropy to first order is immediately guaranteed (Proposition \ref{Proposition:ConformalityAndRealIsotropyToFirstOrderFromHarmonicityToFirstOrderAndSpheres}). Hence, the previous construction allows a (local) unified proof of the twistor correspondence between Jacobi fields and twistor vector fields that are tangent to variations on $\Sigma^+N^4$ which are simultaneously $\J^1$ and $\J^2$-holomorphic (\emph{infinitesimal horizontal holomorphic deformations} in \cite{LemaireWood:07}). We can also conclude that different what properties (namely, conformality, real isotropy or harmonicity) are related with those of the twistor lift (respectively, $\H$, $\J^1$ or $\J^2$-holomorphicity).





\renewcommand{\thechapter}{\Alph{chapter}}
\renewcommand{\chaptername}{}             
\addtocounter{chapter}{-5}                


\titleformat{\chapter}[display]          
{\bf\Huge}{\chaptertitlename}{0pt}{}{}   
\titlespacing{\chapter}{0pt}{-8pt}{15pt} 

\addtocounter{chapter}{-1} 
\chapter{Appendix}\label{Chapter:Appendix}
\addtocounter{chapter}{1}  




\section{Complex Lie Groups and the proof of Theorem \ref{Theorem:ComplexStructureOnTheGrassmannianOfThekDimensionalIsotropicPositiveSubspacesOfcnTwok}}\label{Section:SectionIn:Chapter:Appendix:ComplexLieGroupsAndProofOf:Theorem:ComplexStructureOnTheGrassmannianOfThekDimensionalIsotropicPositiveSubspacesOfcnTwok}


In this section, we prove Theorem \ref{Theorem:ComplexStructureOnTheGrassmannianOfThekDimensionalIsotropicPositiveSubspacesOfcnTwok}.
We start by establishing the following
\begin{proposition}\label{Proposition:ComplexLieGroupsAndDensities}Let $E$ be a (finite-dimensional) complex vector space and let $\L_\cn(E,E)$ denote the space of all complex linear endomorphisms of $E$. Let $\G$ be a complex Lie group contained in $\GL(\cn,E)$ and $F$ a complex linear subspace of $\L_\cn(E,E)$ such that $\G\cap F\subseteq\L_\cn(E,E)$ is a subgroup of $\G$ (not necessarily closed or holomorphic) and $\G\cap V$ is dense in $F$. Then, $\G\cap F$ is a complex Lie subgroup of $\G$ and
\begin{equation}\label{Equation:FirstEquationIn:Proposition:ComplexLieGroupsAndDensities}
T_e(\G\cap F)=T_e\G\cap F\text{.}
\end{equation}
\end{proposition}

In order to prove this proposition, we shall need the following
\begin{lemma}\label{Lemma:LemmaToProve:Proposition:ComplexLieGroupsAndDensities}
Under the same hypothesis as in Proposition \ref{Proposition:ComplexLieGroupsAndDensities}, $\G\cap F$ is a real Lie subgroup of $\G$ and equation \eqref{Equation:FirstEquationIn:Proposition:ComplexLieGroupsAndDensities} is satisfied.
\end{lemma}

\noindent\textit{Proof of Lemma \ref{Lemma:LemmaToProve:Proposition:ComplexLieGroupsAndDensities}.}
Since $F$ is a vector subspace of $\L_\cn(E,E)$, we can deduce it is closed. Thus, $\G\cap F$ is closed in $\G$ and therefore a real Lie subgroup of it. Hence, it makes sense to write $T_e(\G\cap F)$.

It is clear that $T_e(\G\cap F)\subseteq T_e\G\cap F$. Conversely, let $v\in{T}_e\G\cap F$. Since $\G\cap F$ is dense in $F$ (and, hence, in $\G\cap F$), there is a sequence $v_n\in\G\cap F$ which tends to $v$. As $\G$ is a subgroup of $\GL(E)$, its exponential map is just the restriction of the exponential map of the latter group. The same holds for $\exp_{\G\cap F}=\exp_{\G}|_{\G\cap F}$. But
\[\exp(t.v_n)=\sum_{k\geq 0}\frac{(tv_n)^k}{(k)!}\in{F}\]
since
\[v_n\in\G\cap{F}\,\impl\,v_n^2\in\G\cap{F}\,\impl\,...\,\impl\,v_n^k\in{F}\,\impl\,\sum_{k\geq0}\frac{(t.v_n)^k}{k!}\in{F}\text{.}\]
Thus, $v_n\to v\,\impl\,\exp(t.v_n)\to \exp(t.v)$ and $\exp(t.v_n)\in{F}$ implies $\exp(t.v)\in{F}$. Hence, $\exp(t.v)$ lies in $\G\cap F$ and
\[v=\ddtatzero\big(\exp(t.v)\big)\,\impl\,v\in{T}_e(\G\cap F)\]
concluding our proof.\qed

\noindent\textit{Proof of Proposition \ref{Proposition:ComplexLieGroupsAndDensities}.} Since $\G\cap F$ is
closed in $\G$, it is a (real) subgroup of the latter. To prove that it is complex is then enough to show its stability under the complex structure. But $T_e(\G\cap F)=T_e\G\cap F$, by the previous Lemma. Since both these spaces are stable under the complex structure, so is their intersection our argument is concluded.\qed


\subsection{Proof~of~Theorem~\ref{Theorem:ComplexStructureOnTheGrassmannianOfThekDimensionalIsotropicPositiveSubspacesOfcnTwok}}\label{Section:ProofOf:Theorem:ComplexStructureOnTheGrassmannianOfThekDimensionalIsotropicPositiveSubspacesOfcnTwok}


\noindent\textit{1$^{\text{st}}$ step. The action
\eqref{Equation:FirstEquationIn:Theorem:ComplexStructureOnTheGrassmannianOfThekDimensionalIsotropicPositiveSubspacesOfcnTwok}.}

We show that the map defined in \eqref{Equation:FirstEquationIn:Theorem:ComplexStructureOnTheGrassmannianOfThekDimensionalIsotropicPositiveSubspacesOfcnTwok}, \addtolength{\arraycolsep}{-1.2mm}
\[\begin{array}{ccclc}
\SO(\cn,E^\cn)&\times& G^+_{iso}(E^\cn)& \to & G^+_{iso}(E^\cn) \\
(\lambda&,&F)&\to&\lambda(F)
\end{array}\]\addtolength{\arraycolsep}{1.2mm}\noindent
is a well-defined transitive left action of the group $\SO(\cn,E^\cn)$ on the set $G^+_{iso}(E^\cn)$.

\noindent\textit{Proof of the 1$^{\text{st}}$ step.} The first non-trivial task is to show that indeed $\lambda(F)$ lies in $G^+_{iso}(E^\cn)$. It is not difficult to check that $\lambda(F)$ is in fact an isotropic $k$-dimensional complex linear subspace of $E^\cn$; however, it is harder to see that this subspace remains positive\footnote{If we wanted the Grassmannian of just isotropic subspaces (so, dropping the ``positive" condition), we could change our complex Lie group to $\O(\cn,E^\cn)$ and this part of our proof would be unnecessary.}. Take $\{u_i,J_Fu_i\}_{i=1,...,k}$ a positive basis of $E$, where $J_F$ is the Hermitian structure determined by $F$ (hence, $u_1-iJ_Fu_1,...,u_k-iJ_Fu_k$ is a basis of $F$). Then, $w_i=\lambda(u_i-iJ_Fu_i)=\lambda(u_i)-i\lambda(J_Fu_i)$,
$i=1,...,k$ forms a basis for $\lambda(F)$ and $\overline{w}_i$ is a
basis for $\overline{\lambda(F)}$. Notice that, in general, we do not have $\overline{\lambda(F)}=\lambda(\overline{F})$ and this complicates the proof, since we cannot conclude that a basis for $\overline{\lambda(F)}$ is given by $\lambda(u_i)+i\lambda(Ju_i)$. Take
\[\left\{\begin{array}{l}
v_i=w_i+\overline{w}_i\in{E} \quad(w_i\in\lambda(F),\,\overline{w}_i\in\overline{\lambda(F)})\\
\tilde{v}_i=i(w_i-\overline{w}_i)\in{E}\text{.}
\end{array}\right.\]
It is clear that $\{v_i,\tilde{v}_i$, $i=1,...,k\}$ is a basis for $E$ and that
\[J_{\lambda(F)}(v_i)=J^\cn_{\lambda(F)}(w_i+\overline{w}_i)=iw_i-i\overline{w}_i=\tilde{v}_i\text{.}\]
Hence, all we have to verify is that $\{v_1,\tilde{v}_1,...,v_k,\tilde{v}_k\}$ is a positive basis for $E$.
Consider the map\addtolength{\arraycolsep}{-1.2mm}
\[\begin{array}{lcll}
\tilde{\alpha}: &E&\to& E \\
~& u_i & \to & v_i \\
~& J_Fu_i & \to & \tilde{v}_i \,(=J_{\lambda(F)}v_i)
\end{array}\]\addtolength{\arraycolsep}{1.2mm}\noindent
so that what we have to check is that it has positive determinant. Take the new map\addtolength{\arraycolsep}{-1.2mm}
\[\begin{array}{llll}
\alpha: & E&\to& E\\
~ & u&\to&\lambda u-i\lambda{J}_Fu+\overline{\lambda u}+i\overline{\lambda{J}_Fu}\text{.}
\end{array}\]\addtolength{\arraycolsep}{1.2mm}\noindent
Then, $\alpha(u_i)=v_i$ and $\alpha(J_F u_i)=\tilde{v}_i$; consequently, $\alpha=\tilde{\alpha}$. But $\alpha=(\lambda-i\lambda J_F)+(\overline{\lambda}+i\overline{\lambda} J_F)$, where $\overline{\lambda}(u)=\overline{\lambda(u)}$. Next, consider\addtolength{\arraycolsep}{-1.2mm}
\[\begin{array}{lcll}
\beta: & \L_\cn(E^\cn,E^\cn)&\to&\L_\rn(E,E) \\
~& \lambda&\to & \Real\lambda+\Imag\lambda\circ J_F
\end{array}\]\addtolength{\arraycolsep}{1.2mm}\noindent
where $(\Real\lambda)u=\Real(\lambda u)$ and analogously for the imaginary part. This map is continuous and therefore so is the map\addtolength{\arraycolsep}{-1.2mm}
\[\begin{array}{lcll}
\det\circ\beta: & \L_\cn(E^\cn,E^\cn)&\to &\rn \\
~& \lambda &\to &\det(\beta(\lambda))\text{.}
\end{array}\]\addtolength{\arraycolsep}{1.2mm}\noindent
Moreover, it takes the value one at the identity map, as $\Real I+\Imag I\circ J_F=I$. Since $\SO(\cn,E^\cn)$ is connected, $\det(\beta(\lambda))$ is not zero ($\Real\lambda+\Imag\lambda\circ J_F$ maps $u_i$ to $v_i$ and $J_Fu_i$ to $\tilde{v}_i$) and $\det\beta(I)=1$, we can conclude that $\det \beta(\lambda)>0$, for every $\lambda\in\SO(\cn,E^\cn)$, as desired.

That we have a left action does not require any arguments and to prove its transitivity we must show that for any isotropic positive $k$-dimensional linear subspaces $F_1$ and $F_2$, there is
$\lambda\in\SO(\cn,E^\cn)$ with $\lambda(F_1)=F_2$. Indeed, if we take $J_{F_1}$ and $J_{F_2}$, we know that there is a
$\tilde{\lambda}\in\SO(E)$ such that $\tilde{\lambda}\circ{J}_{F_1}\circ\tilde{\lambda}^{-1}=J_{F_2}$ \footnote{Recall the argument to give a differentiable structure on $\Sigma^+E$ as the quotient $\SO(E)/\U(E)$: the left action $\SO(E)\times\Sigma^+E\to\Sigma^+E$ defined by $(\lambda,J)\to\lambda J\lambda^{-1}$ is a transitive left action with isotropy subgroup $\U(E)$.}. Hence, we can consider $\lambda$ as the complexified $\tilde{\lambda}$ and it is clear that
$\lambda\in\SO(\cn,E^\cn)$; that $\lambda(F_1)=F_2$ follows easily. Notice that this argument shows that we could have considered the transitive action of the group $\SO(E)$ on $G^+_{iso}(E^\cn)$ instead of using the group $\SO(\cn,E^\cn)$.

\noindent\textit{2$^{\text{nd}}$ step. The complex structure on the manifold $G^+_{iso}(E^\cn)$.}

We show that the isotropy subgroup of the action
\eqref{Equation:FirstEquationIn:Theorem:ComplexStructureOnTheGrassmannianOfThekDimensionalIsotropicPositiveSubspacesOfcnTwok}
is a complex Lie subgroup of $\SO(\cn,E^\cn)$.

\noindent\textit{Proof of the 2$^{\text{nd}}$ step.} Identifying $E$ with $\rn^{2k}$ (so that $E^\cn\simeq\cn^{2k}$) and choosing the positive isotropy subspace $F_0=T^{10}_{J_0}\rn^{2k}$ associated with the usual complex structure $J_0$ on $\rn^{2k}$, we can reduce ourselves to the study of the isotropy group at this point. This isotropy subgroup $\K_{F_0}$ is given by\addtolength{\arraycolsep}{-1.0mm}
\[\begin{array}{lll}\K_{F_0}&=&\{\lambda\in\SO(\cn,2k):\,\lambda(F_0)=F_0\}=\{\lambda\in\SO(\cn,2k):\,\lambda(F_0)\subseteq{F}_0\}\\
~&=&\underbrace{\SO(\cn,2k)\cap\underbrace{\{\lambda\in\L(\cn,2k):\,\lambda(F_0)=F_0\}}_{\text{\tiny{complex~subspace~of~$\L(\cn,2k)$}}}}_{\text{\tiny{closed~in~$\SO(\cn,2k)$}}}\text{.}
\end{array}\]\addtolength{\arraycolsep}{1.0mm}\noindent
On the other hand,\addtolength{\arraycolsep}{-1.0mm}
\[\begin{array}{ll}~&\SO(\cn,2k)\cap\{\lambda\in\L(\cn,2k):\,\lambda(F_0)\subseteq{F}_0\}\\
=&\SO(\cn,2k)\cap\big(\{\lambda\in\L(\cn,2k):\,\lambda(F_0)\subseteq{F}_0\}\cap\GL(\cn,2k)\big)\text{.}
\end{array}\]\addtolength{\arraycolsep}{1.0mm}\noindent
Now, $\{\lambda\in\L(\cn,2k):\,\lambda(F_0)\subseteq F_0\}\cap\GL(\cn,2k)$ is closed in $\GL(\cn,2k)$, it is a subgroup of $\GL(\cn,2k)$ and dense in $\{\lambda\in\L(\cn,2k):\lambda(F_0)\subseteq F_0\}$ by density of $\GL(\cn,2k)$ in $\L(\cn,2k)$. Thus, by Proposition \ref{Proposition:ComplexLieGroupsAndDensities} it is a complex Lie subgroup of $\GL(\cn,2k)$ with tangent space at the identity given by
\[T_{\Id}\GL(\cn,2k)\cap\{\lambda\in\L(\cn,2k):\,\lambda(F_0)\subseteq{F}_0\}=\{\lambda\in\L(\cn,2k):\,\lambda(F_0)\subseteq{F}_0\}\text{.}\]
Hence, $\K_{F_0}$ is the intersection of two complex Lie subgroups and consequently a complex Lie subgroup, with tangent space at the identity given by the intersection of the two tangent spaces,
\begin{equation}\label{Equation:EquationInTheProofOf:Theorem:ComplexStructureOnTheGrassmannianOfThekDimensionalIsotropicPositiveSubspacesOfcnTwok.the.tgt.space.of.the.isotropy.subgroup}
T_{\Id}\K_{F_0}=T_{\Id}\SO(\cn,2k)\cap\{\lambda\in\L(\cn,2k):\,\lambda(F_0)\subseteq{F}_0\}\text{.}
\end{equation}\qed

\section{Riemann~surfaces~and~pluriharmonic~maps}\label{Section:SectionIn:Chapter:Addendums:RiemannSurfacesAndPluriharmonicMaps}\index{Riemann~surface!pluriharmonic~maps}


\index{Conformal~structure!}\index{Conformal~manifold}\index{Metric!conformally~equivalent}Let
$(M^m,g)$ be a Riemannian manifold. Two metrics $g$ and $\tilde{g}$ on $M$ are said to be \emph{conformally equivalent} if there is $\lambda:M\to (0,+\infty)$ with $\tilde{g}=\lambda g$ (see \cite{BairdWood:03}, p. 30). In other words, if the identity map of $M$ is conformal as a map $(M,\tilde{g})\to (M,g)$. An equivalence class is called a \textit{conformal structure} and a manifold equipped with a conformal structure is called a \emph{conformal manifold}.

When $M^{2}$ is a two dimensional Riemannian manifold, we then have
the following classical result (see \cite{Spivak:79}, Vol. IV):

\begin{proposition}[\textup{Isothermal~coordinates}]\label{Proposition:IsothermalCoordinates}\index{Isothermal~coordinates}
Suppose that $(M^2,g)$ is a two dimensional Riemannian manifold.

\textup{(i)} Given any point of $M$ there exists a local coordinate chart $(x,y)$ on an open neighbourhood $\openU$ of that point such that
\begin{equation}\label{Equation:FirstEquationIn:Proposition:IsothermalCoordinates}
g=\mu^2(\d x^2+\d y^2)
\end{equation}
for some smooth positive real-valued function $\mu$ on $\openU$. Such coordinates $(x,y)$ are called \emph{isothermal coordinates}.

\textup{(ii)} If $(x,y)$ are isothermal coordinates then
\begin{equation}\label{Equation:SecondEquationIn:Proposition:IsothermalCoordinates}
\nabla_{\partial_x}\partial_x+\nabla_{\partial_y}\partial_y=0
\end{equation}
where $\nabla$ is the induced Levi-Civita connection on $M^2$. Moreover, in the case $M^2$ is oriented, we can choose a system $(x_\alpha,y_\alpha)$ of such charts compatible with the orientation on $M^2$ that give to our manifold a system of complex charts on writing $z_\alpha=x_\alpha+iy_\alpha$. We call this the \textup{induced complex structure} on $M^2$.
\end{proposition}

Notice that the induced \textit{complex} structure $J$ on $M^2$ is defined in isothermal coordinates $(x,y)$ by $J\partial_x=\partial_y$ and the manifold $(M^2,g,J)$ is Hermitian. Moreover, if $\tilde{g}$ is conformally equivalent to $g$, it induces the same complex structure $J$. Thus, the complex structure on $M^2$ depends only on the conformal class of the metric $g$: an oriented two-dimensional conformal manifold equipped with this complex structure is called a \textit{Riemann surface}. Furthermore, the concept of harmonic map $\varphi:M^2\to N$ does not depend on the particular choice of metric within the conformal structure (as the characterization of harmonic map in Proposition \ref{Proposition:HarmonicMapsFromRiemannSurfaces} only involves the complex structure of $M^2$ and not its connection). We now wish to examine the converse: given a one-dimensional complex manifold, do we have an induced conformal structure? If this is true, then the concept of harmonic map is also well-defined on one-dimensional complex manifolds. Indeed, we have the following lemma (\cite{EellsLemaire:78}, p. 43):

\begin{lemma}\label{Lemma:OneDimensionalComplexManifoldsAndInducedConformalStructures}\index{Metric!Hermitian}\index{Riemann~surface!}\index{Complex~structure!induced}\index{Harmonic~map!from~a~Riemann~surface}\index{Harmonic~map!from~a~complex~curve}\index{Conformal~structure!from~a~complex~structure}
Let $(M^2,J)$ be a one-dimensional complex manifold. We say that $g$ is a \emph{Hermitian metric} if $J$ is an isometry with respect to $g$. Then, any two Hermitian metrics on $(M^2,J)$ are conformally equivalent.
\end{lemma}
\begin{proof}We start by showing that $g$ in Hermitian if and only if there is a nonzero vector with $g(X,X)=g(JX,JX)$ and $g(X,JX)=0$. The ``only if" implication is trivial; conversely, if there is such a vector, as $g$ is given on $TM$ by the bilinear decomposition of $g(a_1X+a_2JX,b_1X+b_2JX)$, $g$ becomes Hermitian. Now, let $g_1$ and $g_2$ be two
Hermitian metrics on $(M,J)$. Let $X_1$ and $X_2$ as before: $X_i$ with $g_i(X_i,X_i)=g(JX_i,JX_i)$ and $g_i(X_i,JX_i)=0$. Write $X_1=a X_2+bJX_2$ so that also $JX_1=-bX_2+aJX_2$. Hence, $g_2(X_1,X_1)=g_2(X_2,X_2)=(a^2+b^2)\lambda g_1(X_1,X_1)$ where $\lambda=(a^2+b^2)\frac{g_2(X_2,X_2)}{g_1(X_1,X_1)}$, as well as $g_2(JX_1,JX_1)=g_2(X_1,X_1)=\lambda{g}_1(X_1,X_1)=\lambda g_1(JX_1,JX_1)$ and $g_2(X_1,JX_1)=0=\lambda g(X_1,JX_1)$ from which it follows that $g_1$ and $g_2$ are conformally equivalent.
\end{proof}

As a consequence, we have
\begin{corollary}\label{Corollary:GoodDefinitionOfPluriharmonicMapOnAComplexManifold}\index{Pluriharmonic~map!on~a~complex~manifold}
Let $(M,J)$ be a complex manifold and $N$ any Riemannian manifold. Then, the concept of pluriharmonic maps $\varphi:M\to{N}$ it is well defined as those whose restriction to complex curves (\textit{i.e.}, one-dimensional complex submanifolds) on $M$ are harmonic maps.
\end{corollary}

\begin{proposition}[Pluriharmonic~and~$(1,1)$-geodesic~maps]\label{Proposition:PluriharmonicAnd(1,1)GeodesicsMapsFromKahlerManifolds}\index{Pluriharmonic~map!and~$(1,1)$-geodesic~maps}\index{$(1,1)$-geodesic~map!and~pluriharmonic~maps}
If $(M,J,g)$ is a Kähler and $N$ a Riemannian manifold, a smooth map $\varphi:M\to N$ is pluriharmonic if and only if it is $(1,1)$-geodesic.
\end{proposition}
In particular, if we have a complex manifold $(M,J)$ and a pluriharmonic map $\varphi:M\to N$, the existence of a Kähler metric on $(M,J)$ guarantees that equation \eqref{Equation:EquationIn:Definition:11GeodesicMap}:
\[(\nabla\,\d\varphi)(Y^{10},Z^{01})=0,\quad\forall\,Y^{10}\in{T}^{10}M,\,Z^{01}\in{T}^{01}M\]
holds.
\begin{proof}
Let $(M,g,J)$ be a Kähler manifold and $\varphi:M\to N$ pluriharmonic. We show that the above displayed equation holds. Fix complex coordinates $(z_1,...,z_m)$ on $M$. Then,
\[(\nabla\d\varphi)(\partial_{z_i},\partial_{\bar{z}_i})=\nabla^{\varphi^{-1}}_{\partial_{\bar{z}_i}}(\d\varphi(\partial_{z_i}))-\d\varphi(\nabla^M_{\partial_{\bar{z}_i}}\partial_{z_i})=\nabla^{\varphi^{-1}}_{\partial_{\bar{z}_i}}(\d\varphi(\partial_{z_i}))\text{,}\]
since $M$ is Kähler (see the next lemma). Hence, all we have to check is that the first term above vanishes. Now, because $\varphi$ is pluriharmonic, we know that $\varphi\circ c$ is harmonic for any complex curve
$c:\cn\to M$. Thus,
\[\nabla_{\partial_{\bar{z}}} \big(d(\varphi\circ c)(\partial_z)\big)=0\text{,}\]
since $\nabla d(\varphi\circ c)(\partial_{\bar{z}},\partial_z)=0$ and $\d(\varphi\circ c)(\nabla_{\partial_{\bar{z}}}
\partial_z)=0$. Taking for each $X^{10}\in{T}^{10}M$ a complex curve with $\d c(\partial_z)=X^{10}$, we can therefore deduce $\big(\nabla\d\varphi\big)(X^{10},X^{01})=0$. Our proof follows easily by symmetry of the second fundamental form $\nabla \d\varphi$ as usual. Conversely, let $(M,J,g)$ be a Kähler manifold and $\varphi:M\to N$ a $(1,1)$-geodesic map. Taking a complex curve $c$, $(\nabla\d\varphi)(\d{c}(\partial_{\bar{z}}),\d{c}(\partial_z))=0$ and the proof follows as in the first part.
\end{proof}

\begin{lemma}\label{Lemma:ExistenceOfLocallyDefinedKahlerMetricsForComplexManifolds}\index{Metric!Kähler}\index{Levi-Civita~connection!on~a~Kähler~manifold}\index{Kähler~manifolds,~coordinates~and~Levi-Civita~connection}
Let $(M,J)$ be a complex manifold. Then, at least locally, there is a Kähler metric on $M$. For such a metric, if $\nabla$ denotes the Levi-Civita connection induced on $M$ and $(z_1,...,z_m)$ are holomorphic coordinates, we have
\begin{equation}\label{Equation:EquationIn:Lemma:ExistenceOfLocallyDefinedKahlerMetricsForComplexManifolds}
\nabla_{\partial_{z_i}}{\partial_{\bar{z}_j}}=\nabla_{\partial_{\bar{z}_j}}\partial_{z_i}=0.
\end{equation}
\end{lemma}
\begin{proof}
Let $\varphi:M\to\cn^m$ be a holomorphic chart for $M$ and consider the pull-back metric $g$ of the standard metric on $\cn^m$ (or, rather, on the underlying $\rn^{2m}$). Let us prove that this metric is, indeed, Kähler: decidedly $g$ is Hermitian, since $\varphi$ is holomorphic. To check that it is Kähler we proceed as follows: since $(M,g,J)$ is Hermitian, all we have to do is to check that it is also $(1,2)$-symplectic, which will follow if $\nabla_{\partial_{\bar{z}_i}}\partial_{z_j}\in{T}^{10}M$;
equivalently, if $g(\nabla_{\partial_{\bar{z}_i}}\partial_{z_j},\partial_{z_k})=0$. Since $g$ is an isometry, this is equivalent to proving that $<\nabla_{\tilde{\partial}_{\bar{z}_i}}\tilde{\partial}_{z_j},\tilde{\partial}_{z_k}>=0$
where $(<>,\nabla)$ are the canonical metric and connection on $\cn^m$ and $\tilde{\partial}_{\bar{z}_i},\tilde{\partial}_{z_j}$ the canonical vectors on $\cn^m$. As the latter identity is trivial, we deduce that $g$ is Kähler, as required. As for equation \eqref{Equation:EquationIn:Lemma:ExistenceOfLocallyDefinedKahlerMetricsForComplexManifolds},
all we have to do is to show that it holds for any Kähler manifold $(M,g,J)$. Fix a system of complex coordinates
$(z_1,...,z_m)$ (notice that here we need $M$ to be complex!). Then, $\nabla_{\partial_{\bar{z}_i}}\partial_{z_j}$ lies in $T^{10}M$, since $M$ is $(1,2)$-symplectic\footnote{Notice that we do need to use both properties of Kähler manifolds: not only that $(M,g,J)$ is integrable but also $(1,2)$-symplectic.}, using Definition \ref{Definition:TheMainTypesOfAlmostComplexManifolds}. But $\nabla_{\partial_{\bar{z}_i}}\partial_{z_j}=\nabla_{\partial_{z_i}}\partial_{\bar{z}_j}$ as $[\partial_{\bar{z}_i},\partial_{z_j}]=0$ so that $\nabla_{\partial_{\bar{z}_i}}\partial_{z_j}$ also lies in $T^{01}M$ and therefore vanishes.
\end{proof}

\begin{corollary}\label{Corollary:PluriharmonicMapsAre(1,1)GeodesicMaps}
Given a pluriharmonic map $\varphi$ from a complex manifold $(M,J)$, for any Kähler metric on $M$, $\varphi$ is a $(1,1)$-geodesic map. In particular, it is harmonic.
\end{corollary}
\begin{proof}Fix (locally) a Kähler metric on $M$, which is always possible from the previous lemma. Then, using Proposition \ref{Proposition:PluriharmonicAnd(1,1)GeodesicsMapsFromKahlerManifolds} the result is immediate.
\end{proof}

\begin{corollary}\label{Corollary:PluriharmonicMapsAreHarmonic}\index{Pluriharmonic~map!and~harmonic~maps}
If $(M,J)$ is a complex manifold and $\varphi:M\to N$ is pluriharmonic, $\varphi$ is harmonic for any Kähler metric on $M$.
\end{corollary}
\begin{proof}Immediate from the preceding corollary and the fact that $(1,1)$-geodesic maps are harmonic (p.\ \pageref{Page:(11)GeodesicMapsAreHarmonic}).
\end{proof}

\section{Real and complex isotropy}\label{Section:SectionIn:Chapter:Addendums:RealAndComplexIsotropicMaps}


\index{Real~isotropic~map!and~complex~isotropic~map}\index{Complex~isotropic~map!and real~isotropic~map} Recall that given a map $\varphi:M^2\to N$ from a Riemann surface to an arbitrary Riemannian manifold, $\varphi$ is real isotropic if equation \eqref{Equation:EquationIn:Definition:RealIsotropicMapFromARiemannSurface}
holds:
\[<\partial_z^r\varphi,\partial_z^s\varphi>=0,\quad\forall\,r,s\geq1\]
where $\partial_z\varphi=(\d\varphi)^\cn(\partial_z)$, $<,>$ is the complex bilinear extension of the metric on $N$ and
$\partial^r_z\varphi=\nabla^{r-1}_{\partial^{r-1}_z}\partial_z\varphi$. We shall need the following fact:
\begin{lemma}\label{Lemma:AnInductionProofREalIsotropicMapsAndOnlyTheRRCase}
Let $M^2$ be a Riemann surface and $\varphi:M^2\to N^{2n}$. Then, $\varphi$ is real isotropic if and only if
\begin{equation}\label{Equation:FirstEquationIn:EquationIn:Lemma:AnInductionProofREalIsotropicMapsAndOnlyTheRRCase}
<\partial_z^r\varphi,\partial_z^r\varphi>=0,\quad\forall\,r\geq1\text{.}
\end{equation}
Analogously, given a map $\varphi:I\times M^2\to N^{2n}$, $\varphi$ is real isotropic to first order if and only if
\begin{equation}\label{Equation:SecondEquationIn:EquationIn:Lemma:AnInductionProofREalIsotropicMapsAndOnlyTheRRCase}
<\partial_z^r\varphi_t,\partial_z^r\varphi_t>\text{~is $o(t)$.}
\end{equation}
\end{lemma}
Thus, to check isotropy, it is enough to establish equations \eqref{Equation:EquationIn:Definition:RealIsotropicMapFromARiemannSurface} and \eqref{Equation:Definition:RealIsotropicMapsToFirstOrder} for $r=s$.

\begin{proof}In the non-parametric case, we want to prove \eqref{Equation:EquationIn:Definition:RealIsotropicMapFromARiemannSurface} from
\eqref{Equation:FirstEquationIn:EquationIn:Lemma:AnInductionProofREalIsotropicMapsAndOnlyTheRRCase}. We shall prove by induction on $j=|r-s|$: if $j=0$ we obtain
\eqref{Equation:FirstEquationIn:EquationIn:Lemma:AnInductionProofREalIsotropicMapsAndOnlyTheRRCase}
and there is nothing left to prove. Assume now that our result is
valid for all $j\leq n$. Take $r,s\geq1$ with $|r-s|=n+1$. Without loss of generality, we
may assume that $r\geq s$, $r=s+n+1$ and we get
\[<\partial_z^{s+n+1}\varphi,\partial_z^s\varphi>=\partial_z<\partial_z^{s+n}\varphi,\partial_z^s\varphi>-<\partial^{s+n}_z\varphi,\partial_z^{s+1}\varphi>\text{.}\]
Since $|s+n-s|=n$, $<\partial_z^{s+n}\varphi,\partial_z^s\varphi>=0$ and the first term in the above expression vanishes. As for the second, we get $|s+n-s-1|=|n-1|\leq n$ and therefore also the second term vanishes, concluding this part our proof. As for the parametric case, we have $<\partial^r_z\varphi_0,\partial^s_z\varphi_0>=0$ for all $r,s\geq1$ if and only if $<\partial^r_z\varphi_0,\partial^r_z\varphi_0>=0$ for all $r\geq 1$ as in the first part. As for the $t$-derivative at zero: assume that $\ddtatzero<\partial_z^r\varphi_t,\partial^r_t\varphi_t>=0$ for all $r\geq1$. Again setting $j=|r-s|$, the case $j=0$ is trivial. Assuming our result valid for all $j\leq n$ and taking $r,s$ with $|r-s|=n+1$, $r=s+n+1$, we have
\[\ddtatzero\hspace{-2mm}<\partial_z^{s+n+1}\varphi_t,\partial_z^s\varphi_t>=\ddtatzero\hspace{-2mm}\partial_z<\partial^{s+n}_z\varphi_t,\partial^s_z\varphi_t>-\ddtatzero\hspace{-2mm}<\partial^{s+n}_z\varphi_t,\partial_z^{s+1}\varphi_t>\text{.}\]
Since $\ddtatzero<\partial^r_z\varphi_t,\partial^s_z\varphi_t>=0$
for all $|r-s|\leq n$ and $\ddtatzero\partial_z<\partial^{s+n}_z\varphi_t,\partial^s_z\varphi_t>=\partial_z\ddtatzero<\partial^{s+n}_z\varphi_t,\partial^s_z\varphi_t>$
we conclude our proof with the same arguments as in the non-parametric case.
\end{proof}

As we have seen, if $N$ is a Kähler manifold and $\varphi:M^2\to N$ is a smooth map from a Riemann surface $M^2$, $\varphi$ is complex isotropic if equation \eqref{Equation:ComplexIsotropicMapDefinition}:
\[<\nabla^{r-1}_{\partial^{r-1}_z}\partial^{10}_z\varphi,\nabla^{s-1}_{\partial^{s-1}_{\bar{z}}}\partial^{10}_{\bar{z}}\varphi>_{Herm}=0,\quad\forall\,r,s\geq1\]
holds and complex isotropic to first order if \eqref{Equation:ComplexIsotropicToFirstOrder}:
\[\phi_0\text{~is~complex~isotropic~and~}\ddtatzero<\nabla^{r-1}_{\partial^{r-1}_z}\partial^{10}_z\phi,\nabla^{s-1}_{\partial^{s-1}_{\bar{z}}}\partial^{10}_{\bar{z}}\phi_t>_{Herm}=0,\quad\forall\,r,s\geq1\]
is satisfied. Complex isotropy (to first order) is stronger than real isotropy (to first order):
\begin{proposition}\label{Proposition:ComplexIsotropicMapsAreRealIsotropic}\index{Real~isotropic~map!and~complex~isotropic~map}\index{Complex~isotropic~map!and~real~isotropic~map}
Let $\varphi:M^2\to N$ be a smooth map from a Riemann surface $M^2$ into a Kähler manifold $N$. If $\varphi$ is complex isotropic, it is also real isotropic. Moreover, if $\varphi:I\times M^2\to N$ is complex isotropic to first order, it is also real isotropic to first order.
\end{proposition}
\begin{proof}We start by showing the non-parametric case. Let $\varphi:M^2\to N$ be a complex isotropic map. Writing $z=x+iy$,
\[\partial^{10}_z\varphi=\partial_z\varphi-iJ^N\partial_z\varphi=\frac{1}{2}\{\partial_x\varphi-i\partial_y\varphi-iJ^N(\partial_x\varphi-i\partial_y\varphi)\}\text{.}\]
Then, using the fact that $N$ is Kähler, the left-hand side in \eqref{Equation:ComplexIsotropicMapDefinition} can be rewritten as\addtolength{\arraycolsep}{-1.0mm}
\[\begin{array}{ll}
~&\frac{1}{2}<\hspace{-0.1cm}\nabla^{r-1}_{\partial^{r-1}_z}(\d\varphi(\partial_x\hspace{-0.1cm}-\hspace{-0.1cm}i\partial_y)\hspace{-0.1cm}-\hspace{-0.1cm}iJ^N\d\varphi(\partial_x\hspace{-0.1cm}-\hspace{-0.1cm}i\partial_y)),\overline{\nabla^{s-1}_{\partial^{s-1}_{\bar{z}}}(\d\varphi(\partial_x\hspace{-0.1cm}+\hspace{-0.1cm}i\partial_y)\hspace{-0.1cm}-\hspace{-0.1cm}iJ^N\d\varphi(\partial_x\hspace{-0.1cm}+\hspace{-0.1cm}i\partial_y))}>\\
=&\frac{1}{2}<\hspace{-0.1cm}\nabla^{r-1}_{\partial^{r-1}_z}(\d\varphi(\partial_x\hspace{-0.1cm}-\hspace{-0.1cm}i\partial_y)\hspace{-0.1cm}-\hspace{-0.1cm}iJ^N\d\varphi(\partial_x\hspace{-0.1cm}-\hspace{-0.1cm}i\partial_y)),\nabla^{s-1}_{\partial^{s-1}_z}(\d\varphi(\partial_x\hspace{-0.1cm}-\hspace{-0.1cm}i\partial_y)\hspace{-0.1cm}+\hspace{-0.1cm}iJ^N\d\varphi(\partial_x\hspace{-0.1cm}-\hspace{-0.1cm}i\partial_y))>\\
=&\frac{1}{2}\big\{<\nabla^{r-1}_{\partial^{r-1}_z}\d\varphi(\partial_x-i\partial_y),\nabla^{s-1}_{\partial^{s-1}_z}\d\varphi^{\cn}(\partial_x-i\partial_y)>\\
~&+<J^N\nabla^{r-1}_{\partial^{r-1}_z}\d\varphi(\partial_x-i\partial_y),J^N\nabla^{s-1}_{\partial^{s-1}_z}\d\varphi(\partial_x-i\partial_y)>\\
~&+i\{-<J^N\nabla^{r-1}_{\partial^{r-1}_z}\d\varphi(\partial_x-i\partial_y),\nabla^{s-1}_{\partial^{s-1}_z}\d\varphi^{\cn}(\partial_x-i\partial_y)>\\
~&+<\nabla^{r-1}_{\partial^{r-1}_z}\d\varphi(\partial_x-i\partial_y),J^N\nabla^{s-1}_{\partial^{s-1}_z}\d\varphi^{\cn}(\partial_x-i\partial_y)>\}\big\}\\
=&<\nabla^{r-1}_{\partial^{r-1}_z}\partial_z\varphi,\nabla^{s-1}_{\partial^{s-1}_z}\partial_z\varphi>+i<\nabla^{r-1}_{\partial^{r-1}_z}\partial_z\varphi,J^N\nabla^{s-1}_{\partial^{s-1}_z}\partial_z\varphi>\text{.}
\end{array}\]\addtolength{\arraycolsep}{1.0mm}\noindent
In\hspace{4mm} particular,\hspace{4mm} when\hspace{4mm} $r=s$, \hspace{4mm}writing\hspace{4mm} $X=\nabla^{r-1}_{\partial^{r-1}_{z}}\partial_z\varphi$, \hspace{4mm}it\hspace{4mm} follows\hspace{4mm} that $<X,X>+i<X,JX>=0$. If $X=X_1+iX_2$, the latter condition implies that $<X_1,X_1>=<X_2,X_2>$ and $<X_1,X_2>=0$.
Thus, complex isotropy implies
\[<\nabla^{r-1}_{\partial^{r-1}_z}\partial_z\varphi,\nabla^{r-1}_{\partial^{r-1}_z}\partial_z\varphi>=0,\quad\forall\,r\geq1\text{.}\]
From Lemma \ref{Lemma:AnInductionProofREalIsotropicMapsAndOnlyTheRRCase} we conclude that $\varphi$ is real isotropic.

For the parametric case, let $\varphi:I\times M^2\to N$ be a map complex isotropic to first order. Then, $\varphi_0$ is real isotropic to first order, from the above. Hence, we are left with the $t$-derivative at zero. Repeating the argument above with $\ddtatzero$, we conclude that
\[\ddtatzero\{<\nabla^{r-1}_{\partial^{r-1}_z}\partial_z\varphi,\nabla^{r-1}_{\partial^{r-1}_z}\partial_z\varphi>\}=0,\quad\forall\,r\geq1\text{.}\]
Now, using Lemma \ref{Lemma:AnInductionProofREalIsotropicMapsAndOnlyTheRRCase} (equation \eqref{Equation:SecondEquationIn:EquationIn:Lemma:AnInductionProofREalIsotropicMapsAndOnlyTheRRCase}), we deduce that $\varphi$ is real isotropic to first order, concluding the proof.
\end{proof}

\section{A~parametric~Koszul-Malgrange~Theorem}\label{Section:SectionIn:Chapter:Addendums:KoszulMalgrangeVariationalTheorem}


In this section we prove Theorem \ref{Theorem:KoszulMalgrangeTheorem:ParametricVersion}. We shall need a parametric version of the Frobenius Theorem, so that we divide this section in two: in the first, we deal with this parametric version of the Frobenius Theorem and on the second we prove the referred Theorem
\ref{Theorem:KoszulMalgrangeTheorem:ParametricVersion}.


\subsection{Parametric~Frobenius~theorem}\label{Subsection:ParametricFrobeniusTheorem}


The following result can be found in \cite{Machado:97} (Theorem IV.8.7) (see \cite{Dieudonne:60} for the non-parametric version):

\begin{theorem}[\textup{Frobenius~Theorem:~parametric~version}]\label{Theorem:FrobeniusTheoremParametricVersion}\index{Frobenius~Theorem!parametric~version}
Let $E$, $F$ and $G$ be three vector spaces. Let $F_0\,\subseteq F$ be an arbitrary set, $M\subseteq E$ a smooth manifold, $A$ open in $F_0\times G\times M$ and $X:A\to\L(G,E)$ a $\C^1$ map such that for each $(t,x,y),\in{A}$ we have:

\textup{(i)} the linear map $X(t,x,y)\in\L(G,E)$ maps $G$ to $T_yM$.

\textup{(ii)} the bilinear map
\begin{equation}\label{Equation:EquationIn:Theorem:FrobeniusTheoremParametricVersion}\addtolength{\arraycolsep}{-1.2mm}
\begin{array}{cll}
G \times G&\to& E\\
(w,\tilde{w})&\to& \d X_{t,x,y}(0,w,X(t,x,y)(w))(\tilde{w})
\end{array}
\end{equation}\addtolength{\arraycolsep}{1.2mm}\noindent
is symmetric.

For each $(t,x,y)\in{A}$, let $f_{t,x,y}:\V_{t,x,y}\to M$ the maximal solution of the total differential equation defined by $X_t:A_t\to\L(G,E)$ with the initial condition $(x,y)$. Let $\Omega\subseteq F_0\times G\times G\times M$ be the set of elements $(t,x_1,x_2,y)$ such that $(t,x_2,y)\in{A}$ and $x_1\in\V_{t,x_2,y}$ and take $\omega:\Omega\to M$ the parametric solution defined by
\[\omega(t,x_1,x_2,y)=f_{t,x_2,y}(x_1)\text{.}\]
Then:

\textup{(i)} $\Omega$ is open in $F_0\times G\times G\times M$ and $\omega:\Omega\to M$ is a $\C^1$ map.

\textup{(ii)} If $X:A\to\L(G,E)$ is $\C^p$ $(1\leq p\leq+\infty)$, then $\omega$ is also $\C^p$.
\end{theorem}

(i) Note that $F_0$ can be regarded as the space of parameters. Taking $F_0=\{0\}$ we get the non-parametric version
of this result. We can also conclude that the solution $f$ in this case is of class $\C^p$ if $X$ is $\C^p$.

(ii) When we claim that $X$ (or the solution $f$) is $\C^p$ on $A\subseteq F_0\times G\times M$, we mean, as is standard, that it admits a $\C^p$ extension to an open set of the vector space $F\times G\times M$.

We give some consequences of this result. Firstly, we obtain a parametric version of the Frobenius Theorem using differential forms:

\begin{theorem}[\textup{Frobenius~Theorem:parametric~version}]\label{Theorem:IntegrationOnLieGroups:ParametricVersion}
Let $\G$ be a connected (real) Lie group and $M$ a manifold. Let $F_0\,\subseteq F$ be any subset containing the origin of the real vector space $F$ and $A$ open in $F_0\times M$ containing $(0,x_0)$. Then, if $\alpha$ is a smooth $\g$-valued $1$-form on $M$ defined on $A\subseteq F_0\times M$ (\textit{i.e.}, $\alpha:F_0\times M\to \L(TM,\g)$, $\alpha(t,x)=\alpha_t(x):T_xM\to\g$ is smooth in $(t,x)$) with
\begin{equation}\label{Equation:FirstEquationIn:Theorem:IntegrationOnLieGroups:ParametricVersion}
\d\alpha_t+[\alpha_t,\alpha_t]=0,\quad\forall\,t\in{F}_0
\end{equation}
then there is a (unique, locally defined) $\G$-valued smooth solution of the equation
\begin{equation}
f(t,x)^{-1}\d f_{t_x}=\alpha_t(x),\quad\forall\,(t,x),\,f_t(x_0)=e\text{.}
\end{equation}
\end{theorem}
\begin{proof}As every Lie group is locally isomorphic to a linear group (\cite{Cohn:57}, \cite{Gilmore:74}) and we are looking for local solutions, we can assume that $\G\subseteq\GL(E)$. Take a local chart $(\openU,\eta)$ of $M$ around $x_0$ and write $\alpha(t,x)=\sum_i \alpha_i(t,x)\d\eta_i$, where $\alpha_i(t,x)$ are the smooth $\g$-valued functions determined by the condition $\d\alpha_t(x)(X)=\sum_i\alpha_i(t,x)\d\eta(x)(X)$. We are then looking for a function $f(t,x)$ defined around $(0,x_0)$ satisfying the equation (on $x$):
\[f^{-1}\d f_{t_x}(\partial_{\eta_i})=\alpha_i(t,x),\quad\forall\,t\text{.}\]
Composing with $\eta$ we can reduce to the case where
$M=\eta(\openU)\subseteq\rn^m$ and look for solutions
$\tilde{f}:\V\subseteq A\subseteq F\times \rn^m\to \G$ of the
equation
\[\tilde{f}^{-1}\frac{\partial \tilde{f}}{\partial x_i}=\tilde{\alpha}_i(t,x),\,f(t,0)=e\text{.}\]
Consider the map \addtolength{\arraycolsep}{-1.0mm}
\begin{equation}\label{Equation:FirstEquationInTheProofOf:Theorem:FrobeniusTheoremParametricVersion}
\begin{array}{lcccccll}
X:& F_0 &\times &\rn^m&\times&\G&\to& \L(\rn^m,\L(E,E))\\
~& (t&,&x&,&y)&\to& X(t,x,y)=y.\tilde{\alpha}^t(x):\rn^m\to\L(E,E)\\
~&~&~&~&~&~&~&\hspace{3.9cm} e_i\,\,\,\to y.\alpha^t(x)(e_i)\text{.}
\end{array}
\end{equation}\addtolength{\arraycolsep}{1.0mm}
Notice that $\tilde{\alpha}_t(x)$ is a $\g$-valued $1$-form on $\rn^m$ so that $\tilde{\alpha}_t(x):T_{x}\rn^m\to\g\subseteq\L(E,E)$. On the other hand, if $\tilde{f}_t$ is a solution of $X_t$ then $\tilde{f}:\V\to\G$ has $(x,\tilde{f}_t(x))\in{A}_t\subseteq\rn^m\times\G$ and $\d\tilde{f}_{t_x}(e_i)=X_t(x,\tilde{f}(x))(e_i)=\tilde{f}(x).\tilde{\alpha}_t(x)(e_i)$, which is precisely what we want. Therefore, we are left with proving that the conditions of Theorem \ref{Theorem:FrobeniusTheoremParametricVersion} are verified. For the first condition, we show that $X^t(x,y)$ maps $\rn^m$ to $T_y\G$. As $T_y\G=\d L_y (e)(\g)$ we then have $X^t(x,y)(e_i)=y.\tilde{\alpha}_t(x)(e_i)=\d{L}_y(e)(\tilde{\alpha}_t(x)(e_i))$; since $\tilde{\alpha}_t(x)(e_i)\in\g$ this part is proved. For the second condition, we now have\addtolength{\arraycolsep}{-1.0mm}
\[\begin{array}{ll}~&\d X^t_{(x,y)}(u,v)(e_i)=\d X^{t,y}_x(u)(e_i)+\d X^{t,x}_y(v)(e_i)\\
=&\d (x\to y.\tilde{\alpha}_t(x))_x(u)(e_i)+\d (y\to{y}.\tilde{\alpha}_t(x))_y(v)(e_i)\\
=&\d(x\to{y}.\tilde{\alpha}_t(x)(e_i))_x(u)+\d(y\to{y}.\tilde{\alpha}_t(x)(e_i))_y(v)=y.\d\tilde{\alpha}_{t_{i_x}}(u)+v.\tilde{\alpha}_t(x)(e_i)
\end{array}\]\addtolength{\arraycolsep}{1.0mm}\noindent
so that taking $u=e_j$ and
$v=X^t(x,y)(e_j)=y.\tilde{\alpha}_t(x)(e_j)$ we get
\[\d{X}^t_{(x,y)}(e_j,X^t(x,y)(e_j))(e_i)=y.\d\tilde{\alpha}^t_{i_x}(e_j)+y.\tilde{\alpha}^t(x)(e_j).\tilde{\alpha}^t(x)(e_i)\text{.}\]
Hence, condition (ii) of Theorem \ref{Theorem:FrobeniusTheoremParametricVersion} is satisfied if and only if
\[-y.\big(\d\tilde{\alpha}_{t_{i_x}}(e_j)+\tilde{\alpha}_t(x)(e_j).\tilde{\alpha}^t(x)(e_i)-\d\tilde{\alpha}_{t_{j_x}}(e_i)-\tilde{\alpha}_t(x)(e_i)\tilde{\alpha}_t(x)(e_j)\big)=0\text{.}\]
As $y\in\G$ we deduce that this is equivalent to
\[\partial_{x_i}(\tilde{\alpha}_t(x)(e_j))-\partial_{x_j}(\tilde{\alpha}_t(x)(e_i))+\tilde{\alpha}_t(x)(e_i)\tilde{\alpha}_t(x)(e_j)-\tilde{\alpha}_t(x)(e_j)\tilde{\alpha}_t(x)(e_i),\quad\forall\,i,j\]
which is precisely
\eqref{Equation:FirstEquationIn:Theorem:IntegrationOnLieGroups:ParametricVersion},
concluding our proof.
\end{proof}

We now replace the real (vector-valued) parameter $t$ by a complex one and consider the holomorphic dependence on this parameter $z$. Notice that we shall not be solving any kind of holomorphic equation: we still want a solution in real variables for each fixed $z$ but now holomorphically-dependent on $z$. More precisely:

\begin{proposition}\label{Proposition:HolomorphicDependentSolutionOnALieGroup}\index{Holomorphic!dependence~of~the~solution}
Suppose that $\G$ is a complex Lie group and $M$ a (real or complex) manifold. Let $\alpha:F_0\times M\to\L(TM,\g)$ be a $\g$-valued $1$-form on $M$ as in Theorem \ref{Theorem:IntegrationOnLieGroups:ParametricVersion}, holomorphic as a function on $z$ and smooth as a function on $(z,x)$. Around each point $x_0\in{M}$ consider the (unique) smooth solution $f(z,x)$ guaranteed by Theorem \ref{Theorem:IntegrationOnLieGroups:ParametricVersion} to the equation $f^{-1}\d{f}_z=\alpha_z$, $f_z(x_0)=e$. Then, this solution is holomorphic in $z$.
\end{proposition}
Notice that equation $\d f_z=\alpha^z$ is only a ``real" equation: if $\alpha_z$ is the $\g$-valued $1$-form, we write
$\alpha_z=\sum_{i=1}^{2m}\alpha_{i,z}\d\varphi_i$, where $(\openU,\varphi)$ is a chart of $M$ (or, when $M$ is complex, the real chart associated with a complex chart of $M$) and we solve the equations $\d{f}_z(\partial_{\varphi_i})=\alpha_{i,z}$, where $\alpha(z,x)$ are $\g$-valued maps, smooth in $(z,x)$ and holomorphic in $z$.
\begin{proof}As before, we assume that $\G$ is a linear group. We shall show that our solution is itself a solution of a differential equation \textit{on} $z$: as $f^{-1}\partial_{\varphi_i}f_z=\alpha_{i,z}$, we can apply $\partial_{\bar{z}}$ to this identity to obtain (using $\alpha$-holomorphicity)
\[-f^{-1}\partial_{\bar{z}}ff^{-1}\partial_{\varphi_i}f+f^{-1}\partial^{2}_{\bar{z}\varphi_i}f=0\,\impl\,\partial_{\bar{z}}f\alpha_i=\partial_{\varphi_i}\partial_{\bar{z}}f\text{.}\]
Let $\tilde{f}(z,x)=\partial_{\bar{z}}f$; then, $\tilde{f}_z$ is solution of the equation
\begin{equation}\label{Equation:FirstEquationInTheProofOf:Proposition:HolomorphicDependentSolutionOnALieGroup}
\tilde{f}_z\alpha_{i,z}=\partial_{\varphi_i}\tilde{f}_z.
\end{equation}
Notice that now $\tilde{f}$ is a map to $\L(E,E)$ and no longer to the Lie group $\G$. Moreover, as $f_{z}(x_0)=e$, we have $\tilde{f}_z(x_0)=0$ for all $z$; in particular, $\tilde{f}$ is a solution to the equation
\eqref{Equation:FirstEquationInTheProofOf:Proposition:HolomorphicDependentSolutionOnALieGroup} with initial condition $\tilde{f}^z(0)=0$. Then, obviously, $\tilde{f}_z\equiv 0$ is a solution to this equation and initial condition and if it is the unique solution we can deduce $\partial_{\bar{z}}f\equiv0$ which proves that $f$ is holomorphic in $z$. We can now use once again Theorem \ref{Theorem:FrobeniusTheoremParametricVersion} (the non-parametric version) on each $z$ to guarantee the uniqueness of solutions: take (in local coordinates) the map\addtolength{\arraycolsep}{-1.0mm}
\[\begin{array}{lcccll}
X^z: & \rn^m & \times & \L(E,E) & \to & \L(\rn^m,\L(E,E))\\
~& x & ,& \lambda & \to & \lambda.\alpha_z(x):\rn^m\to \L(E,E)\\
~&~&~&~&~& \hspace{1.8cm}e_i\,\,\to \lambda.\alpha_z(x)(e_i)\text{.}
\end{array}\]\addtolength{\arraycolsep}{1.0mm}\noindent
Then, a solution to the total differential equation defined by $X^z$ is a map $\tilde{f}:\V\subseteq\rn^m\to \L(E,E)$ with $\d\tilde{f}=X(x,\tilde{f})=\tilde{f}.\alpha_z(x)$ (as we want) and if there is a solution for a given initial condition it is unique, concluding our proof.
\end{proof}


\subsection{Parametric~Koszul-Malgrange~Theorem}\label{Subsection:ParametricKoszulMalgrangeTheorem}\index{Koszul-Malgrange!Theorem!parametric~version}


In this section we prove Theorem \ref{Theorem:KoszulMalgrangeTheorem:ParametricVersion}. This proof is a parametric version of that in \cite{KoszulMalgrange:58}:
\begin{lemma}\label{Lemma:FirstLemmaToProve:Theorem:KoszulMalgrangeTheorem:ParametricVersion:GeneralizationOfLemma2InKoszulMalgrangePaper}
Let $F$ be a (real or complex) vector space\footnote{That we shall regard as the set of parameters.} and $F_0\subseteq F$ any subset containing the origin. Let $\g$ be a complex vector space and let $\L(\g,\g)$ denote the space of linear endomorphisms of $\g$. Given a smooth\footnote{As in Theorem \ref{Theorem:FrobeniusTheoremParametricVersion}, smoothness in the sense that we have a smooth extension to an open set.} function $M(t_1,...,t_n,z_1,...,z_m)$ (respectively, $\psi(t_1,...,t_n,z_1,...,z_m)$) defined on an open set $\openU\subseteq F_0\times\cn^m$ containing $(0,0)$ with values on $\L(\g,\g)$ (respectively, on $\g$), holomorphic in $(z_1,...,z_p)$ ($0\leq p < m$), there is a smooth function $\varphi(t_1,..,t,t_n,z_1,...,z_m)$ defined on $\V\subseteq\openU$, $(0,0)\in\V$, with values in $\g$ and holomorphic in $(z_1,...,z_p)$ that satisfies the equation
\begin{equation}\label{Equation:EquationIn:Lemma:Lemma2OnKoszulMalgrange:TheParametricVersion}
\partial_{\bar{z}_{p+1}}\varphi=M\varphi+\psi\text{.}
\end{equation}
\end{lemma}
\begin{proof}If $F$ is complex, the result is immediate from the non-parametric version of this result
(\cite{KoszulMalgrange:58}, Lemma 2.). If $F$ is real, identify it with $\rn^n$ and let us we proceed in the natural way: from the functions $M(t_1,...,t_n,z_1,...,z_m)$ and $\psi(t_1,...,t_n,z_1,...,z_m)$, defined on open set in $\rn^n\times\cn^m$, define new functions $\tilde{M}(a_1,...,a_n,z_1,....,z_m)$ and $\tilde{\psi}(a_1,...,a_n,z_1,....,z_m)$ defined on an open set in $\cn^n\times\cn^m$ by $\tilde{M}(a_1,...,a_n,z_1,....,z_m)=M(\Real{a}_1,...,\Real{a}_n,z_1,...,z_m)$ and $\tilde{\psi}(a_1,...,a_n,z_1,....,z_m)=\psi(\Real a_1,...,\Real a_n,z_1,...,z_m)$; these are still smooth and holomorphic in $(z_1,...,z_p)$, for $0\leq p<m$. Hence, we have now a complex vector space of the form $\cn^{n+m}$ and we can use the non-parametric version of this lemma to deduce that there is a smooth solution $\tilde{\varphi}(a_1,...,a_n,z_1,...,z_m)$ defined on an open set around $(0,0)\in\cn^{n+m}$ to the equation $\partial_{\bar{z}_{p+1}}\tilde{\varphi}=\tilde{M}\tilde{\varphi}+\tilde{\psi}$, holomorphic in $(z_1,...,z_p)$. Take
$\varphi(t_1,...,t_n,z_1,...,z_m)=\tilde{\varphi}(t_1+i0,....,t_n+i0,z_1,...,z_m)$; this is smooth and holomorphic in $(z_1,..,z_p)$. Let us show that $\varphi$ satisfies equation \eqref{Equation:EquationIn:Lemma:Lemma2OnKoszulMalgrange:TheParametricVersion}. Indeed,\addtolength{\arraycolsep}{-1.0mm}
\[\begin{array}{ll}~&\partial_{\bar{z}_{p+1}}\varphi(t_1,...,t_n)=\partial_{\bar{z}_{p+1}}\tilde{\varphi}(t_1+i0,....,t_n+i0,z_1,...,z_m)\\
=&\tilde{M}(t_1+i0,....,t_n+i0,z_1,...,z_m)\tilde{\varphi}(t_1+i0,....,t_n+i0,z_1,...,z_m)\\
~&+\tilde{\psi}(t_1+i0,....,t_n+i0,z_1,...,z_m)\\
=&M(t_1,....,t_n,z_1,...,z_m)\varphi(t_1,...,t_n,z_1,...,z_m)+\psi(t_1+,....,t_n,z_1,...,z_m),
\end{array}\]\addtolength{\arraycolsep}{1.0mm}\noindent
concluding our proof.
\end{proof}

\begin{lemma}\label{Lemma:SecondLemmaToProve:Theorem:KoszulMalgrangeTheorem:ParametricVersion:GeneralizationOfImplicitLemmaInKoszulMalgrangePaper}
Let $\G\subseteq\L(E,E)$ be a Lie group and let $\g=T_e\G\subseteq\L(E,E)$ and take $F_0\subseteq F$ as in Lemma
\ref{Lemma:FirstLemmaToProve:Theorem:KoszulMalgrangeTheorem:ParametricVersion:GeneralizationOfLemma2InKoszulMalgrangePaper}.
Let $L(t,z)$ and $K(t,z)$ be $\g$-valued maps, smooth in $F_0\times\openU$ and holomorphic in
$z_1,...,z_p$ ($0\leq p<m$). Assume that
\begin{equation}\label{Equation:FirstEquationIn:Lemma:SecondLemmaToProve:Theorem:KoszulMalgrangeTheorem:ParametricVersion:GeneralizationOfImplicitLemmaInKoszulMalgrangePaper}
\frac{\partial K}{\partial \bar{z}_{p+1}}-\frac{\partial L}{\partial z_{p+1}}+[L,K]=0\text{.}
\end{equation}
Then, for fixed $(t,z_1,...,z_p,z_{p+2},...,z_{m})=(t,\hat{z}_{p+1})$, there is a (unique) function $h$ with values on $\G$ such that
\begin{equation}\label{Equation:SecondEquationIn:Lemma:SecondLemmaToProve:Theorem:KoszulMalgrangeTheorem:ParametricVersion:GeneralizationOfImplicitLemmaInKoszulMalgrangePaper}
h(0)=e,\,h^{-1}\frac{\partial h}{\partial z_{p+1}}=K\text{~and~}h^{-1}\frac{\partial h}{\partial\bar{z}_{p+1}}=L\text{.}
\end{equation}
Moreover, the resulting map $h(t,z_1,...,z_m)=h{t,\hat{z}_{p+1}}(z_{p+1})$ is smooth and holomorphic in the variables $z_1,...,z_p$.
\end{lemma}
\begin{proof} Consider $(t,z_1,...,z_p,z_{p+2},...,z_m)$ as the space of parameters. Then, we are solving the parametric equations given by equation \eqref{Equation:FirstEquationIn:Lemma:SecondLemmaToProve:Theorem:KoszulMalgrangeTheorem:ParametricVersion:GeneralizationOfImplicitLemmaInKoszulMalgrangePaper}.
Moreover, notice that we can rewrite these equations as
\[\left\{\begin{array}{l}
h^{-1}\big(\frac{\partial h}{\partial_{x_{p+1}}}-i\frac{\partial{h}}{\partial y_{p+1}}\big)=2K\\
h^{-1}\big(\frac{\partial h}{\partial_{x_{p+1}}}+i\frac{\partial{h}}{\partial y_{p+1}}\big)=2L\\
\end{array}\right.\,\equi\,\left\{\begin{array}{l}
h^{-1}\frac{\partial h}{\partial_{x_{p+1}}}=K+L\\
h^{-1}\frac{\partial{h}}{\partial_{y_{p+1}}}=i(K-L)\end{array}\right.\,\equi\,h^{-1}\d
h=\alpha\] where $\alpha=(K+L)\d x_{p+1}+i(K-L)\d y_{p+1}$,
$z_{p+1}=x_{p+1}+iy_{p+1}$. Using Theorem \ref{Theorem:IntegrationOnLieGroups:ParametricVersion} and
Proposition \ref{Proposition:HolomorphicDependentSolutionOnALieGroup}, we deduce the existence of $h$ satisfying our conditions if and only if $\d\alpha+[\alpha,\alpha]=0$. But we have\addtolength{\arraycolsep}{-1.0mm}
\[\begin{array}{ll}~&\d\alpha(\partial_{x_{p+1}},\partial_{y_p+1})+[\alpha(\partial_{x_{p+1}}),\alpha(\partial_{y_{p+1}})]\\
=&\partial_{x_{p+1}}(i(K-L))-\partial_{y_{p+1}}((K+L))+i(K+L)(K-L)-i(K-L)(K+L)\\
=&\frac{2i}{2}\big(\partial_{x_{p+1}}K+i\partial_{y_{p+1}}K-(\partial_{x_{p+1}}L-i\partial_{y_{p+1}}L)\big)+i(LK-KL+LK-KL)\\
=&2i\big(\frac{\partial{K}}{\partial\bar{z}_{p+1}}-\frac{\partial{L}}{\partial{z}_{p+1}}+[L,K]\big)\text{.}
\end{array}\]\addtolength{\arraycolsep}{1.0mm}\noindent
This last expression vanishes, from our hypothesis
\eqref{Equation:FirstEquationIn:Lemma:SecondLemmaToProve:Theorem:KoszulMalgrangeTheorem:ParametricVersion:GeneralizationOfImplicitLemmaInKoszulMalgrangePaper}, concluding the proof.
\end{proof}

For the non-parametric version of the following result, see \cite{KoszulMalgrange:58}, Lemma 1:

\begin{lemma}\label{Lemma:ThirdLemmaToProve:Theorem:KoszulMalgrangeTheorem:ParametricVersion:GeneralizationOfLemma1InKoszulMalgrangePaper}
Under the same notations as before, let $L(t,z)$ be smooth in $F_0\times \openU$ with values in $\g$. If $L$ is holomorphic in $z_1,...,z_p$ ($0\leq p< m$), there is a smooth function $h$ on an open set containing $(0,0)$ in $F_0\times \cn^m$, with values in $\G$ and with
\begin{equation}\label{Equation:EquationIn:Lemma:ThirdLemmaToProve:Theorem:KoszulMalgrangeTheorem:ParametricVersion:GeneralizationOfLemma1InKoszulMalgrangePaper}
h^{-1}\frac{\partial h^t}{\partial \bar{z}_{p+1}}=L\text{.}
\end{equation}
\end{lemma}
\begin{proof}We shall prove the latter lemma by means of the two preceding ones. Consider the equation on $K$ given by
\begin{equation}\label{Equation:FirstEquationInTheProofOf:Lemma:ThirdLemmaToProve:Theorem:KoszulMalgrangeTheorem:ParametricVersion:GeneralizationOfLemma1InKoszulMalgrangePaper}
\frac{\partial{K}}{\partial\bar{z}_{p+1}}=\frac{\partial{L}}{\partial{z}_{p+1}}-[L,K]=\frac{\partial{L}}{\partial{z}_{p+1}}-LK+KL\text{.}
\end{equation}
Using Lemma \ref{Lemma:FirstLemmaToProve:Theorem:KoszulMalgrangeTheorem:ParametricVersion:GeneralizationOfLemma2InKoszulMalgrangePaper},
where $\psi=\frac{\partial L}{\partial z_{p+1}}$ and $M$ is the function with values on $\L(\g,\g)$ defined by $M(\lambda)=\lambda L-L\lambda$, we deduce the existence of a smooth map $K$ satisfying
\eqref{Equation:FirstEquationInTheProofOf:Lemma:ThirdLemmaToProve:Theorem:KoszulMalgrangeTheorem:ParametricVersion:GeneralizationOfLemma1InKoszulMalgrangePaper}, holomorphic in $z_1,...,z_p$. Using now Lemma \ref{Lemma:SecondLemmaToProve:Theorem:KoszulMalgrangeTheorem:ParametricVersion:GeneralizationOfImplicitLemmaInKoszulMalgrangePaper}
we can guarantee that there is a smooth map $h$, holomorphic in $z_1,...,z_p$, satisfying equation \eqref{Equation:EquationIn:Lemma:ThirdLemmaToProve:Theorem:KoszulMalgrangeTheorem:ParametricVersion:GeneralizationOfLemma1InKoszulMalgrangePaper}.
\end{proof}

We are finally ready to prove Theorem \ref{Theorem:KoszulMalgrangeTheorem:ParametricVersion}; we shall show that Lemma
\ref{Lemma:ThirdLemmaToProve:Theorem:KoszulMalgrangeTheorem:ParametricVersion:GeneralizationOfLemma1InKoszulMalgrangePaper}
implies our result. Write $\alpha=\sum_j \alpha_j\d z_j$, where the $\alpha_j$ are smooth (but not
necessarily holomorphic in any variable) on the open set $A\subseteq{F}_0\times \cn^m$. Then, take $f_0\equiv e$ and
$L_0=a_1=f_0.\alpha_1.f_0^{-1}-\frac{\partial f_0}{\partial \bar{z}_1}.f_0^{-1}$. Using Lemma
\ref{Lemma:ThirdLemmaToProve:Theorem:KoszulMalgrangeTheorem:ParametricVersion:GeneralizationOfLemma1InKoszulMalgrangePaper}, solve the equation $h_0^{-1}\frac{\partial h^t_0}{\partial \bar{z}_{1}}=L_0$, which solution is
a smooth map, not necessarily holomorphic in any variable, as neither is $L_0$. Writing $f_1=h_0=h_0.f_0$, we have $f_1^{-1}\frac{\partial f_1}{\partial \bar{z}_1}=L_0=a_1$. Let $L_1=f_1.\alpha_2.f_1^{-1}-\frac{\partial{f}_1}{\partial\bar{z}_2}f^{-1}_1$ and let us show that $\frac{\partial{L}_1}{\partial\bar{z}_1}=0$. Indeed, using the fact that $\alpha$ satisfies
\eqref{Equation:SecondEquationIn:Theorem:KoszulMalgrangeTheorem:ParametricVersion}, and that $f^{-1}_1\frac{\partial{f}_1}{\partial\bar{z}_1}=a_1$, we have\addtolength{\arraycolsep}{-1.0mm}
\[\begin{array}{lll}
\frac{\partial{L}_1}{\partial\bar{z}_1}&=&\frac{\partial{f}_1}{\partial\bar{z}_1}\alpha_2f_1^{-1}+f_1\frac{\partial\alpha_2}{\partial\bar{z}_1}f_1^{-1}-f_1\alpha_2{f}^{-1}_1\frac{\partial{f}_1}{\partial\bar{z}_1}f_1^{-1}-\frac{\partial^2{f}_1}{\partial\bar{z}_1\partial\bar{z}_2}f_1^{-1}+\frac{\partial{f}_1}{\partial\bar{z}_2}f_1^{-1}\frac{\partial{f}_1}{\partial\bar{z}_1}f_1^{-1}\\
~&=&f_1\alpha_1\alpha_2{f}_1^{-1}+f_1\frac{\partial\alpha_2}{\partial\bar{z}_1}f_1^{-1}-f_1\alpha_2\alpha_1{f}_1^{-1}-(\frac{\partial{f}_1}{\partial\bar{z}_2}\alpha_1+f\frac{\partial\alpha_1}{\partial\bar{z}_2}){f}^{-1}+\frac{\partial{f}_1}{\partial\bar{z}_2}\alpha_1{f}_1^{-1}\\
~&=&f_1(\alpha_1\alpha_2-\alpha_2\alpha_1+\frac{\partial\alpha_2}{\partial\bar{z}_1}-\frac{\partial\alpha_1}{\partial\bar{z}_1})f_1^{-1}=0
\end{array}\]\addtolength{\arraycolsep}{1.0mm}\noindent
as required. Next, take $h_1$ a smooth solution of the equation $h_1^{-1}\frac{\partial h_1^t}{\partial\bar{z}_2}=L_1^t$; since $L_1$ is holomorphic in $z_1$, so is $h_1$ (Lemma
\ref{Lemma:ThirdLemmaToProve:Theorem:KoszulMalgrangeTheorem:ParametricVersion:GeneralizationOfLemma1InKoszulMalgrangePaper}).
Write $f_2=h_1.f_1$ (which is not necessarily holomorphic in any variable) and let us show that $f_2^{-1}\frac{\partial{f}_2}{\partial\bar{z}_j}=\alpha_j$, $j=1,2$. In fact,
\[\frac{\partial{f}_2}{\partial\bar{z}_j}=\frac{\partial{h}_1}{\partial\bar{z}_j}f_1+h_1\frac{\partial{f}_1}{\partial\bar{z}_j}\text{.}\]
Thus, if $j=1$, from the $z_1$-holomorphicity of $h_1$ and because $\frac{\partial f_1}{\partial\bar{z}_1}=f_1\alpha_1$, then we have $\frac{\partial f_2}{\partial\bar{z}_1}=h_1f_1\alpha_1=f_2\alpha_1$. As for $j=2$, since $\frac{\partial{h}_1}{\partial\bar{z}_2}=h_2 L_1$, we have \[\frac{\partial{f}_2}{\partial\bar{z}_2}=h_1L_1f_1+h_1\frac{\partial{f}_1}{\partial\bar{z}_2}=h_1\big(f_1.\alpha_2.f_1^{-1}-\frac{\partial{f}_1}{\partial\bar{z}_2}f^{-1}_1\big)f_1+h_1\frac{\partial{f}_1}{\partial\bar{z}_2}=f_2\alpha_2\text{.}\] Finally, writing
\[L_2=f_2\alpha_3f_2^{-1}-\frac{\partial f_2}{\partial\bar{z}_3}f_2^{-1}\]
we have that $L_2$ is holomorphic with respect to $z_1$ and $z_2$ since, for $j=1,2$,\addtolength{\arraycolsep}{-1.0mm}
\[\begin{array}{lll}\frac{\partial L_2}{\partial\bar{z}_j}&=&\frac{\partial f_2}{\partial \bar{z}_j}\alpha_3f_2^{-1}+f_2\frac{\partial\alpha_3}{\partial\bar{z}_j}f_2^{-1}-f_2\alpha_3f_2^{-1}\frac{\partial f_2}{\partial\bar{z}_j}f_2^{-1}-\frac{\partial^2 f_2}{\partial\bar{z}_j\partial\bar{z}_3}f_2^{-1}+\frac{\partial f_2}{\partial\bar{z}_3}f_2^{-1}\frac{\partial{f}_2}{\partial\bar{z}_j}f_2^{-1}\\
~&=&f_2\alpha_j\alpha_3f_2^{-1}+f_2\frac{\partial\alpha_3}{\partial \bar{z}_j}f_2^{-1}-f_2\alpha_3\alpha_jf_2^{-1}-\partial_{\bar{z}_3}\big(f_2\alpha_j\big)f_2^{-1}+\frac{\partial f_2}{\partial \bar{z}_3}\alpha_jf_2^{-1}\\
~&=&f_2\big(\alpha_j\alpha_3-\alpha_3\alpha_j+\frac{\partial\alpha_3}{\partial\bar{z}_j}-\frac{\partial\alpha_j}{\partial\bar{z}_3}\big)f_2^{-1}=0
\end{array}\]\addtolength{\arraycolsep}{1.0mm}\noindent
as $\alpha$ satisfied \eqref{Equation:SecondEquationIn:Theorem:KoszulMalgrangeTheorem:ParametricVersion}.
Hence, now solving equation $h_2^{-1}\frac{\partial h}{\partial\bar{z}_3}=L_2$ gives a smooth map which is holomorphic in $z_1$ and $z_2$. We continue by induction until we get the map $f_m$. This map is smooth in $(t,z)$ (although not necessarily holomorphic in any variable) and it satisfies $f_m^{-1}\frac{\partial{f}_m}{\partial\bar{z}_j}=\alpha_j$ for all $1\leq j\leq m$. So, $f_m$ is the desired map $f$ of Theorem \ref{Theorem:KoszulMalgrangeTheorem:ParametricVersion}, finishing our proof.



\addcontentsline{toc}{chapter}{\numberline{}Index}

\printindex

\end{document}